\documentclass[11pt]{book}
\usepackage{amssymb,amsmath,amsfonts,amsthm,epsfig,latexsym,color,makeidx}
\usepackage[francais,english]{babel}
\usepackage[T1]{fontenc}

\makeatletter
\newcommand{\sectionruninhead}{\@startsection{subsection}{2}{0mm}
{-\baselineskip}{-0mm}{\bf\large}}
\newcommand{\subsubsectionruninhead}{\@startsection{subsubsection}{3}{0mm}
{-\baselineskip}{-0mm}{\bf\normalsize}}
\makeatother

\makeatletter
\renewcommand{\l@subsection}{\@dottedtocline{3}{3.8em}{1.8em}}
\makeatother

\newtheorem*{proposition*}{Proposition}
\newtheorem{proposition}{Proposition}[chapter]
\newtheorem*{theoreme*}{Th\'eor\`eme}
\newtheorem{theoreme}[proposition]{Th\'eor\`eme}
\newtheorem*{corollaire*}{Corollaire}
\newtheorem{corollaire}[proposition]{Corollaire}
\newtheorem{lemme}[proposition]{Lemme}
\newtheorem{addendum}[proposition]{Addendum}
\newtheorem*{addendum*}{Addendum}
\newtheorem*{theorem-tolerance}{Conjecture de stabilit\'e molle}
\newtheorem*{theoreme-de-stabilite}{Th\'eor\`eme d'$\Omega$-stabilit\'e}
\newtheorem*{affirmation*}{Affirmation}
\newtheorem{affirmation}{Affirmation}
\newtheorem*{consequence*}{Cons\'equence}
\newtheorem{question}[proposition]{Question}
\newtheorem{question*}{Question}
\newtheorem*{probleme*}{Probl\`eme}

\newtheorem*{conjecture*}{Conjecture}
\newtheorem*{conjecture-finitude}{Conjecture de finitude}
\newtheorem*{conjecture-hyperbolicite}{Conjecture d'hyperbolicit\'e}
\theoremstyle{definition}
\newtheorem{definition}[proposition]{D\'efinition}
\theoremstyle{remark}
\newtheorem{remarque}[proposition]{Remarque}
\newtheorem{remarques}[proposition]{Remarques}
\newtheorem{exemples}[proposition]{Exemples}
\numberwithin{equation}{section}

   \def\DD{{\mathbb D}}

 \def\NN{{\mathbb N}}  
 \def\RR{{\mathbb R}}  \def\TT{{\mathbb T}}
   
 \def\ZZ{{\mathbb Z}}

  \def\cG{\mathcal{G}} \def\cM{\mathcal{M}} \def\cS{\mathcal{S}}
\def\cB{\mathcal{B}}    \def\cT{\mathcal{T}}
\def\cC{\mathcal{C}}  \def\cI{\mathcal{I}} \def\cO{\mathcal{O}} \def\cU{\mathcal{U}}
\def\cD{\mathcal{D}}   \def\cP{\mathcal{P}} \def\cV{\mathcal{V}}
  \def\cK{\mathcal{K}}  \def\cW{\mathcal{W}}
\def\cF{\mathcal{F}}   \def\cR{\mathcal{R}}

\newcommand{\Tang}{\cT}
\newcommand{\Cycl}{\cC}

\newcommand{\interior}{\operatorname{Int}}
\newcommand{\supp}{\operatorname{Supp}}
\newcommand{\diff}{\operatorname{Diff}}
\newcommand{\lip}{\operatorname{Lip}}
\newcommand{\id}{\operatorname{Id}}

\newcommand{\diam}{\operatorname{Diam}}
\newcommand{\per}{\operatorname{\cP er}}
\newcommand{\cper}{\operatorname{\mathbf{Per}}}
\newcommand{\ftrans}{\operatorname{\mathbf{fTrans}}}
\newcommand{\ctrans}{\operatorname{\mathbf{pTrans}}}
\newcommand{\cl}{\operatorname{\mathbf{Cl}}}
\newcommand{\corb}{\operatorname{\mathbf{Orb}}}
\newcommand{\orbf}{\operatorname{\mathbf{OrbF}}}
\newcommand{\eorb}{\operatorname{\mathbf{pOrb}}}
\newcommand{\segment}{\operatorname{\mathbf{Seg}}}
\newcommand{\psegment}{\operatorname{\mathbf{pSeg}}}
\newcommand{\card}{\operatorname{Card}}

\makeindex

\begin{document}
\selectlanguage{french}

\pagestyle{empty}

\begin{centering}

\mbox{}
\vspace{2cm}

\textbf{\huge Perturbation de la dynamique de diff\'eomorphismes en topologie $C^1$}

\vspace{2cm}

\textbf{\Large Sylvain \textsc{Crovisier}}

\vspace{2cm}

\textbf{\large D\'ecembre 2009}

\end{centering}

\chapter*{}
\mbox{}\vspace{1cm}

\section*{Perturbation de la dynamique de diff\'eomorphismes en topologie $C^1$}

Les travaux pr\'esent\'es dans ce m\'emoire portent sur la dynamique
de dif\-f\'e\-o\-mor\-phis\-mes de vari\'et\'es compactes.
Pour l'\'etude des propri\'et\'es g\'en\'eriques ou pour la construction
d'exemples, il est souvent utile de savoir perturber un syst\`eme.
Ceci soul\`eve g\'en\'eralement des probl\`emes d\'elicats~:
une modification locale de la dynamique peut engendrer un changement
brutal du comportement des orbites. En topologie $C^1$,
nous proposons diverses techniques permettant de perturber
tout en contr\^olant la dynamique~: mise en transversalit\'e, connexion d'orbites, perturbation de la dynamique tangente, r\'ealisation d'extensions topologiques,...
Nous en tirons diverses applications \`a la description de
la dynamique des diff\'eomorphismes $C^1$-g\'en\'eriques.
\vspace{2cm}

\selectlanguage{english}
\section*{Perturbation of the dynamics of diffeomorphisms in the $C^1$-topology}

This memoir deals with the dynamics of diffeomorphisms
of compact manifolds.
For the study of generic properties or for the construction of examples,
it is often useful to be able to perturb a system.
This generally leads to delicate problems:
a local modification of the dynamic may cause a radical change
in the behavior of the orbits. For the $C^1$ topology,
we propose various techniques which allow to perturb
while controlling the dynamic: setting in transversal position, connection of orbits, perturbation of the tangent dynamics,...
We derive various applications to the description of
$C^1$-generic diffeomorphisms.
\selectlanguage{francais}


\tableofcontents
\pagestyle{headings}

\setcounter{chapter}{-1}
{\renewcommand{\thechapter}{}
\renewcommand{\chaptername}{}
\chapter[\negthickspace \negthickspace \negthickspace Introduction]{Introduction}}
\setcounter{chapter}{0}

Le dynamicien \'etudie les syst\`emes qui \'evoluent au cours du temps.
De nombreux exemples sont fournis par la m\'ecanique classique \`a travers une \'equation diff\'erentielle.
Poincar\'e a montr\'e dans son m\'emoire~\cite{poincare1} sur la stabilit\'e du syst\`eme solaire que l'on peut d\'ecrire
le comportement asymptotique des trajectoires sans passer par leur calcul explicite (souvent vain).
Nous nous int\'eressons dans ce texte aux syst\`emes \`a temps discret obtenus en it\'erant un diff\'eomorphisme.
Ils sont naturellement reli\'es aux syst\`emes \`a temps continu si l'on consid\`ere une section de Poincar\'e ou le temps $1$ du flot associ\'e.
Les syst\`emes provenant de la physique poss\`edent bien souvent des sym\'etries suppl\'ementaires~:
nous n'abordons pas ici le cas sp\'ecifique des diff\'eomorphismes conservatifs, i.e. qui pr\'eservent une
forme volume ou symplectique (voir~\cite{panorama} pour un panorama des diff\'eomorphismes de surface conservatifs).

La classe des diff\'eomorphismes hyperboliques est l'une des premi\`eres
\'etudi\'ees et des mieux d\'ecrites.
Dans ce m\'emoire\footnote{\emph{Je suis tr\`es reconnaissant \`a J\'er\^ome Buzzi, Rafael Potrie et Charles Pugh
pour l'attention qu'ils ont port\'ee sur ce texte et leurs nombreuses remarques.}}, nous discuterons essentiellement des diff\'eomorphismes non hyperboliques g\'en\'eriques.
(Nous renvoyons \`a~\cite{bdv} pour une vue d'ensemble des dynamiques faiblement hyperboliques.)

\section{Dynamiques g\'en\'eriques}
Pla\c{c}ons-nous sur une vari\'et\'e diff\'erentiable compacte $M$.
L'\'etude de la dynamique d'un diff\'eomorphisme arbitraire de $M$ peut sembler parfois hors de port\'ee~:
on peut imaginer la coexistence et l'accumulation d'une grande vari\'et\'e de comportements diff\'erents,
emp\^echant l'espoir de d\'ecomposer et de structurer la dynamique.
Il se pourrait cependant que de tels diff\'eomorphismes, pathologiques, puissent toujours \^etre
approch\'es par d'autres, plus simples \`a analyser~: il devient plus raisonnable d'esp\'erer
\emph{d\'ecrire la dynamique d'une partie dense des syst\`emes}.
Ce point de vue soul\`eve une difficult\'e, que nous discuterons plus amplement par la suite~:
nous devons pr\'eciser l'espace des diff\'eomorphismes consid\'er\'es, ainsi que le choix d'une
topologie sur cet espace.

Nous travaillerons en g\'en\'eral avec des espace de diff\'eomorphismes
m\'etrisables complets et nous pourrons alors utiliser
la g\'en\'ericit\'e de Baire. Une propri\'et\'e sera dite \emph{g\'en\'erique} si elle est satisfaite
par un ensemble de diff\'eomorphismes contenant une intersection d\'enombrable d'ouverts denses
(un ensemble G$_\delta$ dense).
La r\'ealisation simultan\'ee de deux propri\'et\'es g\'en\'eriques est encore g\'en\'erique.
La g\'en\'ericit\'e de Baire est donc un outil commode pour d\'ecrire des ensembles denses de diff\'eomorphismes.

\paragraph{0.1.a \quad Choix d'un espace de diff\'eomorphismes.}
Les propri\'et\'es g\'en\'eriques que l'on peut obtenir d\'ependent beaucoup
de la r\'egularit\'e des diff\'eomorphismes que l'on \'etudie.
Notons $\diff^k(M)$ l'espace des diff\'eomorphismes de classe $C^k$ de $M$.
Nous allons discuter l'existence de points p\'eriodiques
de diff\'eomorphismes g\'en\'eriques pour diff\'erents espaces.
\smallskip

\begin{itemize}
\item[--] Dans $\diff^0(M)$, les hom\'eomorphismes poss\`edent g\'en\'eriquement un ensemble
non d\'e\-nom\-bra\-ble de points p\'eriodiques (voir~\cite{ppss,hurley-livre}).
\smallskip

\item[--] Dans $\diff^1(M)$, l'existence de points p\'eriodiques est plus difficile \`a obtenir~:
c'est l'objet du lemme de fermeture de Pugh. L'ensemble des points ayant une m\^eme p\'eriode est fini.
\smallskip

\item[--] Dans $\diff^k(M)$, $k>1$, l'existence de points p\'eriodique n'est en g\'en\'eral pas connue.
\smallskip

\item[--] Herman a montr\'e~\cite{herman-closing1,herman-closing2} que sur certains tores symplectiques, pour
des ensembles ouverts de fonctions hamiltoniennes $C^\infty$ il existe des ouverts de surfaces
d'\'energies r\'eguli\`eres sur lesquelles il n'y a pas d'orbites p\'eriodiques.
Les dynamiques hamiltoniennes g\'en\'eriques sur ces surfaces d'\'energies n'ont donc pas d'orbite p\'eriodique.
\smallskip
\end{itemize}

\emph{Nous travaillerons par la suite en r\'egularit\'e $C^1$.} Il y a pour cela plusieurs raisons.
L'existence d'une structure diff\'erentiable donne une rigidit\'e minimale, permettant par exemple
la continuation des points p\'eriodiques, l'\'etude des exposants de la dynamique, l'obtention
de sous-vari\'et\'es invariantes.
Comme nous l'avons vu, la perturbation de diff\'eomorphismes en topologie $C^1$ est
assez souple pour cr\'eer des points p\'eriodiques.
La topologie $C^1$ poss\`ede \'egalement une propri\'et\'e importante~:
si $f$ est un diff\'eomorphisme de $\RR^d$ \`a support compact,
tous les conjugu\'es $x\mapsto a\; f(a^{-1}\;x)$ par des homoth\'eties de rapport $a\in (0,1)$ petit
sont \`a la m\^eme distance \`a l'identit\'e.
Cette remarque est \`a la base des principales techniques de perturbation dans $\diff^1(M)$.
En revanche, la r\'egularit\'e $C^1$ n'est pas suffisante pour certains r\'esultats
requ\'erant un contr\^ole de distorsion.

\paragraph{0.1.b \quad Dynamiques hyperboliques.}
Tr\`es rapidement, les dynamiciens ont cherch\'e quelles \'etaient les dynamiques structurellement
stables, i.e. dont le comportement topologique ne change pas apr\`es perturbation du syst\`eme.
Anosov et Smale ont introduit une classe de diff\'eomorphismes ayant cette propri\'et\'e~:
ce sont les diff\'eomorphismes hyperboliques.
Un diff\'eomorphisme est hyperbolique si, au-dessus de tout ensemble invariant
suffisamment r\'ecurrent, on peut d\'ecomposer l'espace
tangent de $M$ en deux fibr\'es lin\'eaires invariants, le premier (le fibr\'e stable) \'etant uniform\'ement
contract\'e, le second (le fibr\'e instable) uniform\'ement dilat\'e.
Paradoxalement, la stabilit\'e de ces syst\`emes face aux perturbations ne les emp\^eche
pas d'avoir une dynamique parfois ``chaotiques''. La dynamique de ces syst\`emes est tr\`es bien comprise,
notamment du point de vue symbolique (existence de partitions de Markov, de codages) et du point de vue
de la th\'eorie ergodique (existence de mesures physiques).

L'\'equivalence entre la stabilit\'e et l'hyperbolicit\'e a fait l'objet de nombreux
travaux et a finalement \'et\'e obtenue par Ma\~n\'e.
Ces recherches ont fait \'emerger de nombreux outils
int\'eressants pour l'\'etude des dynamiques diff\'erentiables.
Il est apparu cependant que l'ouvert des dynamiques hyperboliques n'est en g\'en\'eral
pas dense dans l'espace des diff\'eomorphismes.
Des objectifs naturels guident alors l'\'etude des dynamiques g\'en\'eriques~:
\smallskip

\begin{itemize}
\item[--] la compr\'ehension des d\'efauts d'hyperbolicit\'e,
\smallskip

\item[--] la g\'en\'eralisations de propri\'et\'es satisfaites par les syst\`emes hyperboliques
\`a des classes plus larges de dynamiques,
\smallskip

\item[--] la recherche de nouveaux ph\'enom\`enes dynamiques.

\end{itemize}
\medskip

Parmi les dynamiques non hyperboliques, deux obstructions apparaissent fr\'equemment~:
les tangences homoclines et les cycles h\'et\'erodimensionnels (voir la figure~\ref{f.bifurcation}).
\begin{figure}[ht]
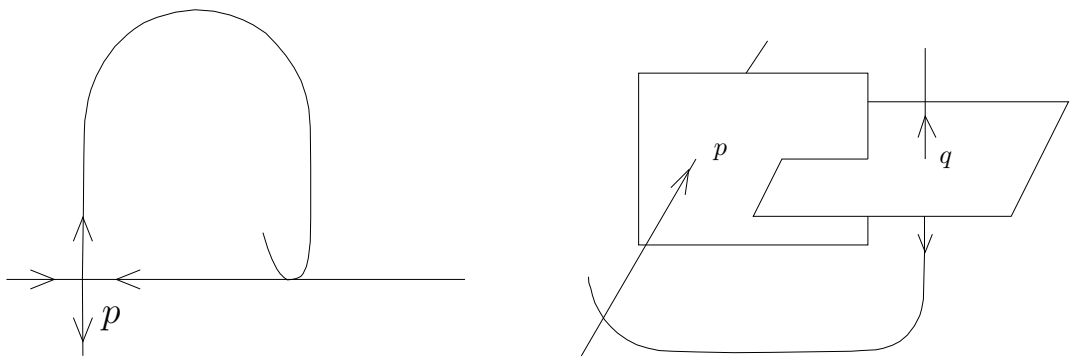

\begin{center}
\mbox{\input{tangence.pstex_t}\quad\quad\quad\quad
\input{cycle.pstex_t}}
\end{center}
\caption{Tangence homocline et cycle h\'et\'erodimensionnel. \label{f.bifurcation}}
\end{figure}
\smallskip

\begin{itemize}
\item[--] Une orbite p\'eriodique hyperbolique poss\`ede une \emph{tangence homocline}
si ses vari\'et\'es stable et instable poss\`edent un point d'intersection non transverse.
Certains vecteurs sont alors uniform\'ement contract\'es par it\'erations pass\'ees et
futures et devraient appartenir simultan\'ement aux fibr\'es stable et instable.
\smallskip

\item[--] Deux orbites p\'eriodiques hyperboliques forment
un \emph{cycle h\'et\'erodimensionnel} si elles sont li\'ees par des orbites h\'et\'eroclines
et si leurs dimensions stables diff\`erent.
Il n'existe donc pas de fibr\'e stable de dimension constante au-dessus du cycle.
\end{itemize}

Pour les diff\'eomorphismes de surface de r\'egularit\'e $C^k$, $k>1$, Newhouse a montr\'e
que l'ensemble des diff\'eomorphismes ayant une tangence homocline est dense dans un ouvert (non vide)
de l'espace des diff\'eomorphismes.
En dimension plus grande (et en toute r\'egularit\'e $C^k$, $k\geq 1$), Abraham et Smale
ont construit un ouvert dans lequel les dynamiques pr\'esentant un cycle h\'et\'erodimensionnel
sont denses. Le cas des diff\'eomorphismes de classe $C^1$ des surfaces n'est pas connu.

\begin{conjecture*}[Smale]\index{conjecture! de Smale}
Pour toute surface compacte $S$, l'ensemble des diff\'eomorphismes hyperboliques est dense dans
$\diff^1(S)$.
\end{conjecture*}

Palis a conjectur\'e que ces bifurcations sont les seules obstructions \`a l'hyperbolicit\'e.

\begin{conjecture*}[Palis]\index{conjecture! de Palis}
Dans $\diff^k(M)$, $k\geq 1$, tout diff\'eomorphisme peut \^etre approch\'e par
une dynamique hyperbolique ou par une dynamique poss\'edant une tangence homocline ou un cycle h\'et\'erodimensionnel.
\end{conjecture*}

\paragraph{0.1.c \quad Les dynamiques g\'en\'eriques repr\'esentent-elles tous les syst\`emes~?}
La g\'e\-n\'e\-ri\-ci\-t\'e donne un sens \`a la notion de partie n\'egligeable de l'ensemble
des diff\'eomorphismes~: nous cherchons \`a d\'ecrire la plupart
des dynamiques. Ce point de vue est cependant tr\`es relatif, puisqu'il se peut que
l'espace des diff\'eomorphismes \'etudi\'es soit n\'egligeable pour une autre notion de
g\'en\'ericit\'e.

Il est ainsi envisageable que l'on puisse trouver deux parties disjointes, chacune dense
dans un ouvert de l'espace des diff\'eomorphismes et repr\'esentant des dynamiques tr\`es diff\'erentes.
Par exemple, lorsque $\dim(M)\leq 3$, on peut alors introduire~:
\begin{itemize}
\item[--] l'ensemble des hom\'eomorphismes ayant un ensemble non d\'enombrable de points p\'e\-ri\-o\-di\-ques
(nous avons vu que c'est une partie g\'en\'erique de $\diff^0(M)$),
\item[--] l'ensemble des diff\'eomorphismes structurellement stables de $M$ (Shub a montr\'e~\cite{shub-topologique} qu'il forment une partie $C^0$-dense de $\diff^\infty(M)$ et
lorsque $\dim(M)\leq 3$,
Munkres a montr\'e~\cite{munkres} que $\diff^\infty(M)$ est dense dans $\diff^0(M)$).
\end{itemize}
Les diff\'eomorphismes structurellement stables ont au plus un nombre d\'enombrable de points p\'eriodiques
et ces deux parties denses de $\diff^0(M)$ sont donc disjointes.

Remarquons que, dans certains cas, en travaillant avec des propri\'et\'es g\'en\'eriques, on peut obtenir
des conclusions satisfaites par un ouvert dense de diff\'eomorphismes, et non pas seulement par un
G$_\delta$.
(Voir par exemple la conjecture faible de Palis, plus bas.)

\section{D\'ecomposition de la dynamique}
La dynamiques d'un diff\'eomorphisme hyperbolique $f$ est port\'ee par
un nombre fini d'ensembles compacts invariants disjoints et transitifs
(i.e. poss\'edant une orbite positive dense), appel\'es pi\`eces basiques.
Plus pr\'ecis\'ement, dans ce cadre chaque pi\`ece basique contient une orbite p\'eriodique
$O$ et co\"\i ncide avec l'adh\'erence $H(O)$ de l'ensemble des intersections
transverses entre les vari\'et\'es stables et instables de $O$.
Un tel ensemble $H(O)$ associ\'e \`a une orbite p\'eriodique hyperbolique $O$
peut \^etre d\'efini pour un diff\'eomorphisme 	arbitraire et est appel\'e \emph{classe homocline} de $O$.

Conley a montr\'e que cette d\'ecomposition se g\'en\'eralise
aux diff\'eomorphismes quelconques, si l'on autorise un nombre infini de pi\`eces
et si l'on remplace la transitivit\'e par la transitivit\'e par cha\^\i nes.
Cette propri\'et\'e est une notion plus faible de r\'ecurrence qui met en jeux les
pseudo-orbites, i.e. les suites $(x_n)$ de $M$ pour lesquelles
$f(x_n)$ et $x_{n+1}$ sont proches pour tout $n$ mais ne co\"\i ncident pas n\'ecessairement.
Les pi\`eces obtenues sont alors appel\'ee \emph{classes de r\'ecurrence par cha\^\i nes}.
\medskip

La d\'ecomposition de Conley est int\'eressante si l'on parvient \`a faire le lien entre
les pseudo-orbites et les vraies orbites de la dynamique.
Il faut pour cela \^etre capable de refermer les sauts d'une pseudo-orbite par perturbation.
Ce probl\`eme a \'et\'e trait\'e progressivement par l'obtention
de lemmes de perturbation dans des situations de plus en plus g\'en\'erales.
\smallskip

\begin{itemize}
\item[--] Pugh a d\'emontr\'e un lemme de fermeture qui permet de cr\'eer un point p\'eriodique
pr\`es d'un point non-errant. Il s'agit en quelque sorte de fermer une pseudo-orbite p\'eriodique
ayant un saut unique.
\smallskip

\item[--] Hayashi a obtenu un lemme de connexion~: lorsque l'orbite positive d'un point $p$
et l'orbite n\'egative d'un point $q$ ont un point d'accumulation commun, $p$ et $q$
sont contenus dans la m\^eme
orbite apr\`es perturbation. L\`a encore, il n'y a qu'un seul saut \`a refermer.
\smallskip

\item[--] Finalement, avec C. Bonatti nous avons abord\'e le cas g\'en\'eral et
\'etabli un lemme de connexion pour les pseudo-orbites.
\end{itemize}

Le lemme de connexion pour les pseudo-orbites permet de montrer que la d\'ecomposition en classes de r\'ecurrence par cha\^\i nes est bien adapt\'ee \`a l'\'etude des diff\'eomorphismes g\'en\'eriques~:
par exemple, les classes de r\'ecurrence par cha\^\i nes
contenant une orbite p\'eriodique sont des ensembles transitifs.
Plus pr\'ecis\'ement, nous avons obtenu le r\'esultat suivant.

\begin{theoreme*}
Il existe un G$_\delta$ dense de $\diff^1(M)$ form\'e de diff\'eomorphismes
pour lesquels la classe de r\'ecurrence par cha\^\i nes de toute orbite p\'eriodique
$O$ co\"\i ncide avec sa classe homocline $H(O)$.
\end{theoreme*}

\medskip

Toutefois le lemme de connexion pour les pseudo-orbites ne permet pas de contr\^oler
com\-pl\`e\-te\-ment le support de l'orbite obtenue. J'ai \'etabli pour cela un autre lemme
de perturbation. En voici une cons\'equence.

\begin{theoreme*}
Pour tout diff\'eomorphisme appartenant \`a un G$_\delta$ dense de $\diff^1(M)$,
chaque classe de r\'ecurrence par cha\^\i nes est limite de Hausdorff d'une suite d'orbites p\'eriodiques.
\end{theoreme*}
\medskip

Ces techniques de perturbation restent valables pour les diff\'eomorphismes
conservatifs. Elles entra\^\i nent un r\'esultat de ``diffusion des orbites''
(obtenu avec C. Bonatti et M.-C. Arnaud).
\begin{theoreme*}
Consid\'erons une vari\'et\'e compacte connexe $M$ munie d'une forme volume ou symplectique $\omega$.
Tout diff\'eomorphisme appartenant \`a une partie g\'en\'erique de l'espace $\diff^1_\omega(M)$ des diff\'eomorphismes de $M$ pr\'eservant $\omega$ est transitif.
\end{theoreme*}
\medskip

Ce dernier th\'eor\`eme souligne \`a nouveau l'importance de la r\'egularit\'e choisie.
En effet, le th\'eor\`eme des tores invariants
de Herman (voir~\cite[chapitre IV]{herman}, \cite[section II.4.c]{moser} et~\cite{yoccoz-tore})
montre par des techniques de th\'eorie KAM qu'il existe un ouvert de diff\'eomorphismes
de $\diff^k_\omega(M)$, $k\geq 4$, qui poss\`edent un tore de codimension $1$. Ce tore s\'epare la vari\'et\'e
en deux composante et emp\^eche la transitivit\'e. Le th\'eor\`eme pr\'ec\'edent n'est donc
pas vrai en r\'egularit\'e sup\'erieure.
\medskip

Nous pouvons \'egalement remarquer certaines particularit\'es du sujet.
D'une part, l'essentiel des r\'esultats s'appuient sur des techniques de perturbation.
D'autre part, les diff\'erents lemmes de perturbation ont parfois des \'enonc\'es
qui peuvent sembler tr\`es similaires, bien que la difficult\'e des d\'emonstrations puisse
\^etre bien diff\'erente.
\medskip

Notre connaissance des dynamiques $C^1$-g\'en\'erique devient suffisamment pr\'ecise pour
pouvoir obtenir des conclusions qui n'\'etaient jusqu'alors connues que dans le cadre
des dynamiques hyperboliques. Par exemple, on s'attend \`a ce que le \emph{centralisateur}\index{centralisateur}
d'un diff\'eomorphisme, i.e. l'ensemble des diff\'eomorphismes qui commutent avec lui,
soit en g\'en\'eral petit.
C'est ce que nous avons v\'erifi\'e avec C. Bonatti et A. Wilkinson.

\begin{theoreme*}
L'ensemble des diff\'eomorphismes $f$ \`a centralisateur trivial
(i.e. r\'eduit aux it\'er\'es $f^n$, $n\in \ZZ$) contient un G$_\delta$ dense de
$\diff^1(M)$.
\end{theoreme*}

\section*{Caract\'erisation des dynamiques non hyperboliques}
En dimension $1$ et en toute r\'egularit\'e, la conjecture de Palis est une cons\'equence
du th\'eor\`eme de Peixoto~: tout diff\'eomorphisme du cercle peut \^etre approch\'e
par un diff\'eomorphisme ayant un nombre fini de points p\'eriodiques, tous hyperboliques
(les autres orbites sont errantes et connectent une source \`a un puits).
En dimension $2$, Pujals et Sambarino ont d\'emontr\'e la conjecture en r\'egularit\'e $C^1$.
En dimension sup\'erieure, seuls des r\'esultats partiels existent (et seulement en topologie $C^1$).
Nous allons les pr\'esenter ci-dessous.

\paragraph{0.3.a \quad Les dynamiques conservatives}
En topologie $C^1$ et dans le cadre conservatif, la conjecture de Palis est v\'erifi\'ee,
nous obtenons m\^eme un r\'esultat plus fort.

En dimension $2$, il existe une obstruction robuste \`a l'hyperbolicit\'e venant de la structure
symplectique~: l'existence d'un point p\'eriodique elliptique, i.e. ayant des valeurs propres non r\'eelles
de module $1$. Newhouse a montr\'e~\cite{newhouse-elliptique} que cette obstruction suffit a caract\'eriser
les diff\'eomorphismes robustement non hyperboliques

Lorsque $\dim(M)\geq 3$, les points p\'eriodiques sont g\'en\'eriquement hyperboliques \cite{robinson-periodique}.
Les travaux, classiques, autour de la stabilit\'e structurelle,
impliquent qu'une dynamique dont on ne peut pas faire bifurquer les points p\'eriodiques
est hyperbolique.
Si l'on parvient au contraire \`a faire bifurquer des points p\'eriodiques,
on obtient par perturbation des points p\'eriodiques ayant des dimensions
stables diff\'erentes; le lemme de connexion pour les pseudo-orbites cr\'ee alors
un cycle h\'et\'erodimensionnel
(dans le cadre conservatif et lorsque $M$ est connexe, il n'y a qu'un seule classe de r\'ecurrence par cha\^\i nes) associ\'e \`a des points p\'eriodiques dont la dimension stable diff\`ere de $1$.

Un r\'esultat de Bonatti et D\'iaz permet de renforcer les cycles h\'et\'erodimensionnels~:
apr\`es perturbation, on obtient un diff\'eomorphisme $f\in \diff^1_\omega(M)$
et une paire d'ensemble hyperboliques transitifs $K,L$ dont la dimension stable diff\`ere,
tels que pour tout diff\'eomorphisme $g$ proche de $f$,
les continuations $K_g,L_g$ sont li\'ees par des orbites
h\'et\'eroclines. Un tel cycle associ\'e \`a $K$ et $L$
est une obstruction robuste \`a l'hyperbolicit\'e et est appel\'e
\emph{cycle h\'et\'erodimensionnel robuste}.
Si l'on choisit des orbites p\'eriodiques $O\in K$ et $O'\in L$, il existe un ensemble
dense de diff\'eomorphismes au voisinage de $f$ pour lesquels $O$ et $O'$ forment un cycle h\'et\'erodimensionnel.

L'\'enonc\'e obtenu porte sur un \emph{ouvert dense}
de l'espace des diff\'eomorphismes conservatifs.
\begin{theoreme*}
Tout diff\'eomorphisme conservatif peut \^etre approch\'e dans $\diff^1_\omega(M)$
par un dif\-f\'e\-o\-mor\-phis\-me hyperbolique ou par un diff\'eomorphisme ayant un cycle h\'et\'erodimensionnel
robuste.
\end{theoreme*}

\paragraph{0.3.b \quad La conjecture faible de Palis.}\index{conjecture! faible de Palis}
On conna\^\i t depuis Poincar\'e~\cite{poincare1,poincare2} et Birkhoff~\cite{birkhoff}
l'importance du r\^ole jou\'e par les \emph{intersections homoclines transverses}, i.e.
par les points d'in\-ter\-sec\-tion
transverses entre les vari\'et\'es stables et instables d'une orbite p\'eriodique.
Cette propri\'et\'e - tr\`es simple - est \`a l'origine d'une dynamique complexe.

Smale a construit~\cite{smale-horseshoe} un mod\`ele g\'eom\'etrique, son c\'el\`ebre fer \`a cheval, qui d\'ecrit la dynamique
pr\`es d'une telle orbite. On trouve des ensembles hyperboliques dont une puissance est conjugu\'ee
au d\'ecalage sur l'ensemble de Cantor $\{0,1\}^\ZZ$. Il y a alors une infinit\'e d'orbite p\'eriodiques
et l'entropie topologique de la dynamique est strictement positive.
Remarquons \'egalement que l'existence d'une intersection homocline transverse est une propri\'et\'e ouverte.
\medskip

\`A l'oppos\'e, Smale a introduit des dynamiques tr\`es simples~: les \emph{diff\'eomorphismes
Morse-Smale}. Ce sont des diff\'eomorphismes hyperboliques pour lesquels chaque pi\`ece basique est
r\'eduite \`a une orbite p\'eriodique. Les autres orbites s'accumulent dans le pass\'e et dans le futur sur
deux orbites p\'eriodiques distinctes. En particulier, ces dynamiques sont stables (elles forment donc un ouvert) et leur entropie topologique est nulle.
\medskip

Palis a conjectur\'e (en toute classe de diff\'erentiabilit\'e $C^k$, $k\geq 1$)
que ces deux comportements forment une partie dense de l'espace des diff\'eomorphismes.
Cette seconde conjecture de Palis est impliqu\'ee par la premi\`ere et est donc v\'erifi\'ee en
dimension $2$ d'apr\`es les r\'esultats de Pujals et Sambarino.
Elle a ensuite \'et\'e d\'emontr\'ee en dimension $3$ par Bonatti, Gan et Wen.
J'ai obtenu une d\'emonstration dans le cas g\'en\'eral.

\begin{theoreme*}
Tout diff\'eomorphisme peut \^etre approch\'e dans $\diff^1(M)$ par un diff\'eomorphisme
Morse-Smale ou par un diff\'eomorphisme ayant une intersection homocline transverse.
\end{theoreme*}

Bien que la preuve fasse appel aux techniques de g\'en\'ericit\'e,
la conclusion d\'ecrit \`a nouveau une partie ouverte et dense de l'espace des diff\'eomorphismes.
Sur cet ouvert, il y a donc une dichotomie parfaite entre les dynamiques ``simples'' et ``compliqu\'ees''.

La d\'emonstration se ram\`ene \`a l'\'etude d'ensembles dont le d\'efaut d'hyperbolicit\'e
est concentr\'e sur un fibr\'e lin\'eaire invariant de dimension $1$.
Aucun vecteur de ce fibr\'e n'est contract\'e ni dilat\'e par les it\'er\'es de l'application
tangente. Les techniques habituelles de dynamiques diff\'erentiables reli\'ees \`a l'existence
d'exposants de Lyapunov non nuls ne s'appliquent donc pas et nous les avons remplac\'ees par l'\'etude
d'objets de nature plus topologique~: les \emph{mod\`eles centraux}.
\medskip

Dans le cas conservatif, on s'attendrait \`a ce que pour un ouvert dense de diff\'eomorphismes,
il existe une orbite p\'eriodique hyperbolique ayant une intersection homocline transverse.
En topologie $C^1$,
c'est une cons\'equence du th\'eor\`eme \'etabli avec Arnaud et Bonatti \'enonc\'e ci-dessus.
En r\'egularit\'e plus grande, des r\'esultats partiels existent sur les surfaces~:
Zehnder a montr\'e~\cite{zehnder} que tout point elliptique est approch\'e par un point p\'eriodique
hyperbolique ayant une intersection homocline transverse.
Plus g\'en\'eralement, les diff\'eomorphismes ayant une intersection homocline transverse sont denses
dans presque toute classe d'homotopie de diff\'eomorphisme conservatif $C^r$, $r\geq 1$,
voir~\cite{robinson-homocline, pixton, oliveira1,oliveira2}.

\paragraph{0.3.c \quad Hyperbolicit\'e partielle.}
Une approche possible pour montrer la conjecture de Palis consiste \`a
\'etudier les diff\'eomorphismes g\'en\'eriques qui ne sont pas approch\'es
par des tangence homoclines ou par des cycles h\'et\'erodimensionnels.
Il s'agit alors de montrer que chaque classe de r\'ecurrence par cha\^\i nes est
un ensemble hyperbolique.

Nous avons obtenu un r\'esultat dans cette direction~:
chaque classe de r\'ecurrence par cha\^\i nes $\Lambda$
v\'erifie une propri\'et\'e d'hyperbolicit\'e partielle.
Plus pr\'ecis\'ement, il existe une d\'ecomposition $T_\Lambda M=E^s\oplus E^c\oplus E^u$
du fibr\'e tangent unitaire. $E^s$ et $E^u$ sont uniform\'ement contract\'es
et dilat\'es. La partie centrale $E^c$
est moins contract\'ee ou dilat\'ee que les fibr\'es extr\^emes $E^s,E^u$.
De plus, sa dimension est au plus \'egale \`a $2$
(lorsqu'elle est de dimension $2$, elle se d\'ecompose \`a nouveau en deux fibr\'es
$E^c=E^c_1\oplus E^c_2$ de dimension $1$).
Un tel diff\'eomorphisme est appel\'e dans ce cadre \emph{diff\'eomorphisme partiellement hyperbolique}.

J'ai d\'emontr\'e que les diff\'eomorphismes loin des bifurcations homoclines peuvent \^etre approch\'es par des diff\'eomorphismes partiellement hyperboliques.
Wen avait obtenu pr\'ec\'edemment une version locale de ce r\'esultat
en s'appuyant sur le lemme de s\'election de Liao.
\begin{theoreme*}
Tout diff\'eomorphisme peut \^etre approch\'e dans $\diff^1(M)$ par un diff\'eomorphisme
partiellement hyperbolique ou par un diff\'eomorphisme ayant une tangence homocline
ou un cycle h\'et\'erodimensionnel.
\end{theoreme*}
\medskip

Pour achever la d\'emonstration de la conjecture de Palis,
il faudrait montrer que pour ces diff\'eomorphismes partiellement hyperboliques,
les fibr\'es centraux de dimension $1$ sont uniform\'ement contract\'es ou dilat\'es.
C'est ce qu'ont r\'eussi Pujals et Sambarino dans le cas des diff\'eomorphismes
de surface. Dans le cas de fibr\'es centraux extr\^emes (i.e. lorsque $E^s$ ou $E^u$
est d\'eg\'en\'er\'e) leur argument se g\'en\'eralise (le cas o\`u $E^c$ est de dimension $1$
a \'et\'e trait\'e par Pujals et Sambarino~; j'ai trait\'e le cas $\dim(E^c)=2$ avec E. Pujals).

\begin{addendum*}
Les diff\'eomorphismes partiellement hyperboliques du th\'eor\`eme pr\'ec\'edent
peuvent \^etre choisis pour satisfaire la propri\'et\'e additionnelle suivante~:
pour chaque classe de r\'ecurrence par cha\^\i nes qui n'est pas un puits ou une source,
les fibr\'es extr\^emes $E^s, E^u$ sont non d\'eg\'en\'er\'es.
\end{addendum*}

\paragraph{0.3.d \quad Hyperbolicit\'e essentielle.}
Il peut \^etre plus facile d'\'etudier l'hyperbolicit\'e des parties attractives de la
dynamique. En dimension $3$,
Pujals a \'etabli sous certaines hypoth\`eses que tout diff\'eomorphisme
ayant un attracteur non hyperbolique peut \^etre approch\'e par un diff\'eomorphisme ayant
une tangence homocline ou un cycle h\'et\'erodimensionnel.

Un diff\'eomorphisme $f$ est \emph{essentiellement hyperbolique} si $f$ et $f^{-1}$
poss\`edent un nombre fini d'attracteurs hyperboliques dont l'union des bassins
est dense dans $M$.
De fa\c{c}on \'equivalente, en restriction \`a un ouvert dense de $M$ on ne distingue
pas la dynamique de $f$ de celle d'un diff\'eomorphisme hyperbolique.

Le r\'esultat suivant a \'et\'e obtenu en collaboration avec E. Pujals.

\begin{theoreme*}
Tout diff\'eomorphisme peut \^etre approch\'e dans $\diff^1(M)$ par un diff\'eomorphisme
essentiellement hyperbolique ou par un diff\'eomorphisme ayant une tangence homocline ou un cycle h\'et\'erodimensionnel.
\end{theoreme*}

\section{Structure de l'espace des dynamiques}
La d\'ecomposition de la dynamique g\'en\'erique a permis la formulation de nombreuses
questions, pr\'esent\'ees dans ce texte (voir aussi l'expos\'e de Bonatti~\cite{bonatti-conjecture}).
Une des directions de recherches actuelles consiste \`a structurer l'espace des
diff\'eomorphismes et \`a classifier les divers types de comportements de la dynamique.
Une telle approche avait d\'ej\`a \'et\'e propos\'ee par Smale \cite{smale-survey,shub-genericite}.

\paragraph{0.4.a \quad Ph\'enom\`enes et m\'ecanismes.}\index{ph\'enom\`enes et m\'ecanismes}
\index{m\'ecanismes et ph\'enom\`enes}
Les \'enonc\'es pr\'esent\'es dans les paragraphes pr\'e\-c\'e\-dents ont un objectif commun
que l'on peut formuler de la fa\c{c}on suivante.
\begin{probleme*}[Pujals]
Structurer l'espace des diff\'eomorphismes en ph\'enom\`enes et m\'ecanismes.
\end{probleme*}
\noindent
Plus pr\'ecis\'ement, ces \'enonc\'es fournissent deux ouverts disjoints
$\cU_p$, $\cU_m$ d'union dense dans l'espace des diff\'eomorphismes et caract\'erisant deux types
de dynamiques tr\`es diff\'erents.
\smallskip

\begin{itemize}
\item[--] $\cU_p$ est l'ouvert associ\'e \`a un \emph{ph\'enom\`ene}~:
pour tout diff\'eomorphisme appartenant \`a une partie (G$_\delta$ parfois ouverte)
dense de $\cU_p$, il existe une description
globale de la dynamique (dynamique Morse-Smale, hyperbolicit\'e, hyperbolicit\'e partielle, hyperbolicit\'e essentielle,...)
\smallskip
 
\item[--] $\cU_m$ est l'ouvert associ\'e \`a un \emph{m\'ecanisme}~:
pour tout diff\'eomorphisme appartenant \`a une partie dense de $\cU_m$, il existe
une configuration semi-locale de la dynamique (une bifurcation) qui est \`a la fois
tr\`es simple (par exemple elle se lit sur les orbites p\'eriodiques, elle peut \^etre
d\'etect\'ee num\'eriquement), qui engendre des changements importants sur les dynamiques voisines
(par exemple l'apparition d'un grand nombre d'orbites p\'eriodiques)
et qui est une obstruction au ph\'enom\`ene associ\'e \`a $\cU_p$.
\smallskip
\end{itemize}

\paragraph{0.4.b \quad Complexit\'e des dynamiques g\'en\'eriques.}
Les r\'esultats r\'ecents sur la dynamique g\'en\'erique permettent d'imaginer une fa\c{c}on
de structurer l'espace des diff\'eomorphismes. Deux conjectures, impliquant toutes deux celle de Palis,
r\'esument l'essentiel des sc\'enarios envisag\'es pas les auteurs du sujet (voir~\cite{bonatti-conjecture}
pour un expos\'e d\'etaill\'e de diff\'erentes conjectures).

La premi\`ere constate que les seuls m\'ecanismes connus engendrant une infinit\'e de classes
distinctes sont li\'ees \`a l'existence de tangences homoclines
(c'est le cas du ph\'enom\`ene de Newhouse pour lequel il existe
une infinit\'e de puits et de sources).
\begin{conjecture-finitude}[Bonatti]\index{conjecture! de finitude}
Les diff\'eomorphismes loin des tangences homoclines dans $\diff^1(M)$ n'ont qu'un nombre fini
de classes de r\'ecurrence par cha\^\i nes.
\end{conjecture-finitude}

La seconde implique \'egalement la conjecture de Smale et est propre \`a la r\'egularit\'e $C^1$.
Pour cette topologie, les m\'ethodes pour obtenir un ensemble localement dense de diff\'eomorphismes
pr\'esentant une tangence homocline font toutes intervenir des cycles h\'et\'erodimensionnels.

\begin{conjecture-hyperbolicite}[Bonatti-D\'\i az]\index{conjecture! d'hyperbolicit\'e}
Tout diff\'eomorphisme de $\diff^1(M)$ peut \^etre approch\'e par un diff\'eomorphisme
hyperbolique ou par un diff\'eomorphisme ayant un cycle h\'e\-t\'e\-ro\-di\-men\-si\-on\-nel robuste.
\end{conjecture-hyperbolicite}
\bigskip

Bonatti propose de d\'ecomposer l'espace des diff\'eomorphismes selon les crit\`eres suivants~:
\begin{itemize}
\item[--] existence de tangences homoclines (dynamique \emph{critique}),\index{critique, dynamique}
\item[--] existence de cycle h\'et\'erodimensionnel (dynamique \emph{h\'et\'erodimensionnelle}),
\index{h\'et\'erodimensionnelle, dynamique}
\item[--] nombre fini ou infini de classes de r\'ecurrence par cha\^\i nes (dynamique \emph{mod\'er\'ee}
ou \emph{sauvage}).\index{mod\'er\'e, diff\'eomorphisme} \index{sauvage, diff\'eomorphisme}
\end{itemize}
Par ailleurs, Bonatti et D\'\i az ont montr\'e l'existence de dynamiques sauvages ayant
une propri\'et\'e remarquable~: en renormalisant la dynamique contenue dans des disques p\'eriodiques,
on obtient une famille de diff\'eomorphismes dense dans l'espace des diff\'eomorphismes du disque
pr\'eservant l'orientation. Une telle dynamique est qualifi\'ee d'\emph{universelle}.
\index{universel, diff\'eomorphisme}
\medskip

Si les deux conjectures pr\'ec\'edentes sont v\'erifi\'ees,
l'espace des diff\'eomorphismes est alors structur\'e en r\'egions de complexit\'e croissante~:
dynamique hyperbolique, mod\'er\'ee non critique, mod\'er\'ee critique, sauvage, dynamique universelle
(voir la figure~\ref{f.structure}). Des exemples de dynamique pour chaque r\'egion seront donn\'es en sections~\ref{s.diffeo-hyperbolique}, \ref{s.exemple-mane}, \ref{s.exemple-abraham-smale}
et~\ref{s.universelle}.

\vspace{2cm}

\begin{figure}[ht]
\begin{center}
\input{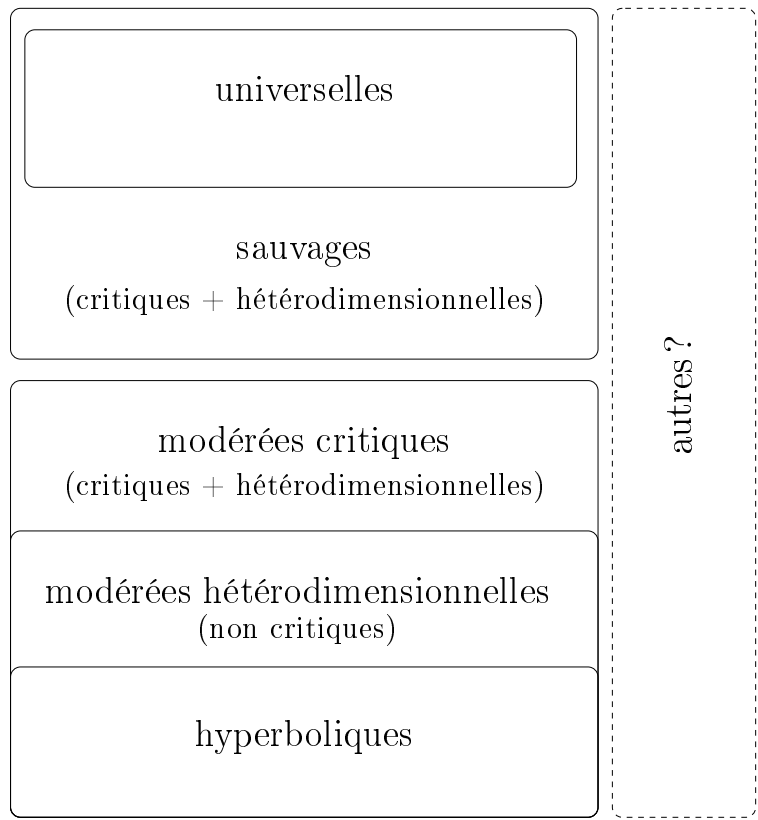}
\end{center}
\caption{Structure de l'espace des diff\'eomorphismes. \label{f.structure}}
\end{figure}

\chapter*{Notations}
\begin{description}
\item
$\cK(X)$ est l'espace (compact) des parties compactes d'un espace m\'etrique compact $X$,
muni de la distance de Hausdorff~:
$$d_H(K,K')=\sup_{x\in K,x'\in K'}d(x,x').$$
(Lorsque $K$ est vide et $K'$ non vide, on pose $d_H(K,K')=\diam(X)$.)
\medskip

\item
$M$ d\'esignera en g\'en\'eral une vari\'et\'e diff\'erentiable compacte sans bord
munie d'une m\'etrique riemannienne.
\medskip

\item
$\interior(A)$ est l'int\'erieur d'une partie $A\subset M$.
\medskip

\item
$\diff^k(M)$, pour $k\geq 0$, est  l'espace
des diff\'eomorphismes $f$ de classe $C^k$ de $M$, muni de la topologie $C^k$.
(C'est un espace m\'etrique complet, donc un espace de Baire.)
\medskip

\item
$\per(f),\Omega(f),\cR(f)$ repr\'esentent l'ensemble des points p\'eriodiques, l'ensemble non-errant et l'ensemble r\'ecurrent par cha\^\i nes du diff\'eomorphisme $f$ de $M$.
\medskip

\item
$W^s(x), pW^s(x)$ sont les ensembles stable et stable par cha\^\i nes du point $x$.
\medskip

\item
$x\prec y$ et $x\dashv y$ sont des relations dynamiques introduites au chapitre~\ref{c.decomposition}.
\medskip

\item Les mesures d'un espace compact m\'etrique $X$ seront toujours des mesures bor\'eliennes.
Leur ensemble est muni de la topologie faible (rendant continues les \'evaluations sur les fonctions continues $X\to \RR$).
\medskip

\item $\cT$ et $\cC$ d\'esignent l'ensemble des diff\'eomorphismes d'une vari\'et\'e
ayant une tangence homocline ou un cycle h\'et\'erodimensionnel.
\end{description}


\chapter{D\'ecomposition de la dynamique}\label{c.decomposition}
Ce chapitre introduit les notations et les d\'efinitions que nous utiliserons
dans la suite du texte. Nous nous pla\c cons sur une vari\'et\'e diff\'erentiable
compacte sans bord $M$ de dimension $d$.

Nous pr\'esentons tout d'abord des notions de r\'ecurrence qui permettent d'analyser les
sys\-t\`e\-mes dynamiques topologiques~:
la r\'ecurrence au sens des pseudo-orbites donne lieu au ``th\'eor\`eme fondamental de la dynamique'' de Conley.
Nous travaillons ensuite avec des dif\-f\'e\-o\-mor\-phis\-mes et nous rappelons la notion d'ensemble hyperbolique.
Nous renvoyons \`a~\cite{bowen,shub-livre,newhouse-cours,yoccoz-hyperbolique,hasselblatt-katok} pour des expos\'es d\'etaill\'es des propri\'et\'es
des dynamiques hyperboliques (stabilit\'e, pistage,...)
Nous d\'efinissons finalement ce qui sera pour nous un \emph{diff\'eomorphisme hyperbolique}.

\section{R\'ecurrence~: transitivit\'e faible, transitivit\'e par cha\^\i nes}
\noindent
Soit $f$ un diff\'eomorphisme d'une vari\'et\'e diff\'erentiable compacte sans bord $M$
munie d'une m\'etrique riemannienne\footnote{Jusqu'en section~\ref{s.hyperbolicite},
il suffirait de consid\'erer un hom\'eomorphisme d'un espace m\'etrique compact.}.

\paragraph{Ensembles faiblement transitifs.}
Un compact invariant $K$ est \emph{transitif} \index{transitivit\'e} si pour tous ouverts $U,V\subset M$ rencontrant $K$,
il existe $z\in U\cap K$ ayant un it\'er\'e positif $f^n(z)$, $n\geq 1$, dans $V$.

Plus g\'en\'eralement, $K$ est \emph{faiblement transitif} \index{transitivit\'e!faible} si, pour tous ouverts $U,V\subset M$ rencontrant $K$
et tout voisinage $W$ de $K$, la propri\'et\'e suivante est satisfaite.
\begin{description}
\item[($\prec$)] \it Il existe un segment d'orbite $\{z,f(z),\dots,f^n(z)\}\subset W$
tel que $z\in U$ et $f^n(z)\in V$.
\end{description}
Si $K\subset M$ est un ensemble ferm\'e, nous noterons
$x\prec_K y$ si pour tous voisinages $U$ de $x$, $V$ de $y$ et $W$ de $K$, la propri\'et\'e~$(\prec)$ a lieu.
Nous \'ecrirons $x\prec y$ pour $x\prec_M y$. Ces relations sont ferm\'ees mais ne sont en g\'en\'eral pas transitives.
\medskip

L'\emph{ensemble non-errant},  \index{ensemble! non-errant} $\Omega(f)$, est l'ensemble des points de $M$
dont tout voisinage $U$ intersecte l'un de ses it\'er\'es $f^n(U)$, $n\geq 1$, c'est-\`a-dire
l'ensemble des points $x$ satisfaisant $x\prec x$.
C'est un ferm\'e invariant qui contient les ensembles faiblement transitifs.
En particulier, il contient tous les points p\'eriodiques de $f$.
Remarquons que la dynamique induite par $f$ sur $\Omega(f)$ peut avoir des points errants.

\paragraph{Ensembles transitifs par cha\^\i nes.}
Nous introduisons une notion de r\'ecurrence plus grossi\`ere mais qui n'a pas les d\'efauts de la pr\'ec\'edente.

Soit $\varepsilon>0$. Les \emph{$\varepsilon$-pseudo-orbites}\index{pseudo-orbite} de $f$ sont
les suites $(x_n)_{n\in I}$ de $M$ telles que la distance $d(f(x_n),x_{n+1})$ est strictement inf\'erieure
\`a $\varepsilon$ pour tout $n\in \ZZ$ tel que $n,n+1$ appartiennent \`a $I$.
\medskip

Un ensemble $K\subset M$ est \emph{transitif par cha\^\i nes} \index{transitivit\'e! par cha\^\i nes} si pour tous $x,y\in K$ et tout $\varepsilon>0$ la propri\'et\'e suivante a lieu.
\begin{description}
\item[($\dashv$)] \it Il existe une $\varepsilon$-pseudo-orbite $\{x_0,\dots, x_n\}\subset K$
telle que $n\geq 1$, $x_0=x$ et $x_n=y$.
\end{description}

Si $K\subset M$ est un ensemble ferm\'e, nous noterons
$x\dashv_K y$ si la propri\'et\'e~($\dashv$) a lieu.
Nous \'ecrirons $x\dashv y$ pour $x\dashv_M y$. Ces relations sont ferm\'ees et transitives.
De plus $x\prec_K y$ implique $x\dashv_K y$. Par cons\'equent les ensembles faiblement transitifs
sont aussi transitifs par cha\^\i nes.

On peut \'etendre la relation $\dashv$ aux ensembles transitifs par cha\^\i nes~:
on a $K\dashv K'$ s'il existe $x\in K$ et $x'\in K'$ tels que $x\dashv x'$.
\medskip

Un point est \emph{r\'ecurrent par cha\^\i nes} \index{r\'ecurrence par cha\^\i nes} si, pour tout $\varepsilon>0$, il appartient \`a
une $\varepsilon$-pseudo-orbite p\'eriodique, c'est \`a dire si $x\dashv x$.
L'ensemble r\'ecurrent par cha\^\i nes $\cR(f)$ est l'ensemble des points r\'ecurrents par cha\^\i nes.
C'est un donc ferm\'e invariant, contenant l'ensemble non-errant.
En restriction \`a $\cR(f)$, la dynamique de $f$
ne poss\`ede que des points r\'ecurrents par cha\^\i nes.

Les \emph{classes de r\'ecurrence par cha\^\i nes}\index{classe!de r\'ecurrence par cha\^\i nes} sont les classes d'\'equivalence
de la relation sym\'etris\'ee ``$x\dashv y$ et $y\dashv x$''.
Elles sont ferm\'ees et elles co\"\i ncident avec les ensembles transitifs par cha\^\i nes maximaux.
L'ensemble r\'ecurrent par cha\^\i nes est naturellement partitionn\'e par les classes
de r\'ecurrence par cha\^\i nes, par cons\'equent l'ensemble r\'ecurrent par cha\^\i nes co\"\i ncide avec l'union des
ensembles transitifs par cha\^\i nes.

\section{Th\'eor\`eme ``fondamental'' de la dynamique, filtrations}
Conley a d\'emontr\'e ce r\'esultat tr\`es g\'en\'eral (voir~\cite{conley,robinson-livre}).

\begin{theoreme}[Conley]\label{t.fondamental}\index{th\'eor\`eme fondamental de la dynamique}
Soit $f$ un hom\'eomorphisme d'un espace compact m\'etrique $M$.
Il existe une fonction continue $h\colon M\to \RR$ ayant les propri\'et\'es suivantes~:
\begin{enumerate}
\item[--] $h$ d\'ecro\^\i t le long des orbites de $f$ et d\'ecro\^\i t strictement
le long des orbites de $f$ contenues dans $M\setminus \cR(f)$~;
\item[--] l'image de $\cR(f)$ par $h$ est totalement discontinue~;
\item[--] $h$ prend des valeurs distinctes sur des classes de r\'ecurrence par cha\^\i nes distinctes.
\end{enumerate}
\end{theoreme}
\medskip

Un ouvert $U$ de $M$ est \emph{attractif} \index{ouvert attractif} si $f(\overline{U})$ est contenu dans $U$.
Une fonction de Lyapunov\footnote{Nous renvoyons \`a~\cite{auslander}
pour une discussion plus d\'etaill\'ee des d\'ecomposition de la dynamique par fonctions
de Lyapunov.} $h$ donn\'ee par le th\'eor\`eme de Conley permet de montrer de nouvelles propri\'et\'es.
\begin{itemize}
 \item[--]
Pour tout point $x\not\in \cR(f)$, il existe un ouvert attractif $U$
tel que $x$ appartient \`a $U\setminus f(\overline U)$~:
l'ensemble r\'ecurrent par cha\^\i nes est donc la notion de r\'ecurrence raisonnable la plus faible.
\item[--]
Pour tout ensemble fini $\cS$ de classes de r\'ecurrence par cha\^\i nes, il existe
une suite embo\^\i t\'ee d'ouverts attractifs $M=U_0\supset U_1 \supset U_2\supset\dots\supset U_k=\emptyset$,
tels que $U_i\setminus U_{i+1}$ contienne un \'el\'ement de $\cS$ et un seul pour chaque entier $0\leq i<k$.
\end{itemize}
Une suite embo\^\i t\'ee d'ouverts attractifs est appel\'ee \emph{filtration} \index{filtration} de $f$.

\section{Ensemble stable par cha\^\i nes, quasi-attracteurs}

L'\emph{ensemble stable} \index{ensemble! stable} d'un point $x\in M$ est l'ensemble
$$W^s(x):=\left\{y\in M,\; d(f^n(x),f^n(y))\underset{n\to+\infty}{\longrightarrow} 0\right\}.$$
Par analogie, on introduit l'\emph{ensemble stable par cha\^\i nes} \index{ensemble! stable par cha\^\i nes} d'un point $x\in \cR(f)$
ou d'une classe de r\'ecurrence par cha\^\i nes de $f$~:
c'est l'ensemble des points $y\in M$ tels que pour tout $\varepsilon>0$ il existe
des $\varepsilon$-pseudo-orbite $(y_0,\dots,y_n)$ et $(x_0,\dots,x_n)$ telles que
$y_0=y$, $x_0=x$ et $x_n=y_n$. On le note $pW^s(x)$.
(Le ``p'' signifie ``au sens des pseudo-orbites'').
On d\'efinit sym\'etriquement l'\emph{ensemble instable} $W^u(x)$ et l'\emph{ensemble instable par cha\^\i nes}
$pW^u(x)$.
\medskip

Les classes de r\'ecurrence par cha\^\i nes sont partiellement ordonn\'ees par la relation
$\dashv$. Un \emph{quasi-attracteur} \index{quasi-attracteur} (on parle aussi parfois d'``attracteur faible'')
\index{attracteur!faible} est
une classe de r\'ecurrence par cha\^\i nes qui est minimale pour l'ordre $\dashv$, ou de fa\c con \'equivalente
qui co\"\i ncide avec son ensemble instable par cha\^\i nes.
Il existe toujours au moins un quasi-attracteur (ce qui n'est pas le cas des attracteurs).
\medskip

Si $K$ est une classe de r\'ecurrence par cha\^\i nes,
pour tout $\varepsilon>0$, l'ensemble des points que l'on peut atteindre depuis $K$
par $\varepsilon$-pseudo-orbites est un voisinage attractif de $K$.
Ceci montre la proposition suivante.
\begin{proposition}
Une classe de r\'ecurrence par cha\^\i nes est un quasi-attracteur si et seulement
si elle poss\`ede une base de voisinages ouverts attractifs.
\end{proposition}

Ces notions sont proches de la stabilit\'e au sens de Lyapunov.
On dit qu'un ensemble invariant $K$ est \emph{stable au sens de Lyapunov}
\index{stabilit\'e!au sens de Lyapunov}
s'il admet une base de voisinages $U$ positivement invariants, i.e.
satisfaisant $f(U)\subset U$.
Un quasi-attracteur est une classe de r\'ecurrence par cha\^\i nes 
stable au sens de Lyapunov, mais la r\'eciproque est fausse en g\'en\'eral.

\section{Hyperbolicit\'e}\label{s.hyperbolicite}
Un ensemble $K$ invariant par $f$ est \emph{hyperbolique} \index{hyperbolicit\'e!ensemble} s'il existe une d\'ecomposition
$T_KM=E^s\oplus E^u$ du fibr\'e tangent au-dessus de $K$ en deux sous-fibr\'es lin\'eaires invariants par l'application tangente
$Df$ et un entier $N\geq 1$ tels que pour tout $x\in K$ on a\footnote{
Nous pouvons bien s\^ur remplacer $e$ par n'importe quelle constante strictement
sup\'erieure \`a $1$, mais nous avons fait ce choix par coh\'erence avec la
d\'efinition de l'exposant de Lyapunov (section~\ref{s.oseledets}).}
\begin{equation}\label{e.hyperbolique}
\|D_xf^N_{|E^s(x)}\|\leq \frac 1 e \text{ et } \|D_xf^{-N}_{|E^u(x)}\|\leq \frac 1 e.
\end{equation}
La dimension du fibr\'e stable est appel\'ee \emph{indice} \index{indice} de $K$.\footnote{Certains auteurs appellent au contraire indice la dimension du fibr\'e instable.}

\noindent
Plus g\'en\'eralement, un ensemble compact invariant est dit \emph{hyperbolique (\`a indice variable)}
s'il admet une partition en ensembles hyperboliques.
\medskip

L'ensemble stable $W^s(x)$ (resp. instable $W^u(x)$)
d'un point $x$ appartenant \`a un ensemble hyperbolique est une sous-vari\'et\'e
injectivement immerg\'ee de $M$, tangente \`a $E^s(x)$ (resp. $E^u(x)$).

Une orbite p\'eriodique hyperbolique est un \emph{puits} \index{puits}
(resp. une \emph{source} \index{source}) si son ensemble stable (resp. instable) contient
un voisinage de l'orbite. Dans les autre cas, on dit que l'orbite p\'eriodique hyperbolique est une \emph{selle}\index{selle}.

\section{Classes homoclines}
Deux orbites p\'eriodiques hyperboliques $O_1,O_2$ sont dites \emph{homocliniquement \index{homocliniquement reli\'ees, orbites} reli\'ees}\footnote{On devrait dire
\emph{transversalement homocliniquement reli\'ees}.}
si la vari\'et\'e stable de $O_i$ coupe transversalement la vari\'et\'e instable de $O_j$
pour $(i,j)=(1,2)$ et $(i,j)=(2,1)$. En particulier, leur indices co\"\i ncident.
\medskip

La \emph{classe homocline} \index{classe!homocline} $H(O)$ d'une orbite p\'eriodique hyperbolique est l'adh\'erence
de l'ensemble des points d'intersection transverse entre les vari\'et\'es stables et instables de $O$.
En utilisant le lemme d'inclinaison (le ``$\lambda$-lemma'', voir~\cite[proposition 2.5]{newhouse-cours})
et le th\'eor\`eme homocline de Smale (voir~\cite[th\'eor\`eme 2.3]{newhouse-cours}), Newhouse a montr\'e~\cite{newhouse-homocline}
que c'est aussi l'adh\'erence de l'ensemble des orbites p\'eriodiques hyperboliques homocliniquement reli\'ees \`a $O$.
La classe homocline $H(O)$ est dite \emph{triviale} lorsqu'elle est r\'eduite \`a $O$.

Remarquons que les classes homoclines de deux orbites p\'eriodiques hyperboliques distinctes peuvent s'intersecter sans co\"\i ncider.
Chaque classe homocline $H(O)$ est un ensemble compact invariant transitif.
Plus pr\'ecis\'ement,
il existe toujours au moins une mesure de probabilit\'e invariante ergodique dont le support co\"\i ncide avec $H(O)$
(voir~\cite[th\'eor\`eme 3.1]{abc-ergodic}).

\section{Diff\'eomorphismes hyperboliques}\label{s.diffeo-hyperbolique}
Un diff\'eomorphisme satisfait l'\emph{axiome A} \index{axiome A} si l'ensemble non-errant $\Omega(f)$ est hyperbolique
et co\-in\-ci\-de avec l'a\-dh\'e\-ren\-ce $\overline{\per(f)}$ des points p\'e\-ri\-o\-di\-ques de $f$.
D'a\-pr\`es, le th\'e\-o\-r\`e\-me de d\'e\-com\-po\-si\-tion spectrale de Smale, il existe une d\'ecomposition disjointe
$$\Omega(f)=K_1\cup\dots\cup K_s$$
en ensembles hyperboliques compacts transitifs localement maximaux\footnote{$K_i$ est localement maximal s'il poss\`ede un voisinage $U$ tel que $K_i=\cap_{n\in \ZZ} f^n(U)$.}, appel\'es \emph{ensembles ba\-si\-ques \index{ensemble! basique} de $f$}.
Chaque ensemble $K_i$ est une classe homocline et r\'eciproquement toute classe homocline de $f$ est
un ensemble basique.

Un diff\'eomorphisme qui satisfait l'axiome $A$ poss\`ede un \emph{cycle} \index{cycle} s'il existe
une suite d'ensembles basiques
$K_{i_1},\dots,K_{i_r}$
tels que $W^u(K_{i_j})$ et $W^s(K_{i_{\ell}})$ se rencontrent hors de $\Omega(f)$
lorsque $j\in\{1,\dots,r-1\}$, $\ell=j+1$ et lorsque $(j,\ell)=(r,1)$.
Lorsque $f$ n'a pas de cycle, $\Omega(f)=\cR(f)$ et les ensembles basiques sont aussi les classes de r\'ecurrence par cha\^\i nes de $f$.

R\'eciproquement, on montre facilement \`a l'aide du lemme de pistage qu'un diff\'eomorphisme
dont l'ensemble r\'ecurrent par cha\^\i nes est hyperbolique v\'erifie l'axiome A et n'a pas de cycle.
Ceci justifie la d\'efinition suivante.\footnote{Cette d\'efinition n'est pas standard.
On lit en g\'en\'eral ``diff\'eomorphisme axiome A sans cycle''.}

\begin{definition}\label{d.hyperbolique}
Un diff\'eomorphisme $f\in \diff^1(M)$ est \emph{hyperbolique} \index{hyperbolicit\'e!diff\'eomorphisme} si $\cR(f)$ est hyperbolique,
ou de fa\c con \'equivalente s'il satisfait l'axiome A et n'a pas de cycle. 
\end{definition}
Puisque l'ensemble r\'ecurrent par cha\^\i nes varie semi-contin\^ument sup\'erieurement avec $f$,
les diff\'eomorphismes hyperboliques forment une partie ouverte de $\diff^1(M)$.

\begin{exemples}
\begin{enumerate}
\item Un diff\'eomorphisme hyperbolique pour lequel $\cR(f)$ co\"\i ncide avec $M$ est un \emph{dif\-f\'e\-o\-mor\-phis\-me d'Anosov}. \index{Anosov, diff\'eomorphisme}
C'est le cas du diff\'eomorphisme du tore $\TT^1=\RR^2/\ZZ^2$ induit par l'action de l'automorphisme
lin\'eaire $\begin{pmatrix}2 &1\\ 1 &1\end{pmatrix}$.
\item Le temps $1$ du flot associ\'e au champ de gradient d'une fonction de Morse $h\colon M\to \RR$
est un diff\'eomorphisme hyperbolique~: les ensembles basiques sont les points critiques de $h$.
Lorsque toutes les valeurs critiques de $h$ sont simples, $-h$ est une fonction de Lyapunov au sens
du th\'eor\`eme~\ref{t.fondamental} de Conley. La figure~\ref{f.filtration}
montre une filtration associ\'ee.
\end{enumerate}
\end{exemples}
\begin{figure}[ht]
\begin{center}
\input{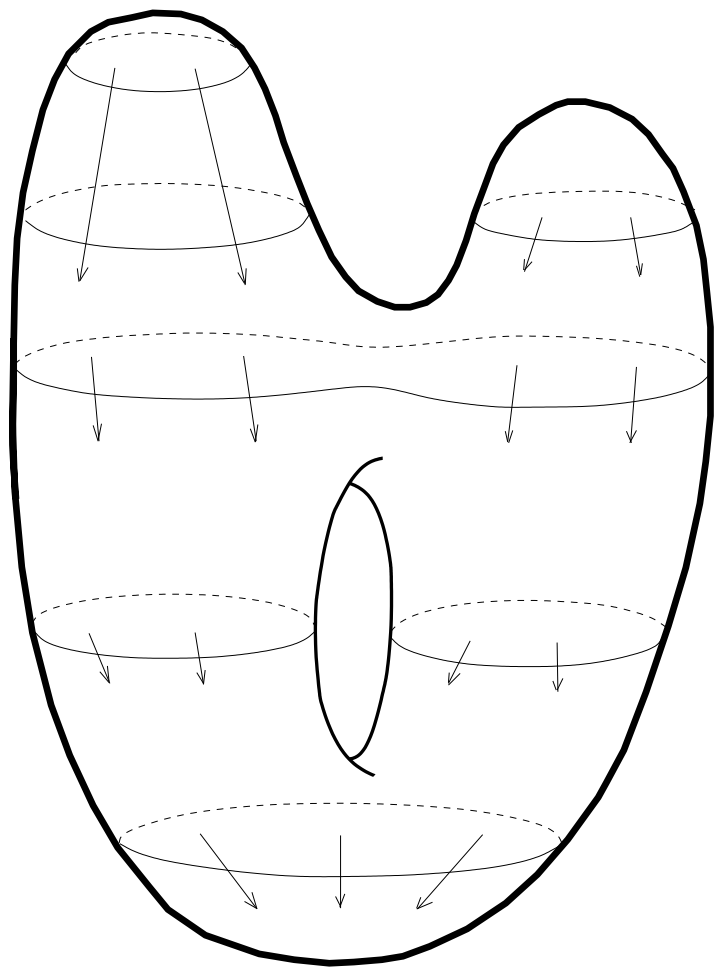}
\end{center}
\caption{Filtration pour le temps $1$ du flot associ\'e au gradient d'une fonction de Morse. \label{f.filtration}}
\end{figure}
Smale a donn\'e~\cite{smale-morse} une g\'en\'eralisation de l'exemple pr\'ec\'edent.

\begin{definition}\label{d.morse-smale}
Un diff\'eomorphisme $f\in \diff^1(M)$ est \emph{Morse-Smale}\index{Morse-Smale, diff\'eomorphisme} si $\cR(f)$ est fini, hyperbolique
et si pour tous points $x,y\in \cR(f)$, l'intersection $W^s(x)\cap W^u(y)$ est transverse~:
pour tout point $z$ dans l'intersection, $T_zM=T_zW^s(x)+T_zW^u(y)$.
\end{definition}
L'ensemble des diff\'eomorphismes Morse-Smale est ouvert~\cite{palis-smale}.


\chapter{Techniques de perturbations, g\'en\'ericit\'e}
Afin d'obtenir des propri\'et\'es satisfaites par des ensembles denses de diff\'eomorphismes,
il est naturel de s'int\'eresser aux propri\'et\'es des diff\'eomorphismes g\'en\'eriques.
La difficult\'e pour montrer qu'une propri\'et\'e est g\'en\'erique reste bien souvent la densit\'e~:
elle requiert de savoir perturber dans $\diff^1(M)$.

Nous introduisons dans ce chapitre une technique de perturbation d'orbites en topologie $C^1$,
due \`a Pugh. Elle permet d'obtenir par un argument combinatoire
(que nous pensons plus simple que l'argument initial de Pugh~\cite{pugh-closing,pugh-robinson})
le lemme de fermeture de Pugh (le ``closing lemma'').
Elle permet \'egalement d'obtenir les autres r\'esultats de connexion d'orbites,
en particulier le lemme de connexion d'Hayashi (le ``connecting lemma'') pr\'esent\'e en fin de chapitre.

Le lemme de fermeture est particuli\`erement important puisqu'il assure l'existence d'orbites p\'eriodiques
pour les diff\'eomorphismes $C^1$-g\'en\'eriques~: un grand nombre de propri\'et\'es dynamiques
s'\'enoncent ou se d\'emontrent en s'appuyant sur les orbites p\'eriodiques.

\section{Notions de g\'en\'ericit\'e, de robustesse}
Un ensemble de diff\'eomorphismes est \emph{$C^r$-g\'en\'erique}\index{g\'en\'ericit\'e}, pour $r\geq 0$,
s'il contient un G$_\delta$ dense de $\diff^r(M)$. Une propri\'et\'e est $C^r$-g\'en\'erique
si elle est satisfaite par un ensemble $C^r$-g\'en\'erique.

Puisque $\diff^r(M)$ est un espace de Baire, une intersection d\'enombrable de parties
$C^r$-g\'en\'eriques est encore $C^r$-g\'en\'erique.
\medskip

Un diff\'eomorphisme $f\in \diff^r(M)$
v\'erifie une propri\'et\'e \emph{de fa\c con robuste}
si cette propri\'et\'e est satisfaite par tous les diff\'eomorphismes
proches de $f$.

\section{Mise en transversalit\'e~: diff\'eomorphismes de Kupka-Smale}
L'une des premi\`eres propri\'et\'es de g\'en\'ericit\'e a \'et\'e montr\'ee par Kupka et Smale~\cite{kupka,smale-kupka}.
En appliquant le th\'eor\`eme de transversalit\'e de Thom
(disponible en r\'egularit\'e $C^k$, $k\geq 1$),
on peut perturber tout diff\'eomorphisme pour supprimer les points fixes non hyperboliques. Plus g\'en\'eralement, on obtient le th\'eor\`eme suivant (voir~\cite{palis-demelo}).

\begin{theoreme}[Kupka-Smale]\label{t.kupka-smale}
Chaque espace $\diff^k(M)$, $\geq 1$ contient un G$_\delta$ dense
form\'e de diff\'eomorphismes v\'erifiant les deux propri\'et\'es suivantes.
\begin{itemize}
\item[--] Toutes les orbites p\'eriodiques de $f$ sont hyperboliques.
\item[--] Pour tous points p\'eriodiques $p,q$ de $f$,
les vari\'et\'es stables $W^s(p)$ et instables $W^u(q)$ sont transverses, i.e.
$T_xW^s(p)+T_xW^u(q)=T_xM$ pour tout $x\in W^s(p)\cap W^u(q)$.
\end{itemize}
\end{theoreme}
Un diff\'eomorphisme v\'erifiant ces deux propri\'et\'es est appel\'e \emph{diff\'eomorphisme Kupka-Smale}\index{Kupka-Smale, diff\'eomorphisme}. 
En particulier pour tout $N\geq 1$, l'ensemble de ses points p\'eriodiques de p\'eriode inf\'erieure \`a $N$
est fini.

\section{Perturbation ponctuelle de la dif\-f\'e\-ren\-ti\-elle~: le lemme de Franks}
Voici un r\'e\-sul\-tat \'e\-l\'e\-men\-tai\-re, souvent u\-ti\-li\-s\'e
pour modifier la dif\-f\'e\-ren\-ti\-elle d'un dif\-f\'e\-o\-mor\-phis\-me le long d'une orbite
p\'eriodique, sans changer l'orbite.

\begin{theoreme}[Lemme de Franks~\cite{franks}\index{lemme de Franks}]\label{t.franks}
Pour tout voisinage $\cU$ de $f\in \diff^1(M)$, il existe $\varepsilon>0$ avec la propri\'et\'e suivante.

Pour tout point $p\in M$ et toute application lin\'eaire $A\colon T_{p}M\to T_{f(p)}M$ telle que
$\|D_pf-A\|\leq \varepsilon$,
il existe un diff\'eomorphisme $g\in \cU$ tel que~:
\begin{itemize}
\item[--] $f$ et $g$ co\"\i ncident en $p$ et hors d'un voisinage arbitrairement petit de $p$~;
\item[--] $D_pg=A$.
\end{itemize}
\end{theoreme}
Ce r\'esultat permet de rendre hyperbolique une orbite p\'eriodique par petite perturbation $C^1$.
Avec le m\^eme argument, on peut aussi ``lin\'eariser'' un diff\'eomorphisme au voisinage d'un point,
c'est-\`a-dire fixer une m\'etrique riemannienne et demander que $g$ co\"\i ncide au voisinage de $p$ avec l'application
$\exp_{f(p)}\circ A \circ \exp_{p}^{-1}$.
Ce r\'esultat est bien s\^ur propre \`a la topologie $C^1$.

\section{Modification \'el\'ementaire d'une orbite}
On souhaite fr\'equemment modifier une orbite d'un diff\'eomorphisme $f$.
Dans les cas les plus simples, on consid\`ere deux points proches $x,y\in M$,
et l'on cherche un diff\'eomorphisme $g$ proche de $f$ pour lequel l'image de $x$ n'est plus $f(x)$ mais $f(y)$.
On contr\^ole la taille du support \`a l'aide du lemme \'el\'ementaire suivant
(voir~\cite{arnaud-closing}).

\begin{lemme}[Perturbation \'el\'ementaire]\label{l.perturbation-elementaire}\index{perturbation \'el\'ementaire}
Pour tout voisinage $\cU$ d'un diff\'eomorphisme $f\in \diff^1(M)$, il existe $\theta>1$ et $r_0>0$ v\'erifiant la propri\'et\'e suivante~:
pour tous points $x,y\in M$, contenus dans une boule $B(z,r)$ de raynon $r<r_0$,
il existe $g\in \cU$ envoyant $x$ sur $f(y)$ et coincidant avec $f$ hors de la boule
$B(z, \theta.r)$.
\end{lemme}

\paragraph{\it Pourquoi travailler en topologie $C^1$~?}

Une telle perturbation s'obtiendrait tr\`es facilement en topologie $C^0$ avec un support
localis\'e autour d'un chemin quelconque joignant $x$ \`a $y$.
Les r\'egularit\'es sup\'erieures font appara\^\i tre des contraintes sur la taille du support de perturbation.

En topologie $C^1$, si l'on cherche une perturbation $g$ de l'identit\'e v\'erifiant $g(x)=y$
et $\|Dg-\id\|\leq \varepsilon$, pour tout point $z=g(z)$ on obtient~:
$$ \|y-x\|=\|(g(x)-x)-(g(z)-z)\|\leq \varepsilon.\|x-z\|.$$
Le support de la perturbation contient donc une boule de rayon $\varepsilon^{-1}.d(x,y)$.
Le lemme~\ref{l.perturbation-elementaire} nous montre en revanche que la taille
du support ne d\'eg\'en\`ere pas aux petites \'echelles~:
l'uniformit\'e de $\theta$ par rapport au couple $(x,y)$ sera essentielle pour les utilisations futures.

En topologie sup\'erieure, on perd cette uniformit\'e. Ceci explique pourquoi les d\'emonstrations
des \'enonc\'es perturbatifs que nous pr\'esentons dans les sections suivantes sont sp\'ecifiques
\`a la topologie $C^1$, ainsi que les difficult\'es que l'on a \`a perturber en r\'egularit\'e $C^k$, $k\geq 2$.

\paragraph{\it Quelle est la difficult\'e \`a travailler en topologie $C^1$~?}
La constante $\theta$ dans le lemme~\ref{l.perturbation-elementaire}
explose lorsque le voisinage $\cU$ d\'ecro\^\i t. Ceci constitue la principale difficult\'e des perturbations en topologie $C^1$. Si l'on suppose par exemple que $x$ est un it\'er\'e futur de $y$, on peut esp\'erer
utiliser le lemme~\ref{l.perturbation-elementaire} pour fermer l'orbite de $x$. Cette id\'ee simple ne fonctionne
pas en g\'en\'eral puisque le support risque de contenir un it\'er\'e futur de $y$ 
ant\'erieur \`a $x$ et la perturbation risque de briser la connexion entre $y$ et $x$.
On voit ainsi l'int\'er\^et de travailler avec des perturbation de support aussi petit que possible.

\section{Modification progressive d'une orbite~: le lemme de Pugh}
Pugh a \'elabor\'e une technique de perturbation en topologie $C^1$
qui permet de supposer la constante $\theta$ du lemme~\ref{l.perturbation-elementaire} petite,
quelle que soit la taille des perturbations. L'id\'ee est de r\'epartir la perturbation dans le temps~:
le diff\'eomorphisme $f$ peut \^etre perturb\'e successivement dans des domaines $U,f(U),\dots, f^{N-1}(U)$ disjoints.
Puisque l'on travaille d\'esormais avec de grands it\'er\'es de la dynamique, le domaine $U$ supportant la perturbation subit une d\'eformation 
et ce sont maintenant des ellipso\"\i des qui jouent le r\^ole des boules apparaissant au lemme~\ref{l.perturbation-elementaire}.
Il sera plus commode par la suite de travailler avec des parall\'el\'epip\`edes.
On introduit donc la d\'efinition suivante.

\begin{definition}
Si $\varphi\colon V\to \RR^d$ est une carte de $M$,
nous appelons \emph{cube} de $\varphi$ tout ensemble $C\subset V$ tel que
$\varphi(C)$ soit l'image d'un cube $[-a,a]^d$ par une translation de $\RR^d$~;
la quantit\'e $a>0$ est le \emph{rayon} du cube.
Si le cube homoth\'etique \`a $\varphi(C)$ de rayon $(1+\varepsilon).a$ et de m\^eme centre
est encore contenu dans $\varphi(C)$, nous noterons $(1+\varepsilon).C$ sa pr\'e-image par $\varphi$.
\end{definition}

Voici l'\'enonc\'e fondamental permettant les modifications d'orbites en topologie $C^1$.
\begin{theoreme}[Lemme de perturbation, Pugh\index{lemme de perturbation}]\label{t.pugh}
Pour tout voisinage $\cU$ de  $f\in \diff^1(M)$ et tous $\varepsilon,\eta>0$,
il existe un entier $N\geq 1$ et, en tout point $p\in M$, une carte $\varphi\colon V\to \RR^d$
avec la propri\'et\'e suivante.

Pour tout cube $(1+\varepsilon).C$ de $\varphi$, disjoint de ses $N-1$-premiers it\'er\'es, et tous points $x,y\in C$,
il existe $g\in \cU$ v\'erifiant~:
\begin{itemize}
\item[--] $g^N(x)=f^N(y)$~;
\item[--] $g$ co\"\i ncide avec $f$ hors de l'union des $f^k((1+\varepsilon).C)$ avec $0\leq k<N$~;
\item[--] pour tout $0< k<N$, le point $g^k(x)$ appartient \`a $f^k((1+\frac \varepsilon 2).C)$~;
\item[--] $g$ et $f$ co\"\i ncident sur $f^k((1+\varepsilon).C)$ pour $0\leq k<N$ hors d'une boule centr\'ee en $g^k(x)$
et de rayon inf\'erieur \`a $\eta$ fois la distance entre $f^k((1+\frac \varepsilon 2).C)$ et le compl\'ementaire de $f^k((1+\varepsilon).C)$.
\end{itemize}
\end{theoreme}
\begin{remarque}
La perturbation $g$ de $f$ est obtenue en composant des perturbations \'el\'ementaires
(au sens du lemme~\ref{l.perturbation-elementaire}) centr\'ee aux diff\'erents points $g^k(x)$, $k=0,\dots,N-1$.
\end{remarque}

C'est essentiellement cet \'enonc\'e que l'on retrouve dans les diff\'erents travaux traitant du ``lemme de fermeture'' de Pugh
(voir~\cite{pugh-closing,pugh-closing3,liao-closing,liao-livre,pugh-robinson}).
La version que nous donnons ici est  l\'eg\`erement plus uniforme
(en particulier $N$ ne d\'epend pas du point $p$) et correspond \`a \cite[th\'eor\`eme A.2]{bc} (voir aussi~\cite{wen-connecting}).
Pour la d\'emonstration, il suffit de travailler avec la suite d'applications lin\'eaires
$D_pf^k\colon T_pM\to T_{f^k(p)}M$, $k\geq 0$.
En effet, le support de la carte $\varphi$ est choisi tr\`es petit de sorte
que les premiers it\'er\'es de $f$ sur $V$ sont proches d'une suite d'applications lin\'eaires.
La preuve initiale de Pugh a \'et\'e simplifi\'ee par Mai~\cite{mai} (voir aussi~\cite{arnaud-closing}).
Nous renvoyons \`a~\cite{wen-xia-basic} pour une d\'emonstration digeste de ce th\'eor\`eme.

\section{Fermeture d'orbites~: le ``closing lemma''}\label{ss.closing}
Nous pouvons d\'eduire du th\'eor\`eme~\ref{t.pugh} le c\'el\`ebre ``closing lemma'' de Pugh.

\begin{theoreme}[Lemme de fermeture, Pugh\index{lemme de fermeture}\index{closing lemma}]\label{t.closing}
Pour tout voisinage $\cU$ de $f\in \diff^1(M)$ et tout $p\in\Omega(f)$,
il existe $g\in \cU$ pour lequel $p$ est un point p\'eriodique.
\end{theoreme}
\`A notre connaissance, il n'y a pas dans la litt\'erature de d\'emonstration directe du th\'eor\`eme de Pugh
\`a l'aide du lemme de perturbation (par exemple, la d\'emonstration de~\cite{pugh-robinson} requiert
de modifier la g\'eom\'etrie des cubes de perturbation). Nous en proposons une, de nature combinatoire.
\begin{proof}
Fixons $\varepsilon>0$ tel que $(1+\varepsilon)^L\leq \frac 3 2$, $L=3^d$, o\`u $d$ est la dimension
de la vari\'et\'e. La constante $\eta$ ne jouera pas de r\^ole et sera prise \'egale \`a $\frac 1 2$.
Consid\'erons un entier $N\geq 1$ et une carte $\varphi\colon V\to \RR^d$ au voisinage de $p$, donn\'es par le th\'eor\`eme~\ref{t.pugh} appliqu\'e \`a un voisinage $\cU_0$ de $f$ plus petit que $\cU$.
Nous pouvons supposer que $p$ n'est pas p\'eriodique et donc que $V$ est disjoint de ses $N-1$
premiers it\'er\'es.

Puisque $p$ est non-errant, il existe deux points $x_0,y_0$ proches de $p$ tels que $x_0$ soit un
it\'er\'e futur de $y_0$. Nous consid\'erons deux cubes $\hat C_0$ et $C_0=\frac 1 2 . \hat C_0$ de $\varphi$
tels que $C_0$ contienne $x_0$ et $y_0$.
Soit $\cP$ l'ensemble (fini) des it\'er\'es interm\'ediaires entre $x_0$ et $y_0$
contenus dans $V$. Nous construisons
\begin{itemize}
\item[--] des paires $(x_k,y_k)$ de points de $\cP$ tels que $x_k$ est un it\'er\'e futur de $y_k$,
\item[--] des cubes $\hat C_k$ de $\varphi$ tels que le cube $C_k=\frac 1 2 .\hat C_k$ contienne $x_k$ et $y_k$,
\end{itemize}
de sorte que $\hat C_{k}$ soit contenu dans $\hat C_{k-1}$ et de rayon moiti\'e.

La construction se fait par r\'ecurrence (voir la figure~\ref{f.nested}). Puisque $\cP$ est fini, elle s'arr\^ete en temps fini.
\begin{figure}[ht]
\begin{center}
\input{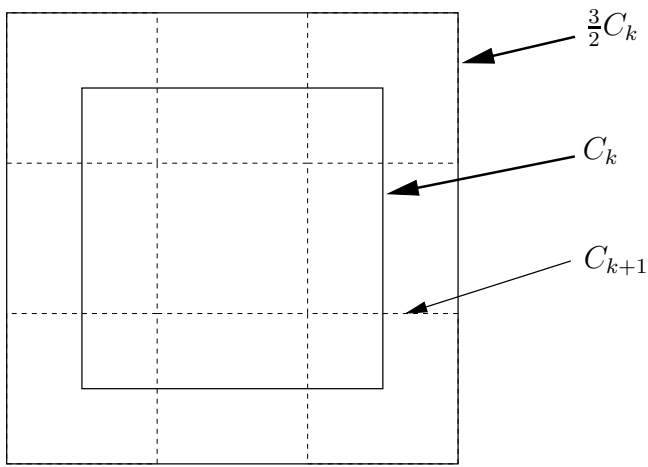}
\end{center}
\caption{Construction des cubes $C_k$. \label{f.nested}}
\end{figure}
Nous allons voir que l'arr\^et correspond \`a l'existence de points $x,y\in \hat C_0\cap \cP$
et d'un cube $\hat C$ de $\varphi$ tels que~:
\begin{enumerate}
\item\label{i.r1} $x=f^t(y)$ pour un entier $t\geq 1$~;
\item\label{i.r2} $x,y$ appartiennent \`a $C=\frac 1 2 . \hat C$~;
\item\label{i.r3} les it\'er\'es interm\'ediaires $f^s(y)$, $0<s<t$ n'appartiennent pas \`a $(1+\varepsilon).C$.
\end{enumerate}
Le th\'eor\`eme~\ref{t.pugh} permet alors de trouver un diff\'eomorphisme $g_0\in \cU_0$
tel que $x$ est p\'eriodique. Puisque $x$ est arbitrairement proche de $p$, il existe $g\in \cU$
conjugu\'e \`a $g_0$ tel que $p$ est p\'eriodique.

Il reste \`a expliquer comment construire $(x_{k+1},y_{k+1},\hat C_{k+1})$ \`a partir de $(x_{k},y_{k},\hat C_{k})$.
Le triplet $T(0)=(x_{k},y_{k},\hat C_{k})$ v\'erifie les propri\'et\'es~\ref{i.r1} et~\ref{i.r2}.
S'il ne v\'erifie pas~\ref{i.r3},
il existe un it\'er\'e $z(1)$ interm\'ediaire entre $x_k$ et $y_k$ qui appartient \`a $(1+\varepsilon).C_{k}$.
Le triplet $T(1)=(x_k,z(1),(1+\varepsilon).\hat C_k)$ v\'erifie alors les propri\'et\'es~\ref{i.r1} et~\ref{i.r2}.
On peut r\'ep\'eter la construction tant que l'on ne parvient pas \`a satisfaire
aux trois propri\'et\'es. On obtient ainsi une suite de triplets $T(j)=(x_k,z(j),(1+\varepsilon)^j.\hat C_k)$
pour $0\leq j \leq L$. Puisque $(1+\varepsilon)^L\leq \frac 3 2$, les points $z(j)$ appartiennent tous
\`a $\frac 3 2 . C_k$. Le cube $\frac 3 2 C_k$ est r\'eunion de $3^d$ cubes de rayon moiti\'e du rayon de $C_k$.
Puisque $L=3^d$, il existe donc deux points $z(j)$ et $z(j')$ qui sont contenus dans un m\^eme de ces cubes, not\'e $C_{k+1}$.
Nous notons $x_{k+1}$ et $y_{k+1}$ ces points et posons $\hat C_{k+1}=2 C_{k+1}$.
\end{proof}

\section[Le th\'eor\`eme de densit\'e de Pugh]{Exemple de d\'emonstration de g\'en\'ericit\'e~: le th\'eor\`eme de densit\'e de Pugh}

Nous montrons \`a pr\'esent comment d\'eduire un r\'esultat de g\'en\'ericit\'e (ici le th\'eor\`eme
de densit\'e de Pugh~\cite{pugh-closing2}) d'un \'enonc\'e perturbatif.

\begin{corollaire}[Pugh]\label{c.pugh}\index{th\'eor\`eme de densit\'e}
Il existe un G$_\delta$ dense $\cG$ de $\diff^1(M)$
tel que l'ensemble des points p\'eriodiques de tout diff\'eomorphisme $f\in \cG$ est dense dans l'ensemble
non-errant $\Omega(f)$.
\end{corollaire}

La d\'emonstration utilise la propri\'et\'e classique suivante des espaces de Baire (voir~\cite{kuratowski}).

\begin{proposition}\label{p.baire}
Soit $\cB$ un espace de Baire, $X$ un espace compact m\'etrique et $h\colon \cB\to \cK(X)$ une application
semi-continue inf\'erieurement ou semi-continue sup\'erieurement,
\`a valeur dans l'espace des sous-ensembles compacts de $X$,
muni de la topologie de Hausdorff. Alors, l'ensemble des points de continuit\'e de $h$ contient un
G$_\delta$ dense de $\cB$.

En d'autres termes,
\begin{itemize}
\item[--] si pour tout ouvert $U$ de $X$,
l'ensemble $\{b\in \cB, \; h(b)\cap U\neq \emptyset\}$ est ouvert, alors
pour tout $b_0$ dans un
G$_\delta$ dense de $\cB$ et pour tout voisinage $U$ de $h(b_0)$, l'ensemble
$\{b\in \cB,\; h(b)\subset U\}$ est un voisinage de $b_0$~;
\item[--]
si pour tout ouvert $U$ de $X$,
l'ensemble $\{b\in \cB, \; h(b)\subset U\}$ est ouvert, alors
pour tout $b_0$ dans un
G$_\delta$ dense de $\cB$ et pour tout voisinage $U$ de $h(b_0)$, l'ensemble
$\{b\in \cB,\; h(b)\cap U\neq \emptyset\}$ est un voisinage de $b_0$.
\end{itemize}

\end{proposition}

\begin{proof}[D\'emonstration du corollaire~\ref{c.pugh}]
Soit $\cG_0\subset \diff^1(M)$ un G$_\delta$ dense de diff\'eomorphismes dont les points p\'eriodiques
sont tous hyperboliques (donn\'e par le th\'eor\`eme~\ref{t.kupka-smale} de Kupka-Smale). C'est un espace de Baire.
Soit $P\colon \diff^1(M)\to \cK(M)$ l'application qui associe \`a tout diff\'eomorphisme $f$
l'adh\'erence $\overline{\per(f)}$ de ses points p\'eriodiques. Chaque point p\'eriodique poss\`ede une continuation hyperbolique
et par cons\'equent $P$ est semi-continue inf\'erieurement. D'apr\`es la proposition~\ref{p.baire} et le th\'eor\`eme de Baire,
l'ensemble de ses points de continuit\'e contient un G$_\delta$ dense $\cG$ de $\diff^1(M)$.

Pour $f\in \cG$ nous devons montrer que $P(f)$ et $\Omega(f)$ co\"\i ncident. Nous avons bien s\^ur l'inclusion
$P(f)\subset \Omega(f)$. Consid\'erons un point $x\in \Omega(f)$ et un voisinage ouvert $U$ de $P(f)$.
Le lemme de fermeture donne l'existence d'un diff\'eomorphisme $g_0\in \diff^1(M)$ proche de $f$ tel que $x$ est p\'eriodique.
Par perturbation (donn\'ee par le lemme de Franks, th\'eor\`eme~\ref{t.franks}),
on peut supposer que $x$ est hyperbolique~; l'orbite de $x$ peut donc \^etre suivie par perturbation, il existe donc
$g\in \cG_0$ proche de $g_0$ et de $f$ ayant un point p\'eriodique arbitrairement proche de $x$.
Puisque $P(g)\subset U$, on en d\'eduit que $x$ appartient \`a $\overline{U}$. Le voisinage $U$ \'etant arbitraire,
ceci permet de conclure $x\in P(f)$. Par cons\'equent $\Omega(f)=P(f)$.
\end{proof}

\section{Connexions d'orbites~: le ``connecting lemma''}\label{s.connecting}
Consid\'erons deux points p\'eriodiques hyperboliques $p,q$ et supposons que la vari\'et\'e
instable de $p$ et la vari\'et\'e stable de $q$ aient un point d'accumulation commun $z$. Peut-on par
perturbation $C^1$ cr\'eer une orbite h\'et\'erocline entre $p$ et $q$~?
Cette question se pose naturellement lorsque l'on cherche \`a caract\'eriser les diff\'eomorphismes structurellement
ou $\Omega$-stables (voir~\cite{hayashi-survey}) et a \'et\'e contourn\'ee dans ce cadre par Ma\~n\'e~\cite{mane-connecting}.
Lorsque $z$ n'est pas p\'eriodique,
la r\'eponse a finalement \'et\'e apport\'ee par Hayashi~\cite{hayashi-connecting} trente ans apr\`es la d\'emonstration du lemme de fermeture.

\begin{theoreme}[Lemme de connexion, Hayashi]\label{t.hayashi}\index{lemme de connexion}\index{connecting lemma}
Soit $\cU$ un voisinage de $f\in \diff^1(M)$ et $p,q,z$ trois points tels que~:
\begin{itemize}
\item[--] les ensembles d'accumulation des orbites future de $p$ et pass\'ee de $q$ contiennent $z$~;
\item[--] le point $z$ n'est pas p\'eriodique.
\end{itemize}
Il existe alors $g\in \cU$ tel que $q$ est un it\'er\'e futur de $p$.
\end{theoreme}
Ce r\'esultat est \'egalement d\'emontr\'e dans~\cite{arnaud-connecting,wen-xia-connecting}.
Il y a eu de nombreuses cons\'equences, voir par exemple~\cite{BD-newhouse,arnaud-connecting,hayashi-application,
cmp,morales-pacifico,arnaud-approximation,gan-wen-connecting} et la section~\ref{s.consequence}.
L'hypoth\`ese que $z$ n'est pas p\'eriodique sert dans la d\'emonstration \`a placer en $z$ un cube perturbatif
(fourni par le th\'eor\`eme~\ref{t.pugh}) disjoint d'un grand nombre d'it\'er\'es.

\paragraph{\it Pourquoi est-il plus difficile de connecter que de fermer~?}
Les lemmes de fermeture et de connexion semblent proches. Si l'on reprend la d\'emonstration du lemme de fermeture,
on place en $z$ une carte $\varphi$ fournie par le th\'eor\`eme~\ref{t.pugh} et dont le support est
disjoint de $N-1$ premiers it\'er\'es.
On consid\`ere un it\'er\'e futur $p'=f^{n(p)}(p)$ de $p$ et un it\'er\'e pass\'e $q'=f^{-n(q)}(q)$ de $q$, proches de $z$.
Comme pour le lemme de fermeture, nous ne pouvons pas directement perturber et construire un diff\'eomorphisme
$g$ tel que $g^N(p')=f^N(q')$, puisqu'une telle perturbation risque de briser les segments d'orbites entre $p$ et $p'$
et entre $f^N(q')$ et $q$~: nous ne serions pas certains que $q$ soit sur l'orbite positive de $p$.

Fixons un voisinage $V$ de $z$ disjoint de ses $N-1$ premiers it\'er\'es.
L'id\'ee naturelle serait de trouver comme dans l'argument combinatoire de la section~\ref{ss.closing},
deux it\'er\'es interm\'ediaires $x$ entre $p$ et $p'$ et $y$ entre $q'$ et $q$, et un
cube $C$ de la carte les contenant, tels que le cube $(1+\varepsilon).C$ ne contiennent pas
d'autre it\'er\'e de $p$ ou de $q$. Pour cela,
nous devons consid\'erer l'ensemble $\cP$ de tous les it\'er\'es interm\'ediaires
de $p$ et $q$ proches de $z$. Une difficult\'e nouvelle appara\^\i t alors~:
nous devons s\'electionner une paire de points de $\cP$, mais \`a la diff\'erence du lemme de fermeture,
cette paire doit contenir \`a la fois un  it\'er\'e de $p$ et un it\'er\'e de $q$.
Les points de $\cP$ ne sont plus tous interchangeables et l'argument pr\'ec\'edent ne fonctionne pas.

\paragraph{\it La strat\'egie d'Hayashi.}
L'argument d'Hayashi consiste \`a s\'electionner plusieurs paires qu'il faudra connecter simultan\'ement~:
on simplifie l'ensemble des retours $\cP$ en effa\c cant certains points.
Si une m\^eme orbite poss\`ede deux it\'er\'es
(par exemple deux it\'er\'es $f^{n_1}(p)$ et $f^{n_2}(p)$ de $p$) qui seraient trop proches relativement
\`a leur distance \`a la second orbite, nous pouvons dans ce cas
consid\'erer que ces deux points sont les m\^emes, oublier les it\'er\'es interm\'ediaires entre ces deux points
et esp\'erer qu'une petite perturbation permettra de les connecter.
En r\'ep\'etant cet argument, on s\'electionne un grand nombre de paires de retour et
nous devons appliquer, pour chacune d'elles, une perturbation donn\'ee par le th\'eor\`eme~\ref{t.pugh}.
Une difficult\'e est de garantir que toutes ces perturbations ont des supports disjoints.

Plus pr\'ecis\'ement, notons $\{p_0,p_1,\dots,p_r\}$
et $\{q_{-s},\dots,q_0\}$ les it\'er\'es de $p$ et de $q$ proches de $z$,
i.e. appartenant \`a l'ensemble $\cP$, et class\'es par ordre chronologique.
On extrait ensuite une sous-suite de la forme
$(x_0, y_0, x_1, y_1, \dots, x_\ell, y_\ell)$ de sorte qu'il existe pour chaque $i$
un diff\'eomorphisme $g_i=\psi_i\circ f$ qui satisfait $g_i^N(x_i)=f^N(y_i)$.
Le support de $\psi_i$ est contenu dans un petit cube $\widehat C_i $ proche de $z$ et dans les $N-1$ premiers
it\'er\'es  de $ \hat C_i $.

Supposons que~:
\begin{enumerate}
\item\label{i.s1} $x_0 = p_0$, $ y_\ell = q_0$,
\item\label{i.s2} pour tout $0\leq i<\ell$, le point $x_{i +1}$ est le premier retour de l'orbite de $y_i$ dans un voisinage de $z$.
\item\label{i.s3}
les supports des diff\'erentes perturbations sont deux \`a deux disjoints.
\end{enumerate}
En composant l'ensemble des perturbations $g_i$ de $f$,
on obtient alors un diff\'eomorphisme $ g = \psi_0 \circ \dots \circ \psi_{N-1} \circ f\in \cU$
envoyant $p$ sur $q$ par it\'erations positives.
Apr\`es perturbation, le segment d'orbite entre $ p $ et $ q $
est plus court que la pseudo-orbite initiale
$(p, f(p), \dots, f^{n(p)}(p)=p', f(q')=f^{-n(q)+1}(q), \dots, f^{-1}(q ), q)$,
voir la figure~\ref{f.combinatoire}.
\begin{figure}[ht]
\begin{center}
\input{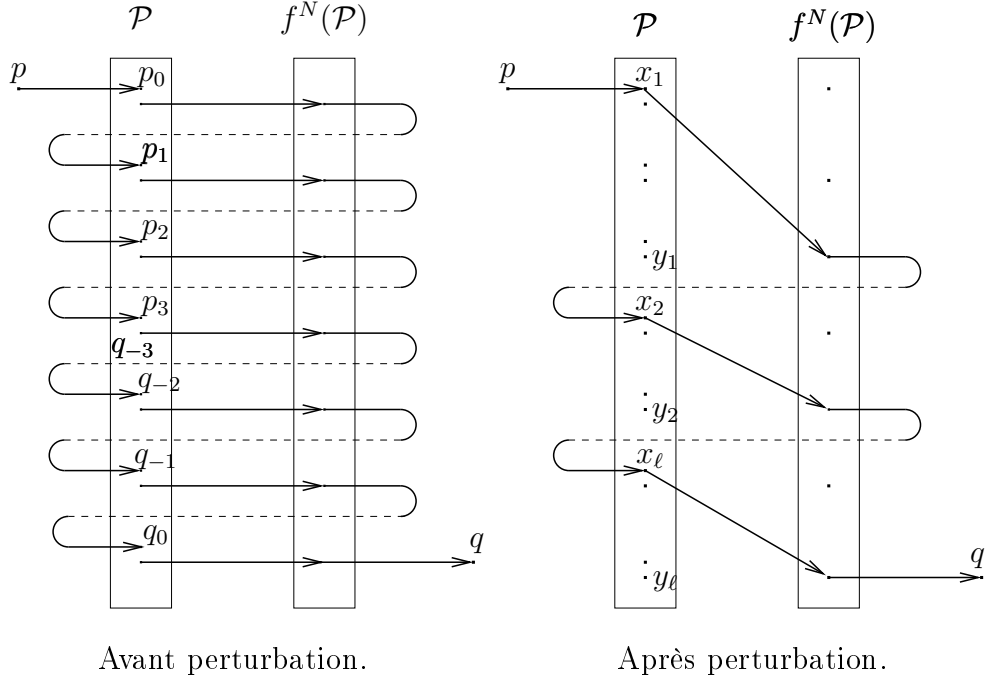}
\end{center}
\caption{Combinatoire des perturbations r\'ealis\'ee par le lemme de connexion. \label{f.combinatoire}}
\end{figure}

Le plus difficile est de choisir la sous-suite $(x_0, y_0, x_1, y_1, \dots, x_\ell, y_\ell)$.
Elle s'obtient apr\`es deux s\'elections.

\paragraph{\it Premi\`ere s\'election~: les cubes quadrill\'es}
Nous pavons tout d'abord un voisinage de $z$ par des cubes de la carte $\varphi$, comme illustr\'e en figure~\ref{f.selection1}.
Le voisinage lui-m\^eme est un cube $\hat C_0$ et le cube central $C_0=\frac 1 3 .\hat C_0$ du pavage contient $z$.
Nous pouvons supposer que les it\'er\'es $p'$ et $q'$ de $p$ et de $q$ sont contenus dans $C_0$.
Nous notons $\cP$ l'ensemble des retours furturs de $p$ et pass\'es de $q$ dans $\hat C_0$.
Le pavage permet de d\'eterminer s'il y a accumulation de points de $\cP$ dans une r\'egion du cube $\hat C_0$.
\begin{figure}[ht]
\begin{center}
\input{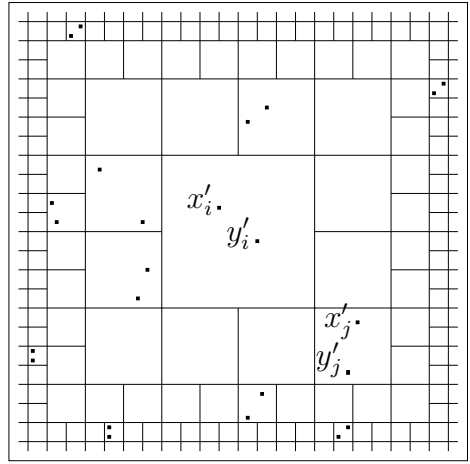}
\end{center}
\caption{Un cube quadrill\'e et la premi\`ere s\'election. \label{f.selection1}}
\end{figure}

Nous extrayons une premi\`ere sous-suite $(x'_0,y'_0,x'_1,y'_1,\dots,x'_{\ell'},y'_{\ell'})$
de $(p_0,\dots,p_r,q_{-s},\dots,q_0)$ satisfaisant les propri\'et\'es~\ref{i.s1} et~\ref{i.s2} ci-dessus, de sorte que
\begin{itemize}
\item[--] pour tout $0\leq i\leq \ell'$, les points $x'_i$ et $y'_i$ appartiennent
\`a une m\^eme tuile $T_i$ du pavage de $\hat C_0$,
\item[--] chaque tuile du pavage contient au plus une paire $(x'_i,y'_i)$.
\end{itemize}
Nous avons utilis\'e ici de fa\c con essentielle que les points $p'$ et $q'$ appartiennent \`a une m\^eme tuile du pavage.
La suite extraite est repr\'esent\'ee en figure~\ref{f.selection1}.

En appliquant le th\'eor\`eme~\ref{t.pugh}, on d\'efinit alors une suite de perturbations $g'_i $
telles que $ (g'_i) ^ N (x'_i) = f ^ N (y'_i) $.
Elle ne permet pas de conclure la d\'emonstration~:
si $x_i,y_i$ appartiennent \`a une tuile $T$ du pavage,
le support de la perturbation $g_i$ est contenu dans le cube $ (1+\varepsilon).T$
et dans ses $N-1$ premiers it\'er\'es.
S'il existe deux perturbations $g'_i,g'_j$ associ\'ees \`a deux tuiles adjacentes, les supports des perturbations correspondantes risquent
de s'intersecter et nous ne pouvons pas composer les perturbations.

\paragraph{\it Seconde s\'election~: les raccourcis.}
Supposons que le support
de deux perturbations $g'_i$ et $g'_j$ d\'efinies ci-dessus se chevauchent
(remarquons que les supports peuvent \^etre disjoints dans $\hat C_0$ et se chevaucher dans des it\'er\'es
$f^k(\hat C_0)$ pour certains $0\leq k <N$).
Cela implique que les points $(x'_i, y'_i)$ (dans la tuile $T_i$)
et les points $(x'_j, y'_j ) $ (dans la tuile $T_j $) ont des images proches pour
un it\'er\'e $f^k$.
Nous fixons un tel entier $k$ et
nous supposerons $ i <j $.

Afin de r\'esoudre le conflit, nous rempla\c cons les deux perturbations $g'_i,g'_j$
qui envoient respectivement $x'_i$ sur $f^N(y'_i)$ et $x'_j$ sur $f^N(y'_j)$,
par une seule perturbation envoyant $x'_i$ sur $f ^ N (y'_j) $ (figure~\ref{f.raccourci})~: il suffit pour cela de conserver la perturbation
$g'_i$ sur les it\'er\'es $f^\ell(\hat C_0)$ pour $0\leq \ell<k$,
de conserver la perturbation $g'_j$
sur les it\'er\'es $f^\ell(\hat C_0)$ pour $k< \ell<N$,
et d'utiliser sur $f^k(\hat C_0)$ une perturbation \'el\'ementaire
qui envoie ${g'}_i^k(x'_i)$ sur ${g'}_j^{k+1}(y'_j)$.
\`A l'issue de cette construction, on efface les paires de points interm\'ediaires $(x'_s, y'_s)$ avec $ s \in \{i +1, \dots, j-1 \}$
de la suite $ (x'_0, y'_0, x'_1, y'_1, \dots, x'_{\ell'}, y'_{\ell'}) $, pour obtenir
une nouvelle pseudo-orbite qui joint $p$ \`a $q$.
En d'autres termes, nous avons r\'ealis\'e un raccourci au sein de la pseudo-orbite obtenue apr\`es la premi\`ere s\'election.

\begin{figure}[ht]
\begin{center}
\input{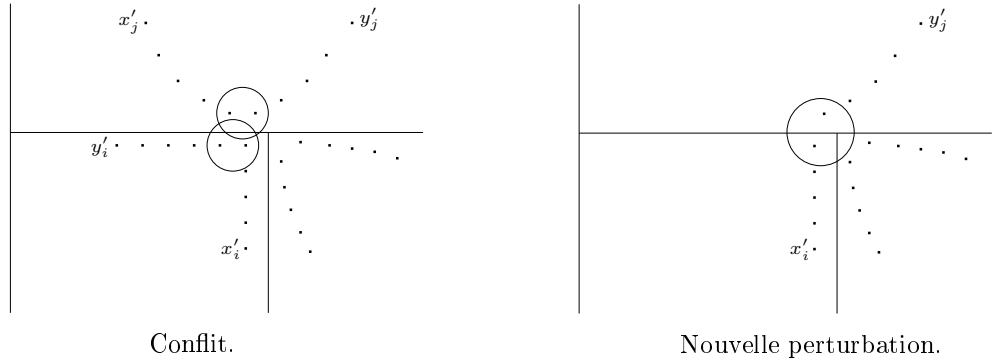}
\end{center}
\caption{Un raccourci. \label{f.raccourci}}
\end{figure}

Les supports des perturbations $g'_i$ et $g'_j$ sont contenus dans des boules de rayon $\eta$ fois plus petites que les distances
$d(\partial T_i, \partial ((1+\varepsilon).T_i))$ et
$d(\partial T_j, \partial ((1+\varepsilon).T_j))$ respectivement.
En ayant choisi $\varepsilon$ et $\eta$ suffisamment petits, on en d\'eduit que
les tuiles $T_i$ et $T_j$ sont adjacentes et que le support de la nouvelle
perturbation reste petit par rapport aux tailles des 
tuiles $T_i,T_j$ et de leur $N$ premiers it\'er\'es.

Nous poursuivons ces modifications tant que subsiste des conflits entre
perturbations. Remarquons qu'\`a chaque fois qu'un conflit est lev\'e,
nous obtenons une nouvelle perturbation de support l\'eg\`erement
plus large que le support des perturbations pr\'ec\'edentes.
Par cons\'equent le support de la nouvelle perturbation peut \`a son tour
rencontrer le support d'une troisi\`eme perturbation $g'_k$.
On peut se demander si le nombre de conflits successifs que l'on rencontre
peut \^etre arbitrairement grand, de sorte que le support de la perturbation
finale devienne bien plus important que la taille de la tuile initiale $T_i$.

Ce n'est pas le cas~: nous contr\^olons a priori la taille des perturbations
si bien que les conflits qui apparaissent sont toujours associ\'es \`a des tuiles
adjacentes \`a la tuile initiale $T_i$. Puisque a g\'eom\'etrie du pavage
est uniforme, le nombre de tuiles adjacentes \`a $T_i$ est born\'e
(par $4^d$) et partant de la perturbation initiale $g'_i$,
il y aura au plus $4^d$ conflits \`a r\'esoudre.
Si l'on choisi la constante $\eta$ suffisamment petite, le support perturbation
que l'on obtient apr\`es $4^d$ conflits reste petit par rapport \`a
la taille de la tuile $T_i$, justifiant ainsi le contr\^ole a priori
des perturbations (voir la figure~\ref{f.conflit}).
\begin{figure}[ht]
\begin{center}
\input{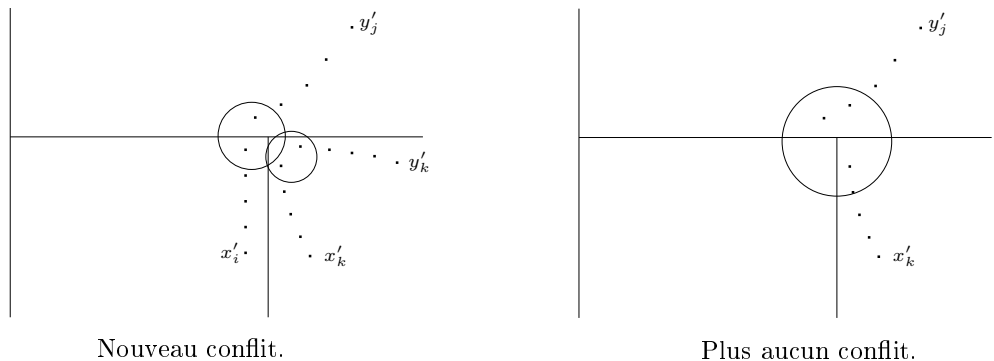}
\end{center}
\caption{Le nombre de conflits est born\'e. \label{f.conflit}}
\end{figure}

Apr\`es traitement des diff\'erents conflits, nous obtenons
une nouvelle pseudo-orbite, une nouvelle suite
$(p_0=x_0,y_0,x_1,y_1,\dots,x_\ell,y_\ell=q_0)$
et une collection de perturbations $(g_i)$
dont les supports sont deux-\`a-deux disjoints
et telle que pour tout $i$ on ait $g_i^N(x_i)=y_i$.

\section{Espaces de perturbation}\label{s.espace}
Les perturbations r\'ealis\'ees dans la preuve du lemme de Pugh et des diff\'erents lemmes de perturbations
qui en d\'ecoulent sont une succession de perturbations \'el\'ementaires
donn\'ees par le lemme~\ref{l.perturbation-elementaire} et de supports disjoints~:
nous appelons \emph{support}\index{support, d'une perturbation} d'un diff\'eomorphisme $h$ l'ensemble des points $x\in M$ tels que $h(x)\neq x$~;
le support d'une perturbation $g$ de $f$ est le support du diff\'eomorphisme $h=f\circ g^{-1}$.

Cette remarque permet de travailler en restant dans des
sous-espaces $\cS$ de $\diff^1(M)$, pourvu que les deux propri\'et\'es suivantes soit v\'erifi\'ees.
\begin{itemize}
\item[--] Le lemme~\ref{l.perturbation-elementaire} de perturbation \'el\'ementaire s'applique au sein de $\cS$.
\item[--] Tout diff\'eomorphisme $f\in \cS$ poss\`ede une base de voisinages $\cU\subset \cS$ \emph{flexibles}, i.e.
satisfaisant pour tous diff\'eomorphismes $h_{1},h_{2}\in \diff^1(M)$ \`a supports disjoints
$$f\circ h_{1}\in \cU \text{ et } f\circ h_{2}\in \cU \Longrightarrow f\circ h_{1}\circ h_{2}\in \cU.$$ 
\end{itemize}

En particulier, le lemme de Pugh reste vrai si l'on travaille avec des diff\'eomorphismes conservatifs
(i.e. v\'erifiant une forme volume ou symplectique), ou m\^eme avec des diff\'eomorphismes de classe $C^r$,
$r\geq 1$, pourvu que la topologie consid\'er\'ee soit la topologie $C^1$.

\section{Probl\`emes}

Il existe tr\`es peu de r\'esultats perturbatifs en classe $C^r$, $r>1$. 

\begin{question}
Existe-t-il un lemme de fermeture en r\'egularit\'e sup\'erieure~?
\end{question}
Des difficult\'es ont \'et\'e mises en \'evidence~\cite{pugh-against,gutierrez-against,herman-closing1,herman-closing2,herman-problems}.
Des cas particuliers ont \'et\'e trait\'es~\cite{pugh-closing-smooth}, ainsi que le cas des flots
sur les surfaces~\cite{peixoto,gutierrez-surface,gutierrez-survey}.
\smallskip

On peut se demander quels r\'esultats persistent pour la dynamique des endomorphismes,
i.e. des applications diff\'erentiables non injectives.
Wen a \'etendu le lemme de fermeture \`a ce cadre~\cite{wen-endo},
mais il n'existe pas de lemme de connexion.

\begin{question}
Existe-t-il un lemme de connexion pour les endomorphismes~?
\end{question}

Pour les diff\'eomorphismes, le lemme de connexion reste valide lorsque $z$ est p\'eriodique, si son
orbite est hyperbolique, ou plus g\'en\'eralement s'il n'y a pas de r\'esonance
non triviale entre ses valeurs propres (voir la remarque~\ref{r.pseudo} plus bas).
Il serait int\'eressant pour les applications de ne plus avoir d'hypoth\`ese
sur le point $z$.

\begin{question}
Le lemme de connexion reste-t-il v\'erifi\'e lorsque le point
$z$ est p\'eriodique~?
\end{question}

Hayashi a montr\'e~\cite{hayashi-application} un ``make or break lemma''~:
si l'orbite future de $x$ et l'orbite pass\'ee de $y$ ont un point d'accumulation commun $z$,
on peut alors perturber la dynamique en topologie $C^1$ pour disjoindre ces ensembles d'accumulation, ou bien pour que $y$ soit un it\'er\'e de $x$.

Terminons ce chapitre en remarquant qu'en l'absence de compacit\'e le lemmes de connexion est en g\'en\'eral faux~\cite{pugh-connecting}.
M.-C. Arnaud a donn\'e une version non compacte~\cite{arnaud-connecting} sous l'hypoth\`ese additionnelle que l'orbite du point $z$ a des points d'accumulation dans $M$.


\chapter{Connexions de pseudo-orbites}
Nous pr\'esentons le lemme de connexion pour les pseudo-orbites
d\'emontr\'e dans~\cite{bc,symplectic}).
La d\'emonstration n\'ecessite de construire une section de la dynamique~:
on \'enonce pour cela un th\'eor\`eme d'existence de ``tours topologiques'',
qui peut \^etre utile pour d'autres applications (voir par exemple
le chapitre~\ref{c.centralisateur}).

\section{\'Enonc\'e du lemme de connexion pour les pseudo-orbites}
Rappelons que si $K\subset M$ est un ensemble compact
et $x,y\in K$ deux points, nous notons
$x\dashv_K y$ lorsque pour tout $\varepsilon>0$ il existe une
$\varepsilon$-pseudo-orbite
$z_0=x,z_1,\dots,z_n=y$ avec $n\geq 1$ et contenue dans $K$.

\begin{theoreme}[Lemme de connexion pour les pseudo-orbites, Bonatti-Crovisier \index{lemme de connexion!pour les pseudo-orbites}]\label{t.bc}
Soit $\cU$ un voisinage d'un diff\'eomorphisme $f\in \diff^1(M)$ dont les points p\'eriodiques sont hyperboliques.
Pour tous points $x,y\in M$ satisfaisant $x\dashv y$,
il existe une perturbation $g\in \cU$ de $f$ et $n\geq 1$ tels que $g^n(x)=y$.
\end{theoreme}

\begin{remarques}\label{r.pseudo}
\begin{enumerate}
\item Si l'on consid\`ere un ensemble compact $K$ tel que $x\dashv_K y$,
alors pour tout voisinage $W$ on peut demander que la perturbation $g$
soit \`a support dans $W$.
\item Dans~\cite{symplectic},
nous avons affaibli l'hypoth\`ese sur le diff\'eomorphisme.
Il suffit de supposer que pour toute orbite p\'eriodique, \emph{il n'y a pas
de r\'esonance non triviale au sein des valeurs propres de module $1$}.

Plus pr\'ecis\'ement, pour toute orbite p\'eriodique, il n'y a pas de valeur propre qui soit racine
de l'unit\'e, les valeurs propres de module $1$ sont simples, et il n'y a pas de relation
de la forme
$$\lambda_{0}=\prod_{j=1}^{s}\lambda_j^{k_{j}},$$
o\`u les $k_{j}$ sont des entiers strictement positifs et les
$\lambda_{0},\bar \lambda_{0}, \lambda_{1},\bar \lambda_{1},\dots,\lambda_{s},\bar \lambda_{s}$
sont des valeurs propres de module $1$ toutes distinctes.
\end{enumerate}
\end{remarques}

\section{Id\'ee de la preuve}
On ne peut clairement pas esp\'erer connecter deux points li\'es par
des pseudo-orbites en effectuant simplement une perturbation locale~:
nous allons devoir perturber ind\'ependemment dans plusieurs
r\'egions en utilisant les cubes quadrill\'es fournis par la d\'emonstration
du lemme de connexion d'Hayashi. En s'assurant que les supports des
diff\'erentes perturbations sont disjoints, la section~\ref{s.espace} garantit
que la perturbation finale reste petite.

\subsection{Les domaines de perturbations\index{domaine de perturbation}}
\label{ss.domaine}
Consid\'erons une carte $\varphi\colon V\to \RR^d$.
Un \emph{domaine quadrill\'e
selon les coordonn\'ees de $\varphi$}
est la donn\'ee d'un ouvert $U\subset V$
et d'une famille $\cC$ de cubes de $\varphi$
(appel\'es \emph{tuiles du domaine}) tels que
\begin{enumerate}
\item les int\'erieurs des tuiles sont deux \`a
deux disjoints~;
\item l'union des tuiles de $\cC$ est \'egale \`a $U$~;
\item la g\'eom\'etrie du quadrillage est born\'ee, i.e.
\begin{itemize}
\item[--]
le nombre de tuiles est uniform\'ement born\'e (par $2^d$)
au voisinage de chaque point,
\item[--]
pour toute paire $(C,C')$
de tuiles adjacentes, le rapport de leur rayon est uniform\'ement
born\'e (par $2$).
\end{itemize}
\end{enumerate}
Par une construction standard, tout ouvert $U\subset V$ peut \^etre quadrill\'e
selon les cordonn\'ees de $\varphi$ (voir la figure~\ref{f.domaine}).
\begin{figure}[ht]
\begin{center}
\input{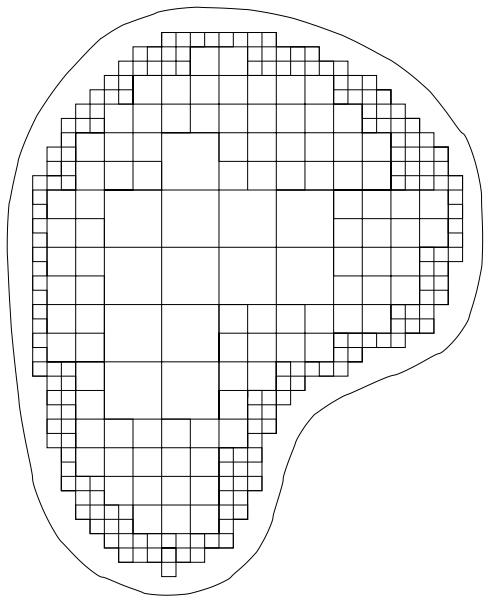}
\end{center}
\caption{Domaine quadrill\'e. \label{f.domaine}}
\end{figure}
\medskip

Une pseudo-orbite $(z_0,\dots,z_n)$ est \emph{\`a sauts dans les tuiles}
du domaine quadrill\'e $(U,\cC)$ si pour tout $0\leq i<n$
les points $f(z_i), z_{i+1}$ co\"\i ncident ou appartiennent
\`a une m\^eme tuile de $\cC$.
\medskip

Finalement, si l'on fixe un voisinage $\cU\subset \diff^1(M)$ de $f$,
un domaine quadrill\'e $(U,\cC)$ et un entier $N\geq 1$,
on dit que $(U,\cC,N)$ est un \emph{domaine de perturbation
pour $(f,\cU)$} si~:
\begin{enumerate}
\item $U$ est disjoint de ses $N-1$ premiers it\'er\'es par $f$~;
\item pour toute pseudo-orbite $(z_0,\dots,z_n)$
avec $z_0\in U$ et $z_n\in f^N(U)$ \`a sauts dans les tuiles $\cC$,
il existe une perturbation $g\in \cU$ de $f$ \`a support dans
$U\cup f(U)\cup\dots\cup f^{N-1}(U)$ et un entier $m\in \{1,\dots, n\}$
tel que $g^m(z_0)=z_n$.
\end{enumerate}
\medskip

Le lemme de perturbation de Pugh (th\'eor\`eme~\ref{t.pugh})
et le lemme de connexion d'Hayashi (th\'eor\`eme~\ref{t.hayashi})
assurent l'existence de domaines de perturbation.
\begin{theoreme}[Existence de domaines de perturbation]
Pour tout voisinage $\cU$ de $f$, il existe un entier $N\geq 1$
et, en tout point $p\in M$, il existe une carte $\varphi\colon V\to \RR^d$
tels que pour tout domaine quadrill\'e $(U,\cC)$ selon les coordonn\'ees
de $V$ qui est disjoint de ses $N-1$ premiers it\'er\'es,
$(U,\cC,N)$ est un domaine de perturbation de $(f,\cU)$.
\end{theoreme}
Le fait que l'entier $N$ ne d\'epende pas du point $p$
sera crucial dans la suite de la d\'emonstration du th\'eor\`eme~\ref{t.bc}.

\subsection{Les tours topologiques\index{tour topologique}}
Afin de traiter des pseudo-orbites arbitraires, nous
devons \^etre en mesure de construire une collection
de domaines de perturbation deux \`a deux disjoints qui couvrent
l'espace des orbites de la dynamique.
Le r\'esultat suivant permet de construire une telle section de la dynamique.

\begin{theoreme}[Tours topologiques, Bonatti-Crovisier]\label{t.tour}
Il existe $\kappa_d>0$ (ne d\'ependant que de la dimension $d$
de la vari\'et\'e $M$) tel que pour tout $m\geq 1$
et tout diff\'eomorphisme $f\in \diff^1(M)$ n'ayant
pas de point p\'eriodique non hyperbolique de p\'eriode inf\'erieure
\`a $\kappa_d.m$, il existe un ouvert $U\subset M$ et un ensemble compact
$D\subset U$ ayant les propri\'et\'es suivantes~:
\begin{itemize}
\item[--] tout point qui n'est pas p\'eriodique de p\'eriode inf\'erieure
\`a $m$ poss\`ede un it\'er\'e dans $D$~;
\item[--] les ouverts $U,\dots,f^{m-1}(U)$ sont deux \`a deux disjoints
\end{itemize}
On peut choisir $U$ pour que le diam\`etre de ses composantes connexes soit arbitrairement petit.
\end{theoreme}
La d\'emonstration utilise le lemme de transversalit\'e de Thom et
se ram\`ene \`a un probl\`eme combinatoire (exprim\'e en termes de coloriage).

\subsection{Lorsqu'il n'y a pas d'orbite p\'eriodique de basse p\'eriode}
Le lemme de perturbation de Pugh donne un entier $N\geq 1$
et permet de couvrir la vari\'et\'e $M$ par une famille finie
de cartes $\varphi_i\colon V_i\to \RR^d$.
Nous supposerons pour simplifier que $f$ n'a pas de point p\'eriodique
de p\'eriode inf\'erieure \`a $3\kappa_d.N$.
Fixons deux points $x,y\in M$ tels que $x\dashv y$.

Consid\'erons un ouvert $U$ disjoint de $3N-1$ it\'er\'es
et un ensemble compact $D\subset U$ donn\'es
par le th\'eor\`eme~\ref{t.tour}.
Quitte \`a remplacer $U$ par $f^N(U)$ ou par $f^{2N}(U)$,
les points $x,y$ n'appartiennent pas \`a $U,f(U),\dots, f^{N-1}(U)$.
Puisque les composantes de $U$ sont de diam\`etre petit, nous pouvons
supposer que chacune d'entre elles est contenue dans l'une des cartes
$V_i$. Finalement, nous quadrillons chaque composante de $U$
selon une des cartes $\varphi_i$ et nous construisons ainsi~:
\begin{itemize}
\item[--] une famille $U_1,\dots,U_s$ de domaines de perturbations
tels que $f^k(U_i), f^\ell(U_j)$ sont disjoints pour tous $0\leq k,\ell<N$
et tous $i\neq j$,
\item[--] une famille finie de tuiles $\cT$ contenue dans l'union des tuiles
des domaines $U_i$, $1\leq i\leq s$,
\end{itemize}
tels que toute orbite rencontre l'une des tuiles de $\cT$.

En travaillant plus, on obtient par un argument de compacit\'e~:
\begin{itemize}
\item[--] pour chaque tuile $T\in \cT$, un ensemble compact $\Delta\subset T$,
\item[--] un entier $n_0\geq 1$ et une constante $\varepsilon_0>0$
tels que toute $\varepsilon_0$-pseudo-orbite de longueur $n_0$
rencontre l'un des ensembles $\Delta$.
\end{itemize}
\medskip

Consid\'erons une $\varepsilon$-pseudo-orbite $x=\tilde z_0,\tilde z_1,\dots,\tilde z_n=y$.
Si $y$ n'est pas d\'ej\`a un it\'er\'e positif de $x$, en prenant
$\varepsilon>0$ suffisamment petit, nous pouvons supposer que
la pseudo-orbite est de longueur sup\'erieure \`a $n_0$.
Nous pouvons diff\'erer les sauts de la pseudo-orbite (sur des intervalles
de temps inf\'erieurs \`a $n_0$) et supposer qu'il n'ont lieu
qu'aux points appartenant aux ensembles compacts $\Delta$.
Si $\varepsilon$ a \'et\'e choisi petit, ceci assure la propri\'et\'e suivante.
\begin{lemme}\label{l.saut}
Il existe
une pseudo-orbite $x=z_0,z_1,\dots,z_n=y$
entre $x$ et $y$ \`a sauts dans les tuiles des
domaines quadrill\'es $U_1,\dots,U_s$.
\end{lemme}

Il ne reste plus qu'\`a perturber successivement
dans chaque domaine de perturbation pour supprimer finalement l'ensemble des
sauts de la pseudo-orbite et obtenir une segment d'orbite
(en g\'en\'eral plus court) qui joint $x$ \`a $y$.

\subsection{Lorsqu'il y a des points p\'eriodiques}
En pr\'esence de points p\'eriodiques de basse p\'eriode,
la tour topologique ne couvre plus l'ensemble des orbites.
De plus, lorsqu'une orbite s'approche d'une orbite p\'eriodique,
le temps de retour dans la tour peut devenir arbitrairement long.

Dans ce cas, on utilise \`a nouveau que les points p\'eriodiques sont
hyperboliques. Il existe une version du th\'eor\`eme~\ref{t.tour}
autorisant l'existence de points p\{eriodiques hyperboliques de petite p\'eriode: il s'applique aux points qui n'appartiennent pas aux vari\'et\'es
invariantes des points p\'eriodiques de petite p\'eriode.
On rajoute de nouveaux domaines de perturbations qui couvrent
des domaines fondamentaux des ensembles stables et instables des orbites
p\'eriodiques de basse p\'eriode. Le lemme~\ref{l.saut} est encore v\'erifi\'e,
ce qui permet de conclure comme dans le cas pr\'ec\'edent.

\section{Cons\'equences imm\'ediates}\label{s.consequence}

Il existe un G$_\delta$ dense de $\diff^1(M)$ form\'e de
diff\'eomorphismes pour lesquel les propri\'et\'es suivantes sont v\'erifi\'ees.

\paragraph{a) Lieu de r\'ecurrence.}
\emph{Les ensembles $\overline{\per(f)}$, $\Omega(f)$ et $\cR(f)$ co\"\i ncident.}

Il n'y a donc qu'une seule notion de r\'ecurrence.
Rappelons que le lemme de fermeture de Pugh permettait d\'ej\`a de conclure
$\overline{\per(f)}=\Omega(f)$.

\paragraph{b) Ordre dynamique.}
\emph{Pour tout ensemble compact $K$, les relations $\prec_K$
et $\dashv_K$ co\"\i ncident.}

En particulier, les ensembles faiblement transitifs et les ensembles
transitifs par cha\^\i nes co\"\i ncident~;
les classes de r\'ecurrence par cha\^\i nes sont
les ensembles faiblement transitifs maximaux.
On en d\'eduit aussi que la relation $\prec_K$
est transitive (ce qui avait \'et\'e montr\'e pr\'ec\'edemment par Arnaud~\cite{arnaud-connecting}, Gan et Wen~\cite{gan-wen-connecting}).

\paragraph{c) Quasi-attracteurs.}
\emph{L'ensemble des points dont
l'ensemble $\omega$-limite \index{ensemble! $\omega$-limite}
\emph{(i.e. l'ensemble des valeurs d'adh\'erence
de son orbite future)} est un quasi-attracteur contient un G$_\delta$ dense de $M$.}

\emph{Les quasi-attracteurs sont exactement
les classes de r\'ecurrence par cha\^\i nes stables au sens de Lyapunov.}

\emph{Une classe homocline $H(p)$ est un quasi-attracteur si et seulement elle contient la vari\'et\'e instable de $p$.}

Ceci r\'epond \`a une conjecture de Hurley~\cite{hurley} dans le cas $C^1$
et am\'eliore des r\'esultats ant\'erieurs de Arnaud~\cite{arnaud-connecting}, Morales et Pacifico~\cite{morales-pacifico}.

Comme cons\'equence, une classe de r\'ecurrence par cha\^\i nes d'int\'erieur non vide
est un quasi-attracteur \`a la fois pour $f$ et $f^{-1}$.

\paragraph{d) Classes homoclines et classes ap\'eriodiques.}
\emph{Les classes de r\'ecurrence par cha\^\i nes contenant une orbite p\'eriodique
sont des classes homoclines.}

\emph{Si $O_1,O_2$ sont deux orbites p\'eriodiques telles que $O_1\dashv O_2$
et si l'indice de $O_1$ est inf\'erieur ou \'egal \`a celui de $O_2$,
alors $W^u(O_1)$ et $W^s(O_2)$ ont un point d'intersection transverse $z$}~:
$$T_zM=T_zW^u(O_1)+ T_zW^s(O_2).$$
Les classes de r\'ecurrence par cha\^\i nes sans orbite p\'eriodique
sont appel\'ees \emph{classes ap\'eriodiques}\index{classe!ap\'eriodique}.
\medskip

On obtient facilement d'autres propri\'et\'es.
\begin{itemize}
\item[--] \emph{Toute composante connexe de $\interior(\Omega(f))$
est enti\`erement contenue dans une classe homocline.}
\item[--] \emph{Deux classes homoclines sont toujours disjointes ou confondues.}
\item[--] \emph{Si deux orbites p\'eriodiques $O_1$ et $O_2$ ont m\^eme indice et sont contenus dans une m\^eme
classe de r\'ecurrence par cha\^\i nes, elles sont homocliniquement reli\'es.}
\item[--] \emph{Si $O_1$ et $O_2$ ont des indices diff\'erents et sont contenues dans une m\^eme
classe de r\'ecurrence par cha\^\i nes, il existe une perturbation $g\in \diff^1(M)$
de $f$ telle que les continuations hyperboliques de $O_1$ et $O_2$ forment un
``cycle h\'et\'erodimensionnel''}~: la vari\'et\'e instable de l'une rencontre la vari\'et\'e stable de l'autre
(voir la section~\ref{s.cycle}).
\end{itemize}
\medskip

Des versions ant\'erieures plus faibles n'utilisant que le lemme de connexion d'Hayashi avaient
\'et\'e donn\'ees par Arnaud~\cite{arnaud-connecting}, Bonatti et D\'\i az~\cite{BD-newhouse}, Carballo, Morales et Pacifico~\cite{cmp},
Gan et Wen~\cite{gan-wen-connecting}.
\medskip

\paragraph{e) Classes isol\'ees.}
\emph{Toute classe de r\'ecurrence par cha\^\i nes isol\'ee dans $\cR(f)$
est une classe homocline $H(p)$. Pour tout diff\'eomorphisme $g$ proche de $f$,
la classe de r\'ecurrence par cha\^\i nes contenant la continuation $p_g$ de $p$
est encore isol\'ee.}

\emph{Si $f$ n'a qu'un nombre fini de classes de r\'ecurrence par cha\^\i nes,
tout diff\'eomorphisme proche de $f$ a le m\^eme nombre de classes que $f$.}
\medskip

La premi\`ere version de ces propri\'et\'es avait \'et\'e donn\'ee par
Abdenur~\cite{abdenur-decomposition,abdenur-attractor}.

\section{Exemples}
Les exemples connus de dynamiques $C^1$-g\'en\'eriques non hyperboliques
ont lieu en dimension sup\'erieure ou \'egale \`a $3$.
Nous en citons quelques uns.

\paragraph{a) Dynamiques robustement transitives.} \index{transitivit\'e!robuste}
Il existe des dynamiques
non hyperboliques ayant une unique classe de r\'ecurrence par cha\^\i nes (voir la section~\ref{s.exemple-mane}).
\begin{theoreme}[Shub~\cite{shub-transitif}]
Il existe un ouvert non vide $\cU\subset \diff^1(\TT^4)$ de diff\'eomorphismes transitifs
et non hyperboliques.
\end{theoreme}
Cet exemple est obtenu en modifiant un diff\'eomorphisme d'Anosov.
D'autres exemples similaires sur d'autres vari\'et\'e ont \'et\'e donn\'es ensuite, voir~\cite{mane-contribution,bv} et~\cite[chapitre 7]{bdv}.
Une autre m\'ethode de construction, due \`a Bonatti-D\'\i az~\cite{blender} consiste \`a perturber l'application
``temps $1$'' d'un flot d'Anosov ou \`a perturber le produit d'un diff\'eomorphisme d'Anosov
et de l'application identit\'e d'une vari\'et\'e compacte.

\paragraph{b) Cycles h\'et\'erodimensionnels.}\index{cycle!h\'et\'erodimensionnel}
L'obstruction \`a l'hyperbolicit\'e dans l'exemple pr\'ec\'edent
provient de l'existence (dans une m\^eme classe de r\'ecurrence par cha\^\i nes) de points p\'eriodiques d'indices diff\'erents. De tels exemples existent \'egalement pour des dynamiques qui ne sont pas
transitives (voir la section~\ref{s.exemple-abraham-smale}).

\begin{theoreme}[Abraham-Smale~\cite{abraham-smale}, Simon~\cite{simon}, Bonatti-D\'\i az~\cite{BD-newhouse}]
\label{t.exemple-cycle}
Pour toute vari\'et\'e compacte $M$ de dimension $d\geq 3$, il existe
un ouvert non vide $\cU\subset \diff^1(M)$ de diff\'eomorphismes non transitifs
et deux familles continues de points p\'eriodiques hyperboliques
$(p_g)_{g\in \cU}$, $(q_g)_{g\in \cU}$ d'indices diff\'erents, ayant la m\^eme classe homocline
pour tout $g\in \cU$.
\end{theoreme}
La construction exploite le fait que l'ensemble stable d'un ensemble hyperbolique
peut contenir des sous-vari\'et\'es de dimension plus grande que la dimension du fibr\'e stable.
Le m\^eme r\'esultat peut s'obtenir en utilisant les m\'elangeurs\index{m\'elangeur}
introduits par Bonatti et D\'\i az~\cite{blender,BD-newhouse}.
\medskip

\paragraph{c) Ph\'enom\`ene de Newhouse.}\index{ph\'enom\`ene de Newhouse}
Newhouse a construit~\cite{newhouse-thesis, newhouse-phenomenon} des exemples de dif\-f\'e\-o\-mor\-phis\-mes de surface $C^2$-g\'en\'eriques
ayant un comportement pathologique, aujourd'hui appel\'e \emph{ph\'enom\`ene de Newhouse}~:
pour toute surface, il existe un ouvert non vide $\cU$ de l'espace des diff\'eomorphismes
de classe $C^2$ et un G$_\delta$ dense $\cG$ de $\cU$ tel que tout diff\'eomorphisme $f\in \cG$
poss\`ede une infinit\'e de puits ou de sources.
Bonatti et D\'\i az ont ensuite utilis\'e les m\'elangeurs
pour obtenir le m\^eme r\'esultat en topologie $C^1$ sur les vari\'et\'es de dimension sup\'erieure ou \'egale
\`a $3$ (voir aussi la section~\ref{s.exemple-abraham-smale}).
\begin{theoreme}[Bonatti-D\'\i az~\cite{BD-newhouse}]\label{t.newhouse}
Pour toute vari\'et\'e compacte $M$ de dimension $d\geq 3$, il existe
un ouvert non vide $\cU\subset \diff^1(M)$ et un G$_\delta$ dense de $\cU$
form\'e de diff\'eomorphismes ayant une infinit\'e de puits et de sources.
\end{theoreme}

\paragraph{d) Classes ap\'eriodiques.}
Bonatti et D\'\i az ont \'egalement montr\'e que
les classes ap\'eriodiques peuvent appara\^\i tre parmi les diff\'eomorphismes $C^1$-g\'en\'eriques
(voir aussi la section~\ref{s.universelle}). Leur construction montre aussi que l'on peut trouver
des classes de r\'ecurrence par cha\^\i nes (des classe homoclines ainsi que des classes ap\'eriodiques)
qui sont accumul\'ees en topologie de Hausdorff
par des classes de r\'ecurrence par cha\^\i nes non triviales (ce ne sont pas
des orbites p\'eriodiques isol\'ees)~;
pour ces dynamiques l'ensemble des classes de r\'ecurrence par cha\^\i nes est non d\'enombrable.

\begin{theoreme}[Bonatti-D\'\i az~\cite{aperiodic}]
Pour toute vari\'et\'e $M$ de dimension sup\'erieure ou \'egale \`a $3$,
il existe un ouvert non vide $\cU\subset \diff^1(M)$ et un G$_\delta$ dense de $\cU$
form\'e de diff\'eomorphismes ayant un nombre non d\'enombrable de classes ap\'eriodiques.
\end{theoreme}

\paragraph{e) Absence d'attracteurs.}
Il est possible de garantir que tous les quasi-attracteurs sont accumul\'es par des sources~:
dans ce cas, il n'y a pas de classe de r\'ecurrence par cha\^\i nes qui est un attracteur.

\begin{theoreme}[Bonatti-Li-Yang~\cite{bly}]
Pour toute vari\'et\'e $M$ de dimension sup\'erieure ou \'egale \`a $3$,
il existe un ouvert non vide $\cU\subset \diff^1(M)$ et un G$_\delta$ dense de $\cU$
form\'e de diff\'eomorphismes n'ayant pas d'attracteur~: il n'y a pas d'ouvert non vide attractif
dont l'intersection des it\'er\'es est une classe de r\'ecurrence par cha\^\i nes.
\end{theoreme}

\section{Probl\`emes}

Voici quelques questions formul\'ees pour les dynamiques $C^1$-g\'en\'eriques~:
nous demandons s'il existe un G$_\delta$
dense de $\diff^1(M)$ sur lequel on peut donner une r\'eponse affirmative
\`a chacune d'entre elle.

\paragraph{a) Structure des classes ap\'eriodiques.}
Rappelons que les classes homoclines sont toujours transitives.
\begin{question}
Les classes ap\'eriodiques sont-elle toujours transitives~?
\end{question}

Les classes ap\'eriodiques d\'ecrites par~\cite{aperiodic} sont
(voir la section~\ref{s.universelle})~:
\begin{itemize}
\item[--] minimales et uniquement ergodiques (ce sont des odom\`etres),
\item[--] Lyapunov stables pour $f$ et $f^{-1}$ (ou ``dynamiquement isol\'ees''), i.e.
il existe des voisinages ouverts arbitrairement petits $U,V$ de la classe ap\'eriodique
v\'erifiant $f(\overline U)\subset U$ et $f^{-1}(\overline V)\subset V$.
\end{itemize}
On peut se demander si c'est toujours le cas.

\paragraph{b) Cardinalit\'e de l'ensemble des classes de r\'ecurrence par cha\^\i nes.}
\begin{question}
Le nombre de classes de r\'ecurrence par cha\^\i nes est-il toujours fini ou infini non d\'enombrable~?
\end{question}

Les diff\'eomorphismes pr\'esentant le ph\'enom\`ene de Newhouse sont souvent
des exemples de dynamiques pour lesquels la cardinalit\'e de l'ensemble des classes de r\'ecurrence par cha\^\i nes n'est pas connue~:
peut-il exister un diff\'eomorphisme qui pr\'esente le ph\'enom\`ene de Newhouse et qui n'aurait qu'un nombre d\'enombrable
de classes de r\'ecurrence par cha\^\i nes~?

\paragraph{c) Robustesse des classes isol\'ees.}
Une classe de r\'ecurrence par cha\^\i ne isol\'ee dans $\cR(f)$ est une classe homocline $H(p)$.
Pour tout diff\'eomorphisme $g$ proche de $f$, la classe de r\'ecurrence par cha\^\i nes contenant la continuation
hyperbolique $p_g$ de $p$ est encore isol\'ee et
pour tout diff\'eomorphisme proche appartenant \`a un G$_\delta$ dense de $\diff^1(M)$, cette classe co\"\i ncide
avec la classe homocline $H(p_g)$ (voir~\cite{abdenur-attractor}).
Cette propri\'et\'e deviendrait \emph{robuste} si l'on parvenait
\`a remplacer l'ensemble $G_\delta$ par un ensemble ouvert.

\begin{question}
Pour toute classe homocline $H(p)$ de $f$ isol\'ee dans $\cR(f)$, existe-t-il un voisinage
$\cU$ de $f$ tel que $H(p_g)$ est encore isol\'ee pour tout diff\'eomorphisme $g\in \cU$~?
\end{question}
Ceci montrerait que les classe homoclines isol\'ees sont des ensembles ``robustement transitifs'' au sens de~\cite{bdp}.

\paragraph{d) Classes d'int\'erieur non vide.}
\begin{question}\label{q.interieur}
Si $M$ est connexe, une classe de r\'ecurrence par cha\^\i nes d'int\'erieur non vide co\"\i ncide-t-elle avec $M$~?
\end{question}
Nous avons montr\'e~\cite{abcd} que c'est le cas en dimension deux.
C'est aussi le cas lorsque la classe est isol\'ee (voir Abdenur-Bonatti-D\'\i az~\cite{ABD-interior}).
En dimension sup\'erieure des r\'esultats partiels existent
(voir~\cite{ABD-interior}, ou les travaux de Potrie et Sambarino~\cite{potrie-sambarino,potrie}),
lorsque la classe est partiellement hyperbolique ou admet certaines d\'ecomposition domin\'ees (voir le chapitre~\ref{c.hyperbolicite-faible}).

Une classe de r\'ecurrence par cha\^\i nes d'int\'erieur non vide est une classe homocline qui est stable au sens de Lyapunov
pour $f$ et $f^{-1}$. Les classes ap\'eriodiques \'etudi\'ees dans~\cite{aperiodic} sont stables au sens de Lyapunov pour $f$
et $f^{-1}$. Nous pouvons donc \'etendre la question pr\'ec\'edente aux classe homoclines stables au sens de Lyapunov pour $f$ et
$f^{-1}$ (voir~\cite{potrie}).

\paragraph{e) Attracteurs.}
Les exemples~\cite{bly} de dynamiques sans attracteurs poss\`edent des
\emph{attracteurs essentiels}\index{attracteur!essentiel}~: ce sont des classes de r\'ecurrence par cha\^\i nes $K$
ayant un voisinage $U$ tel que l'orbite positive de tout point contenu dans un G$_\delta$
dense de $U$ s'accumule dans $K$. De plus, l'union des bassins des attracteurs essentiels est dense
dans la vari\'et\'e.
\begin{question}\label{q.existence-quasi-attracteur}
Existe-t-il toujours un attracteur essentiel~?

L'union des bassins des attracteurs essentiels contient-il un G$_\delta$ dense de $M$~?
un ensemble de mesure de Lebesgue totale~?
\end{question}

On peut aussi \'etudier les attracteurs essentiels.
Remarquons que les classes ap\'eriodiques construites dans~\cite{aperiodic}
sont des quasi-attracteurs dont le bassin est trivial.
\begin{question}\label{q.homocline-essentiel}
Une classe homocline qui est un quasi-attracteur est-elle un attracteur essentiel~?
(ou attire-t-elle un G$_\delta$ dense d'un ouvert non vide de $M$~?)

Un attracteur essentiel est-il toujours une classe homocline~?
\end{question}
Une r\'eponse affirmative \`a la premi\`ere question r\'epondrait
\'egalement \`a la question~\ref{q.interieur}.
Nous donnerons des r\'eponses \`a ces questions pour les dynamiques
loin des tangences homoclines (section~\ref{s.quasi-attracteur-tangence}) et pour les dynamiques loin
des tangences homoclines et des cycles h\'et\'erodimensionnels (chapitre~\ref{c.hyp}).


\chapter{Connexions globales}
Le lemme de fermeture permet de construire des points p\'eriodiques
dans tout ouvert qui rencontre l'ensemble non-errant, mais il ne donne pas de contr\^ole
sur le support des orbites p\'eriodiques cr\'ees.
Dans ce chapitre, nous r\'epondons \`a la question suivante.

{\it \'Etant donn\'es des ouverts $U_{1},U_{2},\dots,U_{n}$,
peut-on perturber la dynamique et obtenir une orbite p\'eriodique qui rencontre
tous ces domaines~?}

En s'appuyant sur le lemme de perturbation de Pugh,
nous d\'emontrons le lemme de fermeture ergodique
de Ma\~n\'e~:
lorsqu'il existe une mesure ergodique $\mu$ chargeant chaque ouvert,
on peut cr\'eer une orbite p\'eriodique
qui passe au moins une fraction de temps proche de $\mu(U_{i})$
dans $U_{i}$ pour tout $i$. 

Nous pr\'esentons \'egalement une propri\'et\'e de pistage faible et donnons de nouvelles
d\'e\-mon\-s\-tra\-tions de certains r\'esultats issus de~\cite{approximation}~:
lorsqu'il existe un ensemble transitif par cha\^\i nes
qui intersecte tous les $U_{i}$, on peut
obtenir une orbite qui visite chacun de ces ouverts
(mais nous ne contr\^olons pas la statistique des visites).

\section{Approximation des mesures ergodiques par orbites p\'e\-ri\-o\-di\-ques~: l' ``ergodic closing lemma''}

Ma\~n\'e a donn\'e~\cite{mane-ergodic-closing,mane-livre} une contrepartie mesur\'ee du lemme de fermeture de Pugh.
\begin{theoreme}[Lemme de fermeture ergodique, Ma\~n\'e\index{lemme de fermeture!ergodique}\index{closing lemma!ergodic}]\label{t.closing-ergodique}
Soit $\cU$ un voisinage de $f\in \diff^1(M)$ et $\mu$ une mesure de probabilit\'e $f$-invariante.
Alors $\mu$-presque tout point $x\in M$ a la propri\'et\'e suivante.
Pour tout $\delta>0$, il existe
$g\in \cU$ et $\tau\geq 1$ tels que $x$ soit $\tau$-p\'eriodique pour $g$
et tels que pour tout $1\leq k\leq \tau$, on ait
$$d(f^k(x),g^k(x))\leq \delta.$$
\end{theoreme}

En fixant le point $x$ et en faisant tendre $\varepsilon$ vers $0$, les orbites p\'eriodiques
$\{x,g(x),\dots,g^{\tau-1}(x)\}$ s'\'equidistribuent comme l'orbite de $x$ sous $f$~:
lorsque $\mu$ est ergodique, elle convergent donc faiblement vers $\mu$.
Ce r\'esultat est am\'elior\'e dans~\cite[Proposition 6.1]{abc-ergodic}
par un contr\^ole des exposants des mesures p\'eriodiques (voir la section~\ref{s.oseledets}).

\begin{addendum}[Abdenur-Bonatti-Crovisier]\label{a.closing-ergodique}
On peut demander dans la conclusion du th\'e\-o\-r\`e\-me~\ref{t.closing-ergodique}
que le $i^\text{\`eme}$ exposant de Lyapunov de $x$ pour $g$ soit $\varepsilon$-proche du
$i^\text{\`eme}$ exposant de Lyapunov de $\mu$ pour $f$.
\end{addendum}

On en d\'eduit un r\'esultat de g\'en\'ericit\'e (voir~\cite[th\'eor\`eme 3.8]{abc-ergodic}).

\begin{corollaire}\label{c.approximation-mesure}
Il existe un G$_\delta$ dense $\cG\subset \diff^1(M)$
tel que toute mesure de probabilit\'e $\mu$ ergodique pour un diff\'eomorphisme $f\in \cG$
est approch\'ee faiblement par une suite de mesures invariantes port\'ees par des orbites p\'eriodiques $O_n$
de $f$. De plus, les orbites $O_n$ convergent vers le support de $\mu$ en topologie de Hausdorff
et le vecteur de Lyapunov de $O_n$ converge vers le vecteur de Lyapunov de la mesure $\mu$.
\end{corollaire}

Nous donnons ci-apr\`es une d\'emonstration du lemme de fermeture ergodique
\`a partir du th\'eor\`eme~\ref{t.pugh} de Pugh.
Pour montrer l'addendum~\ref{a.closing-ergodique}, nous choisissons un point $x$
qui est r\'egulier pour $\mu$, i.e. dont l'espace tangent admet une d\'ecomposition $T_xM=E_1\oplus\dots\oplus E_k$
satisfaisant le th\'eor\`eme d'Oseledets (voir la section~\ref{s.oseledets}).
Dans une carte autour de $x$, la diff\'erentielle $D_xg^\tau$ peut s'\'ecrire sous la forme
$P\circ D_xf^\tau$, o\`u $P$ est une isom\'etrie de $\RR^d$ que l'on peut choisir dans un voisinage de l'identit\'e.
On prendra $P$ de sorte que pour tous $i\leq j$, l'image $D_xg^\tau.E_{i,j}$
de l'espace $E_{i,j}=E_i\oplus\dots\oplus E_j$ est dans un c\^one uniforme autour de $E_{i,j}$.
Pour une p\'eriode $\tau$ suffisamment grande, ceci garantit que les exposants de Lyapunov
de $x$ pour $f$ et $g$ sont proches.

\begin{proof}[\bf D\'emonstration du lemme de fermeture ergodique (th\'eor\`eme~\ref{t.closing-ergodique})]
Il suffit de fixer $\delta>0$ et de montrer que $\mu$-presque tout point $x\in M$ satisfait
la conclusion du th\'eor\`eme~\ref{t.closing-ergodique} pour cette constante $\delta$.

Nous pouvons supposer que $\mu$ n'est pas support\'ee par une orbite p\'eriodique~: $\mu$-presque tout point
$p$ appartient au support de $\mu$ et n'est pas p\'eriodique.
Fixons $L\geq 1$ tel que $(3/2)^L\geq 2^{d+1}$
et $\varepsilon>0$ pour que $(1+\varepsilon)^L<\frac 3 2$ et posons $\eta=\frac 1 2$.
Le th\'eor\`eme~\ref{t.pugh} pour $f^{-1}$ nous donne un entier $N\geq 1$ et une carte
$\varphi\colon V\to \RR^d$ au voisinage de $p$, telle que $\varphi(V)$ co\"\i ncide avec le cube standard $C_0=(-1,1)^d$
et telle que les $N$ premiers it\'er\'es de $V$ par $f^{-1}$ soient disjoints et de diam\`etre inf\'erieur \`a $\delta$.
On note $\nu=\varphi^*\mu$. Soit $X$ l'image par $\varphi$ des points $x\in V$ qui v\'erifient la conclusion du th\'eor\`eme.
On doit montrer que $\nu(X)=\nu(C_0)$.

Comme avant, les cubes de $\RR^d$ que nous consid\'erons sont des images du cube standard $[-1,1]^d$ par une homoth\'etie-translation.
Un \emph{bon cube} $C$ est un cube de $C_0$ v\'erifiant $(1+\varepsilon).C\subset C_0$
et $\nu(C)>\frac 2 3 \nu((1+\varepsilon).C)$. Dans ce cas, il existe un ensemble
ayant mesure plus grande que $\frac 1 2 \nu(C)$ form\'e de points $x\in \varphi^{-1}(C)$
dont le premier retour $f^\tau(x)$ dans $\varphi^{-1}((1+\varepsilon).C)$
est contenu dans $\varphi^{-1}(C)$. On peut alors appliquer le lemme de perturbation de Pugh (th\'eor\`eme~\ref{t.pugh})
au cube $\varphi^{-1}(C)$ de la carte $\varphi$
pour obtenir un diff\'eomorphisme $g\in \cU$ tel que $f$ et $g$ co\"\i ncident le long
de l'orbite $\{x,\dots,f^{\tau-N}(x)\}$ et $g^\tau(x)=x$.
En particulier $\nu(X\cap C)\geq \frac 1 2 \nu(C)$.

Supposons par l'absurde que l'ensemble $C_0\setminus X$ soit de mesure non nulle.
On fixe $\ell\geq 1$ grand et on consid\`ere le pavage de $C_0$ par des cubes d'int\'erieurs deux \`a deux disjoints
et de taille $2^{-\ell}$. Par r\'egularit\'e de la mesure $\nu$,
on peut trouver une collection $\cC$ de cubes du pavage qui approxime $C_0\setminus X$~:
\begin{itemize}
\item[--] d'une part $2.C\subset C_0$ pour tout $C\in \cC$~;
\item[--] d'autre part l'union $Y$ des cubes $C$ et l'union $\hat Y$ des cubes $2.C$ pour $C\in \cC$
v\'erifient~:
\begin{equation}\label{e.regularite}
\nu(\hat Y\cap X)\ll \nu(Y \setminus X).
\end{equation}
\end{itemize}
Remarquons que l'on peut ``s\'eparer'' les cubes de la famille $\cC$~:
il existe une partition de $\cC$ en $2^d$ familles de cubes tels que
pour tous cubes $C_1,C_2$ appartenant \`a une m\^eme famille, les cubes
$2.C_1,2.C_2$ sont d'int\'erieurs disjoint. Pour l'une de ces familles, l'union des cubes
est de mesure sup\'erieure \`a $2^{-d}\nu(Y)$, donc quitte \`a remplacer
$\cC$ par cette famille, (\ref{e.regularite}) est toujours satisfaite et de plus,
\begin{equation}\label{e.disjonction}
\text{Pour tous $C_1,C_2\in \cC$, les cubes
$2.C_1,2.C_2$ sont d'int\'erieurs disjoints.}
\end{equation}

Nous utilisons maintenant le r\'esultat suivant.
\begin{affirmation*}
Pour tout cube $C\in \cC$, il existe un cube $C'$ tel que~:
\begin{enumerate}
\item[i)] $(1+\varepsilon).C'\subset 2.C$ (et en particulier $C'\subset \hat Y$)~;
\item[ii)] $C'$ est un bon cube (et donc $\nu(C'\cap X)\geq \frac 1 2 \nu(C')$)~;
\item[iii)] $\nu(C')\geq 2^{-d}\nu(C)$.
\end{enumerate}
\end{affirmation*}
\begin{proof}
La preuve se fait par l'absurde~: on construit par r\'ecurrence
une suite de cubes $(C_n)_{n\geq 1}$ satisfaisant $C_1=C$ et
$2.C_{n}\subset 2.C_{n-1}$, $\nu(C_{n})\geq 2\nu(C_{n-1})$ pour tout $n>1$.
En particulier $2.C_{n}$ est contenu dans $2.C$ et la mesure des cubes $C_n$
est sup\'erieure \`a $\nu(C)$ et cro\^\i t exponentiellement avec $n$. C'est une contradiction.

On construit $C_{n+1}$ \`a partir de $C_n$ de la fa\c con suivante~:
on divise $C_n$ en $2^d$ cubes de m\^eme rayon et d'int\'erieurs deux \`a deux disjoints.
L'un d'eux, not\'e $C'_n$, est de mesure sup\'erieure \`a $2^{-d}\nu(C_n)\geq 2^{-d}\nu(C)$.
Les cubes $(1+\varepsilon)^k.C'_{n}$, $0\leq k\leq L$, v\'erifient donc (iii).
Puisque $(1+\varepsilon)^L<3/2$, on a
$2.(1+\varepsilon)^L C'_{n}\subset 2C_n$.
Les cubes $(1+\varepsilon)^k.C'_{n}$, $0\leq k\leq L$, v\'erifient donc
\'egalement (i).
Par hypoth\`ese, (ii) n'est donc pas satisfaite et
on obtient pour tout $0\leq k \leq L$,
$$2/3\; \nu((1+\varepsilon)^{k+1}C'_{n})\geq \nu((1+\varepsilon)^kC'_{n}).$$
En posant $C_{n+1}=(1+\varepsilon)^L C_n$, ceci implique
$$\nu(C_{n+1})\geq (3/2)^L\nu(C'_n)\geq (3/2)^L2^{-d}\nu(C_n)\geq 2\nu(C_n).$$
Ceci conclut la construction de la suite $(C_n)$.
\end{proof}

Si l'on applique l'affirmation \`a chaque cube $C\in \cC$, on obtient donc un cube $C'$
contenu dans l'int\'erieur de $2.C$. D'apr\`es~(\ref{e.disjonction}) et (i),
les cubes $C'$ sont deux \`a deux disjoints.
En particulier, leur union $Y'$ v\'erifie d'apr\`es (ii) et (iii)~:
$$\nu(\hat Y\cap X)\geq \nu(Y'\cap X)\geq \frac 1 2 \nu(Y')\geq 2^{-(d+1)}.\nu(Y).$$
Ceci contredit~(\ref{e.regularite}). La mesure de $X$ est donc pleine dans $C_0$.
Ceci ach\`eve la d\'emonstration du lemme de fermeture ergodique.
\end{proof}

\section[Approximation des ensembles transitifs par cha\^\i nes]{Approximation des ensembles transitifs par cha\^\i nes par orbites p\'eriodiques}

Il est naturel de se demander si l'on peut approcher un ensemble transitif par cha\^\i nes $K$ par
une orbite p\'eriodique, ou plus simplement par un segment d'orbite.

\paragraph{\it Une difficult\'e~: les raccourcissements d'orbites.}
Si la dynamique est g\'en\'erique, on peut supposer que $K$ est faiblement transitif.
\`A l'aide du lemme de connexion, on peut assez facilement construire un segment d'orbite $\sigma$ qui visite trois points quelconques $x,y,z\in K$~:
puisque $K$ est faiblement transitif,
il suffit de connecter en $y$ un segment d'orbite joignant $x$ \`a $y$ avec un segment d'orbite joignant $y$ \`a $z$.
Une difficult\'e appara\^\i t lorsque l'on cherche \`a connecter plus de points~:
si l'on connecte $\sigma$ avec un segment d'orbite $\gamma$ joignant $z$ \`a un quatri\`eme point $t\in K$,
nous n'obtenons en g\'en\'eral pas une orbite qui visite les quatre points $x,y,z,t$.
En effet, si l'on applique le lemme de connexion, on obtient un segment d'orbite qui est souvent plus court
que la concat\'enation des segments $\sigma$ et $\gamma$ (voir la section~\ref{s.connecting}).
\medskip

Nous avons montr\'e dans~\cite{approximation} que de telles connexions globales peuvent \^etre r\'ealis\'ees
en prenant plus de pr\'ecautions.
Comme pour les lemmes de connexion pr\'ec\'edents, nous avons besoin d'une hypoth\`ese technique,
un peu plus faible cette fois, sur les orbites p\'eriodiques du diff\'eomorphisme.
\begin{description}
\item (I) \it Pour tout entier $\tau_{0}\geq 1$, les points p\'eriodiques de p\'eriode $\tau_{0}$ sont isol\'es dans $M$.
\end{description}

Nous disons qu'un ensemble ferm\'e invariant $K$ est une \emph{orbite faible}
si la relation $\prec_{K}$ est un ordre total sur $K$, i.e. $\prec_{K}$ est transitive
et pour tous points $x,y\in K$ distincts on a $x\prec_{K} y$ ou $y\prec_{K} x$ (on peut avoir les deux relations
\`a la fois et il peut exister des points $x\in K$ tels que $x\prec _{K} x$ n'ait pas lieu).
Par exemple l'adh\'erence, d'une orbite est toujours une orbite faible.

\begin{theoreme}[Lemme de connexion globale, Crovisier\index{lemme de connexion globale}]\label{t.global}
Soit $\cU$ un voisinage d'un dif\-f\'e\-o\-mor\-phis\-me $f\in \diff^1(M)$ v\'erifiant (I).
Pour tout $\eta>0$, les propri\'et\'es suivantes sont satisfaites.

\begin{itemize}
\item[--] Pour tout ensemble faiblement transitif $K$,
il existe $g\in \cU$ et une orbite p\'eriodique $O$ de $g$ qui est \`a distance de Hausdorff de $K$ inf\'erieure \`a $\eta$.
\item[--] Pour toute orbite faible $K$, il existe $g\in \cU$ et une orbite $\{g^n(x)\}$
de $g$ dont l'adh\'erence est \`a distance de Hausdorff de $K$ inf\'erieure \`a $\eta$.
\end{itemize}
\end{theoreme}

On en d\'eduit un lemme de pistage faible satisfait par les diff\'eomorphismes $C^1$-g\'en\'eriques.

\begin{corollaire}[Pistage faible\index{pistage!faible}]\label{c.pistage-faible}
Il existe un G$_\delta$ dense de $\diff^1(M)$ form\'e de diff\'eomorphismes
v\'erifiant la propri\'et\'e suivante.

Pour tout $\delta>0$ il existe $\varepsilon>0$ tel que toute $\varepsilon$-pseudo-orbite
$\{z_0,\dots,z_n\}$ est $\delta$-proche d'un segment d'orbite fini $\{x,f(x),\dots,f^m(x)\}$
pour la distance de Hausdorff.

Si de plus $Z$ est p\'eriodique (i.e. $z_n=z_0$), on peut choisir $x$ p\'eriodique.
\end{corollaire}

En particulier, on obtient l'approximation des ensembles transitifs par cha\^\i nes
par orbites p\'eriodiques.

\begin{corollaire}\label{c.global}
Il existe un G$_\delta$ dense de $\diff^1(M)$ form\'e de diff\'eomorphismes
v\'erifiant les propri\'et\'es suivantes.
\begin{itemize}
\item[--] Un ensemble compact
est transitif par cha\^\i nes si et seulement s'il est limite de Hausdorff d'une suite d'orbites p\'eriodiques.
\item[--] Les classes ap\'eriodiques sont limites de Hausdorff de classes homoclines.
\end{itemize}
\end{corollaire}

Arnaud avait montr\'e dans un travail ant\'erieur~\cite{arnaud-approximation}
que pour un diff\'eomorphisme $C^1$-g\'e\-n\'e\-ri\-que, les ensembles $\omega$-limites
sont limite de Hausdorff d'orbites p\'eriodiques.

Le corollaire~\ref{c.pistage-faible} donne aussi des informations sur la dynamique entre les classes
de r\'ecurrence par cha\^\i nes~:
pour tout diff\'eomorphisme $C^1$-g\'en\'erique,
si l'on consid\`ere une suite de classes de r\'ecurrence par cha\^\i nes
$K_1,\dots,K_s$, v\'erifiant $K_i\dashv K_{i+1}$ pour tout $1\leq i<s$,
et des voisinages $U_1,\dots, U_s$,
il existe une orbite qui visite successivement chaque ouvert $U_i$.

\section{D\'emonstration du pistage faible}
La d\'emonstration du th\'eor\`eme~\ref{t.global} est assez d\'elicate
puisqu'elle requiert de perturber la dynamique en plusieurs endroits.
Nous allons d\'emontrer un r\'esultat plus faible qui impliquera
le corollaire~\ref{c.pistage-faible}.

\begin{theoreme}[Pistage faible, Crovisier]\label{t.pistage2}
Il existe un G$_\delta$ dense $\cG\subset \diff^1(M)$
tel que, pour tout $f\in \cG$, pour tout ensemble ferm\'e $K\subset M$
(non n\'ecessairement invariant), pour tous points $z_1,\dots,z_n\in K$
v\'erifiant $z_1\dashv_K z_2\dashv_K\dots\dashv_K z_n$,
et pour tout $\delta>0$, il existe
un segment d'orbite $\{x,\dots,f^m(x)\}$ contenu dans le $\delta$-voisinage de $K$
qui rencontre chaque boule centr\'ee en $z_i$, $1\leq i \leq s$, de rayon $\delta$.

Si de plus on a $z_n\dashv_K z_1$, alors $x$ peut \^etre choisi p\'eriodique.
\end{theoreme}

Pour la suite, on introduira pour tout $f\in \diff^1(M)$
l'ensemble $\segment(f)$ des parties compactes de $M$ (non n\'ecessairement invariantes)
qui sont limites de Hausdorff d'une suite de segments d'orbite finis
ainsi que l'ensemble $\psegment(f)$ des parties compactes de $M$ (non n\'ecessairement invariantes)
qui sont limites de Hausdorff pour tout $\varepsilon>0$
d'une suite de segments de $\varepsilon$-pseudo-orbites finis.
Finalement $\cper(f)$ d\'esigne l'ensemble des parties compactes (invariantes) de $M$
qui sont limites d'une suite d'orbites p\'eriodiques et $\ctrans(f)$ celles qui sont transitives par cha\^\i nes.

\begin{proof}[{D\'emonstration du corollaire~\ref{c.pistage-faible} \`a partir du th\'eor\`eme~\ref{t.pistage2}}]
Consid\'erons $f$ appartenant \`a l'en\-sem\-ble r\'esiduel $\cG$ donn\'e par le th\'eor\`eme~\ref{t.pistage2}.
Lorsque $\delta>0$ est fix\'e, il existe $\varepsilon>0$ tel que tout segment de $\varepsilon$-pseudo-orbite
est $\delta/2$-proche d'un \'el\'ement $K$ de $\psegment(f)$.
D'apr\`es le th\'eor\`eme~\ref{t.pistage2}, il existe un segment d'orbite fini de $f$ qui est $\delta/2$-proche de $K$
pour la distance de Hausdorff. Ceci d\'emontre la premi\`ere partie du corollaire~\ref{c.pistage-faible}.

Le m\^eme argument montre que pour $\varepsilon$ assez petit, toute $\varepsilon$-pseudo-orbite p\'eriodique
est $\delta/2$-proche d'un \'el\'ement $K$ de $\ctrans(f)$.
D'apr\`es le th\'eor\`eme~\ref{t.pistage2}, il existe une orbite p\'eriodique de $f$ qui est $\delta/2$-proche de $K$
pour la distance de Hausdorff.
\end{proof}

\begin{proof}[D\'emonstration du th\'eor\`eme~\ref{t.pistage2}: le cas non r\'ecurrent.]
L'application $f\mapsto \segment(f)$ est semi-con\-ti\-nue inf\'erieurement sur $\diff^1(M)$.
Il existe donc un G$_\delta$ dense $\cG\subset \diff^1(M)$ de diff\'eomorphismes qui sont des points de continuit\'e
de cette application, qui sont Kupka-Smale, et qui satisfont \`a la seconde propri\'et\'e de la section~\ref{s.consequence}
($\prec_K=\dashv_K$ pour tout ensemble compact $K$). Nous consid\'erons \`a pr\'esent un \'el\'ement $f\in \cG$.
Nous montrons la premi\`ere partie du th\'eor\`eme par r\'ecurrence sur $n$. Le cas $n=1$ est \'evident.
Remarquons aussi que nous pouvons toujours supposer que les points $z_j$
appartiennent \`a des orbites distinctes (quitte \`a supprimer certains points)
et supposer qu'ils ne sont pas p\'eriodiques (qui \`a les remplacer par des points proches,
puisque $f$ est un diff\'eomorphisme dont tous les points p\'eriodiques sont hyperboliques).

Consid\'erons $n+1$ points $z_1\dashv_K z_2\dashv_K\dots\dashv_K z_{n+1}$ de $K$. Nous devons montrer
qu'ils sont contenus dans un \'el\'ement $\Gamma\in \segment(f)$ inclus dans $K$.
Puisque $f$ est un point de continuit\'e de l'application $g\mapsto \segment(g)$,
il suffit de trouver pour tout $\delta>0$ et tout voisinage $\cU$ de $f$
un diff\'eomorphisme $g\in \cU$ ayant un segment d'orbite fini $\gamma_g$ contenu dans le $\delta$-voisinage
de $K$ et rencontrant chaque boule $B(z_i,\delta)$, $1\leq i\leq n+1$.
Le lemme de perturbation de Pugh (th\'eor\`eme~\ref{t.pugh}) appliqu\'e \`a $f$, $\cU$
et $\varepsilon=\eta=\frac 1 2$ nous donne un entier $N\geq 1$
et des paires de cubes $C_j\subset \widehat C_j$ centr\'es en chaque point $z_j$, $1\leq j\leq n$,
dont les $N$ premiers it\'er\'es sont disjoints deux \`a deux et de diam\`etre inf\'erieur \`a $\delta$.
On choisit enfin $\delta'\ll \delta$ tel que chaque boule $B(z_j,\delta')$ est contenue dans $C_j$.

En appliquant l'hypoth\`ese de r\'ecurrence, on obtient un segment d'orbite $\gamma=\{f^k(x)\}_{0\leq k \leq L}$
qui intersecte chaque cube $C_j$ et qui est contenue dans le $\delta$-voisinage de $K$.
Soit alors $\sigma_0=\{f^k(x)\}_{0\leq k\leq N_0}\subset \gamma$
le plus petit segment d'orbite contenu dans $\gamma$ qui contienne $x$
et rencontre chaque cube $\widehat C_j$, $1\leq j\leq n$~; nous notons
$\zeta_f=f^{N_0}$ le point final de $\sigma_0$.
Remarquons que par minimalit\'e de $\sigma_0$, les points $f^k(x)$
pour $0\leq k<N_0$ ne rencontrent pas le cube $\widehat C_f$ contenant $\zeta_f$.
Il existe donc un it\'er\'e $f^\ell(\zeta_f)$, $\ell\leq L-N_0$, qui appartient \`a $C_f$.

Puisque $\prec_K=\dashv_K$ pour $f$, il existe un segment d'orbite
$\gamma'=\{f^{-m}(y),\dots,y\}$ contenu dans le $\delta$-voisinage de $K$
tel que $f^{-m}(y)$ appartient \`a $C_f$ et $y$ \`a la boule $B(z_{n+1},\delta)$.
On applique alors le lemme de connexion~: il existe $g\in \cU$
tel que $y$ est un it\'er\'e positif de $\zeta_f$.
Les diff\'eomorphismes $f$ et $g$ co\"\i ncident hors de $\widehat C_f$ et de ses $N-1$ premiers it\'er\'es.
Le segment d'orbite $\sigma_0$ n'a pas \'et\'e modifi\'e. On en d\'eduit que
$y$ est un it\'er\'e positif de $x$ et que l'ensemble $\gamma_g$ des it\'er\'es compris entre $x$ et $y$
rencontre chaque boule $B(z_j,\delta)$ pour $1\leq j\leq n+1$.
Par ailleurs, le nouveau segment d'orbite est inclus dans l'union de $\gamma$,
$\gamma'$ et des $N$ premiers it\'er\'es de $\widehat C_f$. Par cons\'equent, il est contenu
dans le $\delta$-voisinage de $K$.
\end{proof}

\begin{proof}[D\'emonstration du th\'eor\`eme~\ref{t.pistage2}: le cas r\'ecurrent.]
Nous montrons \`a pr\'esent la seconde partie du th\'eor\`eme:
nous supposons $z_j\dashv_K z_{j+1}$ pour tout $1\leq j<n$ et $z_n\dashv_K z_1$.
Quitte \`a consid\'erer un G$_\delta$ dense $\cG$ plus petit, nous pouvons supposer
que $f\in \cG$ (dont tous les points p\'eriodiques sont hyperboliques) est un point de continuit\'e
de l'application $g\mapsto \cper(g)$. Il suffit donc de construire $g\in \diff^1(M)$ proche de $f$
poss\'edant une orbite p\'eriodique contenue dans le $\delta$-voisinage de $K$
qui intersecte chaque boule $B(z_j,\delta)$.
Nous introduisons comme pr\'ec\'edemment les cubes $C_j\subset \widehat C_j$.
Nous supposerons de plus que chaque $\widehat C_j$ est un cube quadrill\'e dont $C_j$
est une tuile centrale (comme pour la section~\ref{s.connecting}).

D'apr\`es la premi\`ere partie du th\'eor\`eme, il existe un segment d'orbite $\gamma=\{f^k(x)\}_{0\leq k\leq L}$ qui rencontre chaque cube $C_j$ et qui est contenu
dans le $\delta$-voisinage de $K$.
Nous consid\'erons un segment d'orbite $\sigma_0=\{f^k(x)\}_{L_i\leq k\leq L_f}$
qui rencontre chaque cube $\widehat C_j$ et qui est minimal (pour l'inclusion) pour cette propri\'et\'e.
Nous notons $\zeta_i=f^{L_i}(x)$ et $\zeta_f=f^{L_f}(x)$ ses points initial et final~:
ils sont chacun contenus dans un cube $\widehat C_i$ et $\widehat C_f$.
Nous montrons que, quitte \`a remplacer $\sigma_0$ par un autre segment minimal,
nous pouvons supposer que
\begin{itemize}
\item[a)] $\zeta_i$ a un it\'er\'e n\'egatif $f^{-\ell_i}(\zeta_i)$, $\ell_i\leq L_i$, dans $C_i$,
\item[b)] $\zeta_f$ a un it\'er\'e positif $f^{\ell_f}(\zeta_f)$, $\ell_f\leq L-L_i$ dans $C_f$.
\end{itemize}
Supposons par exemple que a) n'est pas satisfait.
Puisque $\gamma$ rencontre $C_i$ et puisque $\sigma_0$ est minimal,
il existe $L'_f>L_f$ tel que $f^{L'_f}(x)$ appartient \`a $\widehat C_i$ et poss\`ede un it\'er\'e
positif $f^{\ell_f+L'_f}(x)$, $\ell_f+L'_f\leq L$, dans $C_i$.
Il existe alors $L'_i>L_i$ tel que le segment
$\sigma_0'=\{f^k(x)\}_{L_i'\leq k\leq L_f'}$ soit un nouveau segment minimal.
Par construction la propri\'et\'e b) est satisfaite par $\sigma'_0$.
Si a) n'est toujours pas satisfaite, on r\'ep\`ete cette construction.
\`A chaque \'etape, $L'_f$ augmente et puisque $\gamma$ est fini,
ce proc\'ed\'e doit s'arr\^eter~: on obtient alors \`a la fois a) et b).

Introduisons \`a pr\'esent un segment d'orbite $\gamma'$ contenu dans le $\delta$-voisinage de $K$
et rencontrant les cubes $C_i$ et $C_f$. La r\'eunion de $\gamma'$
et de $\gamma$ contient une pseudo-orbite \`a sauts dans les tuiles de $\widehat C_i$ et $\widehat C_f$
et contenant $\sigma_0$.
Par construction $\sigma_0\setminus \{\zeta_i,\zeta_f\}$ \'evite les cubes
$\widehat C_i$ et $\widehat C_f$ et rencontre les autres cubes $\widehat C_j$.
Nous pouvons donc, comme en section~\ref{ss.domaine}, perturber $f$ et construire un diff\'eomorphisme $g\in \cU$
poss\'edant une orbite p\'eriodique qui contient $\sigma_0$ et donc rencontre chaque cube $\widehat C_j$.
Cette orbite est contenue dans le $\delta$-voisinage de $K$.
\end{proof}

\section{Application~: \'etude de la stabilit\'e molle}\label{s.tolerance}

Un des buts des syst\`emes dynamiques consiste \`a d\'ecrire comment
les invariants dynamiques varient lorsque l'on perturbe le syst\`eme.
Ceci conduit \`a la notion de stabilit\'e.
Puisque la stabilit\'e structurelle
et l'$\Omega$-stabilit\'e (voir section~\ref{s.stabilite}) ne sont pas dense dans $\diff^1(M)$,
d'autres formes de stabilit\'e ont \'et\'e propos\'ees.
En suivant une id\'ee de Zeeman, Takens a introduit~\cite{takens-tolerance1} une stabilit\'e affaiblie.
La stabilit\'e structurelle a lieu lorsque l'espace des orbites est rigide~;
la stabilit\'e molle
exprime que l'espace des orbites du syst\`eme change peu lorsque l'on perturbe la dynamique.

Par exemple, si $\cK(M)$
d\'esigne l'espace des ensembles compacts de $M$ muni de la topologie de Hausdorff,
Takens montre le th\'eor\`eme suivant~\cite{takens-tolerance1}.
\begin{theoreme}[Takens]
L'ensemble des points de continuit\'e de chaque application $f\mapsto \overline{\per(f)}$, $f\mapsto \Omega(f)$ et $f\mapsto \cR(f)$
d\'efinies sur $\diff^1(M)$ et \`a valeurs dans $\cK(M)$ contient un G$_\delta$ dense.
\end{theoreme}
Pour $\cR(f)$, c'est une simple cons\'equence de la proposition~\ref{p.baire}.
Pour $\overline{\per(f)}$, c'est une cons\'equence du th\'eor\`eme~\ref{t.kupka-smale} de Kupka-Smale
et pour $\Omega(f)$, c'est une cons\'equence du lemme de fermeture de Pugh (th\'eor\`eme~\ref{t.closing}).
\footnote{Remarquons que puisque g\'en\'eriquement dans $\diff^1(M)$, on a $\Omega(f)=\cR(f)$, il n'y a
pas de $C^0$ $\Omega$-explosion, r\'epondant ainsi au probl\`eme 19 de~\cite{palis-pugh}.}
\medskip

Pour d\'ecrire comment $M$ se d\'ecompose en orbites, on travaille dans l'espace $\cK(\cK(M))$
des familles ferm\'ees d'ensembles compacts de $M$ et on introduit pour tout diff\'eomorphisme $f$,
\begin{description}
\item $\corb(f)$~:
l'ensemble des parties compactes invariantes de $M$ qui sont limite de Hausdorff d'adh\'erences d'orbite de $f$.
\end{description}
Un diff\'eomorphisme $f$ est \emph{mollement stable\index{stabilit\'e!molle}} (``tolerance stable''\index{stabilit\'e!tolerance stability})
si c'est un point de continuit\'e de l'application $f\mapsto \corb(f)$ d\'efinie sur $\diff^1(M)$.

\begin{theorem-tolerance}[Zeeman]
L'ensemble des diff\'eomorphismes mollement stables contient un G$_\delta$ dense de $\diff^1(M)$.
\end{theorem-tolerance}
\medskip

\`A notre connaissance, cette conjecture reste ouverte, mais nous pouvons traiter des questions analogues.
La structure de l'espace des orbites est d\'ecrite par les \'el\'ements suivants de $\cK(\cK(M))$~:
\begin{description}
\item[$\cl(f)$~:] l'ensemble des parties compactes invariantes de $M$,
\item[$\cper(f)$~:] l'ensemble des parties compactes invariantes de $M$ qui sont limites de Hausdorff d'or\-bi\-tes p\'eriodiques,
\item[$\ftrans(f)$~:] l'ensemble des parties compactes invariantes de $M$ qui sont faiblement transitives,
\item[$\ctrans(f)$~:] l'ensemble des parties compactes invariantes de $M$ qui sont transitives par cha\^\i nes,
\item[$\orbf(f)$~:] l'ensemble des parties compactes invariantes de $M$ qui sont limites de Hausdorff de segments
d'orbites finis,
\item[$\eorb(f)$~:] l'ensemble des parties compactes invariantes de $M$ qui sont limites de $\varepsilon$-pseudo-orbites pour tout $\varepsilon>0$ (parfois appel\'ees \emph{orbites \'etendues} de $f$).
\end{description}
\bigskip

Les propri\'et\'es de stabilit\'e correspondantes sont v\'erifi\'ees.
\begin{proposition}
L'ensemble des points de continuit\'e des ensembles $\cl(f)$, $\cper(f)$, $\ftrans(f)$,
$\ctrans(f)$, $\orbf(f)$ et $\eorb(f)$ contient un G$_\delta$ dense de $\diff^1(M)$.
\end{proposition}
Takens avait trait\'e~\cite{takens-tolerance1,takens-tolerance2} le cas des ensembles $\cl(f)$, $\cper(f)$
et $\eorb(f)$.
La proposition s'obtient par des arguments de semi-continuit\'e.
Pour $\cper(f)$, on utilise le lemme de fermeture et pour $\ftrans(f)$, on utilise le corollaire~\ref{c.global}.
\bigskip

Concernant la conjecture initiale de Zeeman, Takens a donn\'e~\cite{takens-tolerance2} le crit\`ere suivant.

\begin{theoreme}[Takens]
Si l'ensemble des diff\'eomorphismes $f$ tels que $\corb(f)=\eorb(f)$ contient un G$_\delta$ dense de
$\diff^1(M)$, alors la conjecture de stabilit\'e molle est vraie.
\end{theoreme}

Du th\'eor\`eme~\ref{t.global}, on d\'eduit que $\orbf(f)=\eorb(f)$ pour un G$_\delta$ dense de $\diff^1(M)$.
La conjecture de Zeeman est donc reli\'ee au probl\`eme suivant.
\begin{question}
A-t-on $\corb(f)=\orbf(f)$ pour $f$ dans un G$_\delta$ dense de $\diff^1(M)$~?
\end{question}

Le th\'eor\`eme~\ref{t.global} montre aussi que
$\cper(f)=\ftrans(f)=\ctrans(f)$.

\section{Probl\`emes}

D'autres probl\`emes de connexion d'orbites restent ouverts et peuvent \^etre int\'eressants
pour des applications.

\paragraph{a) Fermeture asymptotique.}
\begin{question}[Fermeture asymptotique]\label{q.closing-asymptotique}
\emph{Consid\'erons un voisinage $\cU$ de $f\in \diff^1(M)$ et $x$ un point de $M$.}
\begin{itemize}
\item[--] Existe-t-il $g\in \cU$ pour lequel $x$ appartient \`a la vari\'et\'e stable d'une orbite
p\'eriodique hyperbolique $\cO$~?
\item[--] Peut-on demander de plus que les adh\'erences des orbites positives de $x$ sous $f$ et sous $g$
restent proches en topologie de Hausdorff~?
Que $\cO$ et l'ensemble des valeurs d'adh\'erences de l'orbite positive de $x$ sous $f$ 
soient proches pour la distance de Hausdorff~?
\end{itemize}
\end{question}
Une r\'eponse positive impliquerait la densit\'e des vari\'et\'es stables et instables
d'orbites p\'e\-ri\-o\-di\-ques pour un diff\'eomorphisme $C^1$-g\'en\'erique
(voir~\cite{bgw} pour une r\'eponse partielle).
Elle permettrait \'egalement de montrer la conjecture de stabilit\'e de Zeeman (voir~\cite{approximation}).

\paragraph{b) Orbites avec visites rares.}
\begin{question}\label{q.proportion}
\emph{Supposons que $f$ v\'erifie une condition $C^1$-g\'en\'erique
et consid\'erons un ensemble compact et transitif par cha\^\i nes $\Lambda$
contenu dans une classe de r\'ecurrence par cha\^\i nes $K$.
Fixons un point $x\in K\setminus \Lambda$, un voisinage $U$ de $\Lambda$ et $\theta>0$.}

Existe-t-il une orbite p\'eriodique ayant un point proche de $x$
et passant une proportion de temps sup\'erieure \`a $\theta$ dans $U$~?
\end{question}
Un tel r\'esultat permettrait d'\'etendre \`a toute la classe $K$ certaines propri\'et\'es
satisfaites sur $\Lambda$. Voir la section~\ref{s.probleme-newhouse}.
On renvoie \`a~\cite{gan-wen-connecting} pour d'autres probl\`emes de connexion d'orbite.


\chapter{Hyperbolicit\'e non uniforme}\label{c.hyperbolicite-faible}
Certaines propri\'et\'es connues pour les syst\`emes hyperboliques s'\'etendent
\`a des classes plus g\'e\-n\'e\-ra\-les de dynamiques. C'est l'objet par exemple de la th\'eorie de Pesin~\cite{pesin}
qui d\'ecrit la dynamique associ\'ee \`a une mesure ergodique dont aucun exposant de Lyapunov ne s'annule,
pour des diff\'eomorphismes de classe $C^{1+\varepsilon}$ (les arguments ne se g\'en\'eralisent pas
\`a la topologie $C^1$, voir~\cite{pugh-hypotheseC2}).
Nous pr\'esentons dans ce chapitre de tels r\'esultats, non perturbatifs et valables en classe $C^1$. 

\section{D\'ecomposition domin\'ee}\label{s.domination}
\paragraph{Domination.}
Consid\'erons un ensemble invariant $K$ dont le fibr\'e unitaire tangent est la somme directe de deux
sous-fibr\'es lin\'eaires invariants par l'application tangente~: $T_KM=E\oplus F$.
C'est une \emph{d\'ecomposition domin\'ee}\index{d\'ecomposition domin\'ee}
s'il existe un entier $N\geq 1$ tel que pour tout $x\in K$ et tous $u\in E(x)$, $v\in F(x)$
unitaires on a
$$\|D_xf^N.u\|\leq \frac 1 e \|D_xf^N.v\|.$$
On dira aussi que la d\'ecomposition est \emph{$N$-domin\'ee} lorsque l'on souhaitera
pr\'eciser l'entier $N$.
La d\'ecomposition est \emph{non triviale} si les dimensions de $E$ et $F$ ne sont pas nulles.
On \'etend cette d\'efinition et on consid\'erera aussi des d\'ecompositions domin\'ees ayant plus de deux facteurs.

\paragraph{Propri\'et\'es.}
Voici quelques propri\'et\'es des d\'ecompositions domin\'ees (voir~\cite[appendice B]{bdv}).
\smallskip

\begin{itemize}
\item[--] Tout ensemble invariant $K$ poss\`ede une d\'ecomposition domin\'ee
$T_KM=E_1\oplus\dots\oplus E_\ell$ la plus fine~: pour toute autre d\'ecomposition domin\'ee $T_KM=E\oplus F$, il existe
$0\leq k \leq \ell$ tel que $E=\bigoplus_{1\leq i \leq k} E_i$ et $F=\bigoplus_{k+1\leq i\leq \ell} E_i$.
\smallskip

\item[--] Les d\'ecompositions domin\'ees s'\'etendent \`a l'adh\'erence de $K$.
Elles passent \`a la limite~: si $(f_n)$ est une suite de diff\'eomorphismes convergeant vers $f$,
si $(K_n)$ est une suite d'ensembles compacts $f_n$-invariants qui converge en topologie de Hausdorff vers $K$,
si chaque ensemble $K_n$ porte une d\'ecomposition $N$-domin\'ee $E_n\oplus F_n$ telle que la dimension de $E_n$
ne d\'epende pas de $n$, alors les fibr\'es $E_n$ et $F_n$ convergent vers des fibr\'es $E,F$ au-dessus de $K$
qui induisent une d\'ecomposition $N$-domin\'ee $T_KM=E\oplus F$.
\smallskip

\item[--] Si l'ensemble compact $K$ poss\`ede une d\'ecomposition domin\'ee $E\oplus F$ pour $f$,
tout ensemble contenu dans un voisinage de $K$ et invariant par un diff\'eomorphisme $g$ proche de $f$ poss\`ede \'egalement une
d\'ecomposition domin\'ee en sous-fibr\'es de la m\^eme dimension.
\end{itemize}

\paragraph{Hyperbolicit\'e partielle.}
Un ensemble $K$ est dit~\emph{partiellement hyperbolique}\index{hyperbolicit\'e!partielle} s'il admet une d\'ecomposition domin\'ee
$T_KM=E^s\oplus E^c\oplus E^u$ et un entier $N\geq 1$ tels que $E^s$ et $E^u$
soient respectivement $N$-uniform\'ement contract\'es et $N$-uniform\'ement dilat\'es (i.e. on a~(\ref{e.hyperbolique}))
et ne soient pas tous deux triviaux.
\medskip

Les notions d'hyperbolicit\'e, hyperbolicit\'e partielle, d\'ecompositions domin\'ees ne d\'ependent pas
de la m\'etrique riemannienne. Gourmelon a montr\'e~\cite{gourmelon-metrique} que l'on peut toujours trouver une m\'etrique pour que dans ces
d\'efinitions l'hyperbolicit\'e ou la domination se voient d\`es la premi\`ere it\'eration,
i.e. en tout point $x\in K$ on a~:

$$\sup\left(\|D_xf_{|E^s(x)}\|, \; \|D_xf^{-1}_{|E^u(x)}\|,\;
\|D_xf_{|E^s(x)}\|.\|D_{f(x)} f^{-1}_{E^c(f(x))}\|,\;
\|D_xf_{|E^c(x)}\|.\|D_{f(x)} f^{-1}_{E^u(f(x))}\|\right) \; < 1.$$

\section{Familles de plaques}
On peut associer \`a tout ensemble ayant une d\'ecomposition domin\'ee une famille de sous-vari\'et\'es
qui g\'en\'eralise les vari\'et\'es invariantes locales des ensembles hyperboliques.

\begin{definition}
Soit $K$ un ensemble invariant avec une d\'ecomposition domin\'ee
$T_KM=F_1\oplus E\oplus F_2$.\\
Une \emph{famille de plaques}\index{famille de plaques} tangente \`a $E$ est une application continue $\cW\colon E \to M$ satisfaisant~:
\begin{itemize}
\item[--] pour tout $x\in K$, l'application induite $\cW_x\colon E_x\to M$ est un plongement $C^1$
pour lequel $\cW_x(0)=x$ et dont l'image est tangente en $x$ \`a $E_x$~;
\item[--] $(\cW_x)_{x\in K}$ est une famille continue de plongements $C^1$.
\end{itemize}
La famille de plaques $\cW$ est \emph{localement invariante}\index{famille de plaques!localement invariante} s'il existe $\rho>0$ tel que pour tout $x\in K$,
l'image de la boule $B(0,\rho)\subset E_x$ par $f\circ \cW_x$ est contenue dans la plaque $\cW_{f(x)}$.
\end{definition}

D'apr\`es~\cite[th\'eor\`eme 5.5]{hps}, il existe toujours des familles de plaques localement invariantes.
\begin{theoreme}[Hirsch-Pugh-Shub]\label{t.hps}
Pour tout ensemble compact invariant $K$ dont l'espace tangent poss\`ede une d\'ecomposition domin\'ee
$T_KM=E\oplus F$, il existe une famille de plaques localement invariante tangente \`a $E$.
\end{theoreme}

\begin{remarque}
\begin{itemize}
\item[a)] En g\'en\'eral, les plaques ne sont pas d\'efinies dynamiquement.
Par con\-s\'e\-quent, deux plaques peuvent avoir des points d'intersection isol\'es
et la famille de plaques n'est pas unique a priori.
\item[b)] On peut \'enoncer une version uniforme de ce r\'esultat~:
il existe des voisinages $U$ de $K$ et $\cU\subset \diff^1(M)$ de $f$
et une collection de familles de plaques $(\cW_g)_{g\in \cU}$
tangentes aux continuations $(E_g)_{g\in \cU}$ du fibr\'e $E$
et d\'efinies au-dessus des ensembles invariants maximaux $(K_g)$ de $U$
tels que
\begin{itemize}
\item[--] $(\cW_{g,x})_{g\in \cU,x\in K_g}$ est une famille continue de plongements $C^1$,
\item[--] les familles de plaques sont uniform\'ement localement invariantes~: il existe
$\rho>0$ tel que pour tous $g\in \cU$ et $x\in K_g$,
l'image de la boule $B(0,\rho)\subset E_{g,x}$ par $f\circ \cW_{g,x}$ est contenue dans la plaque $\cW_{g,g(x)}$.
\end{itemize}
\item[c)] Lorsqu'il y a une d\'ecomposition domin\'ee $T_KM=F_1\oplus E\oplus F_2$ en trois fibr\'es,
on peut obtenir une famille de plaques localement invariante tangente \`a $E$ en intersectant des familles
de plaques localement invariantes tangentes \`a $F_1\oplus E$ et \`a $E\oplus F_2$ respectivement.
\item[d)] Les plaques seront g\'en\'eralement de petit diam\`etre. 
\end{itemize}
\end{remarque}

\section{Points hyperboliques}
Consid\'erons un entier $N\geq 1$ et
un ensemble compact invariant $K$ muni d'une d\'ecomposition
domin\'ee $T_KM=E\oplus F$.
\begin{definition}
Un point $x\in K$ est \emph{$N$-hyperbolique le long de $E$}\index{hyperbolicit\'e!ponctuelle} si
pour tout $k\geq 0$ on a:
$$\prod_{i=0}^{k-1} \|D_{f^{iN}(x)}f^N_{|E}\|\leq e^{-k}.$$
\end{definition}

L'existence de points hyperboliques fait souvent appel au lemme
de Pliss~\cite{pliss}.
Dans le cas d'orbites p\'eriodiques, on obtient
le r\'esultat suivant.

\begin{proposition}[Cons\'equence du lemme de Pliss]\label{p.pliss-periodique}
Pour ensemble invariant $K$ ayant une d\'ecomposition domin\'ee $T_KM=E\oplus F$,
il existe $\rho\in (0,1)$ avec la propri\'et\'e suivante.
Pour tout $N\geq 1$ suffisamment grand et
pour toute orbite p\'eriodique hyperbolique
$O=\{x,\dots,f^{\tau}(x)=x\}$ de $K$ satisfaisant
\begin{equation}\label{e.periode}
\prod_{i=0}^{\tau-1} \|D_{f^{i}(x)}f_{|E}\|\leq e^{-\tau},
\end{equation}
l'ensemble des points $N$-hyperboliques de $O$ a une proportion sup\'erieure \`a $\rho$.
\end{proposition}
\begin{remarque}
Lorsque la condition~\ref{e.periode} est \'egalement satisfaite pour
les it\'erations de $Df^{-1}$ le long de $F$, on obtient
des points simultan\'ement hyperboliques le
long des espaces $E$ et $F$.

En effet, on choisit un entier $N$ premier.
On applique la proposition au fibr\'e $E$
et on consid\`ere un point $x\in O$ qui est $N$-hyperbolique le long du fibr\'e $E$.
On applique ensuite la proposition au fibr\'e $F$~: 
il existe un point $y\in O$ qui est $N$-hyperbolique le long de $F$ pour $f^{-1}$.
Puisque $N$ est premier, on peut le mettre sous la forme $y=f^{-kN}(x)$.
On choisit $k\geq 0$ minimal avec cette propri\'et\'e.
Le fait que les points $f^{N}(y),f^{2N}(y)$,\dots, $f^{kN}(y)$ ne soient pas $N$-hyperboliques
le long de $F$ implique pour tout $1\leq j\leq k$,
$$\prod_{i=1}^{j}\|D_{f^{iN}(y)}f^{-N}_{|F}\|>e^{-j}.$$
La domination entre $E$ et $F$ entra\^\i ne alors (si $N$ est suffisamment grand),
$$\prod_{i=0}^{j-1}\|D_{f^{iN}(y)}f^{N}_{|E}\|\leq e^{-j}$$
pour tout $1\leq j \leq k$. Puisque $x=f^{kN}(y)$ est $N$-hyperbolique pour $E$,
ceci est encore vrai pour tout $k\geq 0$.
Par cons\'equent, $y$ est simultan\'ement $N$-hyperbolique le long de $E$ et $F$.
\end{remarque}

\section{Vari\'et\'es invariantes}
Par un argument classique (voir par exemple~\cite[section 8]{abc-ergodic}),
tout point qui est $N$-hy\-per\-bo\-li\-que le long de $E$ poss\`ede une vari\'et\'e stable tangente \`a $E$.

\begin{proposition}\label{p.variete}
Consid\'erons un ensemble compact invariant $K$
dont l'espace tangent poss\`ede une d\'ecomposition domin\'ee $T_KM=E\oplus F$,
une famille de plaques localement invariante $\cW$ tangente \`a $E$
et un entier $N\geq 1$. Il existe $\varepsilon,\delta>0$ tels que pour tout point $x\in K$
qui est $N$-hyperbolique le long de $E$, l'ensemble
$$W_{E^s}(x):=\left\{y\in M,\quad \lim_{n\to +\infty}\frac {d(f^n(y), f^n(x))}{e^{n\varepsilon}\|D_xf^n_{|E^s}\|}<+\infty\right\} $$
a les propri\'et\'es suivantes~:
\begin{itemize}
\item[--] $W_{E^s}(x)$ est une sous-vari\'et\'e injectivement immerg\'ee tangente \`a $E^s(x)$ et ne d\'epend pas de $\varepsilon$~;
\item[--] la boule $B(0,\delta)\subset \cW_x$
est contenue dans $W_{E^s}(x)$ et son image par $f^k$ est contenue dans $\cW_{f^k(x)}$ pour tout $k\geq 0$.
\end{itemize}
\end{proposition}
L'ensemble $W_{E^s}(x)$ est appel\'e \emph{vari\'et\'e stable forte de $x$ associ\'ee \`a $E^s$}\index{vari\'et\'e stable forte} et not\'e $W^{ss}(x)$ lorsqu'il n'y a pas d'ambigu\"\i t\'e sur le fibr\'e $E^s$. On d\'efinit de fa\c con sym\'etrique la \emph{vari\'et\'e instable forte} $W_{E^u}(x)$ (encore not\'ee $W^{uu}(x)$).
Dans le cas de la d\'ecomposition $E^s\oplus E^u$ d'un ensemble hyperbolique,
la vari\'et\'e $W_{E^s}(x)$ co\"\i ncide avec l'ensemble stable $W^s(x)$.

\begin{remarque}
On peut \'enoncer une version uniforme de ce r\'esultat.
\end{remarque}

Le r\'esultat suivant d\'ecoule des propositions~\ref{p.pliss-periodique} et~\ref{p.variete}
et peut servir \`a borner le nombre
de classes homoclines d'un diff\'eomorphisme.
Il a \'et\'e d\'emontr\'e initialement par Pliss~\cite{pliss}
dans le cas d'une d\'ecomposition domin\'ee triviale ($T_KM=E$)
pour majorer le nombre de puits d'un diff\'eomorphisme.
\begin{corollaire}\label{c.homocline-bornee}
Pour tout ensemble invariant $K$ ayant une d\'ecomposition domin\'ee
$T_KM=E\oplus F$ et tout entier $N\geq 1$,
il existe $k\geq 1$ ayant la propri\'et\'e suivante.

Dans toute famille d'orbites p\'eriodiques
$\{O_1,\dots,O_k\}$ de $K$ v\'erifiant pour tout $1\leq i\leq k$,
\begin{equation}\label{e.lie}
\prod_{x\in O_i}\|D_xf^N_{|E}\|\leq e^{-\card (O_i)} \text{ et }
\prod_{x\in O_i}\|D_xf^{-N}_{|F}\|\leq e^{-\card (O_i)},
\end{equation}
il existe au moins deux orbites $O_i,O_j$ homocliniquement reli\'ees.
\end{corollaire}

\section{Mesures hyperboliques}\label{s.oseledets}

Si $\mu$ est une mesure de probabilit\'e invariante sur $M$, le th\'eor\`eme d'Oseledets (voir~\cite[th\'eor\`eme S.2.9]{hasselblatt-katok})
associe \`a $\mu$-presque tout point $x\in M$ une unique d\'ecomposition invariante $T_xM=E_1\oplus\dots\oplus E_s$
et des r\'eels $\lambda_1<\lambda_2<\dots<\lambda_s$
tels que pour tout $u\in E_i\setminus \{0\}$ la quantit\'e
$$\frac 1 n \log\|D_xf^{n}.u\|$$
converge vers $\lambda_i$ lorsque $n$ tend vers $\pm \infty$.

On appelle $\lambda_i$ l'\emph{exposant de Lyapunov}\index{exposant de Lyapunov} de $x$ selon l'espace $E_i$
et on lui affecte la multiplicit\'e $\dim(E_i)$.
La suite ordonn\'ee des exposants de Lyapunov de $\mu$, compt\'es avec multiplicit\'e, est le \emph{vecteur de Lyapunov}\index{vecteur de Lyapunov} $L(x,\mu)$ de $x$.
Lorsque $\mu$ est ergodique, le vecteur de Lyapunov
ne d\'ependent pas du point $x$.
\bigskip

Consid\'erons une mesure ergodique dont le support $K$
ait une d\'ecomposition domin\'ee $T_KM=E\oplus F$.
Alors l'exposant de Lyapunov maximal de $\mu$ en restriction au fibr\'e $E$
est \'egal (voir~\cite{ledrappier})
\`a la limite
$$\lambda^+(\mu,E)=\lim_{n\to +\infty} \frac 1 n \int \log\|Df^n_{|E}\| d\mu.$$
La proposition suivante issue de~\cite{abc-ergodic}
permet de construire des vari\'et\'es
invariante en tout point r\'egulier d'une mesure ergodique.
Il permet de retrouver une partie de la th\'eorie de Pesin
lorsque la r\'egularit\'e du diff\'eomorphisme est seulement $C^1$
mais en pr\'esence d'une d\'ecomposition domin\'ee.

\begin{proposition}[th\'eor\`eme 3.11 de~\cite{abc-ergodic}]\label{p.mesure}
Consid\'erons une mesure ergodique dont le support $K$
ait une d\'ecomposition domin\'ee $T_KM=E\oplus F$.
Lorsque l'exposant de Lyapunov maximal de $\mu$ en restriction au fibr\'e $E$
est strictement n\'egatif, $\mu$-presque tout point
est hyperbolique le long de $E$ et poss\`ede donc une vari\'et\'e stable forte
tangente \`a $E$.
\end{proposition}

Lorsque $\mu$ est ergodique et que les exposants de Lyapunov sont tous non nuls en $\mu$-presque tout point,
nous disons que $\mu$ est \emph{hyperbolique}.\index{hyperbolicit\'e!mesure}

\section{Pistage g\'en\'eralis\'e}
\begin{definition}
Fixons $N\geq 1$ et un ensemble invariant $K$ muni d'une d\'ecomposition domin\'ee $T_KM=E\oplus F$.
Un segment d'orbite $\{x,f(x),\dots,f^n(x)\}$ de longueur $n=\ell.N>0$
contenu dans $K$ 
est \emph{$N$-hyperbolique}\index{hyperbolicit\'e!segment d'orbite} pour la d\'ecomposition domin\'ee $E\oplus F$
si pour tout $1\leq k\leq \ell$ on a:
$$\prod_{i=0}^{k-1} \|D_{f^{i.N}(x)}f^N_{|E}\|\leq e^{-k} \text{ et }
\prod_{i=0}^{k-1} \|D_{f^{(\ell-i).N}(x)}f^{-N}_{|F}\|\leq e^{-k}.$$
\end{definition}

Liao a d\'emontr\'e~\cite{liao-shadowing,liao-livre} le lemme de pistage suivant
qui g\'en\'eralise le lemme de pistage classique de la th\'eorie hyperbolique
(voir aussi~\cite{gan-shadowing}).

\begin{theoreme}[Pistage g\'en\'eralis\'e, Liao\index{pistage!g\'en\'eralis\'e}]
Soit $K$ un ensemble invariant muni d'une d\'e\-com\-po\-si\-tion domin\'ee
$T_KM=E\oplus F$. Fixons $N\geq 1$ et $\delta>0$.

Il existe alors $\varepsilon>0$ tel que pour toute
famille de segments d'orbites $\{x_i,f(x_i),\dots,f^{n_i}(x_i)\}$,
$i\in \ZZ/s\ZZ$, contenus dans $K$, qui sont $N$-hyperboliques et satisfont
$$\forall i\in \ZZ/s\ZZ,\quad d(f^{n_i}(x_i),x_{i+1})<\varepsilon,$$
il existe une orbite p\'eriodique
$\{y,f(y),\dots, f^\tau(y)=y\}$ de p\'eriode
$\tau=n_1+\dots+n_s$ telle que
$$\forall i\in \{1,\dots, s\} \text{ et } k\in \{0,\dots,n_i\},\quad d(f^{n_1+\dots+n_{i-1}+k}(y),f^k(x_i))<\delta.$$
\end{theoreme}

Avec la proposition~\ref{p.mesure}, ceci permet d'approximer les mesures
ergodiques par des orbites p\'eriodiques.

\begin{corollaire}\label{c.anosov-closing}
Soit $\mu$ une mesure ergodique dont le support $K$ poss\`ede une d\'ecomposition domin\'ee non triviale $T_KM=E\oplus F$ telle que les exposants de Lyapunov de $\mu$
soient strictement n\'egatifs le long de $E$ et strictement positifs
le long de $F$.

Il existe alors une suite d'orbites p\'eriodiques hyperboliques
$(O_i)_{i\in \NN}$ qui converge vers $K$
en topologie de Hausdorff et dont les mesures induites convergent vers $\mu$.

De plus, il existe un entier $N\geq 1$ tel que~(\ref{e.lie})
ait lieu pour tout $i$ assez grand. En particulier, toutes les orbites $O_i$,
sauf un nombre fini, sont homocliniquement li\'ees et $K$ est contenu dans
leur classe homocline.
\end{corollaire}

\section{Lemmes de s\'election}\index{lemme de s\'election}
Liao~\cite{liao-obstruction2,liao-livre} (voir aussi~\cite{wen-selecting}) et Ma\~n\'e~\cite{mane-stabilite}
ont donn\'e d'autres cadres o\`u le lemme de pistage g\'en\'eralis\'e
s'applique~: la difficult\'e est de s\'electionner des segments d'orbites
hyperboliques.

\begin{theoreme}[Lemme de s\'election de Liao]\label{t.selection-liao}
Consid\'erons un ensemble compact invariant $K$ muni d'une d\'ecomposition
$1$-domin\'ee non triviale $T_KM=E\oplus F$ et $\lambda\in (0,1)$,
tels que les deux conditions suivantes
soient satisfaites.
\begin{itemize}
\item[--] Tout sous-ensemble compact invariant de $K$ contient un point $x$
v\'erifiant pour tout $n\geq 1$~:
$$\prod_{i=0}^{n-1} \|Df_{|E}(f^i(x))\|\leq \lambda^n.$$
\item[--] Il existe un point $y\in K$ tel que pour tout $n\geq 1$ on ait~:
$$\prod_{i=0}^{n-1} \|Df_{|E}(f^i(y))\|\geq 1.$$
\end{itemize}
Pour tous $\lambda_-<\lambda_+<1$ proches de $1$,
il existe alors dans tout voisinage de $K$
une orbite p\'eriodique contenant un point $p$ satisfaisant pour tout $n\geq 1$~:
$$\prod_{k=0}^n\|Df_{|E}(f^k(p))\|\leq \lambda_+^n
\text{ et }
\prod_{k=0}^n\|Df_{|E}(f^{-k}(p))\|\geq \lambda_-^n.$$
En particulier, on peut trouver une suite de points p\'eriodiques
homocliniquement reli\'es entre eux et qui converge vers un point de $K$.
\end{theoreme}
\bigskip

Le lemme de s\'election de Ma\~n\'e suppose que l'un des fibr\'es de la d\'ecomposition est uniforme.

\begin{theoreme}[Lemme de s\'election de Ma\~n\'e]\label{t.selection-mane}
Consid\'erons un ensemble compact invariant $K$ muni d'une d\'ecomposition
$1$-domin\'ee non triviale $T_KM=E\oplus F$ et $\lambda\in (0,1)$ tels que
\begin{itemize}
\item[--] le fibr\'e $F$ est uniform\'ement dilat\'e,
\item[--] la dynamique restreinte \`a $K$ n'a pas d'ouvert errant,
\end{itemize}
et tels que les deux conditions suivantes soient satisfaites.
\begin{itemize}
\item[--] Il existe un ensemble dense $\cD\subset K$ de points $x$ v\'erifiant~:
$$\liminf_{n\to \infty} \prod_{i=0}^{n-1}\|Df_{|E}(f^{-i}(x))\|^{1/n}\leq \lambda.$$
\item[--] Il existe un point $y\in K$ tel que pour tout $n\geq 1$ on ait~:
$$\prod_{i=0}^{n-1} \|Df_{|E}(f^i(y))\|\geq 1.$$
\end{itemize}
Alors, la conclusion du th\'eor\`eme~\ref{t.selection-liao} a lieu.
\end{theoreme}

Dans le cas o\`u $K$ est une classe homocline $H(q)$, on peut choisir un ensemble
dense $\cD$ de points p\'eriodiques homocliniquement li\'es \`a $q$.
Les points p\'eriodiques obtenus par le th\'eor\`eme~\ref{t.selection-mane}
peuvent \^etre choisis ``homocliniquement reli\'es'' \`a un point de $\cD$.
On obtient alors le r\'esultat suivant, d\'emontr\'e dans~\cite{bgy}.

\begin{corollaire}[Bonatti-Gan-Yang]
Consid\'erons une classe homocline $H=H(p)$ munie d'une d\'ecomposition domin\'ee
$T_HM=E\oplus F$ telle que $F$ soit uniform\'ement dilat\'e et $\dim(F)=\dim(E^u(p))$.
Si $E$ n'est pas uniform\'ement contract\'e, il existe alors une suite d'orbites p\'eriodiques hyperboliques $(O_i)$ homocliniquement li\'ees \`a $p$ et
pour tout $N\geq 1$, l'une de ces orbites n'a pas de point $N$-hyperbolique
le long de $E$.
\end{corollaire}

\section{Fibr\'es non uniformes}
Voici une application
du lemme de s\'election de Liao, issue de~\cite{model},
qui permet d'analyser l'existence de fibr\'es non uniformes.

\begin{theoreme}\label{t.trichotomie}
Supposons que pour tout $1\leq i < d$ et tout diff\'eomorphisme $g$ $C^1$-proche de $f$
l'ensemble $\per_i(g)$ des points p\'eriodiques d'indice $i$ ait une d\'ecomposition
domin\'ee $T_{\per_i(g)}M=E_i\oplus F_i$ telle que $\dim(E_i)=i$.

Consid\'erons un ensemble compact invariant $K$ ayant une d\'ecomposition domin\'ee
$T_KM=E\oplus F$.
Si le fibr\'e $E$ n'est pas uniform\'ement contract\'e, l'un des cas suivant se produit.
\begin{enumerate}
\item $K$ intersecte une classe homocline $H(p)$ associ\'ee \`a un point p\'eriodique d'indice strictement plus
petit que $\dim(E)$.
\item $K$ intersecte des classes homoclines $H(p_n)$ associ\'ees \`a des points p\'eriodiques ayant un indice \'egal \`a $\dim(E)$
et ayant un exposant de Lyapunov le long de $E$ arbitrairement proche de $0$.
\item $K$ contient un ensemble compact invariant $\Lambda$ muni d'une structure partiellement hyperbolique
$T_\Lambda M=E^s\oplus E^c\oplus E^u$ telle que $E^c$ est de dimension $1$ et $\dim(E^s)<\dim(E)$.
De plus, pour toute mesure ergodique support\'ee par $\Lambda$, l'exposant de Lyapunov le long de $E^c$
est \'egal \`a $0$.
\end{enumerate}
\end{theoreme}

Ce th\'eor\`eme fait naturellement appara\^\i tre des ensembles hyperboliques
ayant une structure partiellement hyperbolique avec un fibr\'e central de dimension $1$.
Puisque les exposants de Lyapunov de toute mesure invariante support\'ee sur cet ensemble
sont nuls le long du fibr\'e central, les techniques pr\'esent\'ees dans ce chapitre
ne permettent pas de d\'ecrire plus pr\'ecis\'ement la dynamique. Nous introduirons
au chapitre~\ref{c.model} les mod\`eles centraux qui permettent d'analyser la dynamique
dans la direction centrale d'un point de vue topologique.

\begin{proof}
Puisque $E$ n'est pas uniform\'ement contract\'e, il existe une mesure ergodique $\mu$
support\'ee par $K$ dont l'exposant de Lyapunov maximal le long de $E$ est positif ou nul.

D'apr\`es le lemme de fermeture ergodique (th\'eor\`eme~\ref{t.closing-ergodique}) et son addendum,
il existe une suite de diff\'eomorphismes $g_n\to f$, et une suite d'orbites p\'eriodiques associ\'ees $(O_n)$ qui converge vers le support de $\mu$
en topologie de Hausdorff et dont les exposants de Lyapunov convergent vers ceux de $\mu$.
D'apr\`es l'hypoth\`ese sur la domination des orbites p\'eriodiques, les orbites $O_n$ ont au plus
un exposant proche de $0$. Par cons\'equent $\mu$ a au plus un exposant de Lyapunov proche de $0$.

Si $\mu$ est hyperbolique, l'indice des orbites $O_n$ est \'egal \`a la dimension des espaces stables de
$\mu$. Par passage \`a la limite, il existe donc une d\'ecomposition domin\'ee $T_{\supp(\mu)}M=E'\oplus F'$,
avec $\dim(E')<\dim(E)$. D'apr\`es le corollaire~\ref{c.anosov-closing},
$K$ intersecte une classe homocline d'indice $\dim(E')$. Ceci donne le premier cas du th\'eor\`eme.

Si $\mu$ n'est pas hyperbolique, on construit de la m\^eme fa\c{c}on une d\'ecomposition
domin\'ee $T_{\supp(\mu)}M=E'\oplus E^c\oplus F$, telle que les exposants de $\mu$
sont strictement n\'egatifs le long de $E'$, strictement positifs le long de $F'$ et nuls le long de $E^c$.
On a $\dim(E^c)=1$.

On peut choisir $\mu$ pour que la dimension de $E'$ soit minimale
et on note $K=\supp(\mu)$.
On en d\'eduit que pour toute mesure $\nu$ support\'ee par $K$,
l'exposant le long du fibr\'e central $E^c$ de $K$ est n\'egatif ou nul.

On peut aussi avoir choisi $K$ minimal pour l'inclusion et ces propri\'et\'es.
Ainsi, pour tout sous-ensemble compact invariant propre $K'\subsetneq K$, 
l'exposant de toute mesure $\nu$ le long du fibr\'e central est strictement n\'egatif.
Il est m\^eme inf\'erieur \`a une constante $-\varepsilon$ car dans le cas contraire
la mesure $\nu$ serait hyperbolique (il y a au plus un exposant proche de $0$).
Comme pr\'ec\'edemment, $K$ intersecterait une classe homocline ayant des
orbites p\'eriodiques d'indice $\dim(E')+1\leq \dim(E)$ dont l'exposant
le long de $E^c$ est $\varepsilon$-proches de $0$. On serait alors dans le cas 1)
(si $\dim(E')+1<\dim(E)$) ou dans le cas 2) (si $\dim(E')+1=\dim(E)$) du th\'eor\`eme.

Si $K$ contient une mesure d'exposant central strictement n\'egatif, on peut appliquer
le lemme de s\'election de Liao et $K$ intersecte une classe homocline ayant des
orbites p\'eriodiques d'indice $\dim(E')+1$ dont l'exposant
le long de $E^c$ est arbitrairement proche de $0$. On est alors \`a nouveau dans le cas 1) ou 2).

Si toutes les mesures support\'ees par $K$ ont un exposant central nul, on est dans le cas 3) du th\'eor\`eme.
\end{proof}

\section{Classes hyperboliques par cha\^\i nes}\label{s.chain-hyp}
Consid\'erons un ensemble invariant $K$, une d\'ecomposition domin\'ee
$T_KM=E\oplus F$ et une famille de plaques $\cW$ tangente \`a $E$.
On d\'efinit les notions suivantes.
\smallskip
\begin{itemize}
\item[--]
$\cW$ est \emph{pi\'eg\'ee}\index{famille de plaques!pi\'eg\'ee} si pour tout $x\in K$ on a
$$f(\overline{\cW_x})\subset \cW_{f(x)}.$$

\item[--] $\cW$ est \emph{finement pi\'eg\'ee}\index{famille de plaques!finement pi\'eg\'ee} si pour une base de voisinages
$U$ de la section $0$ de $E$, il existe~:
\begin{enumerate}
\item une famille de diff\'eomorphismes
$(\varphi_x)_{x\in K}$ de $(E_x)_{x\in K}$ qui est continue en topologie $C^1$
et support\'ee dans $U$~;
\item une constante $\rho>0$ telle que pour tous $x\in K$ on a $B(0,\rho)\subset U\cap E_x$ et~:
$$f(\overline{\cW_x\circ\varphi_x(B(0,\rho))})\subset \cW_{f(x)}\circ \varphi_{f(x)}(B(0,\rho)).$$
\end{enumerate}
\smallskip

Bien s\^ur, si $\cW$ est finement pi\'eg\'ee, il existe une famille de plaques tangentes \`a $E$
(et de diam\`etres arbitrairement petits) qui est pi\'eg\'ee.
Par ailleurs, toute autre famille de plaques $\cW'$ tangente \`a $E$ et localement invariante
est \'egalement finement pi\'eg\'ee~: il existe $\rho>0$ tel que pour tout $x\in K$
la boule $B(0,\rho)\subset E_x$ est envoy\'ee par $\cW'_x$ dans $\cW_x$.
Ceci justifie la d\'efinition~:

\item[--] $E$ est \emph{finement pi\'eg\'ee} s'il existe une famille de plaques tangente \`a $E$ et finement pi\'eg\'ee.
\end{itemize}
\medskip

Nous avons propos\'e~\cite{cp1} une nouvelle d\'efinition d'hyperbolicit\'e affaiblie.

\begin{definition}
Une classe homocline $H(p)$ est \emph{hyperbolique par cha\^\i nes} si
\begin{itemize}
\item[--] elle poss\`ede une d\'ecomposition domin\'ee $T_{H(p)}M=E^{cs}\oplus E^{cu}$ en deux fibr\'es~;
\item[--] il existe une famille de plaques $\cW^{cs}$ tangente \`a $E^{cs}$ pi\'eg\'ee par $f$
et une famille de plaques $\cW^{cu}$ tangente \`a $E^{cu}$ et pi\'eg\'ee par $f^{-1}$~;
\item[--] il existe un point p\'eriodique hyperbolique $q_s$ homocliniquement reli\'e \`a
l'orbite de $p$ dont l'ensemble stable contient $\cW^{cs}_{q_s}$ et
il existe un point p\'eriodique hyperbolique $q_u$ homocliniquement reli\'e \`a
l'orbite de $p$ dont l'ensemble instable contient $\cW^{cu}_{q_u}$.
\end{itemize}
\end{definition}
\medskip

Les familles de plaques $\cW^{cs}$ et $\cW^{cu}$ jouent alors le r\^ole des vari\'et\'es stables et instables locales des ensembles hyperboliques~: elles sont respectivement contenues dans les ensembles stables et instables par cha\^\i nes de la classe $H(p)$. Ceci justifie la terminologie ``hyperbolicit\'e par cha\^\i nes''.
En particulier, une propri\'et\'e de produit local est satisfaite
(voir~\cite{cp1}).

\begin{lemme}\label{l.hyp-chain}
Consid\'erons une classe homocline $H(p)$ hyperbolique par cha\^\i nes.
\begin{enumerate}
\item il existe un sous-ensemble dense d'orbite p\'eriodiques $O$ homocliniquement reli\'ees \`a $p$
telles que pour tout $q\in O$ on a $\cW^{cs}_q\subset W^s(q)$ et $\cW^{cu}_q\subset W^u(q)$~;
\item pour tout $x\in H(p)$ on a $\cW^{cs}_x\subset pW^s(x)$ et $\cW^{cu}_x\subset pW^u(x)$~;
\item tout point d'intersection transverse entre deux plaques $\cW^{cs}_x$ et $\cW^{cu}_y$,
$x,y\in H(p)$, est contenu dans $H(p)$.
\end{enumerate}
\end{lemme}
\medskip

L'hyperbolicit\'e par cha\^\i nes est robuste aux perturbations
(voir~\cite{cp1}).

\begin{theoreme}[Crovisier-Pujals]\label{t.chain-robust}
Consid\'erons une classe homocline $H(p)$ hyperbolique par cha\^\i nes telle que~:
\begin{itemize}
\item[--] $H(p)$ co\"\i ncide avec sa classe de r\'ecurrence par cha\^\i nes, 
\item[--] les familles de plaques $\cW^{cs}$ et $\cW^{cu}$ sont finement pi\'eg\'ees
respectivement par $f$ et $f^{-1}$.
\end{itemize}
Alors pour tout diff\'eomorphisme $g\in \diff^1(M)$ proche de $f$
la classe homocline $H(p_g)$ de $g$ 
associ\'ee \`a la continuation hyperbolique de $p$ est encore hyperbolique par cha\^\i nes.
\end{theoreme}

Nous donnons des exemples de classes hyperboliques par cha\^\i nes robustement non hyperbolique
en section suivante.

\section[Diff\'eomorphismes robustement non hyperboliques]{Diff\'eomorphismes d\'eriv\'es d'Anosov robustement non hyperboliques}
\label{s.exemple-mane}
\index{d\'eriv\'e d'Anosov, diff\'eomorphisme}

Smale a construit~\cite{smale-dynamics} un diff\'eomorphisme hyperbolique ayant un attracteur non trivial
en modifiant un diff\'eomorphisme d'Anosov lin\'eaire du tore $\TT^2$.
Il est obtenu en d\'eformant le diff\'eomorphisme initial pr\`es d'un point fixe~: la perturbation
est petite en topologie $C^0$ mais grande en topologie $C^1$.

Cette id\'ee de d\'eformer au voisinage d'un point fixe a \'et\'e reprise par Ma\~n\'e~\cite{mane-contribution},
puis par Bonatti-Viana~\cite{bv} pour construire un diff\'eomorphisme robustement transitif et non hyperbolique. (Avec cet argument,
on peut aussi construire des exemples d'attracteurs robustes non hyperboliques (voir~\cite{carvalho}).)
Nous expliquons ici comment construire de telles dynamiques dans le cas le plus simple.
Voir aussi~\cite[section 7.1]{bdv}.
\medskip

On consid\`ere un diff\'eomorphisme d'Anosov lin\'eaire $A$ du tore $\TT^3$
avec valeurs propres r\'eelles $0<\lambda_1^s<\lambda_2^s<1<\lambda^u$,
et un point fixe $p$. L'application $A$ peut donc \^etre \'ecrite localement
$$A:(x,y,z)\mapsto (\lambda^u.x, \; \lambda_1^s.y,\; \lambda_2^s.z).$$
\smallskip

Introduisons un diff\'eomorphisme $g$ qui fixe $p$, co\"\i ncide avec $A$
hors d'un petit voisinage de $p$ et de la forme
$$g:(x,y,z)\mapsto (\lambda^u.x,\; g_x(y,z)).$$
Il pr\'eserve donc le feuilletage stable $\cF^{cs}$ de $A$.
On demande \'egalement que
\begin{itemize}
\item[--] $\|Dg_{|E^{cs}}\|<\lambda^u$, de sorte
que $g$ pr\'eserve une domination entre l'espace centre-stable (tangent aux
feuilles de $\cF^{cs}$) et un fibr\'e instable,
\item[--] $g$ contracte strictement les aires le long des feuilles de $\cF^{cs}$.
\end{itemize}
Le diff\'eomorphisme initial $A$ poss\`ede ces propri\'et\'es mais nous allons voir que d'autres
diff\'eomorphismes peuvent \^etre int\'eressants.
\medskip

Fixons alors $a,b\gg 1$ et d\'efinissons un diff\'eomorphisme $f$ qui co\"\i ncide avec $A$
hors d'un voisinage de $p$ et prend la forme suivante au voisinage de $p$
$$f:(x,y,z)\mapsto (\lambda^u.x,\; (ab)^{-1}. g_{ax}(aby,abz)).$$
Pour $a$ grand, $f$ diff\`ere de $A$ dans une boule centr\'ee en $p$
de rayon arbitrairement petit.
Pour $b$ grand, le fibr\'e instable de $f$ est arbitrairement proche de celui de
$A$ et sa dilatation est arbitrairement proche de $\lambda^u$.
\medskip

Les diff\'eomorphismes $C^1$-proches de $f$ poss\`edent encore
un feuilletage centre-stable. Ceci d\'ecoule de~\cite[th\'eor\`emes 7.1 et 7.2]{hps} puisque le feuilletage centre-stable de $f$ est lisse et normalement hyperbolique.
\medskip

Les diff\'eomorphismes proches de $f$ n'ont
qu'une seule classe de r\'ecurrence par cha\^\i nes.
\begin{proposition}\label{p.ex1}
Si l'on choisit $a,b$ assez grands, tout diff\'eomorphisme $C^1$-proche de
$f$ est transitif.
Plus pr\'ecis\'ement, $\TT^3$ est une classe homocline.
\end{proposition}
\begin{proof}
L'argument est le m\^eme que dans~\cite[section 6.2]{bv}.
\end{proof}

\medskip

\begin{proposition}\label{p.ex2}
Si $a,b$ sont assez grands, $\TT^3$ est une classe hyperbolique par cha\^\i nes pour tout diff\'eomorphisem $C^1$-proche de $f$.
\end{proposition}
\begin{proof}
Consid\'erons un point p\'eriodique $q\neq p$ de $A$ et une famille de plaques centre-stable $\cW^{cs}$
pi\'eg\'ee pour $A$ (i.e. une famille continue de vari\'et\'es stables locales).
Puisque $f$ pr\'eserve le feuilletage stable $\cF^{cs}$ de $A$ et est arbitrairement proche de $A$
en topologie $C^0$, les plaques $\cW^{cs}$ sont pi\'eg\'ees par $f$.
Pour $a,b$ suffisamment grands, l'orbite de $q$ co\"\i ncide pour $A$ et pour $f$
et de plus $\cW^{cs}_q\subset W^s(q)$ pour $f$.
Pour les diff\'eomorphismes $h$ proches de $f$ en topologie $C^1$,
il existe un feuilletage centre-stable proche de $\cF^{cs}$.
Par cons\'equent il existe encore une famille de plaques centre-stable $\cW^{cs}_h$
qui est proche de la famille $\cW^{cs}$ pour la topologie $C^1$.
Cette famille est donc \'egalement pi\'eg\'ee pour $h$ et la plaque de la continuation $q_h$
de $q$ est contenue dans la vari\'et\'e stable de $q_h$.

Puisque le fibr\'e $E^u$ est uniform\'ement dilat\'e par $f$, il existe une famille de plaques $\cW^u$
tangente \`a $E^u$ pi\'eg\'ee par $f^{-1}$ et satisfaisant $\cW^u_q\subset W^u(q)$. Ceci est \'egalement v\'erifi\'e
par tout diff\'eomorphisme $h$ proche de $f$ en topologie $C^1$.

Avec la proposition~\ref{p.ex1},
ceci montre que $\TT^3$ est une classe hyperbolique par cha\^\i nes.
\end{proof}

\bigskip

Pour obtenir une dynamique robustement non hyperbolique, il suffit
de choisir $g$ avec des points fixes hyperboliques $p_1,p_2$ d'indices 1 et 2.
Voir la figure~\ref{f.deformation}.
Puisque $p_1,p_2$ appartiennent \`a une m\^eme classe homocline robustement,
tout diff\'eomorphisme proche de $f$ peut \^etre approch\'e par un diff\'eomorphisme
ayant un cycle h\'et\'erodimensionnel.
\begin{figure}[ht]
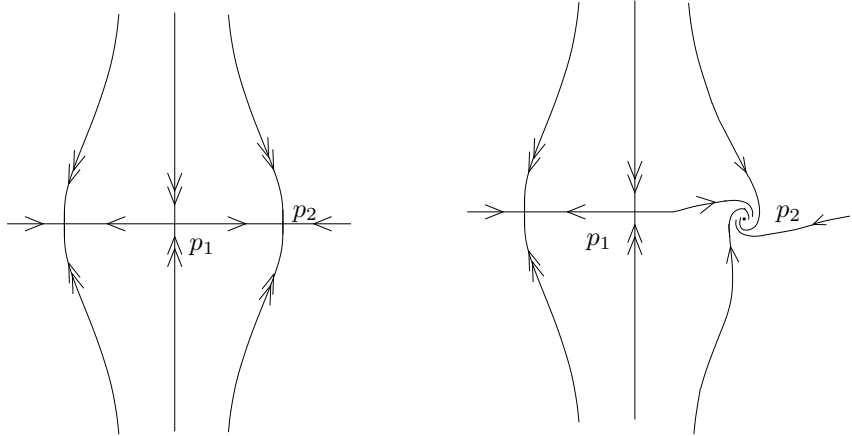

\begin{center}
\mbox{\input{deformation1.pstex_t}\quad\quad\quad\quad
\input{deformation2.pstex_t}}
\end{center}
\caption{D\'eformations de $A$ dans la vari\'et\'e stable locale de $p$. \label{f.deformation}}
\end{figure}
\smallskip

On peut construire $g$ avec une structure partiellement hyperbolique
$TM=E^s\oplus E^c\oplus E^u$ avec trois fibr\'es de dimension $1$.
Aucun diff\'eomorphisme au voisinage de $f$ ne peut alors avoir de tangence homocline
(voir la section~\ref{s.tangence}).
\smallskip

On peut aussi choisir $p_2$ avec des valeurs propres stables complexes.
Dans ce cas, pour aucun diff\'eomorphisme proche de $f$, la classe homocline
$H(p_1)$ (contenant $p_2$) n'a pas de d\'ecomposition domin\'ee de la forme
$TM=E\oplus F$ avec $\dim(E)=1$. On en d\'eduit (voir le th\'eor\`eme~\ref{t.dichotomie-tangence} plus loin)
que tout diff\'eomorphisme proche de $f$ est accumul\'e par des diff\'eomorphismes pour lesquels $p_1$
a une tangence homocline.
\smallskip

Ceci fournit des exemples de dynamiques mod\'er\'ees h\'et\'erodimensionnelle et critique mentionn\'ees en introduction.

\begin{corollaire}
Il existe sur $\TT^3$ des dynamiques g\'en\'eriques mod\'er\'ees h\'et\'erodimensionnelles
critiques et non critiques.
\end{corollaire}


\chapter{R\'eduction de la dimension ambiante}

On peut toujours repr\'esenter
les dynamiques $C^1$-g\'en\'eriques de dimension $d$ au sein des dynamiques g\'en\'eriques
de dimension sup\'erieure~: il suffit de les r\'ealiser sur des sous-vari\'et\'es invariantes
normalement hyperboliques (voir la construction en section~\ref{s.exemple-abraham-smale}).
Ceci permet d'obtenir tr\`es simplement de nouvelles classes d'exemples~:
du ph\'enom\`ene de Newhouse, on d\'eduit l'existence (en dimension $4$) de classes de r\'ecurrence
par cha\^\i nes accumul\'ees par des selles isol\'ees
(des orbites p\'eriodiques hyperboliques qui ne sont ni des sources ni des puits et dont la classe homocline
est triviale).
Dans ce chapitre nous \'etudions le probl\`eme r\'eciproque~: \emph{\'etant donn\'e un ensemble invariant $K$,
peut-on d\'etecter l'existence d'une sous-vari\'et\'e invariante qui le contient~?}
Nous pr\'esentons le crit\`ere issu de~\cite{whitney}.

\section{Vari\'et\'e normalement hyperbolique}
Consid\'erons un ensemble compact invariant $K$
muni d'une d\'ecomposition domin\'ee $T_KM=E\oplus E^u$ telle que $E^u$ est uniform\'ement dilat\'e.
Supposons qu'il existe une sous-vari\'et\'e $N\subset M$ contenant $K$,
tangente \`a $E$ (et de dimension $\dim(E)$), qui soit \emph{localement invariante}~:
il existe un voisinage $U$ de $K$ dans $N$ tel que $f(U)\subset N$
(voir la figure~\ref{f.whitney}).
\begin{figure}[ht]
\begin{center}
\input{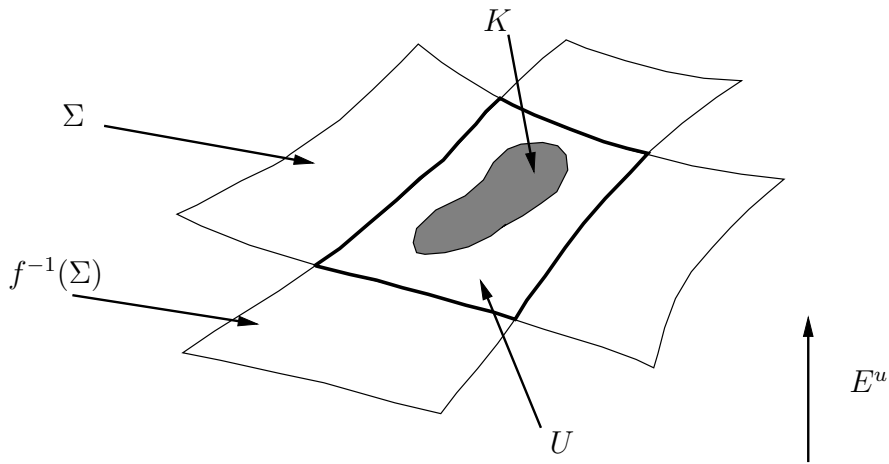}
\end{center}
\caption{Vari\'et\'e normalement hyperbolique
localement invariante au voisinage de $K$. \label{f.whitney}}
\end{figure}

Il est facile de voir que
tout point dont l'orbite est contenue dans un petit voisinage de $K$ appartient \`a $N$.
Un argument classique de transform\'ee de graphe (voir~\cite{hps,whitney})
entra\^\i ne que cette propri\'et\'e est encore v\'erifi\'ee pour les diff\'eomorphismes
proches.

\begin{theoreme}\label{t.variete-normalement-hyperbolique}
Consid\'erons un ensemble compact invariant $K$
muni d'une d\'ecomposition domin\'ee $T_KM=E\oplus E^u$ telle que $E^u$ est uniform\'ement dilat\'e,
et une sous-vari\'et\'e $N\subset M$ contenant $K$,
tangente \`a $E$, qui est localement invariante au voisinage de $K$.

Il existe alors un voisinage $\Sigma$ de $K$ dans $N$ qui est une sous-vari\'et\'e \`a bord
localement invariante par $f$, un voisinage $U$ de $K$ dans $M$
et un voisinage $\cU$ de $f$ dans $\diff^1(M)$ tels que pour tout $g\in \cU$
il existe $\Sigma_g$ sous-vari\'et\'e \`a bord $C^1$-proche de $\Sigma$ v\'erifiant~:
l'ensemble maximal invariant de $g$ dans $U$ est contenu dans $\Sigma_g$
et $g(\Sigma_g\cap U)\subset \Sigma_g$.
\end{theoreme}

\section{Existence de sous-vari\'et\'e localement invariante}
Reprenons le cadre de la section pr\'ec\'edente.
Tout point $x\in K$ poss\`ede une vari\'et\'e instable forte
$W^{uu}(x)=W_{E^u}(x)$ tangente \`a $E^u(x)$.
Il est alors facile de voir que pour tout $x\in K$,
la vari\'et\'e $W^{uu}(x)$ n'intersecte $K$ qu'au point $x$.
En effet, si $W^{uu}(x)$ recoupe $K$ en $y$, en consid\'erant les images
de $x,y$ pour un it\'er\'e $f^{-n}$ avec $n\geq 0$ large, on obtient deux points
$f^{-n}(x),f^{-n}(y)$ arbitrairement proches et joints par une courbe tangente \`a un champ
de c\^one instable, transverse \`a la direction $E$.
\medskip

Cette propri\'et\'e admet une r\'eciproque (voir~\cite{whitney}).
\begin{theoreme}[Bonatti-Crovisier]\label{t.whitney}
Consid\'erons un ensemble compact invariant $K$ ayant une d\'ecomposition domin\'ee
$T_K=E\oplus E^u$ telle que $E^u$ est uniform\'ement dilat\'e.\\
Il existe alors une sous-vari\'et\'e \`a bord $N\subset M$ de dimension $\dim(E)$ qui~:
\begin{itemize}
\item[--] contient $K$ dans son int\'erieur,
\item[--] est tangente \`a $E$ aux points de $K$,
\item[--] est localement invariante au voisinage de $K$,
\end{itemize}
si (et seulement si) la propri\'et\'e suivante est v\'erifi\'ee.
\begin{itemize}
\item[(I)] Pour tout $x\in K$, la vari\'et\'e instable $W^{uu}(x)$
ne rencontre $K$ qu'au point $x$.
\end{itemize}
\end{theoreme}

\section{Id\'ee de la preuve du th\'eor\`eme~\ref{t.whitney}}
Nous utilisons le th\'eor\`eme d'extension de Whitney en classe $C^1$
(voir par exemple~\cite[Appendice A]{abraham-robbin}).

\begin{theoreme}[Whitney]
Consid\'erons une partie ferm\'ee $A\subset \RR^{s}$ et $\varphi\colon A\to \RR^{d-s}$
une application continue. Les deux propri\'et\'es suivantes sont \'equivalentes~:
\begin{enumerate}
\item $\varphi$ s'\'etend en une fonction $\Phi\colon \RR^s\to \RR^{d-s}$ de classe $C^1$,
\item\label{i.whitney2} il existe une application continue $D$ d\'efinie sur $A$ et \`a valeurs dans l'espace
d'applications lin\'eraires $L(\RR^s,\RR^{d-s})$, telle que l'application $R\colon A\times A \to \RR^{d-s}$
d\'efinie par~:
$$R(x,y)=(\varphi(y)-\varphi(x))- D(x).(y-x),$$
v\'erifie~: pour tout $z\in A$, la quantit\'e $\frac{\|R(x,y)\|}{\|y-x\|}$ tend vers $0$ lorsque
les points $x\neq y$ de $K$ tendent vers $z$.
\end{enumerate}
\end{theoreme}

En travaillant dans des cartes, l'ensemble $K$ du th\'eor\`eme~\ref{t.whitney}
peut \^etre localement d\'ecrit comme le graphe d'une application $\varphi \colon \RR^s\to \RR^{d-s}$,
o\`u $s=\dim(E)$~: les axes $\RR^{s}\times \{y\}$ et $\{x\}\times \RR^{d-s}$ sont respectivement
tangents \`a des champs de c\^ones autour des directions $E$ et $E^u$.
En effet, si l'on fixe une section locale $S$ de la lamination par vari\'et\'es instables fortes,
on peut projeter localement $K$ sur $S$ par holonomie.
L'hypoth\`ese (I) du th\'eor\`eme~\ref{t.whitney} se traduit par l'injectivit\'e de cette application.

Le fibr\'e $E$ d\'efinit une application $K\to L(\RR^{s},\RR^{d-s})$.
Si la condition~(\ref{i.whitney2}) n'est pas satisfaite, on peut trouver
des paires de points proches $x,y$ telles que $y-x$ soit uniform\'ement transverse
\`a un champ de c\^ones autour de la direction $E$. Par cons\'equent, la dynamique dilate le vecteur
$y-x$ par it\'erations positives~: il existe $n\geq 1$ tel que la distance $d(f^n(y),f^n(x))$ soit proche de $1$~;
par ailleurs, le vecteur $f^n(y)-f^n(x)$ (vu dans une carte) est proche du champ $E^u$.
En prenant une suite de points $(x,y)$ telle que $d(x,y)$ tende vers $0$, l'entier $n$ tend vers $+\infty$.
Le long d'une suite extraite, $f^n(x)$ converge vers un point $x_0\in K$
et $f^n(y)$ vers un point $y_0\in K\cap W^{uu}(x_0)$, contredisant l'hypoth\`ese (I).
La condition~(\ref{i.whitney2}) du th\'eor\`eme de Whitney est donc satisfaite.
\bigskip

\`A l'aide d'une partition de l'unit\'e, on obtient une vari\'et\'e diff\'erentiable \`a bord $\Sigma_0$
de dimension $\dim(E)$ qui contient $K$ dans son int\'erieur
et qui est tangente \`a $E$ aux points de $K$. Un argument de transform\'ee de graphe,
permet de modifier $\Sigma_0$ pour obtenir une vari\'et\'e localement invariante au voisinage de $K$.

\section{Application~: existence de feuilletages stables}

Nous pouvons retrouver tr\`es simplement l'existence d'un feuilletage stable
au voisinage de tout ensemble hyperbolique pour les diff\'eomorphismes de surface de classe $C^2$
(voir~\cite[appendice A]{palis-takens} pour une d\'emonstration classique).
\begin{theoreme}
Consid\'erons un diff\'eomorphisme $f$ de classe $C^2$ d'une surface $M$
et un ensemble hyperbolique $K\subset M$. Il existe alors un feuilletage $\cF^s$
de classe $C^1$, d\'efini au voisinage de $K$ et localement invariant par $f$,
qui est tangent au fibr\'e stable de $K$.
\end{theoreme}
\begin{proof}
Puisque $f$ est de classe $C^2$, il induit un diff\'eomorphisme
$\widehat f$ de classe $C^1$ sur le fibr\'e unitaire tangent $\pi\colon T^1M\to M$.
Le fibr\'e stable $E^s$ au-dessus de $K$ d\'efini un ensemble compact invariant $\widehat K\subset T^1M$
qui rel\`eve $K$.

En tout point $\widehat x\in \widehat K$, notons $\widehat E^{uu}(x)$
le sous-espace de $T_{\widehat x}T^1M$ tangent aux fibres de $\pi$.
Ceci d\'efinit un fibr\'e continu invariant $\widehat E^{uu}$ au-dessus de $K$.
On v\'erifie facilement qu'il existe $C>0$ tel que pour tout $n\geq 1$, on ait
pour tout $\widehat x\in \widehat K$,
$$\|D_{\widehat x}\widehat f^n_{|\widehat E^{uu}}\|\geq C^{-1}.\frac{\|D_{\pi(\widehat x)}f^n_{|E^u}\|}{\|D_{\pi(\widehat x)}f^n_{|E^s}\|}.$$

Transversalement aux fibres, la norme de $D\widehat f^n$ est born\'ee par
$\|D_{\pi(\widehat x)}f^n_{|E^u}\|$. Ceci montre que $\widehat K$ poss\`ede une d\'ecomposition
domin\'ee en trois fibr\'es $\widehat E^s\oplus \widehat E^u\oplus \widehat E^{uu}$
telle que $\widehat E^s$ et $\widehat E^u$ se projettent par $\pi$ respectivement sur $E^s$ et $E^u$.

En chaque point $\widehat x$, on peut consid\'erer la vari\'et\'e instable forte
$\widehat W^{uu}(x)$, tangente \`a $\widehat E^{uu}(x)$~: par construction c'est
l'ensemble $T_{\pi(\widehat x)}^1M\setminus \{E^u(x)\}$.
Par cons\'equent, elle coupe $\widehat K$ en $\widehat x$ uniquement.
Nous pouvons donc appliquer le th\'eor\`eme~\ref{t.whitney}
et consid\'erer une sous-vari\'et\'e $\widehat N\subset T^1M$
contenant $\widehat K$, tangente \`a $\widehat E^s\oplus \widehat E^u$ et localement invariante.
La projection $\widehat N\to M$ induite par $\pi$ est un diff\'eomorphisme local, injectif sur
$\widehat K$, donc injectif au voisinage de $\widehat K$.
On peut donc interpr\'eter $\widehat N$ comme un champ de droites $C^1$ au voisinage de
$K$ qui est localement invariant par $Df$ et qui \'etend le champ $E^s$ d\'efini au-dessus $K$.
Le feuilletage $\cF^s$ s'obtient alors en int\'egrant ce champ de droites.
\end{proof}

\section{Probl\`emes}

\paragraph{a) Classes pelliculaires.}
Nous avons vu en introduction qu'il existe des classes
de r\'ecurrence par cha\^\i nes accumul\'ees par des orbites p\'eriodiques
qui ne sont ni des puits ni des sources~: pour les exemples construits,
la dynamique est support\'ee par une sous-vari\'et\'e normalement hyperbolique.
On peut se demander s'il existe des exemples plus int\'eressants.

\begin{question}\label{q.pelliculaire}
Existe-t-il un ouvert non vide $\cU\subset \diff^1(M)$
et un G$_\delta$ dense $\cG\subset \cU$ form\'e de diff\'eomorphismes
ayant une classe homocline $H(p)$ avec la propri\'et\'e suivante?
La classe $H(p)$ n'est pas contenue dans une sous-vari\'et\'e
normalement hyperbolique et est accumul\'ee par des orbites p\'eriodiques selles
dont l'indice n'appartient pas \`a l'ensemble des indices des points
p\'eriodiques de $H(p)$.
\end{question}
L. D\'\i az appelle \emph{classe pelliculaire}\index{classe!pelliculaire}
une classe homocline ayant de telles propri\'et\'es.

\paragraph{b) Connexions fortes.}
Lorsque la condition (I) n'est pas v\'erifi\'ee, il peut \^etre utile
de chercher un point p\'eriodique $x\in K$ tel que $W^{uu}(x)\setminus \{x\}$ rencontre $K$.
Ceci peut permettre de cr\'eer des connexions fortes
(voir la section~\ref{s.cycle} pour la d\'efinition).
Nous sommes alors int\'eress\'es par des \'enonc\'es de la forme~:
\emph{$K$ satisfait (I) ou bien poss\`ede une connexion forte.}

Nous obtiendrons dans certains cas de tels \'enonc\'es en sections~\ref{s.discontinu} et~\ref{s.hyperbolicite-quasi-attracteur}.


\chapter{Bifurcations de points p\'eriodiques}\label{c.periodique}
Nous avons rassembl\'e dans ce chapitre divers r\'esultats perturbatifs
sur les suites p\'eriodiques d'applications lin\'eaires.
Gr\^ace au lemme de Franks (th\'eor\`eme~\ref{t.franks}), on obtient des r\'esultats
perturbatifs sur les orbites p\'eriodiques des diff\'eomorphismes.
Comme cons\'equence, nous d\'eduisons que les diff\'eomorphismes $C^1$-g\'en\'eriques
dont l'ensemble r\'ecurrent par cha\^\i nes n'admet pas de d\'ecomposition domin\'e
poss\`ede une infinit\'e de puits ou de sources (ph\'enom\`ene de Newhouse).
Nous d\'emontrons aussi le th\'eor\`eme d'$\Omega$-stabilit\'e
en utilisant le lemme de connexion pour les pseudo-orbites.
Finalement, nous pr\'esentons le th\'eor\`eme de Pujals-Sambarino-Wen-Gourmelon
qui fait le lien entre les tangences homoclines et l'absence de d\'ecomposition domin\'ee.

\section{Cocycles lin\'eaires p\'eriodiques}
\paragraph{D\'efinition.}
Un \emph{cocycle lin\'eaire p\'eriodique}\index{cocycle lin\'eaire p\'eriodique} $A$ est la donn\'ee d'une suite $E$ d'espaces euclidiens $(E_i)_{i\in \ZZ}$
de dimension $d$, d'une suite d'isomorphismes lin\'eaires $A_i\colon E_i\to E_{i+1}$
et d'un entier $\tau\geq 1$,
tels que $E_{i+\tau}=E_i$ et $A_{i+\tau}=A_i$ pour tout $i$.
On appelle $\tau$ sa \emph{p\'eriode}.
Les \emph{valeurs propres de} $A$ sont les valeurs propres de l'endomorphisme $A_\tau\dots A_1$.
Ses \emph{exposants} sont les quantit\'es $\lambda=\frac 1 \tau \log |\sigma|$ o\`u $\sigma$ parcourt l'ensemble
des valeurs propres.

On dit que $A$ est \emph{born\'e} par $K\geq 1$
si pour tout $i$ on a
$$\sup\left(\|A_i\|,\;\|A_{i-1}^{-1}\|\right)\leq K.$$
On dit que les cocycles $A$ et $B$ sont $\varepsilon$-proches si pour tout $i$ on a
$$\sup\left(\|A_i-B_i\|,\;\|A_{i}^{-1}-B_i^{-1}\|\right)\leq \varepsilon.$$

\paragraph{Sous-fibr\'es.}
Si $E'=(E'_i)$ est une suite de sous-espaces de dimension $d'$, \'egalement $\tau$-p\'e\-ri\-o\-di\-que, qui est
invariante par $A$, on obtient par restriction un cocycle $A'$ de $E'$.
Si $A$ est born\'e par $K$, le cocycle $A'$ l'est \'egalement.
R\'eciproquement si $A'$ est un cocycle de $E'$ born\'e par $K$, c'est la restriction d'un cocycle $A$
de $E$ born\'e par $K$. Si $A'$ et $B'$ sont $\varepsilon$-proches, leurs extensions $A$ et $B$ seront
\'egalement $\varepsilon$-proches.

\paragraph{Hyperbolicit\'e.}
On dit que $E$ est \emph{uniform\'ement contract\'e \`a la p\'eriode},
s'il existe $N\geq 1$ tel que pour tout $i\in \ZZ$ on ait~:
\begin{equation}\label{e.contraction}
\prod_{0\leq j < \tau /N} \|(A_{(j+1).N-1+i}\dots A_{j.N+i})\|\leq e^ {-\tau/N}.
\end{equation}
C'est une propri\'et\'e plus forte que la contraction \`a la p\'eriode (toutes les valeurs propres du produit
$A_\tau\dots A_1$ sont de module strictement plus petit que $1$)
mais plus faible que la contraction uniforme (il existe $N\geq 1$ tel que pour tout $i\in \ZZ$
on a $ \|(A_{i+N-1}\dots A_i)\|\leq e^ {-1}$).

On dit que $E$ admet une \emph{d\'ecomposition domin\'ee} $E=F\oplus G$,
s'il existe deux suites $\tau$-p\'eriodiques $F=(F_i)$ et $G=(G_i)$ de sous-espaces suppl\'ementaires
qui sont $A$-invariantes et s'il existe $N\geq 1$ tel que, pour tout $i\in \ZZ$ et tous $u\in F_i$, $v\in G_i$ unitaires,
on ait~:
\begin{equation}\label{e.domination}
\|(A_{i+N-1}\dots A_i).u\|\leq \frac 1 e \|(A_{i+N-1}\dots A_i).v\|.
\end{equation}
La d\'ecomposition est \emph{non triviale} si les dimensions de $F$ et de $H$ ne sont pas nulles.

On dit aussi que $E$ est \emph{$N$-uniform\'ement contract\'e \`a la p\'eriode} ou que la d\'ecomposition $E=F\oplus G$ est \emph{$N$-domin\'ee} pour la dynamique de $A$, lorsque~(\ref{e.contraction}) ou~(\ref{e.domination}) se produisent.

\section{Contraction uniforme \`a la p\'eriode}
Le r\'esultat suivant a \'et\'e obtenu par Pliss~\cite{pliss}
(voir aussi~\cite{liao-stability,liao-livre,mane-ergodic-closing}).
\begin{theoreme}[Pliss]\label{t.pliss}
Pour tous $d,K\geq 1$, $\varepsilon>0$, il existe $N,\tau_0\geq 1$ et $\delta>0$ tels que
pour tout cocycle p\'eriodique lin\'eaire $A$ de dimension $d$, born\'e par $K$
et de p\'eriode $\tau\geq \tau_0$ l'un des cas suivants se produit~:
\begin{itemize}
\item[--] $E$ est $N$-uniform\'ement contract\'e \`a la p\'eriode~;
\item[--] il existe une $\varepsilon$-perturbation $B$ de $A$ qui poss\`ede un exposant positif.
\end{itemize}
\end{theoreme}

\section{Valeurs propres r\'eelles simples}

Il est possible par pertubation de rendre les valeurs propres d'un cocycle lin\'eaire p\'eriodique r\'eelles.
Ceci a \'et\'e d\'emontr\'e dans~\cite{bc} pour la dimension $2$ et g\'en\'eralis\'e en
dimension sup\'erieure dans~\cite{bgv}.

\begin{proposition}[lemme 6.6 de~\cite{bc} et th\'eor\`eme 2.1 de~\cite{bgv}]\label{p.reel}
Pour tous $d,K \geq 1$, $\varepsilon>0$, il existe $\tau_0\geq 1$
tel que tout cocycle p\'eriodique lin\'eaire $A$ de dimension $d$, born\'e par $K$
et de p\'eriode $\tau\geq \tau_0$ poss\`ede une $\varepsilon$-perturbation $B$ telle que~:
\begin{itemize}
\item[--] les valeurs propres de $B$ sont toutes r\'eelles et simples~;
\item[--] pour $1\leq i \leq d$, les $i^{\text{\`emes}}$ exposants de $A$ et de $B$
(compt\'es avec multiplicit\'e) sont $\varepsilon$-proches.
\end{itemize}
\end{proposition}

\section{Domination}

Ma\~n\'e a d\'emontr\'e~\cite{mane-ergodic-closing}
qu'en l'absence de d\'ecomposition domin\'ee,
on peut faire bifurquer les orbites p\'eriodiques de type selle d'un diff\'eomorphisme
de surface pour les transformer en puits ou source.
Ceci a \'et\'e g\'en\'eralis\'e en dimension $3$ par D\'\i az, Pujals et Ures~\cite{dpu} puis
en dimension quelconque par Bonatti, Pujals et D\'\i az dans le cadre des dynamiques
robustement transitives~\cite{bdp}~:
supposons que les diff\'eomorphismes proches de $f\in \diff^1(M)$ soient tous transitifs,
alors $M$ poss\`ede une d\'ecomposition domin\'ee pour $f$.
Plus r\'ecemment, Bonatti, Gourmelon et Vivier ont donn\'e~\cite{bgv} une preuve diff\'erente de ce r\'esultat
qui autorise \`a travailler avec des orbites p\'eriodiques n'appartenant pas \`a un m\^eme ensemble transitif.

\begin{theoreme}[Bonatti-Gourmelon-Vivier]\label{t.bgv}
Pour tous $d,K \geq 1$, $\varepsilon>0$, il existe $N,\tau_0\geq 1$
tel que pour tout cocycle p\'eriodique lin\'eaire $A$ de dimension $d$, born\'e par $K$
et de p\'eriode $\tau\geq \tau_0$ l'un des cas suivants se produit~:
\begin{itemize}
\item[--] $E$ poss\`ede une d\'ecomposition $N$-domin\'ee non triviale~;
\item[--] il existe une $\varepsilon$-perturbation $B$ de $A$ dont toutes les valeurs propres sont r\'eelles
et de m\^eme module.
\end{itemize}
\end{theoreme}

Comme cons\'equence des th\'eor\`emes~\ref{t.pliss} et~\ref{t.bgv}, on retrouve
un r\'esultat de Ma\~n\'e~\cite{mane-ergodic-closing}.

\begin{corollaire}[Ma\~n\'e]\label{c.mane-splitting}
Pour tous $d,K \geq 1$, $\varepsilon>0$, il existe $N,\tau_0\geq 1$
tel que pour tout cocycle p\'eriodique lin\'eaire $A$ de dimension $d$, born\'e par $K$
et de p\'eriode $\tau\geq \tau_0$ l'un des cas suivants se produit~:
\begin{itemize}
\item[--] il existe une d\'ecomposition $N$-domin\'ee $E=E^s\oplus E^u$
telle que $E^s$ et $E^u$ sont $N$-u\-ni\-for\-m\'e\-ment contract\'es \`a la p\'eriode par $A$ et $A^ {-1}$ respectivement~;
\item[--] il existe une $\varepsilon$-perturbation $B$ de $A$ ayant une valeur propre de module $1$.
\end{itemize}
\end{corollaire}
\begin{proof}
Consid\'erons des entiers $N,\tau_0$ sup\'erieur aux entiers donn\'es par les th\'eor\`emes~\ref{t.pliss}
et~\ref{t.bgv} et consid\'erons un cocycle $A$ pour lequel
le deuxi\`eme cas du corollaire ne se produit pas.

Consid\'erons une d\'ecomposition $N$-domin\'ee $E=F\oplus G\oplus H$
(\'eventuellement triviale)
telle que tous les exposants selon $F$ soient strictement n\'egatifs
et tous les exposants selon $H$ soient strictement positifs.
Il existe une telle d\'ecomposition qui maximise les dimensions de $F$ et de $G$.
Nous pouvons appliquer le th\'eor\`eme~\ref{t.bgv} \`a la restriction de $A$
au fibr\'e $G$. Puisqu'il n'existe pas de d\'ecomposition $N$-domin\'ee de $G$,
on peut perturber le cocycle $A$ pour rendre tous ses exposants selon $G$ de m\^emes signes.
Puisque nous ne sommes pas dans le second cas du corollaire, les exposants selon $G$
doivent d\'ej\`a \^etre de m\^emes signes pour $A$, ce qui contredit la maximalit\'e de $F$ ou de $H$.
Nous avons donc montr\'e qu'il existe une d\'ecomposition $N$-domin\'ee
$E=E^s\oplus E^u$
telle que tous les exposants de $A$ selon $E^s$ sont strictement n\'egatifs et tous ceux selon $E^u$ sont strictement positifs.

En appliquant \`a pr\'esent le th\'eor\`eme~\ref{t.pliss} \`a chacun des fibr\'es $E^ {s}$ et $E^u$,
on obtient que $E^s$ et $E^u$ sont $N$-uniform\'ement contract\'es \`a la p\'eriode par $A$ et $A^ {-1}$ respectivement.
\end{proof}

\section{Tangences homoclines}\label{s.tangence}
D'autres bifurcations associ\'ees aux orbites p\'eriodiques hyperboliques
font intervenir les vari\'et\'es invariantes.
C'est la cas des tangences homoclines, particuli\`erement importantes
puisqu'elles engendrent des modifications importantes de la dynamique (voir~\cite{palis-takens,bdv}).

\begin{definition}
Une orbite p\'eriodique hyperbolique $O$
a une \emph{tangence homocline}\index{tangence homocline} s'il existe un point d'intersection
non transverse $z$ entre les vari\'et\'es stables et instables de $O$.
\end{definition}

En dimension deux,
Pujals et Sambarino ont montr\'e~\cite{pujals-sambarino} que
si l'on ne peut pas approcher $f$
par des diff\'eomorphismes pr\'esentant des tangences homoclines,
l'ensemble des points p\'eriodiques de type selle de $f$ poss\`ede
une d\'ecomposition domin\'ee.
Ceci a \'et\'e g\'en\'eralis\'e par Wen~\cite{wen-tangences}
en dimension quelconque.
Gourmelon~\cite{gourmelon-tangence} a am\'elior\'e ces \'enonc\'es en donnant une version
quantitative qui autorise \`a consid\'erer chaque orbite p\'eriodique s\'epar\'ement.

\begin{theoreme}[Pujals-Sambarino, Wen, Gourmelon]\label{t.tangence}
Pour tout voisinage $\cU\subset \diff^1(M)$ de $f$, il existe $N,\tau_0\geq 1$
et un voisinage $\cV$ de $f$ tels que pour $g_0\in \cV$ et toute orbite p\'eriodique hyperbolique $O$ de $g_0$ de p\'eriode $\tau\geq \tau_0$, l'un des deux cas suivants se produit~:
\begin{itemize}
\item[--] la d\'ecomposition $T_OM$ en espaces stables et instables est $N$-domin\'ee,
\item[--] il existe une perturbation $g\in \cU$ qui co\"\i ncide avec $g_0$
sur $O$ et sur un voisinage arbitrairement petit de $O$ telle que
$O$ a une orbite de tangence homocline contenue dans un petit voisinage de $O$.
\end{itemize}
\end{theoreme}

Avec le th\'eor\`eme~\ref{t.pliss} et la proposition~\ref{p.reel},
on en d\'eduit (par une d\'emonstration similaire \`a celle du corollaire~\ref{c.mane-splitting},
voir~\cite[lemmes 3.3 et 3.4]{wen-palis}),

\begin{corollaire}[Wen]\label{c.tangence}
Supposons que $f$ n'est pas limite dans $\diff^1(M)$ de diff\'eomorphismes ayant
une tangence homocline.

Il existe alors
un voisinage $\cU\subset \diff^1(M)$ de $f$ et $N,\tau_0,\delta\geq 1$ tels que
pour tout $g\in \cU$ et toute orbite p\'eriodique $O$ de $g$,
la d\'ecomposition $T_OM=E^s_\delta\oplus E^c_\delta\oplus E^u_\delta$ en espaces caract\'eristiques
dont les exposants de Lyapunov appartiennent respectivement \`a
$(-\infty,-\delta]$, $(-\delta,\delta)$, $[\delta,+\infty)$ v\'erifie~:
\begin{itemize}
\item[--] $E^c_\delta$ est de dimension $0$ ou $1$,
\item[--] la d\'ecomposition $E^s_\delta\oplus E^c_\delta\oplus E^u_\delta$ est $N$-domin\'ee,
\item[--] si la p\'eriode de $O$ est sup\'erieure \`a $\tau_0$, alors
$E^s_\delta$ et $E^u_\delta$ sont $N$-uniform\'ement contract\'es \`a la p\'eriode par $g$ et $g^{-1}$ respectivement.
\end{itemize}
\end{corollaire}
Ce r\'esultat se g\'en\'eralise aux mesures ergodiques
(voir~\cite[Corollaire 1.3]{model} et~\cite{yang-mesure}).
\medskip

L'existence d'une tangence homocline est un ph\'enom\`ene de codimension $1$
et d'apr\`es le th\'eor\`eme~\ref{t.kupka-smale} de Kupka-Smale
l'ensemble des diff\'eomorphismes ayant une tangence homocline forme une partie maigre.
On peut toutefois g\'en\'eraliser la d\'efinition pr\'ec\'edente et d\'efinir une propri\'et\'e qui appara\^\i t sur des ouverts de
diff\'eomorphismes.

\begin{definition}
Une orbite p\'eriodique hyperbolique $O$ poss\`ede une \emph{tangence homocline robuste}
\index{tangence homocline!robuste}
si elle appartient \`a un ensemble hyperbolique transitif $K$ tel que pour tout
diff\'eomorphisme $g\in \diff^1(M)$ proche de $f$ la continuation hyperbolique $K_{g}$
de $K$ pour $g$ poss\`ede deux points $x,y$ tels que $W^s(x)$ et $W^u(y)$ aient une intersection non transverse.
\end{definition}

Toute vari\'et\'e de dimension sup\'erieure ou \'egale \`a $3$ poss\`ede des diff\'eomorphismes
ayant des tangences homocline robustes
(voir~\cite[section 8]{newhouse-cours}, \cite{asaoka} et la section~\ref{s.exemple-abraham-smale}).
En dimension $2$, Newhouse a montr\'e~\cite{newhouse-phenomenon} qu'en topologie $C^2$, il existe des tangences homoclines robustes.
Moreira a montr\'e~\cite{gugu} que ce n'est pas le cas en topologie $C^1$.

\begin{theoreme}[Moreira]\label{t.gugu}
Lorsque $\dim(M)=2$, il n'existe pas de tangence homocline robuste dans $\diff^1(M)$.
\end{theoreme}

\section[Application (1)~: ph\'enom\`ene de Newhouse]{Application (1)~: ph\'enom\`ene de Newhouse en l'absence de d\'ecomposition domin\'ee}
D'apr\`es le th\'eor\`eme~\ref{t.newhouse},
le ph\'enom\`ene de Newhouse\index{ph\'enom\`ene de Newhouse} existe en topologie $C^1$
sur toute vari\'et\'e de dimension sup\'erieure ou \'egale \`a $3$.
Parall\`element \`a ces exemples, Ma\~n\'e a montr\'e~\cite{mane-ergodic-closing}
que tout diff\'eomorphisme de surface
$C^1$-g\'en\'erique est hyperbolique ou bien poss\`ede une infinit\'e de puits ou de sources.
D\'\i az, Pujals, Ures en dimension $3$ (voir~\cite{dpu}), puis Bonatti, D\'\i az, Pujals en dimension quelconque~\cite{bdp} ont obtenu un r\'esultat similaire.

\begin{theoreme}[Bonatti-D\'\i az-Pujals-Ures]\label{t.bdpu}
Pour tout diff\'eomorphisme appartenant \`a un G$_\delta$ dense $\cG\subset \diff^1(M)$,
toute classe homocline poss\`ede une d\'ecomposition
domin\'ee non triviale ou est contenue dans l'adh\'erence des puits ou des sources.
\end{theoreme}

En particulier, toute classe de r\'ecurrence par cha\^\i nes ayant de l'int\'erieur
ou isol\'ee dans $\cR(f)$ poss\`ede une d\'ecomposition domin\'ee non triviale.
\smallskip

En appliquant le corollaire~\ref{c.global} et le th\'eor\`eme~\ref{t.bgv},
nous obtenon~\cite{} une version du r\'esultat pr\'ec\'edent pour les ensembles transitifs
par cha\^\i nes.

\begin{theoreme}[Abdenur-Bonatti-Crovisier]\label{t.newhouseC1}
Il existe un G$_\delta$ dense $\cG\subset \diff^1(M)$
tel que tout ensemble transitif par cha\^\i nes d'un diff\'eomorphisme $f\in \cG$
poss\`ede une d\'ecomposition domin\'ee non triviale ou bien est
limite pour la distance de Hausdorff d'une suite de puits ou de sources.
\end{theoreme}

On en d\'eduit un \'enonc\'e sur la dynamique globale.

\begin{corollaire}[Abdenur-Bonatti-Crovisier]
Pour tout diff\'eomorphisme $C^1$-g\'en\'erique,
\begin{itemize}
\item[--] ou bien il existe une d\'ecomposition de $\Omega(f)$
en un nombre fini d'ensembles compacts invariants
$$\Omega(f)=\Lambda_0\cup\dots\cup\Lambda_d$$
tels que $\Lambda_0\cup \Lambda_d$ est l'union d'un nombre fini de puits et de sources
et $\Lambda_i$ poss\`ede une d\'ecomposition domin\'ee non triviale,
\item[--] ou bien $f$ poss\`ede une infinit\'e de puits ou de sources.
\end{itemize}
\end{corollaire}
\bigskip

Au voisinage d'une classe homocline isol\'ee, le ph\'enom\`ene de Newhouse n'a pas lieu
et on peut obtenir dans ce cadre un r\'esultat plus fort.
\begin{theoreme}[Bonatti-Diaz-Pujals]\label{t.hyperbolicite-volume}
Pour tout diff\'eomorphisme $f$ appartenant \`a un $G_{\delta}$ dense
$\cG\subset \diff^1(M)$, toute classe homocline $H$
qui est isol\'ee dans $\cR(f)$ est hyperbolique en volume\index{hyperbolicit\'e!en volume}.

Plus pr\'ecis\'ement, $H$ est une source, ou un puits, ou poss\`ede une d\'ecomposition
domin\'ee $T_{H(O)}M=E\oplus E^c \oplus F$ telle que
\begin{itemize}
\item[--] $E$ et $F$ sont non d\'eg\'en\'er\'es,
\item[--] il existe $N\geq 1$ tel que $Df^N$ contracte uniform\'ement le volume le long de $E$
et dilate uniform\'ement le volume le long de $F$.
\end{itemize}
\end{theoreme}
\begin{proof}
Si $H(O)$ n'est pas un puits ou une source, il existe une d\'ecomposition
domin\'ee non triviale $T_{H(O)}M=E\oplus F$.
Nous pouvons supposer que $\dim(E)$ est minimale et il s'agit de montrer qu'il existe $N\geq 1$
tel que $Df^N$ contracte uniform\'ement le volume le long de $E$.

La structure riemannienne sur $M$ permet de d\'efinir le jacobien $J(f,E)$ le long de $E$
en chaque point de $H(O)$ comme le module du d\'eterminant de $Df_{|E}$.
On obtient une fonction multiplicative sur $K$, i.e. $J(f^n,E)=(J(f,E)\circ f^{n-1})\dots J(f,E)$ pour tout $n\geq 1$.
Si l'on suppose par l'absurde que le volume le long de $E$ n'est uniform\'ement contract\'e par
$Df^N$ pour aucun entier $N\geq 1$, il existe une mesure de probabilit\'e invariante support\'ee par $K$
telle que $\int J(f,E)d\mu\geq 0$.
Le lemme de fermeture ergodique (th\'eor\`eme~\ref{t.closing-ergodique})
implique qu'il existe une suite d'orbite p\'eriodiques $(O_{n})$ s'accumulant sur une partie de $H(O)$
telles que pour tout $\varepsilon>0$,
la moyenne de $J(f,E)$ long de $O_{n}$ soit sup\'erieure \`a $-\varepsilon$ pour tout $n$ suffisamment grand.
Puisque $f$ satisfait une condition de g\'en\'ericit\'e, on en d\'eduit que tout voisinage de $H(O)$
contient une orbite p\'eriodique $O_{E}$ telle que la moyenne de $J(f,E)$ le long de $O_{E}$ soit strictement positive.

Nous utilisons \`a pr\'esent le fait que $H(O)$ est isol\'ee pour conclure que $O_{E}$ est contenue dans $H(O)$.
La g\'en\'ericit\'e de $f$ implique alors que $H(O)=H(O_{E})$.
La d\'efinition des classes homoclines implique que $H(O)$ est limite de Hausdorff d'orbites p\'eriodiques
sur lesquelles la moyenne de $J(f,E)$ est strictement positive.
Puisque $\dim(E)$ est minimale,
en appliquant le th\'eor\`eme~\ref{t.bgv}, on peut faire bifurquer l'une de
ces orbites p\'eriodiques pour que tous les exposants de Lyapunov le long de $E$ soient strictement positifs~: on obtient
alors une source dans un voisinage de $H(O)$ arbitraire.
Par un argument de g\'en\'ericit\'e, nous en d\'eduisons que $H(O)$ est limite de Hausdorff de sources,
contredisant le fait que cette classe soit isol\'ee et ne soit pas elle-m\^eme une source.
\end{proof}

\begin{remarque}\label{r.dichotomie}
Le point cl\'e de la d\'emonstration consiste \`a montrer que s'il existe
une suite d'orbite p\'eriodiques $(O_{n})$ s'accumulant sur une partie $K$ de $H(O)$
et dont la moyenne de $J(f,E)$ long de $O_{n}$ soit positive, on peut se ramener \`a $K=H(O)$.

Nous avons utilis\'e la fait que $H(O)$ est isol\'e pour voir que les orbites $O_{n}$ sont contenues dans $H(O)$
puis nous nous sommes servi de la notion de points p\'eriodiques homocliniquement reli\'es pour conclure
$K=H(O)$. Ce second argument sera pleinement exploit\'e au chapitre~\ref{c.homocline}.
\end{remarque}

\section{Application (2)~: caract\'erisation de la stabilit\'e}\label{s.stabilite}
Un probl\`eme, formul\'e par Palis et Smale~\cite{palis-smale} et qui a occup\'e les dynamiciens pendant les ann\'ees 1970-1980,
est la caract\'erisation des dynamiques structurellement stables.
Abraham et Smale ont montr\'e~\cite{abraham-smale,smale-structural-stability} que ces dynamiques ne sont pas dense dans $\diff^1(M)$
lorsque $\dim(M)\geq 3$.

\paragraph{D\'efinition.}
Nous disons qu'un diff\'eomorphisme $f$ est \emph{structurellement stable}\index{stabilit\'e!structurelle}
(pour la topologie $C^1$) si tout diff\'eomorphisme $g$
appartenant \`a  un voisinage de $f$ dans $\diff^1(M)$
est conjugu\'e \`a $f$ par un hom\'eomorphisme de $M$.
Plus g\'en\'eralement, $f$ est \emph{$\Omega$-stable}\index{stabilit\'e!$\Omega$-stabilit\'e} (pour la topologie $C^1$)
si pour tout diff\'eomorphisme $g$ appartenant \`a  un voisinage de $f$,
les syst\`emes $(\Omega(f),f)$ et $(\Omega(g),g)$ sont conjugu\'es par un hom\'eomorphisme
$h\colon \Omega(f)\to \Omega(g)$.

\paragraph{Conditions suffisantes.}
Smale a montr\'e~\cite{smale-dynamics,smale-stability} que les diff\'eomorphismes axiome A sans cycle
sont $\Omega$-stables.

Une fois que l'on sait que les diff\'eomorphismes axiome A sans cycles sont les diff\'eomorphismes
hyperboliques (voir la section~\ref{s.diffeo-hyperbolique}), son r\'esultat est une cons\'equence du
lemme de pistage pour les dynamiques hyperboliques.

Pour un diff\'eomorphisme axiome A satisfaisant l'hypoth\`ese suivante~:
\begin{description}
\item[\it (transversalit\'e forte)\index{transversalit\'e forte}] \quad \quad $\forall x,y\in \Omega(f),\; \forall z\in W^s(x)\cap W^u(y),\;\; T_zM=T_zW^s(x)+T_zW^u(y)$,
\end{description}
Robbin~\cite{robbin} et Robinson~\cite{robinson-stabilite} ont montr\'e que $f$ est structurellement stable
(la d\'emonstration est bien plus d\'elicate que pour l'$\Omega$-stabilit\'e).

\paragraph{Conditions n\'ecessaires.}
Palis et Robinson ont montr\'e~\cite{palis-omega,robinson-kupka-smale} que pour un dif\-f\'e\-o\-mor\-phis\-me axiome A,
les conditions ``sans cycle'' ou ``transversalit\'e forte'' sont n\'ecessaires pour
avoir l'$\Omega$-stabilit\'e ou la stabilit\'e structurelle respectivement.
Palis et Ma\~n\'e ont montr\'e plus tard que l'axiome A lui-m\^eme est n\'ecessaire~\cite{mane-stabilite,palis-stabilite}.
\medskip

On obtient alors le th\'eor\`eme suivant.
\begin{theoreme-de-stabilite}
Pour tout diff\'eomorphisme il y a \'equivalence entre les propri\'et\'es~:
\begin{itemize}
\item[--] $f$ est $\Omega$-stable (pour la topologie $C^1$),
\item[--] $f$ est hyperbolique (d\'efinition~\ref{d.hyperbolique}),
\item[--] $f$ poss\`ede la propri\'et\'e (*) suivante~:
\begin{description}
\item[\it (*)] \it Toute orbite p\'eriodique de
tout $g\in \diff^1(M)$ proche de $f$ est hyperbolique.
\end{description}
\end{itemize}
\end{theoreme-de-stabilite}
L'\'equivalence avec la troisi\`eme propri\'et\'e a \'et\'e montr\'ee par Franks, Aoki et Hayashi~\cite{franks,aoki,hayashi-star}.
\medskip

\begin{proof}[\bf D\'emonstration de l'$\Omega$-stabilit\'e]
Nous avons d\'ej\`a expliqu\'e pourquoi les diff\'eomorphismes hyperboliques sont
$\Omega$-stables.
Supposons \`a pr\'esent que $f\in\diff^1(M)$ est $\Omega$-stable.
D'apr\`es le th\'eor\`eme de Kupka-Smale, il existe des perturbations arbitrairement petites de $f$
dont les nombre d'orbites p\'eriodiques de chaque p\'eriode est fini.
La stabilit\'e implique que $f$ a aussi cette propri\'et\'e.
Supposons que $f$ ait un point p\'eriodique $p$ non hyperbolique~:
avec le lemme de Franks, il est possible de perturber $f$ pour que $p$ ait une valeur propre
racine de l'unit\'e et que la dynamique au voisinage de l'orbite de $p$ soit topologiquement conjugu\'ee
\`a la partie lin\'eaire de $f$ le long de l'orbite de $p$~: on obtient ainsi une infinit\'e de points p\'eriodiques
de m\^eme p\'eriode au voisinage de $p$ pour une petite perturbation de $f$, ce qui contredit la stabilit\'e.
Puisque l'$\Omega$-stabilit\'e est une propri\'et\'e ouverte, nous avons montr\'e que
l'$\Omega$-stabilit\'e entra\^\i ne la propri\'et\'e~(*).
\medskip

Supposons finalement que $f$ satisfait (*).
Remarquons que cette propri\'et\'e permet de suivre contin\^ument toute orbite p\'eriodique le long
d'un chemin de diff\'eomorphismes proches de $f$.
D'apr\`es le corollaire~\ref{c.mane-splitting},
il existe un entier $N\geq 1$ tel que pour tout $0\leq i \leq d$ et tout diff\'eomorphisme $g$ proche de $f$, l'ensemble $\per_i(g)$ des points p\'eriodiques d'indice $i$ poss\`ede une d\'ecomposition $N$-domin\'ee et leurs directions stables et instables sont $N$-uniform\'ement
contract\'ees \`a la p\'eriode par $g$ et $g^{-1}$ respectivement.
Le corollaire~\ref{c.homocline-bornee}
montre que $f$ n'a qu'un nombre fini de classes homoclines distinctes.
Nous en d\'eduisons~:
\begin{affirmation*}
$f$ n'a qu'un nombre fini de classe de r\'ecurrence par cha\^\i nes.
\end{affirmation*}
\begin{proof}
Si ce n'\'etait pas le cas, on peut consid\'erer une suite arbitrairement
longue $U_0\subset U_1\subset \dots\subset U_k=M$ d'ouverts filtrants
(i.e. $f(\overline{U_i})\subset U_i$) telle que $U_{i+1}\setminus \overline{U_i}$
contient une classe de r\'ecurrence
par cha\^\i nes. Par une perturbation arbitrairement petite donn\'ee par le lemme de fermeture,
on peut cr\'eer dans
$U_{i+1}\setminus \overline{U_i}$ une orbite p\'eriodique. Puisque (*) est satisfaite
cette orbite peut \^etre associ\'ee \`a une orbite p\'eriodique $O$ de $f$ et puisque
la suite $(U_0,U_1,\dots,U_k)$ est filtrante, $O$ appartient \`a $U_{i+1}\setminus \overline{U_i}$.
On peut donc trouver $k$ orbites p\'eriodiques pour $f$ qui ne sont pas homocliniquement reli\'ees,
contredisant la borne sur le nombre de classes homoclines de $f$.
\end{proof}
\medskip

Nous montrons par l'absurde que $f$ est hyperbolique.
Si ce n'est pas le cas, il existe une classe de r\'ecurrence par cha\^\i nes $K$
qui n'est pas hyperbolique. Nous avons vu qu'elle est isol\'ee dans $\cR(f)$.
Puisque $f$ a la propri\'et\'e (*), toute orbite p\'eriodique contenue dans un voisinage de $K$
pour un diff\'eomorphisme $g$ proche de $f$ peut \^etre associ\'ee \`a une orbite p\'eriodique de $K$.
D'apr\`es le lemme de fermeture, $K$ contient donc des orbites p\'eriodiques.
Soit $i$ l'indice maximal des points p\'eriodiques contenus dans $K$.
\begin{affirmation*}
$\per_i(K)$ n'est pas hyperbolique.
\end{affirmation*}
\begin{proof}
En effet si $\overline{\per_i(K)}$ est un ensemble hyperbolique $\Lambda$,
il se d\'ecompose en un nombre fini de classes homoclines hyperboliques
disjointes. Soit $H(p)$ l'une d'entre elles. Puisque $H(p)$ est strictement contenu
dans la classe de r\'ecurrence par cha\^\i nes $K$, on peut trouver $x^s\in K\cap W^s(H(p))\setminus \Lambda$
et $x^u\in K\cap W^u(H(p))\setminus \Lambda$. Le lemme de connexion pour les pseudo-orbites
permet par une perturbation de cr\'eer une orbite homocline transverse en $x^s$
pour l'ensemble hyperbolique $H(p)$ et un diff\'eomorphisme proche $g$.
Le diff\'eomorphisme $g$ poss\`ede donc un ensemble hyperbolique $\Lambda_g$ qui est la continuation hyperbolique
de $\Lambda$, ainsi que des points p\'eriodiques d'indice $i$ proches de $x^s$~: nous avons cr\'e\'e par perturbation de nouvelles orbites p\'eriodiques, contredisant (*).
\end{proof}

L'ensemble $\overline{\per_i(K)}$ est union d'un nombre fini de classes homoclines
associ\'ees \`a des points p\'eriodiques d'indice $i$.
L'une d'entre elles, not\'ee $H(p)$, n'est pas hyperbolique.
D'apr\`es ce qui pr\'ec\`ede, elle poss\`ede une d\'ecomposition domin\'ee
$T_{H(p)}M=E\oplus F$ telle que $\dim(E)$ est \'egale \`a la dimension stable de $p$.

\begin{affirmation*}
Le fibr\'e $F$ est uniform\'ement dilat\'e.
\end{affirmation*}
\begin{proof}
Si l'on suppose par l'absurde que ce n'est pas le cas, il existe une mesure
ergodique $\mu$ support\'ee par $H(p)$ qui v\'erifie $\int\|Df^{-N}_{|F}\|d\mu\geq 1$.
En appliquant le lemme de fermeture ergodique, on obtient une orbite p\'eriodique $O$ proche
de $H(p)$ pour un diff\'eomorphisme $g$ proche de $f$ pour laquelle le fibr\'e $F$ n'est pas
$N$-uniform\'ement contract\'e \`a la p\'eriode par $g^{-1}$.
Si $O$ est d'indice $i$, ceci contredit la propri\'et\'e (*).
Sinon $O$ est d'indice sup\'erieur \`a $i$, contredisant le choix d'un indice maximal $i$.
\end{proof}

Pour conclure, nous appliquons le lemme de s\'election de Ma\~n\'e (th\'eor\`eme~\ref{t.selection-mane})
\`a l'ensemble $H(p)$ et \`a la d\'ecomposition $E\oplus F$.
Nous en d\'eduisons que $f$ poss\`ede une orbite p\'eriodique qui n'est pas $N$-uniform\'ement contract\'ee
\`a la p\'eriode le long de $E$. C'est une contradiction.
\end{proof}

\section{Contr\^ole des vari\'et\'es invariantes}

Les techniques de Gourmelon permettent de faire bifurquer une orbite p\'eriodique
comme dans les sections pr\'ec\'edentes, tout en contr\^olant certaines orbites homoclines
(voir~\cite{gourmelon-franks}).
En particulier, ceci permet~\cite{gourmelon-tangence} de cr\'eer des tangences \`a l'int\'erieur
d'une classe homocline.

\begin{theoreme}[Gourmelon]\label{t.dichotomie-tangence}
Toute classe homocline $H(O)$ de $f$ a l'une de ces propri\'et\'es~:
\begin{itemize}
\item[--] il existe une perturbation $g$ de $f$ telle que $O$ poss\`ede une tangence homocline,
\item[--] il existe une d\'ecomposition domin\'ee
$T_{H(O)}M=E\oplus F$ telle que $\dim(E)=\dim(E^s(O))$.
\end{itemize}
\end{theoreme}

Potrie a appliqu\'e~\cite{potrie} les techniques de ce chapitre \`a l'\'etude des classes homoclines
qui sont stable au sens de Lyapunov \`a la fois pour $f$ et $f^{-1}$~:
puisqu'une telle classe $H(p)$ reste un quasi-attracteur pour tout diff\'eomorphisme g\'en\'erique
proche, il n'est pas possible de faire bifurquer l'une de ses orbites p\'eriodiques $O$
en un puits tout en pr\'eservant l'intersection $W^u(p)\cap W^s(O)$.

\begin{theoreme}[Potrie]
Il existe un G$_\delta$ dense $\cG\subset \diff^1(M)$ tel que
pour tout $f\in \cG$, toute classe homocline qui est stable au sens de Lyapunov \`a la fois
pour $f$ et $f^{-1}$ poss\`ede une d\'ecomposition domin\'ee non triviale.
\end{theoreme}

Les bifurcations d'orbites p\'eriodiques avec contr\^ole de la classe homocline seront
plus largement \'etudi\'ee au chapitre~\ref{c.homocline}.

\section{Exemples d'Abraham-Smale}\label{s.exemple-abraham-smale}
Nous reprenons la m\'ethode de construction de Abraham-Smale pour obtenir
les r\'esultats d\'ej\`a mentionn\'es de~\cite{abraham-smale,simon,BD-newhouse,asaoka}
et~\cite[section 8]{newhouse-cours}.
L'id\'ee est de consid\'erer des ensembles hyperboliques transitifs d'une vari\'et\'e
de dimension trois dont le fibr\'e stable est un fibr\'e en droites,
mais dont la vari\'et\'e stable est de dimension deux (elle est feuillet\'ee
par les vari\'et\'es stables des points de l'ensemble hyperbolique
qui sont des sous-vari\'et\'es de dimension un).
\medskip

\begin{theoreme}[Newhouse~\cite{newhouse-cours}, Asaoka~\cite{asaoka}]
Toute vari\'et\'e $M$ de dimension $\geq 3$,
poss\`ede un ouvert non vide $\cU\subset \diff^1(M)$ de dif\-f\'e\-o\-mor\-phis\-mes
ayant une orbite p\'eriodique $O$ qui pr\'esente une tangence homocline robuste.
De plus, la classe homocline $H(O)$
n'a pas de d\'ecomposition domin\'ee non triviale.
\end{theoreme}

D'apr\`es le th\'eor\`eme~\ref{t.newhouseC1}, pour tout diff\'eomorphisme $f$
appartenant \`a un G$_\delta$ dense de $\cU$,
la classe homocline $H$ est donc limite d'une infinit\'e de puits ou de sources.
(Ph\'enom\`ene de Newhouse identifi\'e par Bonatti-D\'\i az.)
Ceci fournit un exemple de dynamique sauvage, mentionn\'e en introduction.
\medskip

\begin{proof}
La construction utilise l'exemple d\^u \`a Plykin~\cite{plykin} (voir aussi~\cite[section 8.9 ]{robinson-livre})
en dimension deux d'un diff\'eomorphisme $f_0$
ayant un ensemble hyperbolique transitif $\Lambda$
d'indice $1$ qui est un attracteur.
Sa vari\'et\'e stable locale est un voisinage $\Sigma$ de $\Lambda$ feuillet\'e
par les vari\'et\'es locales $W^{s}_{loc}(x)$, $x\in\Lambda$.
Voir la figure~\ref{f.plykin}.
\begin{figure}[ht]
\begin{center}
\includegraphics[width=12cm]{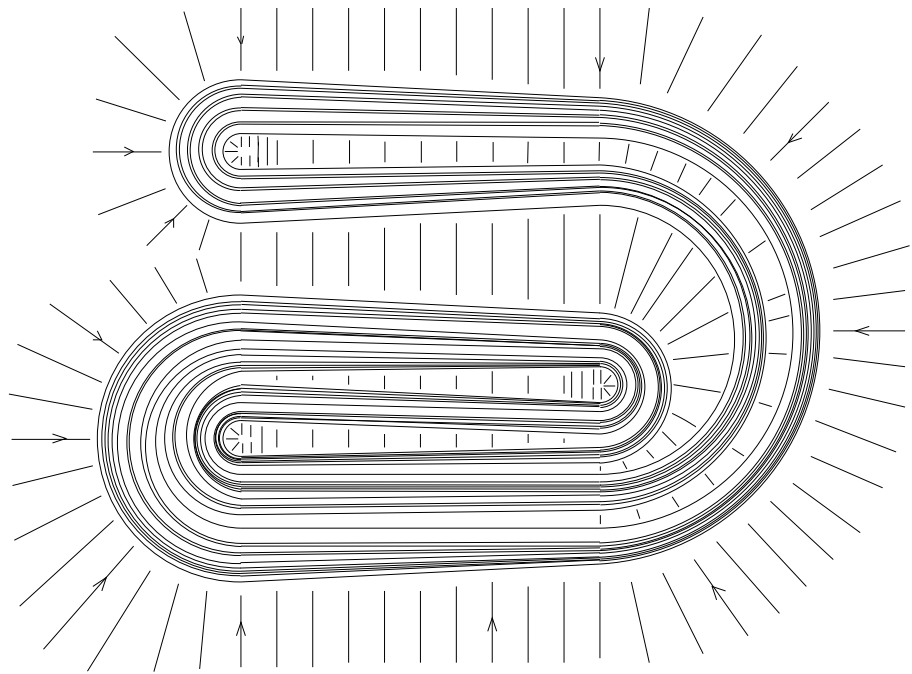}
\end{center}
\caption{Exemple d'attracteur de Plykin. \label{f.plykin}}
\end{figure}
\medskip

On peut alors plonger $\Sigma\times (-1,1)$ dans $M$
et consid\'erer des diff\'eomorphismes $f\in \diff^1(M)$
coincidant avec l'application $(x,y)\mapsto (f_0(x),\mu.y)$
sur $\Sigma\times (-1,1)$.
Pour $\mu>1$ suffisamment grand, la d\'ecomposition $T_zM=E^x\oplus E^y$
aux points $z\in \Lambda$ correspondant aux coordonn\'ees $(x,y)$ est domin\'ee.
On en d\'eduit (voir par exemple le th\'eor\`eme~\ref{t.variete-normalement-hyperbolique})
que pour tout diff\'eomorphisme $g$ proche de $f$, la continuation hyperbolique $\Lambda_g$
de $\Lambda$ est contenue dans une sous-vari\'et\'e invariante $\Sigma_g$ qui est $C^1$-proche de $\Sigma$.

Pour tout $g$ proche de $f$, la surface $\Sigma_g$ est encore contenue dans l'ensemble stable de $\Lambda_g$
et poss\`ede un feuilletage $\cF^{ss}$ en vari\'et\'es stables $W^{ss}_{loc}(z)$, $z\in\Lambda_g$.
Fixons une orbite p\'eriodique $O\subset \Lambda$.
Sa continuation $O_g$ appartient \`a une vari\'et\'e instable $W^u(O_g)$ de dimension $2$
poss\'edant un feuilletage $\cF^{uu}$ par vari\'et\'es instables fortes de dimension $1$.
Voir la figure~\ref{f.plykin3}.
\begin{figure}[ht]
\begin{center}
\input{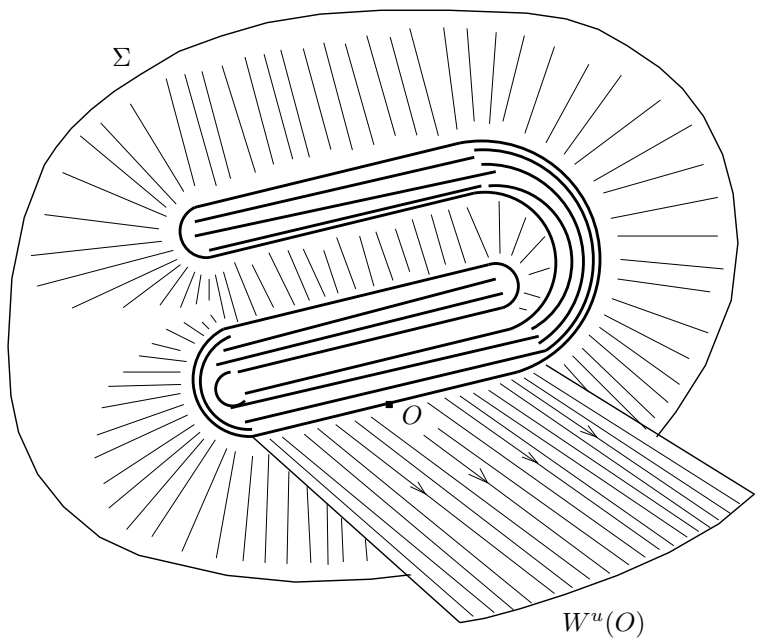}
\end{center}
\caption{Plongement de l'attracteur de Plykin en dimension $3$. \label{f.plykin3}}
\end{figure}
\medskip

On peut alors modifier $f$ hors d'un voisinage de $\Sigma$
pour que $W^{uu}(O)$ et $\Sigma$ aient un point d'intersection transverse $\zeta$
hors de $\Lambda$ tel que la feuille $\cF^{ss}(\zeta)$
soit tangente \`a $W^{uu}(O)$ et de sorte que ces propri\'et\'es restent
satisfaites pour tout diff\'eomorphisme $g\in \diff^1(M)$ proche de $f$
et un point $\zeta_g$ proche de $\zeta$.
Par cons\'equent, l'ensemble hyperbolique $\Lambda_g$ poss\`ede une tangence homocline robuste.
Voir la figure~\ref{f.tangence-robuste}.
\begin{figure}[ht]
\begin{center}
\input{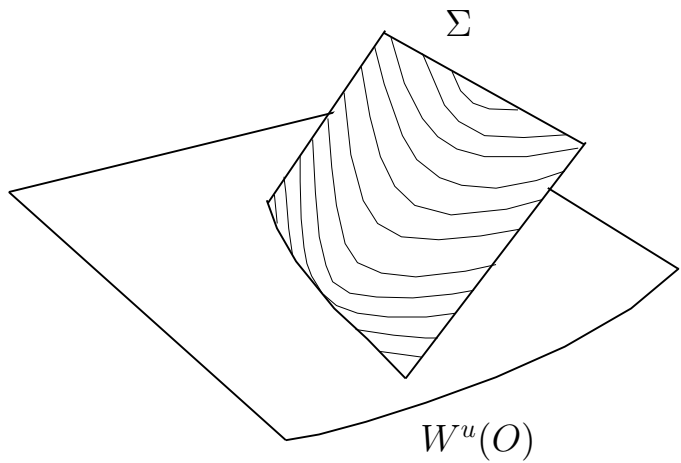}
\end{center}
\caption{Intersection critique de $\cF^{ss}$ et $W^u(O)$. \label{f.tangence-robuste}}
\end{figure}

On peut demander qu'au voisinage de $\xi$, il existe une feuille de $\cF^{ss}$
qui soit transverse \`a $W^u(O)$. Pour tout $g$ proche de $f$, il existe donc un point d'intersection
non transverse $\zeta_g$ entre une feuille de $\cF^{ss}$ et $W^u(O)$
qui soit accumul\'e par des points d'intersections transverses.
Ceci implique que $\zeta_g$ appartient \`a la classe homocline $H(O_g)$ associ\'ee \`a $\Lambda_g$.

De la m\^eme mani\`ere, on peut construire ind\'ependemment un point
d'intersection transverse $\xi$ hors de $\Lambda$ entre $W^u(O)$ et $\Sigma$
tel que la feuille $\cF^{uu}(\zeta)$
soit tangente \`a $\Sigma$ et de sorte que ces propri\'et\'es restent
satisfaites pour tout diff\'eomorphisme $g\in \diff^1(M)$ proche de $f$
et un point $\xi_g$ proche de $\xi$.

La classe homocline $H(O_g)$ ne poss\`ede pas de d\'ecomposition domin\'ee non triviale.
En effet, si elle poss\'edait par exemple une d\'ecomposition $T_{H(O_g)}M=E\oplus F$ avec $\dim(E)=1$,
il est facile de voir que l'espace $E_{\zeta_g}$ co\"\i nciderait avec l'espace tangent \`a la feuille
$\cF^{ss}(\zeta_g)$ en $\zeta_g$ et l'espace $F_{\zeta_g}$ avec l'espace tangent \`a $W^u(O)$.
Par construction ces espaces ne sont pas transverse et on obtiendrait une contradiction.
\end{proof}

\paragraph{Addenda.}
Cette construction offre beaucoup de souplesse.
On peut plonger ind\'ependemment les dynamiques
$(x,y)\mapsto (f_0(x),\mu.y)$ et $(x,y)\mapsto (f_0^{-1}(x),\mu^{-1}.y)$
et ainsi construire un dif\-f\'e\-o\-mor\-phis\-me $f$ de $M$
qui poss\`ede deux ensembles hyperboliques $\Lambda,\Lambda'$, l'un d'indice $1$
et l'autre d'indice $2$, tels que~:
\begin{itemize}
\item[--] $\Lambda$ et $\Lambda'$ sont chacun contenu dans une surface $\Sigma,\Sigma'$
invariante par $f$ et $f^{-1}$ respectivement, ayant un point d'intersection transverse.
\item[--] $\Lambda,\Lambda'$ poss\`edent chacun une orbite p\'eriodique $O,O'$
telles que les vari\'et\'es invariantes (de dimension deux) $W^{u}(O)$ et $W^{s}(O')$
aient un point d'intersection transverse.
\end{itemize}
On obtient donc la propri\'et\'e suivante qui permet d'obtenir le th\'eor\`eme~\ref{t.exemple-cycle}.
\medskip

\emph{Il existe un ouvert non vide $\cU\subset \diff^1(M)$ de diff\'eomorphismes
ayant des ensembles hyperboliques $\Lambda,\Lambda'$ d'indice $1$ et $2$ respectivement
qui pr\'esentent un cycle robuste~: $W^u(\Lambda)\cap W^s(\Lambda')$
et $W^u(\Lambda')\cap W^s(\Lambda)$ sont non vides.}
\medskip

En dimension plus grande, la m\^eme construction permet de lier robustement des ensembles
hyperboliques de tout indice compris entre $1$ et $d-1$ et d'emp\^echer l'existence
de d\'ecomposition domin\'ee non triviale.

\section{Probl\`emes}\label{s.probleme-newhouse}
\paragraph{a) Autres perturbations d'orbites p\'eriodiques.}
D'autres types de perturbations de cocycles pourraient se r\'ev\'eler int\'eressants
et donner lieu \`a des r\'esultats
perturbatifs dans le m\^eme esprit que ceux qui ont \'et\'e pr\'esent\'es dans ce chapitre.
Par exemple, on peut poser les questions suivantes.
\begin{itemize}
\item[--] \emph{Sous quelles conditions peut-on perturber une orbite p\'eriodique pour obtenir
$k>1$ valeur propres de module $1$~?}
\item[--] \emph{Quels sont les indices que l'on peut atteindre par bifurcation~?}
\item[--] \emph{Sous quelles conditions peut-on cr\'eer des valeurs propres complexes~?}
\end{itemize}

Nous renvoyons aussi \`a~\cite{gourmelon-franks} pour d'autres types de contraintes
sur les perturbations de cocycles.

\paragraph{b) Hyperbolicit\'e en volume des classes isol\'ees.}
La d\'emonstration du th\'eor\`eme~\ref{t.hyperbolicite-volume}
permet de pr\'eciser la question~\ref{q.proportion}.
\begin{question}
\emph{Supposons que $f$ v\'erifie une condition $C^1$-g\'en\'erique,
consid\'erons une mesure ergodique $\mu$ pour $f$ et un point $x$
appartenant \`a la classe de r\'ecurrence par cha\^\i nes du support de $\mu$.}

Existe-t-il une suite de points p\'eriodiques $(p_{n})$
convergeant vers $x$ telles que la suite des mesures p\'eriodiques associ\'ees
converge faiblement vers $\mu$~?
\end{question}

On obtiendrait alors une version plus forte des th\'eor\`emes~\ref{t.bdpu} et~\ref{t.hyperbolicite-volume}
et une r\'eponse affirmative \`a la question suivante.

\begin{question}\label{q.dichotomie}
\emph{Supposons que $f$ v\'erifie une condition $C^1$-g\'en\'erique.}

Est-ce que toute classe homocline $H$ (ou plus g\'en\'eralement
tout ensemble transitif par cha\^\i \-nes)
v\'erifie la dichotomie suivante~?
\begin{itemize}
\item[--] $H$ est hyperbolique en volume~;
\item[--] $H$ (ou une partie de $H$) est limite de Hausdorff d'une infinit\'e de puits ou de sources.
\end{itemize}
\end{question}

\paragraph{c) Stabilit\'e structurelle en r\'egularit\'e sup\'erieure.}
Nous avons d\'ej\`a signal\'e que peu de r\'esultats perturbatifs sont connus en r\'egularit\'e
plus grande. En particulier, le lemme de Franks (qui est la motivation principale pour les r\'esultats
de ce chapitre) devient faux (voir~\cite[th\'eor\`eme D]{pujals-sambarino2}).
Nous renvoyons \`a~\cite{pujals-pb} pour une discussion de la stabilit\'e
structurelle $C^r$, $r>1$.


\chapter{Points p\'eriodiques homocliniquement li\'es}\label{c.homocline}

Au chapitre~\ref{c.periodique}, nous avons consid\'er\'e s\'epar\'ement les orbites p\'eriodiques d'un diff\'eomorphisme.
Nous \'etudions \`a pr\'esent les orbites contenues dans une m\^eme classe de r\'ecurrence par cha\^\i nes~:
elles sont li\'ees par la dynamique.
Ceci permet d'une part de propager certaines propri\'et\'es v\'erifi\'ees sur une partie
\`a l'ensemble de la classe et d'autre part d'obtenir de nouvelles orbites p\'eriodiques
comme moyenne d'orbites p\'eriodiques de la classe (sp\'ecification).

\section{Sp\'ecification au sein d'une classe homocline (indice fix\'e)}
Bowen a montr\'e~\cite{bowen-specification} que l'ensemble des orbites p\'eriodiques
contenues dans un ensemble basique d'un diff\'eomorphisme hyperbolique poss\`ede une propri\'et\'e de
sp\'ecification.

\begin{definition}
Un ensemble d'orbites p\'eriodiques $\cO$ a la \emph{propri\'et\'e
de sp\'ecification}\index{sp\'ecification} si pour tout $\varepsilon>0$ et pour toute
collection finie d'orbites $O_1,\dots,O_s\in \cO$, il existe $N\geq 1$
v\'erifiant~:
pour toute suites finies $\{i_1,\dots,i_\ell\}\subset \{1,\dots,s\}$
et $\{n_1,\dots,n_\ell\}\subset \NN$, il existe une orbite p\'eriodique
$O=\{x,f(x),\dots,f^{\tau}(x)=x\}$ appartenant \`a $\cO$
et des entiers $k_1,\dots,k_\ell$ tels que~:
\begin{itemize}
\item[--] $0\leq k_{j+1}-(k_j+n_j)\leq N$ pour tout $j\in \{1,\dots,\ell-1\}$
et $0\leq (k_1+\tau)-(k_{\ell}+n_\ell)\leq N$~;
\item[--] $f^k(x)$ appartient au $\varepsilon$-voisinage de $O_{i_j}$ pour
tous $j\in \{1,\dots \ell\}$ et $k_j\leq k \leq k_j+n_j-1$.
\end{itemize}
\end{definition}

Pour les ensembles hyperboliques, l'entier $N$ ne d\'epend pas des orbites
$O_1,\dots,O_s$. Le r\'esultat de Bowen a \'et\'e g\'en\'eralis\'ee~\cite{bdp}
aux classes homoclines.\footnote{
La notion de sp\'ecification d\'efinie dans~\cite{bowen-specification} est plus g\'en\'erale, puisqu'elle
autorise \`a travailler avec des segments d'orbites qui ne sont pas p\'eriodiques.
La d\'efinition propos\'ee par~\cite{bdp}, sous le terme de ``transition''\index{transition},
est, elle aussi, un peu diff\'erente~: elle d\'ecrit les cocycles lin\'eaires au-dessus d'une
famille d'orbites p\'eriodiques.}

\begin{proposition}[lemme 1.9 de~\cite{bdp}]
Pour tout point p\'eriodique hyperbolique $p$,
l'ensemble des orbites p\'eriodiques hyperboliques homocliniquement reli\'ees \`a $p$
poss\`ede la propri\'et\'e de sp\'e\-ci\-fi\-ca\-tion.
\end{proposition}

Les cocycles lin\'eaires au-dessus d'une familles d'orbites p\'eriodiques
ayant la propri\'et\'e de sp\'ecification permettent plus de perturbations
que les cocycles lin\'eaires g\'en\'eraux \'etudi\'es au chapitre~\ref{c.periodique}.
(Il devient possible de remplacer une orbite p\'eriodique $O$ par une orbite p\'eriodique $O'$
passant une grande proportion de temps proche de $O$. La nouvelle orbite ressemble \`a $O$,
mais tout \'ev\'enement qui appara\^\i t le long de l'orbite
de $O$ appara\^\i t un grand nombre de fois le long de l'orbite de $O'$.)
On obtient une version plus forte du th\'eor\`eme~\ref{t.bgv}.

\begin{theoreme}[Bonatti-D\'\i az-Pujals]\label{t.bdp}
Pour tout $\varepsilon>0$, il existe $N\geq 1$ tel que
toute famille d'orbites p\'eriodiques $\cO$ de $f$ ayant la propri\'et\'e de sp\'ecification
satisfait l'un de ces deux cas~:
\begin{itemize}
\item[--] il existe une d\'ecomposition $N$-domin\'ee non triviale au-dessus de $\cO$,
\item[--] il existe une orbite $O=\{x,f(x),\dots,f^{\tau}(x)=x\}$
dans $\cO$ et une $\varepsilon$-perturbation $B_1,\dots,B_{\tau}$
de la diff\'erentielle $Df(f(x)),\dots,Df(f^\tau(x))$ de $f$ le long de $O$
tels que le produit $B_\tau\dots B_1$ soit une homoth\'etie.
\end{itemize}
\end{theoreme}

\section{Cycles h\'et\'erodimensionnels}\label{s.cycle}

Une autre bifurcation introduite dans~\cite{newhouse-palis} et mettant en jeu les vari\'et\'es invariantes d'orbites p\'eriodiques s'est r\'ev\'el\'ee importante (voir~\cite[chapitre 6]{bdv}).
Elle permet en particulier d'obtenir une propri\'et\'e de sp\'ecification entre points
p\'eriodiques d'indices diff\'erents.

\begin{definition}
Deux orbites p\'eriodiques hyperboliques $O_1,O_2$
forment un \emph{cycle h\'e\-t\'e\-ro\-di\-men\-si\-on\-nel}\index{cycle!h\'et\'erodimensionnel} si leurs indices sont diff\'erents et si
$W^u(O_1)\cap W^s(O_2)$ et $W^u(O_2)\cap W^s(O_1)$ ne sont pas vides.
\end{definition}

Ce type de bifurcation a \'et\'e \'etudi\'e intensivement notamment par L. D\'\i az et ses collaborateurs.
Bien s\^ur, l'ensemble des diff\'eomorphismes ayant un cycle h\'et\'erodimensionnel forme une partie
maigre de $\diff^1(M)$. On peut toutefois renforcer la d\'efinition pr\'ec\'edente.

\begin{definition}
Deux orbites p\'eriodiques hyperboliques $O_1,O_2$
forment un \emph{cycle h\'e\-t\'e\-ro\-di\-men\-si\-on\-nel robuste}\index{cycle!h\'et\'erodimensionnel robuste} s'il existe des ensembles hyperboliques transitifs $K, L$ contenant respectivement $O_1$ et $O_2$,
tels que pour tout diff\'eomorphisme $g\in \diff^1(M)$ proche de $f$, on a
$W^u(K_g)\cap W^s(L_g)\neq \emptyset$ et $W^s(K_g)\cap W^s(L_g)\neq \emptyset$
pour les continuations hyperboliques $K_g, L_g$ de $K$ et $L$.
\end{definition}

La section~\ref{s.exemple-abraham-smale} pr\'esente des diff\'eomorphismes avec cycles h\'et\'erodimensionnels robustes.
\medskip

\paragraph{Connexions fortes.}
Nous pr\'esentons maintenant un cadre qui permet d'obtenir un cycle h\'et\'erodimensionnel
par perturbation.
Consid\'erons le cas o\`u $K$ est muni d'une structure partiellement hyperbolique
$T_{K}M=E^s\oplus E^c\oplus E^u$ avec $\dim(E^c)=1$ et supposons que pour tout voisinage
$U$ de $K$ et tout $\delta>0$, il existe un diff\'eomorphisme $f'$ proche de $f$
et une orbite p\'eriodique $O$ contenue dans $U$ dont l'exposant central appartient
\`a $(-\delta,\delta)$ et telle que $W^{uu}(O)\setminus O$ et $W^{ss}(O)\setminus O$ s'intersectent.
On peut alors cr\'eer un cycle h\'et\'erodimensionnel en appliquant la remarque suivante.

\begin{lemme}\label{l.connexion-forte}
Soit $O$ une orbite p\'eriodique ayant un unique exposant nul
et telle que $W^{uu}(O)\setminus O$ et $W^{ss}(O)\setminus O$ s'intersectent.
Il existe alors un diff\'eomorphisme $C^1$-proche ayant un cycle h\'et\'erodimensionnel.
\end{lemme}
\begin{proof}
Par perturbation, on cr\'ee au voisinage de $O$ deux orbites p\'eriodiques $O_1,O_2$ ayant des exposants
respectivement positifs et n\'egatifs dans la direction centrale et connect\'e par un segment central.
L'intersection entre $W^{uu}(O)\setminus O$ et $W^{ss}(O)\setminus O$ permet
d'obtenir une intersection
entre $W^{uu}(O_2)\setminus O_2$ et $W^{ss}(O_1)\setminus O_1$.
Voir la figure~\ref{f.connexion}.
\end{proof}

\begin{figure}[ht]
\begin{center}
\input{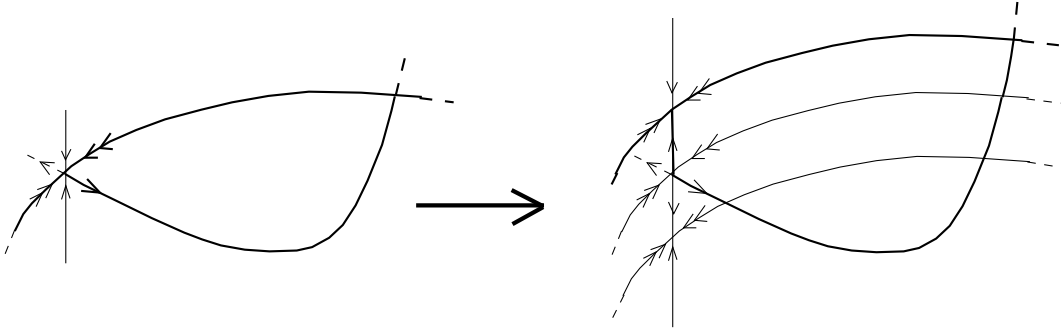}
\end{center}
\caption{Cr\'eation d'un cycle h\'et\'erodimensionnel.\label{f.connexion}}
\end{figure}

Nous dirons qu'une orbite p\'eriodique $O\subset K$
poss\`ede une \emph{connexion forte}\index{connexion forte} lorsque
$W^{uu}(O)\setminus O$ et $W^{ss}(O)\setminus O$ s'intersectent.
\medskip

\paragraph{Connexions fortes g\'en\'eralis\'ees.}
Nous donnons maintenant un m\'ecanisme plus g\'en\'eral que le lemme pr\'ec\'edent pour engendrer des
cycles h\'et\'erodimensionnels (voir par exemple~\cite{cp1}).

\begin{lemme}\label{l.generalisee}
Consid\'erons une classe homocline $H(p)$ v\'erifiant~:
\begin{itemize}
\item[--] il existe une d\'ecomposition domin\'ee $T_{H(p)}M=E\oplus E^c\oplus F$
telle que $\dim(E^c)=1$ et $\dim(E\oplus E^c)$ co\"\i ncide avec la dimension stable de $p$,
\item[--] la vari\'et\'e stable forte $W^{ss}(p)$ tangente \`a $E_p$ coupe $H(p)$ en un point diff\'erent de $p$,
\end{itemize}
et supposons qu'il existe des orbites p\'eriodiques homocliniquement reli\'ees \`a $p$ ayant des exposants centraux
arbitrairement proches de $0$.

Il existe alors un diff\'eomorphisme $C^1$-proche ayant un cycle h\'et\'erodimensionnel.
\end{lemme}

Nous dirons qu'une classe homocline poss\`ede \emph{une connexion forte g\'en\'eralis\'ee}\index{connexion forte!g\'en\'eralis\'ee}
si elle v\'erifie les deux propri\'et\'es du lemme pr\'ec\'edent.

\section{Sp\'ecification au sein d'une classe homocline (indice variable)}
Afin de travailler avec des orbites p\'eriodiques d'indices diff\'erents nous affaiblissons
la propri\'et\'e de sp\'ecification.\index{sp\'ecification}

\begin{definition}
Un ensemble d'orbites p\'eriodiques $\cO$ a la \emph{propri\'et\'e
du barycentre} si pour tous $\varepsilon>0$, $\rho\in (0,1)$
et $O_1,O_2\in \cO$, il existe $O=\{x,f(x),\dots,f^{\tau}(x)=x\}$ appartenant \`a $\cO$
et des parties disjointes $I,J\subset \{1,\dots,\tau\}$ tels que~:
\begin{itemize}
\item[--] $\card(I)\geq (\rho-\varepsilon).\tau$ et
$\card(J)\geq (1-\rho-\varepsilon).\tau$~;
\item[--] $f^k(x)$ appartient au $\varepsilon$-voisinage de $O_{1}$ (resp. de $O_2$)
pour tout $k\in I$ (resp. $k\in J$).
\end{itemize}
\end{definition}

Pour un diff\'eomorphisme $C^1$-g\'en\'erique, la propri\'et\'e du barycentre
est v\'erifi\'ee~\cite{abc-ergodic} au sein des classes homoclines.

\begin{theoreme}[Abdenur-Bonatti-Crovisier]\label{t.barycentre}
Il existe un G$_\delta$ dense $\cG\subset \diff^1(M)$
tel que pour tout $f\in \cG$ et toute classe homocline $H(p)$ de $f$,
l'ensemble des orbites p\'eriodiques de $f$ contenues dans $H(p)$
a la propri\'et\'e du barycentre.
\end{theoreme}

La d\'emonstration reprend des arguments perturbatifs de~\cite{bdpr,abcdw}~:
consid\'erons deux orbites p\'eriodiques $O_1,O_2$ de p\'eriodes $\tau_1,\tau_2$,
contenues dans $H(p)$.
En utilisant la proposition~\ref{p.reel},
on peut toujours supposer que leurs valeurs propres sont r\'eelles et de multiplicit\'e
\'egale \`a $1$. On peut alors perturber $f$ et obtenir un diff\'eomorphisme $\tilde f$
pour lequel $O_1$ et $O_2$ sont reli\'ee par un cycle h\'et\'erodimensionnel
(voir la section~\ref{s.consequence}).

On peut supposer de plus que le cycle est \emph{lin\'eaire}~:
il existe $x\in W^u(O_1)\cap W^s(O_2)$,
$y\in W^u(O_2)\cap W^s(O_1)$
et des cartes $\varphi_1\colon V_1\to \RR^d$,
$\varphi_2\colon V_2\to \RR^d$ au voisinage de $O_1$ et $O_2$ tels que
\begin{itemize}
\item[--] dans les cartes $\varphi_1,\varphi_2$,
l'application $f$ est affine et sa diff\'erentielle est diagonale~;
\item[--] il existe un it\'er\'e n\'egatif $f^{n^-_x}(x)$ (resp. $f^{n^-_y}(y)$)
dont l'orbite n\'egative reste proche de $O_1$ (resp. $O_2$)
et un it\'er\'e positif $f^{n^+_x}(x^+)$ (resp. $f^{n^+_y}(x^-)$)
dont l'orbite positive reste proche de $O_2$ (resp. $O_1$)
tels que l'application $f^{n^-_x+n^+_x}$ (resp. $f^{n^-_y+n^+_y}$)
est affine au voisinage de $f^{n_x^-}(x)$
(resp. $f^{n_y^-}(y)$) et sa diff\'erentielle est diagonale.
\end{itemize}
Pour tous $\ell_1,\ell_2$ suffisamment grands,
on cr\'ee alors par perturbation de $f$ au voisinage de $x$,
une orbite p\'eriodique de p\'eriode $\ell_1.\tau_1+\ell_2.\tau_2+n_x^-+n_x^++n^-_y+n^+_y$
qui poss\`ede $\ell_1.\tau_1$ it\'er\'es cons\'ecutifs dans $V_1$ et
$\ell_2.\tau_2$ it\'er\'es cons\'ecutifs dans $V_2$.
Par construction l'indice de l'orbite cr\'ee $O$ est compris entre
ceux de $O_1$ et $O_2$. De plus, pour un ouvert de diff\'eomorphismes,
$W^u(O)$ intersecte la vari\'et\'e stable de l'une des orbites $O_1,O_2$
et $W^s(O)$ intersecte la vari\'et\'e instable de l'autre.
Puisque pour tout diff\'eomorphisme $C^1$-g\'en\'erique au voisinage de $f$
les classes homoclines de $O_1$ et $O_2$ co\"\i ncident, il existe un diff\'eomorphisme
proche de $f$ pour lequel $O$ appartient \`a cette classe.

\section{Application (1)~: indices des classes homoclines}
Ceci permet de retrouver des r\'esultats de~\cite{bdpr,abcdw}.

\begin{corollaire}[\cite{bdpr,abcdw}]\label{c.indice}
Il existe un G$_\delta$ dense $\cG\subset \diff^1(M)$
tel que pour tout $f\in \cG$ et toute classe homocline $H(p)$ de $f$,
l'ensemble des indices des orbites p\'eriodiques de $f$ contenues dans $H(p)$
est un intervalle de $\NN$.
\end{corollaire}

En combinant ce r\'esultat avec le th\'eor\`eme~\ref{t.dichotomie-tangence},
on obtient l'existence de fibr\'es centraux de dimension $1$ au-dessus des classes
homoclines loin des tangences homoclines.

\begin{corollaire}[Gourmelon~\cite{gourmelon-tangence}]
Il existe un G$_\delta$ dense $\cG\subset \diff^1(M)$
tel que pour tout $f\in \cG$ et toute classe homocline $H(p)$ de $f$
l'une des deux propri\'et\'es suivantes est v\'erifi\'ee.
\begin{itemize}
\item[--] $H(p)$ poss\`ede une d\'ecomposition domin\'ee
$$T_{H(p)}M=E\oplus E^c_1\oplus\dots\oplus E^c_\ell\oplus F,$$
telle que $\dim(E)$ et $d-\dim(F)$ sont respectivement l'indice minimal
et l'indice maximal de $H(p)$ et chaque fibr\'e $E^c_i$ est de dimension $1$.
\item[--] Il existe une orbite p\'eriodique $O$ homocliniquement reli\'ee \`a $p$
et une perturbation $g\in \diff^1(M)$ de $f$ telle que la continuation hyperbolique de $O$ a une
tangence homocline.
\end{itemize}
\end{corollaire}

\section{Application (2)~: mesures g\'en\'eriques port\'ees par une classe homocline isol\'ee}

Voici d'autres cons\'equences tir\'ees de~\cite{abc-ergodic}.

\begin{corollaire}[Abdenur-Bonatti-Crovisier]\label{c.fermeture-isolee}
Il existe un G$_\delta$ dense $\cG\subset \diff^1(M)$
tel que pour tout $f\in \cG$ et toute classe de r\'ecurrence par cha\^\i nes $K$
de $f$, qui est isol\'ee dans $\cR(f)$, toute mesure de probabilit\'e invariante
support\'ee sur $K$ est limite faible de mesures de probabilit\'e support\'ees
sur des orbites p\'eriodiques de $K$.
\end{corollaire}

En effet, la propri\'et\'e est v\'erifi\'ee pour les mesures ergodiques d'apr\`es
le lemme de fermeture ergodique (th\'eor\`eme~\ref{t.closing-ergodique}).
Les orbites p\'eriodiques appartiennent n\'ecessairement \`a $K$.
Les mesures invariantes non ergodiques sont approch\'ees par des barycentres de
mesures ergodiques et on conclut en utilisant la propri\'et\'e du barycentre.

On en d\'eduit le r\'esultat suivant, d\'emontr\'e par Sigmund~\cite{sigmund} dans le cas
des ensembles hyperboliques localement maximaux transitifs et annonc\'e par Ma\~n\'e~\cite{mane-oseledets}.
\begin{theoreme}[Abdenur-Bonatti-Crovisier]
Il existe un G$_\delta$ dense $\cG\subset \diff^1(M)$
tel que pour tout $f\in \cG$ et toute classe de r\'ecurrence par cha\^\i nes $K$
de $f$, qui est isol\'ee dans $\cR(f)$, l'ensemble des mesures de probabilit\'e invariantes $\mu$
support\'ees dans $K$ contient un G$_\delta$ dense de mesures qui sont~:
\begin{itemize}
\item[--] ergodiques,
\item[--] de support total ($\supp(\mu)=K$),
\item[--] d'entropie nulle,
\item[--] hyperboliques,
\item[--] et dont la d\'ecomposition d'Oseledets $T_xM=E_1\oplus\dots\oplus E_s$
en $\mu$-presque tout point s'\'etend \`a l'espace tangent au-dessus de $K$
en une d\'ecomposition domin\'ee.
\end{itemize}
\end{theoreme}

Voici une cons\'equence du th\'eor\`eme pr\'ec\'edent et de la proposition~\ref{p.mesure}.
\begin{corollaire}[Abdenur-Bonatti-Crovisier]
Pour toute classe homocline isol\'ee $H(O)$ d'un diff\'eomorphisme $C^1$-g\'en\'erique et
toute mesure g\'en\'erique $\mu$ support\'ees sur $H(O)$,
presque tout point poss\`ede une vari\'et\'e stable et une vari\'et\'e instable. 
\end{corollaire}

Utilisant la sp\'ecification, D\'\i az et Gorodetski~\cite{dg} obtiennent l'existence de
mesures ergodiques non hyperboliques.

\begin{theoreme}[D\'\i az-Gorodetski]
Il existe un G$_\delta$ dense $\cG\subset \diff^1(M)$ tel que pour tout $f\in \cG$
toute classe homocline contenant des points p\'eriodiques d'indices diff\'erents
contient \'egalement une mesure invariante ergodique ayant au moins un exposant de Lyapunov nul.
\end{theoreme}

\section{Application (3)~: dynamique universelle}\label{s.universelle}

Pour $k\geq 1$,
notons $\DD^k$ la boule standard ferm\'ee unit\'e de $\RR^k$
et $\diff^r_{\interior,+}(\DD^k)$ l'espace des plongements $C^r$ pr\'eservant l'orientation
de $\DD^k$ dans l'int\'erieur d'elle-m\^eme, muni de la topologie $C^r$, $r\in [0,+\infty]$.

Un diff\'eomorphisme est $k$-universel si sa dynamique contient les dynamiques
associ\'ees \`a une partie dense d'\'el\'ements de $\diff^1_{\interior,+}(\DD^k)$.
Donnons une d\'efinition plus pr\'ecise.

\begin{definition}\label{d.universel}
Consid\'erons un entier $1\leq k\leq d$.
Un diff\'eomorphisme $f\in \diff^1(M)$ est \emph{$k$-universel}\index{universel, diff\'eomorphisme}
au voisinage d'un ensemble compact invariant $K$ par $f$ si pour tout ouvert $\cO\subset \diff^1_{\interior,+}(\DD^k)$ et tout voisinage $U$ de $K$, il existe
un plongement $\varphi\colon \DD^k\to M$ et un entier $\ell\geq 1$ v\'erifiant~:
\begin{enumerate}
\item\label{i.universel1} l'image $D$ de $\DD^k$ dans $M$ est disjointe de ses $\ell-1$
premiers it\'er\'es,
\item\label{i.universel2} les ensembles $f^i(D)$, $i\in \{0,\dots,\ell-1\}$ sont contenus dans $U$,
\item\label{i.universel3} $f^\ell(D)$ est contenu dans $D$ et l'application $\varphi^{-1}\circ f^\ell\circ \varphi$
d\'efinie sur $\DD^k$ appartient \`a $\cO$.
\end{enumerate}
Lorsque $K=M$ nous dirons simplement que $f$ est $k$-universel.
\end{definition}

\begin{remarque}
\begin{enumerate}
\item Si $f$ est $k$-universel au voisinage de $K$, l'application $f^{-1}$ l'est \'egalement
puisque pour tout diff\'eomorphisme $h\in \diff^1_{\interior,+}(\DD^k)$, il existe
$g\in \diff^1_{\interior,+}(\DD^k)$ tel que l'image de la boule standard $\frac 1 2 \DD^k$ de rayon $1/2$
par $g$ contient dans son int\'erieur $\frac 1 2 \DD^k$ et l'application $x\mapsto 2.g^{-1}(x/2)$
co\"\i ncide avec $h$ sur $\DD^k$.
\item Les dynamiques g\'en\'eriques $d$-universelles sur une vari\'et\'e de dimension $d$
sont les dynamiques g\'en\'eriques les plus compliqu\'ees puisqu'elles contiennent
toutes les pathologies associ\'ees aux diff\'eomorphismes g\'en\'eriques de
$\diff^1_{\interior,+}(\DD^d)$~: lorsque $d\geq 3$, une telle dynamique poss\`ede
une infinit\'e de puits et de sources,
des cycles h\'et\'erodimensionnels robustes, des tangences homoclines robustes,...
\end{enumerate}
\end{remarque}
\medskip

Voici un crit\`ere~\cite{blender} d'existence de dynamique universelle.
\begin{theoreme}[Bonatti-D\'\i az]\label{t.universel}
Il existe un G$_\delta$ dense de $\diff^1(M)$ dont les dif\-f\'e\-o\-mor\-phis\-mes
sont $d$-universels au voisinage de toute classe homocline $H(O)$ v\'erifiant les propri\'et\'es suivantes~:
\begin{itemize}
\item[--] $H(O)$ n'a pas de d\'ecomposition domin\'ee non triviale,
\item[--] $H(O)$ contient des orbites p\'eriodiques
$\{x,f(x),\dots,f^{\tau-1}(x)\}$ dont le jacobien
\`a la p\'eriode $|\det Df_x^\tau|$ est de module strictement plus grand que $1$,
et des orbites dont le jacobien moyen \`a la p\'eriode est de module strictement plus petit que $1$.
\end{itemize}
\end{theoreme}
\begin{proof}[Id\'ee de la d\'emonstration]
On utilisera le lemme \'el\'ementaire suivant
(voir~\cite[Chapitre 8]{hirsch}).

\begin{lemme}\label{l.isotopie}
Pour tout $k\geq 1$, et tout \'el\'ement $h\in \diff^\infty_{\interior,+}(\DD^k)$,
il existe une application diff\'erentiable $H\colon \RR^k\times [0,1]\to \RR^k$
telle que
\begin{itemize}
\item[--] pour tout $t\in [0,1]$, $h_t=H(.,t)$ est un diff\'eomorphisme de $\RR^k$
qui co\"\i ncide avec l'identit\'e hors d'une partie compacte de $\RR^k$~;
\item[--] $h_0$ est l'identit\'e de $\RR^k$~;
\item[--] la restriction de $h_1$ \`a $\DD^k$ co\"\i ncide avec $h$.
\end{itemize}
\end{lemme}
\begin{proof}[D\'emonstration du lemme]
On peut se ramener au cas o\`u $h$ fixe $0$.

On construit tout d'abord une application $H\colon \DD^k\times [0,1]\to \RR^k$ de classe $C^\infty$
telle que $h_0$ est l'identit\'e sur $\DD^k$ et $h_1$ co\"\i ncide avec $h$ sur $\DD^k$.
En effet, on peut choisir un chemin diff\'erentiable dans $GL(n,\RR)$
entre $\id$ et $D_0h$ param\'etr\'e par $[0,1/2]$.
Pour $s\in (0,1/2]$, on d\'efinit $h_{s+1/2}$ comme restriction \`a $\DD^d$ de l'application
$x\mapsto s^{-1}.h(s.x)$.

En diff\'erentiant l'application $(x,t)\mapsto (h_t(x),t)$ on obtient un champs
de vecteurs de la forme $(\frac \partial{\partial t} h_t, 1)$. En l'\'etendant au moyen d'une partition de l'unit\'e, on d\'efinit un champs de vecteurs d\'ependant du temps $(X_t)_{t\in [0,1]}$
sur $\RR^d$ \`a support compact. L'application $H$ est alors le flot associ\'e.
\end{proof}
\medskip

Consid\'erons un voisinage ouvert $U$ de $H(O)$
et un \'el\'ement $h\in \diff^\infty_{\interior,+}(\DD^d)$.
\smallskip

En utilisant le th\'eor\`eme~\ref{t.barycentre}, on d\'eduit
qu'il existe dans $H(O)$ un ensemble dense de points p\'eriodiques
$\{x,f(x),\dots,f^{\tau-1}(x)\}$ dont le jacobien
$\det Df_x^\tau$ est positif et dont le jacobien moyen
\`a la p\'eriode $(\det Df_x^\tau)^{1/\tau}$ est arbitrairement proche de $1$.
On peut aussi choisir cet ensemble pour qu'il ait la propri\'et\'e de sp\'ecification.
Puisque $H(O)$ n'a pas de d\'ecomposition domin\'ee, il existe
(th\'eor\`eme~\ref{t.bdp}) une telle orbite p\'eriodique
$\{x,f(x),\dots,f^{\tau-1}(x)\}$ et
une perturbation arbitrairement petite $f_1$ de $f$ pour laquelle
$D_xf_1^\tau$ est une homoth\'etie. Puisque le jacobien moyen \`a la p\'eriode
de l'orbite de $x$ est proche de $1$, on peut aussi perturber pour que $D_xf_1^\tau$ 
soit l'identit\'e de $T_xM$.
Une nouvelle perturbation $f_2$ permet de cr\'eer un plongement
$\varphi_0\colon \DD^d\to M$ v\'erifiant les propri\'et\'es
\ref{i.universel1} et \ref{i.universel2} de la d\'efinition~\ref{d.universel}
et telle que $f_2^\tau$ co\"\i ncide avec l'identit\'e sur $D_0=\varphi_0(\DD^d)$.

D'apr\`es le lemme~\ref{l.isotopie}, il existe
une famille de diff\'eomorphismes
$(h_t)_{t\in [0,1]}$ de $\RR^d$ et $\rho>0$ tels que
$h_0=\id$, $h_1$ co\"\i ncide avec $h$ sur $\DD^d$ et
pour tout $t\in [0,1]$ le support du diff\'eomorphisme $h_t$ soit
contenu dans la boule de rayon $\rho$.
On choisit alors un entier $L\geq 1$ suffisamment grand pour que chaque
diff\'eomorphisme $g_i:=h_{(i+1)/L}\circ h_{i/L}^{-1}$ pour $0\leq i<L$
soit arbitrairement proche de l'identit\'e de $\RR^d$ en topologie
$C^1$. Nous avons ainsi fragment\'e le diff\'eomorphisme $h$~:
$$h_1=g_{L-1}\circ\dots\circ g_0.$$

On peut perturber $f_2$ en un diff\'eomorphisme $f_3$ ayant les m\^emes
propri\'et\'es sauf que $\varphi_0^{-1}\circ f_3^\tau\circ \varphi_0$
co\"\i ncide avec la rotation $R$ d'angle $1/L\in [0,1]$ sur $\DD^d$.
Si $L$ est suffisamment grand, cette perturbation de $f_3$ est suffisamment petite.
Il existe une boule $B\subset \DD^3$ disjointe de ses $L-1$ premiers
it\'er\'es par $R$~: c'est l'image de la boule standard $\rho.\DD^d$ de rayon $\rho$
par une homoth\'etie $T$.
Sur $\varphi_0(R^i(B))$, pour $0\leq i<L$,
on remplace $f_3$ par la composition
$$f_3\circ(\varphi_0\circ R^i\circ T)\circ g_i\circ (T^{-1}\circ R^{-i}\circ\varphi_0^{-1}).$$
Par construction, cette application $f_4$ est une petite perturbation de $f$ dans $\diff^1(M)$.
Si l'on pose $\ell=L.\tau$,
la boule $\varphi_0(B)$ est $\ell$-p\'eriodique, disjointe de ses $\ell-1$ premiers it\'er\'es et
le plongement $\varphi=\varphi_0\circ T$ envoie $\DD^d$ sur une boule $D\subset B$ et v\'erifie
$h=\varphi^{-1}\circ f_4^\ell\circ  \varphi$ sur $\DD^d$.
\medskip

Puisque $\diff^\infty_{\interior,+}(\DD^d)$ est dense dans $\diff^1_{\interior,+}(\DD^d)$, nous avons
montr\'e que pour tout dif\-f\'e\-o\-mor\-phis\-me $f$ ayant une classe homocline $H(O)$
v\'erifiant les hypoth\`ese du th\'eor\`eme~\ref{t.universel}, pour tout voisinage $U$
de $H(O)$, pour tout ouvert non vide $\cO\subset \diff^1_{\interior,+}(\DD^d)$, il existe
une perturbation $h'\in \diff^1(M)$ de $f$ et un plongement $\varphi\colon \DD^d\to M$
tel que les propri\'et\'es \ref{i.universel1}, \ref{i.universel2}, \ref{i.universel3}
de la d\'efinition~\ref{d.universel} soient satisfaites.
On conclut la d\'emonstration par un argument de g\'en\'ericit\'e.
\end{proof}

\bigskip

Les constructions de la section~\ref{s.exemple-abraham-smale} montrent
que pour toute vari\'et\'e $M$ de dimension $d\geq 3$,
il existe un ouvert non vide $\cO\subset \diff^1(M)$
de diff\'eomorphismes ayant une classe homocline qui v\'erifie
les propri\'et\'es du th\'eor\`eme~\ref{t.universel}.
Voici une cons\'equence.

\begin{corollaire}
Pour toute vari\'et\'e de dimension $d\geq 3$, il existe
un ouvert non vide $\cO\subset \diff^1(M)$
et un G$_\delta$ dense de $\cO$ form\'e de diff\'eomorphismes
qui sont $d$-universels au voisinage d'une classe homocline $H(O)$.
\end{corollaire}
\bigskip

Les dynamiques $d$-universelles impliquent l'existence de classes ap\'eriodiques.

\begin{proposition}[Bonatti-D\'\i az]
Pour $d\geq 3$,
il existe un G$_\delta$ dense $\cG\subset \diff^1(M)$
tel que tout diff\'eomorphisme $d$-universel $f\in \cG$ poss\`ede un ensemble
non d\'enombrable de classes ap\'eriodiques qui sont stables au sens de Lyapunov pour $f$ et pour $f^{-1}$.
\end{proposition}
\begin{proof}
Il existe un ouvert $\cO\subset \diff^1_{\interior,+}(\DD^d)$
de diff\'eomorphismes ayant une classe homocline $H(O)$
v\'erifiant les propri\'et\'es du th\'eor\`eme~\ref{t.universel}.
En appliquant inductivement le th\'eor\`eme~\ref{t.universel},
on en d\'eduit qu'il existe une suite de plongements $\varphi_n\colon \DD^d\to M$
et d'entiers $\ell_n$ tels que pour tout $n\geq 1$~:
\begin{itemize}
\item[--] l'image $D_n=\varphi_n(\DD^d)$ est disjointe de ses $\ell_n$ premiers it\'er\'es
et v\'erifie $f^{\ell_n}(D_n)\subset \interior(D_n)$ pour $n$ pair
et $f^{-\ell_n}(D_n)\subset \interior(D_n)$ pour $n$ impair~;
\item[--] le supremum des diam\`etres des ensembles $f^k(\varphi^n(\DD^d))$ pour $n\geq 1$ et
$|k|\leq \ell_n$ tend vers $0$ lorsque $n\to +\infty$~;
\item[--] $\varphi_n^{-1}\circ f^{\ell_n}\varphi_n$ pour $n$ pair
et $\varphi_n^{-1}\circ f^{-\ell_n}\varphi_n$ pour $n$ impair
appartiennent \`a $\cO$~;
\item[--] la suite $(D_n)$ est d\'ecroissante.
\end{itemize}
L'intersection d\'ecroissante
$$K=\bigcap_n \bigcup_{0\leq k < \ell_n} f^k(D_n)$$
poss\`ede une base de voisinages attractifs pour $f$ et une base
de voisinages attractifs pour $f^{-1}$. Par cons\'equent $K$
est une classe de r\'ecurrence par cha\^\i nes qui est stable au sens de
Lyapunov pour $f$ et $f^{-1}$. Par construction elle n'a pas de point p\'eriodique.

Remarquons que dans chaque orbite $D_n\cup\dots\cup f^{\ell_n-1}(D_n)$
on peut choisir deux orbites $D_{n+1}^+,\dots, f^{\ell^+_{n+1}-1}(D^+_{n+1})$
et $D_{n+1}^-,\dots, f^{\ell^-_{n+1}-1}(D^-_{n+1})$ disjointes.
Il existe donc une collection non d\'e\-nom\-bra\-ble de suites $(D_n)$
ayant les propri\'et\'es \'enonc\'ees ci-dessus qui engendrent des
classes ap\'eriodiques deux \`a deux disjointes.
\end{proof}

\begin{remarque}
Les classes ap\'eriodiques ainsi obtenues sont des odom\`etres.
En particulier, leur dynamique est minimale, uniquement ergodique.
\end{remarque}

\section{M\'elangeurs, obtention de bifurcations robustes}
Dans le cadre des dynamiques $C^1$-g\'en\'eriques, Bonatti et D\'\i az ont identifi\'e~\cite{blender,cycles} un m\'ecanisme
(les \emph{m\'elangeurs}\index{m\'elangeur}) permettant de construire des cycles h\'et\'erodimensionnels robustes. On peut se demander si tout diff\'eomorphisme ayant un cycle h\'et\'erodimensionnel
peut \^etre approch\'e par un diff\'eomorphisme ayant un cycle h\'et\'erodimensionnel robuste.

\begin{theoreme}[Bonatti-D\'\i az]
Il existe un G$_\delta$ dense $\cG\subset \diff^1(M)$
tel que, tout $f\in \cG$ ayant deux orbites p\'eriodiques hyperboliques
$O_1,O_2$ d'indices diff\'erents et contenues dans une m\^eme classe de r\'ecurrence par cha\^\i nes
poss\`ede un cycle h\'et\'erodimensionnel robuste.
\end{theoreme}
Le cycle h\'et\'erodimensionnel robuste est contenu dans un petit voisinage de la classe
$H(O_1)=H(O_2)$. (Dans les bons cas, il appartient \`a la classe.)

Bonatti et D\'\i az ont annonc\'e~\cite{tangences} pouvoir rendre robuste une tangence homocline
li\'ee \`a une classe homocline ayant plusieurs indices.
Ceci compl\`ete le th\'eor\`eme~\ref{t.dichotomie-tangence}.

\begin{theoreme}[Bonatti-D\'\i az]
Il existe un G$_\delta$ dense $\cG\subset \diff^1(M)$
tel que pour tout $f\in \cG$ et toute classe homocline $H(O)$
contenant un point p\'eriodique d'indice diff\'erent de $O$,
la dichotomie (robuste) suivante est v\'erifi\'e~:
\begin{itemize}
\item[--] $O$ poss\`ede une tangence robuste,
\item[--] il existe une d\'ecomposition domin\'ee $T_{H(O)}M=E\oplus F$ telle que
$\dim(E)=\dim(E^s(O)$. 
\end{itemize}
\end{theoreme}

\section{Probl\`emes}

Au cours des chapitres~\ref{c.periodique} et~\ref{c.homocline}, nous avons obtenu
des r\'esultats d'existence d'orbites p\'e\-ri\-o\-di\-ques par bifurcation de la dynamique
au sein d'une classe de r\'ecurrence par cha\^\i nes $K$.
Nous pouvons distinguer trois sortes d'orbites p\'eriodiques~:
\begin{itemize}
\item[--] les orbites p\'eriodiques contenues dans la classe $K$,
\item[--] celles qui sont proches de $K$ en topologie de Hausdorff,
\item[--] celles contenues dans un voisinage de $K$.
\end{itemize}

\paragraph{a) Compl\'etude des indices.}\index{indice!compl\'etude}
Par exemple, le corollaire~\ref{c.indice} donne des informations sur l'ensemble
des indices d'une classe homocline.
Nous pouvons naturellement formuler des questions sur l'ensemble
des indices des orbites p\'eriodiques contenues dans un voisinage d'une classe homocline.
Le probl\`eme de l'existence de classes pelliculaires (question~\ref{q.pelliculaire}) consiste \`a d\'ecrire
le lien entre l'ensemble des indices d'une classe et l'ensemble des indices des orbites p\'eriodiques
proches de la classe.

\begin{question}\label{q.completude}
\emph{Supposons que $f$ v\'erifie une condition $C^1$-g\'en\'erique
et soit $K$ une classe de r\'ecurrence par cha\^\i nes.}
\begin{itemize}
\item[--] Supposons que $K$ soit limite de Hausdorff d'orbites p\'eriodiques d'indice $i$ et $j>i$.
Est-elle \'egalement limite d'orbites d'indice $\ell$ pour tout $i<\ell<j$~?
\item[--] Supposons qu'une partie de $K$ soit limite de Hausdorff d'orbites p\'eriodiques d'indice $i$.
Est-ce \'egalement le cas de la classe $K$ toute enti\`ere~?
\end{itemize}
\end{question}
Cette seconde question est bien s\^ur reli\'ee \`a la question~\ref{q.proportion}.
Une r\'eponse positive permettrait \`a partir de la remarque~\ref{r.dichotomie} de r\'epondre \`a la question~\ref{q.dichotomie}.
(Voir aussi la discussion de la section~\ref{s.minimaux})

\paragraph{b) Dynamique au sein des classes homoclines.}
Il serait tr\`es int\'eressant d'\'etendre le corollaire~\ref{c.fermeture-isolee}
aux classes homoclines non isol\'ees. Les r\'esultats de ce chapitre montrent que
l'on peut se ramener au cas des mesures ergodiques.
\begin{question}\label{q.fermeture-homocline}
\emph{Supposons que $f$ v\'erifie une condition $C^1$-g\'en\'erique
et soit $H(O)$ une classe homocline de $f$.}

Les mesures ergodiques support\'ees par $H(O)$ sont-elles limites faibles de mesures
de probabilit\'e invariantes support\'ees par des orbites p\'eriodiques de $H(O)$~?
\end{question}
Un tel ``lemme de fermeture ergodique au sein des classes homoclines'' permettrait par exemple
de r\'epondre \`a la question~\ref{q.dichotomie} ainsi qu'\`a la question~\ref{q.proportion}
pour les classes homoclines.
\bigskip

Nous avons une compr\'ehension relativement fine des orbites p\'eriodiques
\`a l'int\'erieur d'une classe homocline, mais la dynamique reste mal
comprise, en particulier les questions de th\'eorie ergodique.
\begin{question}
Consid\'erons une classe homocline $H(O)$ d'un diff\'eomorphisme
$C^1$-g\'en\'erique. Existe-t-il une mesure d'entropie maximale~?
\end{question}

\bigskip

\paragraph{c) Importance des cycles h\'et\'erodimensionnels.}
Tous les exemples connus de dynamiques robustement non hyperboliques poss\`edent un cycle
h\'et\'erodimensionnel.
Bonatti et D\'\i az~\cite{cycles} ont conjectur\'e que c'\'est toujours le cas.
\begin{conjecture-hyperbolicite}[Bonatti-D\'\i az]\label{conjecture-cycle}
\index{conjecture! d'hyperbolicit\'e}
Tout diff\'eomorphisme peut \^etre approch\'e dans $\diff^1(M)$ par un diff\'eomorphisme qui est hyperbolique
ou poss\`ede un cycle h\'et\'erodimensionnel robuste.
\end{conjecture-hyperbolicite}
Cette conjecture concerne exclusivement la topologie $C^1$ puisqu'elle n'est pas v\'erifi\'ee
par les diff\'eomorphismes $C^2$ des surfaces (du fait du ph\'enom\`ene de Newhouse).
Une \'etape dans cette direction serait de montrer que la pr\'esence de tangences homoclines
implique l'existence de cycles h\'et\'erodimensionnels.

On peut \'enoncer des versions semi-locales de cette conjecture.

\begin{question}
\emph{Supposons que $f$ v\'erifie une condition $C^1$-g\'en\'erique.}

Est-ce que tout voisinage d'un ensemble transitif par cha\^\i ne non hyperbolique contient 
un cycle h\'et\'erodimensionnel robuste~?

Est-ce que toute classe homocline non hyperbolique contient un cycle h\'et\'erodimensionnel
robuste~?
\end{question}

\paragraph{d) Dynamique sur les surfaces.}
Le th\'eor\`eme~\ref{t.exemple-cycle} montre que sur toute vari\'et\'e de dimension sup\'erieure ou \'egale \`a $3$,
il existe des ouverts $C^1$ de diff\'eomorphismes non hyperboliques.
Cette question reste ouvert en dimension $2$. Smale a conjectur\'e~\footnote{Je n'ai pas trouv\'e de trace de cette conjecture sur les surfaces. Au d\'ebut des ann\'ees 1960, il a explicitement pos\'e le probl\`eme de
la densit\'e des dynamiques structurellement stables~\cite{smale-ICM-1962}.
Dans ses textes r\'ecents~\cite{Sm1,Sm2}, Smale conjecturait la densit\'e des dynamiques hyperboliques pour
les endomorphismes en dimension $1$.} que ce n'\'etait pas le cas.
C'est un cas particulier de la conjecture d'hyperbolicit\'e.

\begin{conjecture*}[Smale]\index{conjecture! de Smale}
Lorsque $\dim(M)=2$, l'ensemble des diff\'eomorphismes hyperboliques est dense dans $\diff^1(M)$.
\end{conjecture*}

\paragraph{e) Tangences robustes.}
Par analogie avec les cycles h\'et\'erodimensionnels~\cite{cycles},
on aimerait savoir dans quelle mesure un diff\'eomorphisme pr\'esentant
une tangence homocline est approch\'e par des diff\'eomorphismes
ayant des tangences robustes.

\begin{question}\label{q.tangence-robuste}
Consid\'erons un ouvert $\cU\subset \diff^1(M)$ qui contient
un ensemble dense de diff\'eomorphismes pr\'esentant une tangence homocline.
Existe-t-il un diff\'eomorphisme dans $\cU$ qui poss\`ede une tangence robuste~?
\end{question}

Cette question peut-\^etre d\'ecompos\'ee en deux sous-probl\`emes
(voir~\cite{abcd,bonatti-conjecture}).

\begin{question}
Consid\'erons un diff\'eomorphisme $C^1$-g\'en\'erique
qui est limite de dif\-f\'e\-o\-mor\-phis\-mes ayant une tangence homocline.
Existe-t-il une classe homocline $H(O)$ sans d\'ecomposition
domin\'ee $T_{H(O)}M=E\oplus F$ avec $\dim(E)=\dim(E^s(O))$~?
\end{question}

\begin{question}
Consid\'erons un diff\'eomorphisme $C^1$-g\'en\'erique
et une classe homocline $H(O)$ n'ayant pas de d\'ecomposition domin\'ee
$T_{H(O)}=E\oplus F$ avec $\dim(E)=\dim(E^s(O))$.
Est-ce que $O$ poss\`ede une tangence robuste~?
\end{question}

Sur les surfaces, le th\'eor\`eme~\ref{t.gugu} montre qu'il n'y a pas de tangence robuste
et une r\'eponse positive \`a la question~\ref{q.tangence-robuste}
entra\^\i nerait la conjecture de Smale.
Une approche possible consiste \`a contr\^oler les \emph{points critiques} de la dynamique~:
il r\'esulte du travail de Pujals et Sambarino~\cite{pujals-sambarino} que les diff\'eomorphismes
de surface g\'en\'eriques et non hyperboliques poss\`edent des classes sans d\'ecomposition domin\'ee.
Lorsque la dynamique sur une telle classe est dissipative (le jacobien de toute mesure invariante
est strictement n\'egatif), Pujals et Rodriguez-Hertz~\cite{pujals-hertz}
ont introduit un ensemble critique~: un ensemble ferm\'e tel que tout compact invariant transitif
qui ne le rencontre pas est un puits ou poss\`ede une d\'ecomposition domin\'ee.


\chapter[Loin des tangences homoclines, mod\`eles centraux]{Dynamique loin des tangences homoclines~: mod\`eles centraux}\label{c.model}
L'un des objectifs de ce chapitre est l'\'etude des dynamiques
qui ne peuvent pas \^etre approch\'es par des diff\'eomorphismes
ayant des tangences homoclines. Une motivation est la conjecture de Palis
pr\'esent\'ee en section~\ref{s.structure}.

Un aspect fondamental de cette approche consiste
\`a comprendre la dynamique le long d'un fibr\'e de dimension un.
Lorsque les exposants de Lyapunov sont nuls pour toute mesure invariante,
la dynamique de l'application tangente donne tr\`es peu d'informations.
Nous d\'eveloppons de nouveaux outils, introduits dans~\cite{palis-faible,model}
pour obtenir des propri\'et\'es d'hyperbolicit\'e topologique
ou pour trouver des obstructions \`a l'hyperbolicit\'e qui apparaissent sur les orbites p\'eriodiques.

Nous utilisons les mod\`eles centraux pour montrer~\cite{palis-faible} que tout diff\'eomorphisme
peut \^etre approch\'e dans $\diff^1(M)$ par un diff\'eomorphisme Morse-Smale ou par un diff\'eomorphisme
ayant une intersection homocline transverse.
Nous \'etablissons un r\'esultat~\cite{model} en direction de la conjecture de Palis~:
tout diff\'eomorphisme qui ne peut pas \^etre approch\'e dans $\diff^1(M)$ par un diff\'eomorphisme
ayant une tangence homocline ou un cycle h\'et\'erodimensionnel
est partiellement hyperbolique~: l'ensemble r\'ecurrent par cha\^\i nes se d\'ecompose en un nombre fini
d'ensembles compacts invariants partiellement hyperboliques dont le fibr\'e central est de dimension
au plus \'egale \`a deux. Finalement, nous \'etudions les quasi-attracteurs des dynamiques
g\'en\'eriques loin des tangences homoclines~: en particulier, J. Yang a montr\'e~\cite{yang-lyapunov-stable}
que ce sont des classes homoclines.

\section{M\'ecanismes versus ph\'enom\`enes}\label{s.structure}
Un des buts de l'\'etude des diff\'eomorphismes $C^1$-g\'en\'eriques
est de structurer l'espace des diff\'eomorphismes en fonction des types de dynamiques
qui apparaissent. Voir~\cite{shub-genericite,smale-survey}.
Par exemple, les diff\'eomorphismes hyperboliques peuvent \^etre class\'es en trois
ouverts de complexit\'e croissante~:
les diff\'eomorphismes Morse-Smale, les diff\'eomorphismes structurellement stables
et les diff\'eomorphismes $\Omega$-stables.
Plus g\'en\'eralement, on cherche des parties (de pr\'ef\'erence ouvertes) de l'espace des diff\'eomorphismes
 pour lesquelles la dynamique peut \^etre d\'ecrite
globalement, avec un degr\'e de pr\'ecision variable~: on peut simplement chercher \`a
d\'ecrire la d\'e\-com\-po\-si\-tion en classes de r\'ecurrence par cha\^\i nes
(\emph{Est-elle finie~? sans classe ap\'eriodiques~?...})
ou vouloir comprendre la dynamique au sein de chaque classe de r\'ecurrence par cha\^\i nes.
\smallskip

Une autre approche consiste \` a caract\'eriser les ouverts
pr\'ec\'ed\'ement introduits en exhibant des obstructions qui apparaissent dans leur
compl\'ementaire. Par exemple, nous verrons en
section~\ref{s.palis-faible} que l'ensemble des diff\'eomorphismes
ayant une intersection homocline non triviale est dense dans
le compl\'ementaire de l'ensemble (ouvert) des dynamiques Morse-Smale.

On cherche des obstructions qui soient simples et faciles \`a d\'etecter
(en g\'en\'eral mettant en jeu des orbites p\'eriodiques).
Elle ne d\'ecrivent pas la dynamique de fa\c con satisfaisante puisque ce sont souvent
des bifurcations locales, mais elles peuvent engendrer une modification importante de la dynamique.
Une telle d\'ecomposition de l'espace des diff\'eomorphismes
en parties poss\'edant un ensemble dense de bifurcations et en ouverts de dynamiques globalement bien d\'ecrite
est appel\'ee par E. Pujals d\'ecomposition par \emph{m\'ecanismes et ph\'enom\`enes}.
\index{m\'ecanismes et ph\'enom\`enes}\index{ph\'enom\`enes et m\'ecanismes}
\bigskip

\section{Caract\'erisation des dynamiques hyperboliques}
Un autre exemple de dichotomie est une des conjecture
propos\'ee par Palis~\cite{palis-conjecture1,palis-conjecture2,palis-conjecture3}
qui cherche \`a caract\'eriser l'ouvert des dynamiques hyperboliques.
\begin{conjecture*}[Palis]\index{conjecture! de Palis}
Tout diff\'eomorphisme peut \^etre approch\'e
dans $\diff^1(M)$ par un dif\-f\'e\-o\-mor\-phis\-me hyperbolique ou
par un dif\-f\'e\-o\-mor\-phis\-me qui pr\'esente une tangence homocline ou un cycle h\'e\-t\'e\-ro\-di\-men\-si\-on\-nel.
\end{conjecture*}

Nous verrons au chapitre~\ref{c.hyp} que cette conjecture a \'et\'e r\'esolue sur les surfaces par Pujals et Sambarino.
Une r\'eponse positive \`a la conjecture d'hyperbolicit\'e
entra\^\i nerait la conjecture de Palis. (Pr\'ecisons toutefois que la conjecture de Palis a \'et\'e
\'egalement formul\'ee en r\'egularit\'e $C^k$, $k> 1$ et que la conjecture
d'hyperbolicit\'e n'est pas satisfaite dans ce cadre.)
Nous allons voir maintenant que cette conjecture peut \^etre
g\'en\'eralis\'ee d'une autre fa\c{c}on.

\paragraph{Notation.}
\emph{Par la suite, nous noterons $\cT$ l'ensemble des diff\'eomorphismes
ayant une tangence homocline et $\cC$ l'ensemble de ceux qui poss\`edent
un cycle h\'et\'erodimensionnel.}
\medskip

La conjecture de Palis est v\'erifi\'ee si l'on se restreint
aux diff\'eomorphismes
dont le nombre de classe de r\'ecurrence par cha\^\i nes
est robustement fini (voir~\cite{abdenur-decomposition,gan-wen-connecting}).

\begin{theoreme}[Abdenur,Gan-Wen]\label{t.fini}
Tout diff\'eomorphisme
peut \^etre approch\'e dans $\diff^1(M)$
par un diff\'eomorphisme $g$ ayant l'une des propri\'et\'es suivantes~:
\begin{itemize}
\item[--] $g$ est hyperbolique,
\item[--] $g$ pr\'esente un cycle h\'e\-t\'e\-ro\-di\-men\-sion\-nel,
\item[--] $g$ poss\`ede une infinit\'e de classes de r\'ecurrence par cha\^\i nes.
\end{itemize}
\end{theoreme}
\begin{proof}
Consid\'erons un diff\'eomorphisme $f$ qui ne peut pas \^etre approch\'e
par un dif\-f\'e\-o\-mor\-phi\-sme ayant une infinit\'e de classes de r\'ecurrence
par cha\^\i nes. Quitte \`a le perturber, on peut supposer qu'il
v\'erifie une condition de g\'en\'ericit\'e.
En particulier, ses classes sont des classes homoclines
$H(p_1),\dots,H(p_k)$ et tout dif\-f\'e\-o\-mor\-phis\-me proche de $f$ poss\`ede
exactement $k$ classes de r\'ecurrence par cha\^\i nes.

Supposons que la classe $H(p_i)$ (non triviale) n'est pas hyperbolique.
Nous montrons que par perturbation, il est possible de cr\'eer au voisinage
de $H(p_i)$ un point p\'eriodique $q$ d'indice diff\'erent de l'indice de $p_i$.
Par hypoth\`ese le point $q$ appartient \`a la m\^eme classe que $p_i$
et une nouvelle perturbation permet alors de cr\'eer par perturbation
un cycle h\'et\'erodimensionnel
(section~\ref{s.consequence}).

Si la d\'ecomposition du fibr\'e tangent au-dessus des orbites p\'eriodiques
homocliniquement reli\'ee \`a $p_i$ en espaces stables et instables
n'est pas uniform\'ement domin\'ee, le th\'eor\`eme~\ref{t.bgv} permet de conclure
directement.

Si $H(p_i)$ poss\`ede une d\'ecomposition domin\'ee $T_{H(p_i)}=E\oplus F$
avec $\dim(E)=\dim(E^s(p_i))$, et si par exemple le fibr\'e $E$ n'est pas uniform\'ement contract\'e, il existe une mesure ergodique support\'ee
sur $H(p_i)$ ayant un exposant de Lyapunov selon $E$ positif ou nul.
Le th\'eor\`eme~\ref{t.closing-ergodique} et son addendum
permettent alors de cr\'eer un point p\'eriodique d'indice
strictement plus petit que $\dim(E)$.
\end{proof}
\medskip

Nous avons d\'ej\`a mentionn\'e l'existence
de diff\'eomorphismes $C^1$-g\'en\'eriques dont le
nombre de classes de r\'ecurrence par cha\^\i nes n'est pas fini.
Les exemples connus sont toujours associ\'es \`a la pr\'esence de tangences
homoclines et Bonatti a conjectur\'e~\cite{bonatti-conjecture} que c'est toujours le cas.
Cette nouvelle conjecture entra\^\i nerait donc la conjecture de Palis.

\begin{conjecture-finitude}[Bonatti]\index{conjecture! de finitude}
Il existe un ouvert dense de $\diff^1(M)\setminus \overline{\cT}$
form\'e de diff\'eomorphismes dont le nombre de classes de r\'ecurrence
par cha\^\i nes est fini~?
\end{conjecture-finitude}

Newhouse, Palis-Viana et Romero~\cite{newhouse-wild,palis-viana,romero} ont montr\'e que l'on obtient une infinit\'e de puits ou de sources
pour certains diff\'eomorphismes proches des tangences homoclines.
Une r\'eponse positive \`a la conjecture de finitude donnerait une
r\'eciproque en topologie $C^1$ en r\'epondant \`a la question suivante.

\noindent
\emph{Est-ce que tout diff\'eomorphisme $C^1$-g\'en\'erique pr\'esentant le ph\'enom\`ene de Newhouse peut \^etre approch\'e par un diff\'eomorphisme ayant une tangence homocline~?}

\section{Dynamiques non critiques}
La discussion qui pr\'ec\`ede motive l'\'etude des diff\'eomorphismes (non hyperboliques)
loin des tangences homoclines.~
Nous connaissons plusieurs classes d'exemples~:
\begin{itemize}
\item[--] certaines dynamiques de type ``d\'eriv\'e d'Anosov'' d\'ecrites en section~\ref{s.exemple-mane},
\item[--] les dynamiques proches du temps $1$ d'un flot d'Anosov
ou du produit d'un diff\'eomorphisme d'Anosov par l'identit\'e sur le cercle~\cite{blender}.
\end{itemize}
Une classe d'exemples est fournie par
la construction de Ma\~n\'e (section~\ref{s.exemple-mane}).
La d\'emonstration du th\'eor\`eme~\ref{t.fini} nous montre que,
pour de telles dynamiques, toute classe homocline non hyperbolique
est accumul\'ee par des points p\'eriodiques d'indices diff\'erents.
La question principale consiste alors \`a comprendre si ces points
appartiennent tous \`a une m\^eme classe (on obtient ainsi
des cycles h\'et\'erodimensionnels par perturbation)
ou si le diff\'eomorphisme poss\`ede une infinit\'e de classes distinctes.

Pour des dynamiques loin des tangences homoclines, le th\'eor\`eme~\ref{t.tangence}
montre que la d\'e\-com\-po\-si\-tion du fibr\'e tangent en espaces stables et instables
au-dessus des orbites p\'eriodiques est uniform\'ement domin\'ee.
Le th\'eor\`eme~\ref{t.pliss} implique que tout fibr\'e qui n'est pas uniform\'ement contract\'e ou dilat\'e permet de changer l'indice d'orbites p\'eriodiques.
Une r\'eponse positive \`a la conjecture de finitude permettrait donc de r\'epondre
affirmativement au probl\`eme suivant (voir~\cite{abcdw}).

\begin{question}\label{q.structure}
Existe-t-il un G$_\delta$ dense de $\diff^1(M)\setminus \overline{\cT}$
form\'e de diff\'eomorphismes dont toutes les classe homoclines $H$ ont
une d\'ecomposition domin\'ee
$$T_HM=E^s\oplus E^c\oplus E^u,$$
o\`u $E^s$ et $E^u$ sont uniform\'ement contract\'es et dilat\'es et o\`u
$\dim(E^s)$ et $\dim(E^s\oplus E^c)$ sont respectivement le plus petit et le plus
grand indice de la classe $H$~?
\end{question}

\section{Ensembles minimaux non hyperboliques}\label{s.minimaux}
Consid\'erons un diff\'eomorphisme $f$ non hyperbolique.
Il poss\`ede un ensemble transitif par cha\^\i nes $K$ qui n'est
pas hyperbolique. On peut choisir $K$ minimal pour l'inclusion~: un tel ensemble est alors appel\'e
\emph{ensemble minimal non hyperbolique}\index{ensemble! minimal non hyperbolique}.
\medskip

Lorsque $f$ appartient \`a $\diff^1(M)\setminus \overline{\cT}$,
le th\'eor\`eme~\ref{t.trichotomie} implique le r\'esultat suivant~\cite{gwy}
qui est une premi\`ere r\'eponse locale \`a la question~\ref{q.structure}.
\begin{theoreme}[Gan-Wen-Yang]\label{t.minimaux}
Pour tout diff\'eomorphisme non hyperbolique $f$ appartenant \`a un G$_{\delta}$ dense de $\diff^1(M)\setminus \overline{\cT}$,
tout ensemble minimal non hyperbolique $K$ admet une d\'e\-com\-po\-si\-tion domin\'ee
$$T_{K}M=E^s\oplus E^c_{1}\oplus\dots\oplus E^c_{k}\oplus E^u,$$
telle que $E^s$ et $E^u$ sont uniform\'ement contract\'es respectivement par $f$ et $f^{-1}$ et chaque fibr\'e  central $E^c_{i}$ est de dimension $1$.

Pour tout $i\in \{1,\dots,k\}$, l'ensemble $K$ est limite de Hausdorff d'une suite d'orbites p\'e\-ri\-o\-di\-ques
$(O_n)$ dont l'exposant de Lyapunov le long de $E^c_i$ est arbitrairement proche de $0$.
Lorsque $k>1$, $K$ est contenu dans une classe homocline $H(p)$ et
on peut choisir les orbites $O_n$ dans $H(p)$.
\end{theoreme}
En particulier lorsque $f$ est loin des cycles h\'et\'erodimensionnels,
une classe homocline poss\`ede un unique indice (section~\ref{s.consequence}). La dimension centrale est donc \'egale \`a $1$ ou $2$.
\begin{corollaire}[Wen~\cite{wen-palis}]\label{c.wen}
Pour tout diff\'eomorphisme non hyperbolique $f$ appartenant \`a un G$_{\delta}$ dense de $\diff^1(M)\setminus \overline{\cT\cup\cC}$,
tout ensemble minimal non hyperbolique $K$ admet une d\'ecomposition domin\'ee de type $T_{K}M=E^s\oplus E^c\oplus E^u$ ou
$T_{K}M=E^s\oplus E^c_{1}\oplus E^c_{2}\oplus E^u$, telle que $E^s$ et $E^u$ sont uniform\'ement contract\'es respectivement par $f$ et $f^{-1}$ et chaque fibr\'e  central $E^c$ ou $E^c_{1},E^c_{2}$ est de dimension $1$.

Dans le deuxi\`eme cas, $K$ est contenu dans une classe homocline d'indice $\dim(E^s)+1$.
\end{corollaire}

\begin{proof}[D\'emonstration du th\'eor\`eme~\ref{t.minimaux}]
Le corollaire~\ref{c.tangence} implique que pour tout $1\leq i<d$, l'ensemble
$\per_{i}(f)$ des points p\'eriodiques de $f$ d'indice $i$ a une d\'ecomposition
domin\'ee $T_{\per_{i}(f)}M=E_{i}\oplus F_{i}$
telle que $\dim(E_{i})=i$.
En particulier, le th\'eor\`eme~\ref{t.trichotomie} s'applique.

Si $K$ est contenu dans une classe ap\'eriodique (en consid\'erant par exemple la d\'ecomposition
domin\'ee triviale $T_{K}M=E$) seul le troisi\`eme cas du th\'eor\`eme~\ref{t.trichotomie} peut avoir lieu~:
$K$ contient un ensemble partiellement hyperbolique $\Lambda$ ayant un fibr\'e central de dimension $1$.
Puisque $K$ est un ensemble minimal non hyperbolique, il co\"\i ncide avec $\Lambda$.

Si $K$ est contenu dans une classe homocline $H$, le th\'eor\`eme~\ref{t.trichotomie}
assure que l'ensemble des indices de $H$ est un intervalle.
Ceci entra\^\i ne l'existence d'une d\'ecomposition domin\'ee
$$T_{H}M=F_{0}\oplus F^c_{1}\oplus\dots\oplus F^c_{\ell}\oplus F_{\ell+1},$$
telle que les fibr\'es $F^c_{i}$ soient de dimension $1$
et telle que $H$ ne contienne pas de point d'indice $i<\dim(F_{0})$
ou $i>d-\dim(F_{\ell+1})$.
On peut supposer de plus que $\dim(F_{0})$ et $d-\dim(F_{\ell+1})$ sont minimales pour ces propri\'et\'es.

Si $F_{0}$ n'est pas uniform\'ement contract\'e au-dessus de $K$, nous envisageons les trois cas
du th\'eor\`eme~\ref{t.trichotomie}.
\begin{itemize}
\item[--] Le premier cas n'appara\^\i t pas puisque $H$ ne contient pas de point d'indice $i<\dim(F_{0})$.
\item[--] Dans le second cas, $H$ contiendrait des points p\'eriodiques d'indice $\dim(F_{0})$
ayant un exposant de Lyapunov stable proche de $0$.
Puisque $H$ est une classe homocline, l'ensemble de ces points p\'eriodiques est dense dans $H$.
Le th\'eor\`eme~\ref{t.trichotomie} implique alors que $F_{0}$ poss\`ede une d\'ecomposition domin\'ee
$F_0=E\oplus E^c$ telle que $\dim(E^c)=1$.
Ceci contredit la minimalit\'e de $\dim(F_{0})$.
\item[--] Dans le troisi\`eme cas, l'ensemble $K=\Lambda$ est partiellement hyperbolique et la dimension du fibr\'e central vaut $1$.
\end{itemize}
Nous obtenons donc une d\'ecomposition domin\'ee
$T_KM=E^s\oplus E^c_1\oplus \dots\oplus E^c_k\oplus E^u$
o\`u $E^s$ et $E^u$ sont uniform\'ement contract\'e et dilat\'e
et o\`u chaque fibr\'e $E^c_i$
est de dimension $1$ et n'est pas uniform\'ement contract\'e ni dilat\'e.
\medskip

Si l'on suppose que $k>1$, le troisi\`eme cas du th\'eor\`eme~\ref{t.trichotomie} n'a pas lieu.
Le th\'eor\`eme appliqu\'e \`a la d\'ecomposition $(E^s\oplus E^c_1)\oplus
(E^c_2\oplus \dots\oplus E^u)$ pour $f^{-1}$
montre alors que $K$ est contenu dans une classe homocline $H$ ayant des orbites
p\'eriodiques d'indice sup\'erieur ou \'egal \`a $\dim(E^s)+1$.
Appliquons le th\'eor\`eme \`a la d\'ecomposition $(E^s\oplus E^c_1)\oplus
(E^c_2\oplus \dots\oplus E^u)$ pour $f$~:
\begin{itemize}
\item[--] dans le premier cas du th\'eor\`eme, $H$ poss\`ede aussi des orbites p\'eriodiques d'indice $\dim(E^s)$ et, d'apr\`es le th\'eor\`eme~\ref{t.barycentre}, $H$ contient des orbites p\'eriodiques dont le $(\dim(E^s)+1)$-\`eme exposant est arbitrairement proche de $0$~;
\item[--] dans le second cas, nous concluons directement que $H$ poss\`ede des orbites p\'eriodiques dont le $(\dim(E^s)+1)$-\`eme exposant est arbitrairement proche de $0$.
\end{itemize}
\end{proof}

\paragraph{Passage du local au global.}
Nous souhaitons \'etendre la d\'ecomposition domin\'ee donn\'ee par le
th\'eor\`eme~\ref{t.minimaux} sur l'ensemble minimal non hyperbolique $K$
\`a toute la classe de r\'ecurrence par cha\^\i nes $C$ contenant $K$.

Pour chaque d\'ecomposition domin\'ee $E\oplus F$ sur $K$ donn\'ee par ce th\'eor\`eme,
il existe une suite d'orbite p\'eriodiques $(O_{n})$ qui convergent vers $K$ en topologie de Hausdorff
et dont l'indice est \'egal \`a $\dim(E)$. On peut esp\'erer utiliser la transitivit\'e par cha\^\i nes de $C$
pour construire une suite d'orbites p\'eriodiques $(O'_{n})$ ayant le m\^eme indice que les orbites $O_{n}$
mais convergeant vers $C$.
(Ce serait le cas si les questions~\ref{q.completude} ou~\ref{q.fermeture-homocline} admettaient des r\'eponses positives.)
Si le diff\'eomorphisme $f$ est loin des tangences homoclines,
le corollaire~\ref{c.tangence} impliquerait alors que la d\'ecomposition $E\oplus F$ s'\'etend \`a toute la classe $C$.

La construction d'une telle suite d'orbites p\'eriodiques $O_{n}'$ est d\'elicate et fait l'objet
des r\'esultats principaux de ce chapitre. Si l'on sait montrer que les orbites p\'eriodiques $O_{n}$
appartiennent \`a la classe $C$, nous concluons directement puisque $C$
co\"\i ncide alors avec la classe homocline $H(O_{n})$ et contient une orbite p\'eriodique
du m\^eme indice que $O_{n}$ qui est $\varepsilon$-dense dans $C$ pour tout $\varepsilon>0$.

\section{Mod\`eles centraux}
\subsection{D\'efinition}
Afin de d\'ecrire la dynamique centrale pour un ensemble
$K$ ayant une d\'ecomposition
domin\'ee $T_{K}M=E_1\oplus E^c\oplus E_3$ avec $\dim(E^c)=1$,
nous introduisons le mod\`ele suivant.

\begin{definition}
Un \emph{mod\`ele central}\index{mod\`ele central} est une paire $(\widehat K,\widehat f)$ form\'ee
\begin{itemize}
\item[--] d'un espace compact m\'etrique $\widehat K$, (sa \emph{base}),
\item[--] d'une application continue $\widehat f\colon \widehat K\times [0,1]\to \widehat K\times [0,+\infty)$,
\end{itemize}
satisfaisant les propri\'et\'es suivantes~:
\begin{itemize}
\item[--] $\widehat f(\widehat K\times \{0\})=\widehat K\times \{0\}$~;
\item[--] $\widehat f$ est un hom\'eomorphisme local au voisinage de  $\widehat K\times \{0\}$~:
il existe une application continue $\widehat g\colon \widehat K\times [0,1]\to \widehat K\times [0,+\infty)$
telle que $\widehat f\circ \widehat g$ et $\widehat g\circ \widehat f$ co\"\i ncident avec l'identit\'e respectivement sur $\widehat g^{-1}(\widehat K\times [0,1])$
et $\widehat f^{-1}(\widehat K\times [0,1])$~;
\item[--] $\widehat f$ est un produit fibr\'e de la forme
$ \widehat f (x,t)= (\widehat f_1(x),\widehat f_2(x,t))$.
\end{itemize}
\end{definition}

Puisque $\widehat K\times [0,1]$ n'est pas pr\'eserv\'e, la dynamique hors de $\widehat K\times \{0\}$ n'est pas toujours bien d\'efinie.
C'est pourtant cette partie qui nous int\'eresse ici.

\begin{definition}
\begin{itemize}
\item[--]
Un sous-ensemble $S\subset \widehat K\times [0,+\infty)$ est une \emph{bande} si pour tout $\widehat x\in \widehat K$, l'intersection
$S\cap \{\widehat x\}\times [0,+\infty)$ est un intervalle.

\item[--]
Une bande ouverte $S$ qui satisfait $\widehat f(\overline S)\subset S$ est dite \emph{attractive} pour $\widehat f$.

\item[--]
Un intervalle $\{\widehat x\}\times [0,a]$ avec $a>0$ est un \emph{segment r\'ecurrent par cha\^\i nes}\index{segment r\'ecurrent par cha\^\i nes} s'il est contenu
dans un ensemble compact invariant transitif par cha\^\i nes $\Lambda\subset \widehat K\times [0,+\infty)$.
\end{itemize}
\end{definition}

Des arguments similaires \`a ceux de Conley pour obtenir le th\'eor\`eme~\ref{t.fondamental}
donnent
un crit\`ere d'existence de bandes attractives.
\begin{proposition}[proposition 2.5 de \cite{palis-faible} et proposition 2.2 de \cite{model}]
Consid\'erons un mod\`ele central $(\widehat K,\widehat f)$ dont la base est transitive par cha\^\i nes.
Il y a $4$ cas disjoints possibles.
\begin{itemize}
\item[--] Les ensembles stables et instables par cha\^\i nes de $\widehat K\times \{0\}$ ne sont pas triviaux.
De fa\c con \'equivalente, il existe un segment r\'ecurrent par cha\^\i nes.
\item[--] Les ensembles stables et instables par cha\^\i nes de $\widehat K\times \{0\}$ sont triviaux.
De fa\c con \'e\-qui\-va\-len\-te, il existe une base de voisinages de $\widehat K\times \{0\}$ par bandes attractives pour $\widehat f$
et une base de voisinages de $\widehat K\times \{0\}$ par bandes attractives pour $\widehat f^{-1}$.
\item[--] L'ensemble instable par cha\^\i nes est trivial, l'ensemble stable par cha\^\i nes contient un voisinage
de $\widehat K\times \{0\}$. Il existe alors une base de voisinages de $\widehat K\times \{0\}$ par bandes attractives.
\item[--] L'ensemble stable par cha\^\i nes est trivial, l'ensemble instable par cha\^\i nes contient un voisinage
de $\widehat K\times \{0\}$. (C'est le cas sym\'etrique du pr\'ec\'edent.)
\end{itemize}
\end{proposition}
\noindent
En particulier, l'existence d'un segment r\'ecurrent par cha\^\i nes est une propri\'et\'e locale au voisinage de $\widehat K\times \{0\}$.

\subsection{Mod\`eles centraux en dynamique diff\'erentiable}
Consid\'erons un ensemble compact $f$-invariant $K\subset M$ ayant une d\'ecomposition domin\'ee
$T_{K}M=E_1\oplus E^c\oplus E_3$ avec $\dim(E^c)=1$.

\begin{definition}
Un mod\`ele central $(\widehat K,\widehat f)$ est \emph{associ\'e \`a $(K,f)$} s'il existe une application
continue
$\pi\colon \widehat K\times [0,+\infty)\to M$ ayant les propri\'et\'es suivantes~:
\begin{itemize}
\item[--] $\pi(\widehat K\times \{0\})=K$~;
\item[--] $\pi\circ \widehat f=f\circ \pi$ sur $\widehat K\times [0,1]$~;
\item[--] les applications $t\mapsto \pi(\widehat x,t)$ pour $\widehat x\in \widehat K$ forment une famille
continues de plongements $C^1$ de $[0,+\infty)$ dans $M$~;
\item[--] pour tout $\widehat x\in \widehat K$, la courbe $\pi(\{\widehat x\}\times[0,+\infty))$ est tangente \`a $E^c_{\pi(\widehat x,0)}$.
\end{itemize}
\end{definition}

Deux cas peuvent \^etre distingu\'es~:
\begin{description}
\item[\bf -- Le cas orientable.] Il existe une orientation continue de $E^c$ pr\'eserv\'ee par $Df$.
Ceci permet de d\'efinir les fibr\'es en demi-droites $E^{c,+}$ et $E^{c,-}$.
\item[\bf -- Le cas non orientable.]
Lorsque $K$ est transitif par cha\^\i nes, la dynamique induite par $Df$ sur le fibr\'e unitaire
associ\'e \`a $E^c$ est encore transitif par cha\^\i nes.
\end{description}
\medskip

Consid\'erons une famille de plaques localement invariante $\cW$
et tangente \`a $E^c$, donn\'ee par le th\'eor\`eme~\ref{t.hps}.
La dynamique de $f$ se rel\`eve par $\cW$ en une dynamique $\widehat f$ localement d\'efinie au voisinage de la
section $0\subset E^c$. Dans le cas orientable, on note $\widehat K=K$ et
les fibr\'es $E^c,E^{c,+},E^{c,-}$ s'identifient \`a $K\times \RR$,
$K\times [0,+\infty)$ et $K\times (-\infty,0]$ respectivement.
Dans le cas non orientable, on note $\widehat K$
l'espace unitaire de $E^c$ et on consid\`ere l'application surjective $\widehat K\times [0,+\infty)\to E^c$
d\'efinie par $(\widehat x,t)\mapsto t.\widehat x$. Ceci entra\^\i ne l'existence d'un mod\`ele central pour $K$.

\begin{proposition}[proposition 3.2 de \cite{palis-faible}]\label{p.existence-modele}
Associ\'e \`a un ensemble compact invariant $K$ ayant une d\'ecomposition domin\'ee
$T_{K}M=E_{1}\oplus E^c\oplus E_{3}$ avec $\dim(E^c)=1$, il existe un mod\`ele central $(\widehat K,\widehat f)$.
Lorsque $K$ est transitif par cha\^\i nes, on peut choisir $\widehat K$ transitive par cha\^\i nes~:
\begin{itemize}
\item[--]
dans le cas non orientable, l'ensemble $\widehat K$ co\"\i ncide avec
l'espace unitaire de $E^c$ et $\pi$ est la restriction de la projection $E^c\to K$~;
\item[--]
dans le cas orientable,
l'application $\pi$ est un hom\'eomorphisme et on peut choisir l'o\-ri\-en\-ta\-tion de chaque courbe
$\pi(\{\widehat x\}\times [0,+\infty))$ compatible avec $E^{c,+}$.
\end{itemize}
\end{proposition}

\subsection{Classification des dynamiques centrales}
Supposons que $K$ soit transitif par cha\^\i nes. Consid\'erons
une famille de plaques localement invariante $\cW$ et la dynamique $\widehat f$ induite par $f$ sur un
voisinage de la section $0$ dans $E^c$.
\medskip

Nous obtenons les cas suivants, non disjoints en g\'en\'eral (voir la figure~\ref{f.type}).
\index{mod\`ele central!type}
\begin{description}
\item[-- Le type (R), r\'ecurrent par cha\^\i nes.]
Pour tout $\varepsilon>0$, il existe un point $x\in K$
et un segment $\gamma\subset \cW_x$ contenant $x$ tel que tous les it\'er\'es
$f^n(\gamma)$, $n\in \ZZ$, soient de longueur born\'ee par $\varepsilon$
et tel que $\gamma$ est inclus dans un ensemble transitif par cha\^\i nes
contenu dans le $\varepsilon$-voisinage de $K$.

Une telle courbe $\gamma$ est appel\'ee \emph{segment central} de $K$.
\item[-- Le type (N), neutre par cha\^\i nes.] Il existe une base de voisinages ouverts $U$ de la section $0\subset E^c$
qui sont attractifs (i.e. $\widehat f(\overline U)\subset U$)
et une base de voisinages r\'epulsifs.
\item[-- Le type (H), hyperbolique par cha\^\i nes.]
Il existe une base de voisinages ouverts de la section $0\subset E^c$
qui sont attractifs pour $\widehat f$ (le cas \emph{attractif}) ou pour $\widehat f^{-1}$
(le cas \emph{r\'epulsif}) et dont l'image par $\cW$ est respectivement contenue
dans l'ensemble stable par cha\^\i nes ou dans l'ensemble instable par cha\^\i nes de $K$.
\item[-- Le type (P), parabolique par cha\^\i nes.] Nous sommes alors dans le cas orientable.
Trois possibilit\'es apparaissent.
\begin{description}
\item[-- Le type (P$_{SU}$).] Quitte \`a renverser l'orientation sur $E^c$,
il existe une base de voisinages ouverts de la section $0\subset E^{c,+}$ qui sont attractifs pour $\widehat f$
et qui se projettent par $\cW$ dans l'ensemble stable par cha\^\i nes de $K$~;
il existe \'egalement une base de voisinages ouverts de la section $0\subset E^{c,-}$ qui sont attractifs pour $\widehat f^{-1}$
et qui se projettent par $\cW$ dans l'ensemble instable par cha\^\i nes de $K$.
\item[-- Les types (P$_{SN}$) et (P$_{UN}$)] respectivement.
Quitte \`a renverser l'orientation sur $E^c$,
il existe une base de voisinages ouverts de la section $0\subset E^{c,+}$ qui sont attractifs pour $\widehat f$
(resp. $\widehat f^{-1}$) et qui se projettent par $\cW$ dans l'ensemble stable (resp. instable)
par cha\^\i nes de $K$~;
il existe \'egalement une base de voisinages ouverts de la section $0\subset E^{c,-}$
qui sont attractifs (i.e. $\widehat f(\overline U)\subset U$)
et une base de voisinages de $0\subset E^{c,-}$ qui sont r\'epulsifs.
\end{description}
\end{description}
\begin{figure}[ht]
\begin{center}
\input{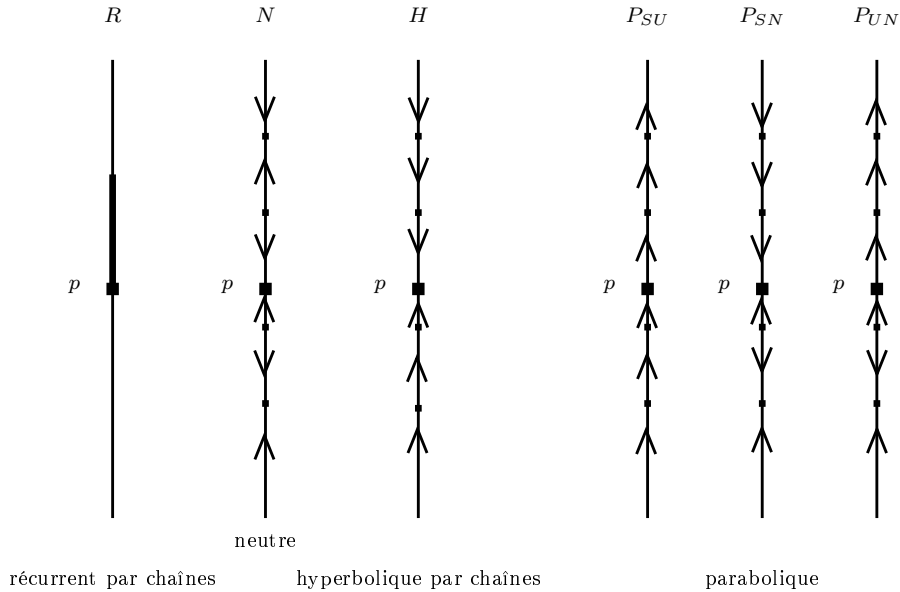}
\end{center}
\caption{Types de dynamiques centrales.\label{f.type}}
\end{figure}

Le type de la dynamique centrale de d\'epend pas du choix d'une famille de plaques $\cW$.

\begin{proposition}[Corollaire 2.6 de~\cite{model}]
Consid\'erons un ensemble compact invariant $K$ transitif par cha\^\i nes
muni d'une d\'ecomposition domin\'ee $T_KM=E_1\oplus E^c\oplus E_3$ avec $\dim(E^c)=1$
et deux familles de plaques $\cW,\cW'$ localement invariantes tangentes \`a $E^c$.
Alors les types de dynamiques centrales associ\'es \`a $\cW,\cW'$ co\"\i ncident.
\end{proposition}

Les \emph{types du fibr\'e $E^c$} seront donc d\'efinis comme \'etant les types de dynamiques centrales
associ\'e \`a un choix de famille de plaques centrales localement invariantes.

\section{Lorsqu'un fibr\'e est d\'eg\'en\'er\'e}\label{s.degenere}

Nous nous int\'eressons au cas o\`u le fibr\'e $E_3$ est d\'eg\'en\'er\'e~: dans les bons cas, l'ensemble $K$
est un puits ou bien le fibr\'e $E^c$ est ``topologiquement r\'epulsif''.

\begin{proposition}[proposition 2.7 de~\cite{model}]\label{p.degenere}
Consid\'erons un ensemble compact invariant $K$ transitif par cha\^\i nes
muni d'une d\'ecomposition domin\'ee $T_KM=E_1\oplus E^c$ avec $\dim(E^c)=1$.
Si $E^c$ est de type (N), (H)-attractif ou (P), l'ensemble $K$ contient un point p\'eriodique.
\end{proposition}
En particulier si $f$ est un diff\'eomorphisme Kupka-Smale, le fibr\'e $E^c$
ne peut pas avoir les types (P) ou (N). S'il est de type (H)-attractif, l'ensemble $K$ est un puits.
\medskip

Si $f$ satisfait une condition de g\'en\'ericit\'e, le type (R) ne peut pas non plus appara\^\i tre\footnote{Voir~\cite[proposition 4.3]{model}~; l'\'enonc\'e suppose $E_1$ uniform\'ement contract\'e, mais la d\'emonstration ne l'utilise pas.}. On en d\'eduit la proposition suivante.

\begin{proposition}\label{p.codim-generique}
Si $f$ appartient \`a un G$_\delta$ dense de $\diff^1(M)$
et si $K$ est un ensemble compact invariant transitif par cha\^\i nes
muni d'une d\'ecomposition domin\'ee $T_KM=E_1\oplus E^c$ avec $\dim(E^c)=1$,
alors ou bien $K$ est un puits, ou bien $E^c$ est de type (H)-r\'epulsif.
\end{proposition}

Lorsque $f$ est de classe $C^2$, Pujals et Sambarino remplacent l'hypoth\`ese de g\'en\'ericit\'e par un argument ``\`a la Denjoy-Schwartz''.
En voici une re-formulation.

\begin{theoreme}[Pujals-Sambarino, section 3.3 de~\cite{pujals-sambarino} et th\'eor\`eme 3.1 de~\cite{pujals-sambarino-integrabilite}]\label{t.codim-c2}
Consid\'erons un ensemble compact invariant
muni d'une d\'ecomposition domin\'ee $T_KM=E_1\oplus E^c$ avec $\dim(E^c)=1$
et $\cW$ une famille de plaques localement invariante tangente \`a $E^c$.
Si $f$ est de classe $C^2$, l'un des cas suivants se produit.
\begin{enumerate}
\item[a)] $K$ contient un point p\'eriodique dont l'exposant selon $E^c$ est n\'egatif ou nul.
\item[b)] $K$ contient une courbe ferm\'ee $C$ tangente au fibr\'e $E^c$
et invariante par un it\'er\'e $f^n$, $n\geq 1$.
La dynamique sur la courbe est conjugu\'ee \`a une rotation irrationnelle.
\item[c)] La dynamique centrale de $K$ est stable au sens de Lyapunov
pour $\widehat f^{-1}$~:
\begin{itemize}
\item[--] pour tout $\delta>0$, il existe $\varepsilon>0$ tel que pour tout $x\in K$,
les it\'er\'es n\'egatifs du $\varepsilon$-voisinage de $x$ dans $\cW_x$ sont tous de taille
inf\'erieure \`a $\delta$~;
\item[--] pour tout $x\in K$, il existe $\eta(x)>0$ tel que le $\eta(x)$-voisinage de $x$ dans
$\cW_x$ est contenu dans l'ensemble instable de $x$.
\end{itemize}
\end{enumerate}
\end{theoreme}

\begin{remarque}
Lorsque $K$ est transitif par cha\^\i nes,
on peut compl\'eter le cas c) de la proposition pr\'ec\'edente
et garantir que le fibr\'e $E^c$ est de type (H)-r\'epulsif, comme
pour la proposition~\ref{p.codim-generique}.

En effet, la dynamique centrale ne peut poss\'eder de bandes attractives
contenues dans un voisinage arbitrairement petit de la section $O\subset E^c$
(puisque, d'apr\`es l'item c), chaque point de $K$ poss\`ede une vari\'et\'e instable tangente \`a $E^c$).

Il ne peut y avoir non plus de segment transitif par cha\^\i nes $\gamma$~:
d'apr\`es le th\'eor\`eme 3.1 de~\cite{pujals-sambarino-integrabilite},
quitte \`a remplacer $\gamma$ par son ensemble $\omega$-limite, $\gamma$ est une courbe
p\'eriodique bord\'ee par un point p\'eriodique hyperbolique $p\in K$
d'indice $d-1$ et un point p\'eriodique $q$ ayant un exposant n\'egatif ou nul selon $E^c$.
Il doit alors exister une pseudo-orbite de segments centraux $(\gamma_n)_{0\leq n\leq s}$
au-dessus d'une pseudo-orbite $(x_n)$ de $K$ v\'erifiant $\gamma_0=\gamma$
et $|\gamma_s|= 0$. Ceci implique que $K$ contient un point p\'eriodique ayant un
exposant n\'egatif ou nul selon $E^c$.
\end{remarque}

\section{Dynamique centrale r\'ecurrente par cha\^\i nes}
Lorsque $K$ est partiellement hyperbolique avec un fibr\'e central r\'ecurrent par cha\^\i nes,
on en d\'eduit facilement que certaines orbites p\'eriodiques sont contenues dans la classe de r\'ecurrence
par cha\^\i nes de $K$.

\begin{proposition}[section 3.3 de~\cite{palis-faible} et proposition 4.1 de~\cite{model}]\label{p.R}
Si $f$ appartient \`a un G$_\delta$ dense de $\diff^1(M)$
et si $K$ est un ensemble compact transitif par cha\^\i nes
ayant une d\'ecomposition domin\'ee $T_KM=E^s\oplus E^c\oplus E^u$
telle que
\begin{itemize}
\item[--] $E^s$ et $E^u$ sont uniform\'ement contract\'es par $f$ et
$f^{-1}$ respectivement,
\item[--] $E^c$ est de dimension $1$ et de type (R),
\end{itemize}
alors tout voisinage de $K$ contient une orbite p\'eriodique
appartenant \`a la m\^eme classe de r\'ecurrence par cha\^\i nes que $K$.

Plus pr\'ecis\'ement,
il existe un voisinage $U$ de $K$ tel que pour tout segment central $\gamma$ de $K$ ayant ses it\'er\'es contenus dans $U$,
pour tout point $z\in \gamma$ qui n'est pas une extr\'emit\'e
et pour toute orbite p\'eriodique hyperbolique $O\subset U$ ayant un point suffisamment proche de $z$,
\begin{itemize}
\item[--] l'orbite $O$ est contenue dans la classe de r\'ecurrence par cha\^\i nes de $K$~;
\item[--] il existe $g\in\diff^1(M)$ arbitrairement proche de $f$ tel que
les vari\'et\'es fortes $W^{ss}(O_{g})\setminus O_{g}$ et $W^{uu}(O_{g})\setminus O_{g}$ de la continuation de $O$ pour $g$ s'intersectent.
\end{itemize}
\end{proposition}

La d\'emonstration s'appuie sur l'id\'ee suivante~: l'union des vari\'et\'es stables et instables locales fortes
des points de $\gamma$ d\'efinissent des vari\'et\'es topologiques $\cD^{cs}, \cD^{cu}$ de dimension $\dim(E^s)+1$ et $\dim(E^u)+1$. Si $p\in O$ est proche de $z\in \gamma$, les vari\'et\'es $W^{ss}(p)$ et $W^{uu}(p)$
intersectent respectivement $\cD^{cu}$ et $\cD^{cs}$. Par cons\'equent $p$ appartient \`a la classe de r\'ecurrence par cha\^\i nes de $K$. Le lemme de fermeture pour les pseudo-orbites permet alors
de cr\'eer une intersection entre $W^{ss}(O)$ et $W^{uu}(O)$. Voir la figure~\ref{f.segment}.
\begin{figure}[ht]
\begin{center}
\input{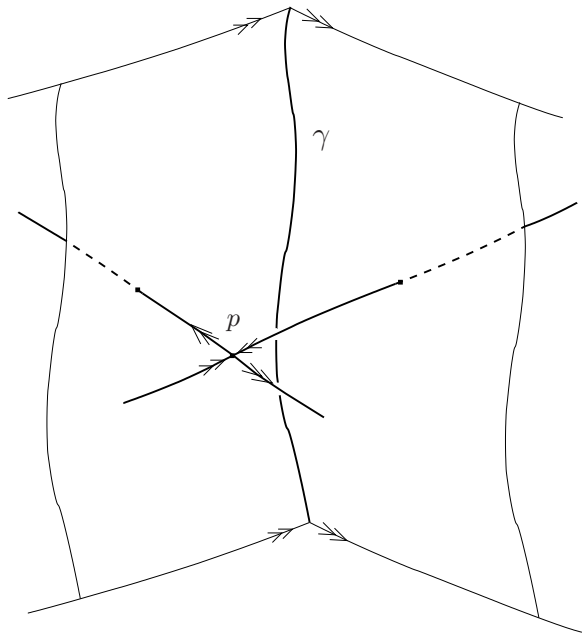}
\end{center}
\caption{Un point $p$ proche d'un segment central $\gamma$. \label{f.segment}}
\end{figure}
\bigskip

Parfois, le m\^eme argument s'applique sans que l'ensemble $K$ soit partiellement hyperbolique.
\begin{proposition}[proposition 4.2 de~\cite{model}]\label{p.R2}
Si $f$ appartient \`a un G$_\delta$ dense de $\diff^1(M)$
et si $H$ est une classe homocline
ayant une d\'ecomposition domin\'ee $T_KM=E^s\oplus E^c\oplus E_{3}$ telle que
\begin{itemize}
\item[--] $E^s$ est uniform\'ement contract\'ee,
\item[--] $E^c$ est de dimension $1$ et de type (R),
\item[--] $H$ contient des orbites p\'eriodiques dont l'exposant de Lyapunov le long de $E^c$ est arbitrairement proche de $0$,
\end{itemize}
la classe $H$ contient alors des points p\'eriodiques d'indice $\dim(E^s)$.
\end{proposition}

\section{Le type hyperbolique par cha\^\i nes et situations similaires}
\subsection{Dynamique centrale hyperbolique par cha\^\i nes}
Comme dans le cas d'une dynamique hyperbolique, les ensembles
partiellement hyperboliques ayant un fibr\'e central
hyperboliques par cha\^\i nes sont contenus dans une classe homocline.

\begin{proposition}[section 3.4 de~\cite{palis-faible} et proposition 4.4 de~\cite{model}]\label{p.H}
Si $f$ appartient \`a un G$_\delta$ dense de $\diff^1(M)$
et si $K$ est un ensemble compact transitif par cha\^\i nes
ayant une d\'ecomposition domin\'ee $T_KM=E^s\oplus E^c\oplus E^u$
telle que
\begin{itemize}
\item[--] $E^s$ et $E^u$ sont uniform\'ement contract\'es par $f$ et
$f^{-1}$ respectivement,
\item[--] $E^c$ est de dimension $1$ et de type (H)-attractif,
\end{itemize}
il existe alors, dans tout voisinage $U$ de $K$,
une orbite p\'eriodique
d'indice $\dim(E^s)+1$ contenue dans la classe de r\'ecurrence par cha\^\i nes
de $K$.
\end{proposition}
\begin{proof}[Id\'ee de la d\'emonstration]
Consid\'erons une suite d'orbites p\'eriodiques $(O_n)$ qui s'accumule sur une partie de $K$.
Puisque $E^c$ est de type (H)-attractif, il existe une famille de plaques $\cW^{cs}$
tangentes \`a $E^s\oplus E^c$ qui sont pi\'eg\'ees et contenues dans l'ensemble stable par cha\^\i nes de $K$.
La famille s'\'etend aux points des orbites $O_n$ (pour $n$ suffisamment grand).

Consid\'erons un point $x\in K$ proche d'un point p\'eriodique $p$ appartenant \`a une orbite $O_n$.
La plaque $W^{uu}_{loc}(x)$ intersecte $\cW^{cs}_p$ en un point $z$. Puisque la famille $\cW^{cs}$ est pi\'eg\'ee
et puisque $p$ est p\'eriodique, $z$ appartient \`a l'ensemble stable d'une orbite p\'eriodique, $q_0$.
Si l'indice de $q_0$ est \'egal \`a $\dim(E^s)+1$, on d\'efinit $q=q_0$. Sinon, $W^u(q_0)$ intersecte
l'ensemble stable d'un autre point p\'eriodique $q\in \cW^{cs}_p$, d'indice $\dim(E^s)+1$.
Dans tous les cas, $z$ appartient \`a l'ensemble stable par cha\^\i nes de $q$.
On remarque que puisque $\cW^{cs}$ est pi\'eg\'ee, $q$ appartient aussi \`a $f(\overline{\cW^{cs}_{f^{-1}(p)}})$.
Puisque $p$ et $x$ sont proches, on en d\'eduit que $W^{uu}_{loc}(q)$ intersecte $\cW^{cs}_x$ et en particulier l'ensemble
stable par cha\^\i nes de $K$. Par cons\'equent $q$ et $K$ ont la m\^eme classe de r\'ecurrence par cha\^\i nes.
\end{proof}

\subsection{Ensembles non vrill\'es}
Dans le cas parabolique, l'ensemble $K$ est ``hyperbolique par cha\^\i nes
d'un c\^ot\'e''. On peut appliquer le m\^eme raisonnement lorsque
la dynamique de $K$ ``voit'' l'hyperbolicit\'e par cha\^\i nes dans la direction
centrale~: c'est le cas sauf si la g\'eom\'etrie de $K$ est ``vrill\'ee''.
\medskip

Consid\'erons un ensemble compact invariant $K$
muni d'une structure partiellement hyperbolique $E^s\oplus E^c \oplus E^u$,
telle que $E^s,E^u$ ne soient pas d\'eg\'en\'er\'es et telle que $\dim(E^c)=1$.

On peut prolonger contin\^ument le fibr\'e $E^c$ sur un voisinage de $K$.
Si l'on consid\`ere une boule $B\subset M$ suffisamment petite
rencontrant $K$, le fibr\'e $E^c_{|B}$ est orientable.
Deux points distincts et proches $p,q\in K$ sont en \emph{position vrill\'ee}
si l'on peut joindre $W^{ss}_{loc}(p)$ \`a $W^{uu}_{loc}(q)$
et $W^{ss}_{loc}(q)$ \`a $W^{uu}_{loc}(p)$ par deux courbes
tangentes \`a $E^c$ ayant la m\^eme orientation.
(En particulier, lorsque $W^{ss}_{loc}(p)$ et $W^{uu}_{loc}(q)$
s'intersectent, $p$ et $q$ sont en position vrill\'ee.)
Voir la figure~\ref{f.vrille}.
Cette notion d\'epend du choix d'un prolongement $E^c$, mais
est bien d\'efinie lorsque la distance $d(p,q)$ tend vers $0$.
\begin{figure}[ht]
\begin{center}
\input{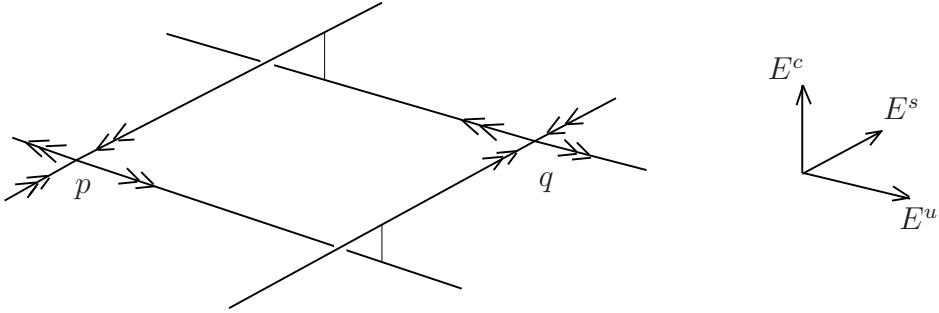}
\end{center}
\caption{Deux points en position vrill\'ee. \label{f.vrille}}
\end{figure}

\begin{definition}
L'ensemble $K$ est \emph{vrill\'e}\index{ensemble! vrill\'e} s'il existe $\varepsilon>0$
tel que toute paire de points $p,q\in K$ v\'erifiant $d(p,q)<\varepsilon$
est en position vrill\'ee.
\end{definition}

La d\'emonstration de la proposition~\ref{p.H} entra\^\i ne le r\'esultat suivant.
\begin{proposition}[section 3.4 de~\cite{palis-faible} et proposition 4.4 de~\cite{model}]\label{p.non-vrillee}
Si $f$ appartient \`a un G$_\delta$ dense de $\diff^1(M)$
et si $K$ est un ensemble compact transitif par cha\^\i nes
ayant une d\'ecomposition domin\'ee $T_KM=E^s\oplus E^c\oplus E^u$
telle que
\begin{itemize}
\item[--] $E^s$ et $E^u$ sont uniform\'ement contract\'es par $f$ et
$f^{-1}$ respectivement,
\item[--] $E^c$ est de dimension $1$ et de type (P$_{SU}$) ou (P$_{SN}$),
\item[--] $K$ n'est pas vrill\'e,
\end{itemize}
il existe alors,
dans tout voisinage $U$ de $K$, une orbite p\'eriodique
d'indice $\dim(E^s)+1$ contenue dans la classe de r\'ecurrence par cha\^\i nes
de $K$.
\end{proposition}

\subsection{Dynamique centrale pi\'eg\'ee}
Nous avons vu que lorsque le fibr\'e central de $K$
est hyperbolique par cha\^\i nes, l'ensemble $K$ est contenu dans une classe
homocline. Des hypoth\`eses plus faibles entra\^\i nent par le m\^eme argument l'existence
d'une classe homocline non triviale au voisinage de $K$.

\begin{proposition}\label{p.piege}
Si $f$ appartient \`a un G$_\delta$ dense de $\diff^1(M)$
et si $K$ est un ensemble compact transitif par cha\^\i nes
ayant une d\'ecomposition domin\'ee $T_KM=E^s\oplus E^c\oplus E^u$
telle que
\begin{itemize}
\item[--] $E^s$ et $E^u$ sont uniform\'ement contract\'es par $f$ et
$f^{-1}$ respectivement,
\item[--] $E^c$ est de dimension $1$ et de type (H)-attractif,
(P$_{SN}$) ou (N),
\item[--] $K$ n'est pas une orbite p\'eriodique,
\end{itemize}
il existe alors,
dans tout voisinage $U$ de $K$,
\begin{itemize}
\item[--]
une orbite p\'eriodique $O$
d'indice $\dim(E^s)+1$ dont la classe homocline est non triviale,
\item[--] un point $z\in K$,
\end{itemize}
tels que $W^{uu}(z)$ intersecte $W^s(O)$.
\end{proposition}

\section{Les autres types}
\subsection{Ensembles vrill\'es}

La g\'eom\'etrie des ensembles vrill\'es permet de cr\'eer une intersection
entre les vari\'et\'es invariantes fortes d'un point p\'eriodique.

\begin{proposition}[proposition 3.2 de~\cite{model}]\label{p.vrillee}
Consid\'erons un ensemble compact partiellement hyperbolique vrill\'e $K$.
Si $K\cap \Omega(f)$ contient un point non p\'eriodique, alors
pour tout voisinage $\cU$ de $f\in \diff^1(M)$ et $U$ de $K$,
il existe $g\in \cU$ ayant une orbite p\'eriodique $O\subset U$
telle que les vari\'et\'es fortes $W^{ss}(O)\setminus O$ et $W^{uu}(O)\setminus O$ s'intersectent.
\end{proposition}

Une premi\`ere \'etape consiste \`a cr\'eer une orbite p\'eriodique $O$ proche de $K$.
En reprenant la d\'emonstration du lemme de fermeture, on s'assure que l'orbite ainsi obtenue
est ``presque vrill\'ee''~: pour tous points $p,q\in O$ proches, les vari\'et\'es
fortes $W^{ss}_{loc}(p)$ et $W^{uu}_{loc}(q)$ d'une part ainsi que
$W^{ss}_{loc}(q)$ et $W^{uu}_{loc}(p)$ d'autre part, peuvent \^etre reli\'ees par
des courbes tangentes \`a un fibr\'e central et dont la longueur est arbitrairement petite
par rapport \`a la distance entre $p$ et $q$.

On peut alors appliquer l'argument de~\cite[section 6.1]{bgw}
ou~\cite[th\'eor\`eme 3.18]{palis-faible} pour les orbites p\'eriodiques vrill\'ees~:
on choisit une paire de points distincts $p,q\in O$ qui minimise
la distance $d(p,q)$. Les vari\'et\'es $W^{uu}_{loc}(p)$
et $W^{ss}_{loc}(q)$ contiennent des points $x_p$ et $x_q$ proches.
On obtient les propri\'et\'es suivantes~:
\begin{itemize}
\item[--]
la distance $d(x_p,x_q)$ est arbitrairement petite par rapport aux distances
de $x_p,x_q$ \`a $p$ et $q$, car $O$ est ``presque vrill\'ee''~;
\item[--]
la distance $d(x_p,x_q)$ est arbitrairement petite par rapport aux distances
de $x_p,x_q$ aux autres points de $O$, car si $z\in O$ est proche de $x_q\in W^{ss}(q)$,
il poss\`ede des it\'er\'es futurs $f^k(z)$ proches de $f^k(q)$, contredisant le fait que $d(p,q)$
est minimale.
\end{itemize}
Une perturbation \'el\'ementaire $g$ de $f$ permet d'envoyer $x_p$ sur $f(x_q)$,
sans perturber l'orbite $O$, l'orbite pass\'ee de $x_p$, ni l'orbite future de $f(x_q)$.
Par cons\'equent les vari\'et\'es fortes $W^{uu}_{loc}(p)$ et $W^{ss}_{loc}(q)$ pour $g$
s'intersectent.

\subsection{Dynamique centrale neutre}
Si $K$ est de type N et est strictement contenu dans un ensemble transitif par cha\^\i nes,
alors ou bien $K$ est contenu dans une classe homocline ayant des orbites p\'eriodiques faibles,
ou bien on peut cr\'eer un cycle h\'et\'erodimensionnel par perturbation.

\begin{proposition}[proposition 5.1 de~\cite{model}]\label{p.N}
Si $f$ appartient \`a un G$_\delta$ dense de $\diff^1(M)\setminus \overline{\cT}$
et si $K$ est un ensemble compact transitif par cha\^\i nes
ayant une d\'ecomposition domin\'ee $T_KM=E^s\oplus E^c\oplus E^u$
telle que
\begin{itemize}
\item[--] $E^s$ et $E^u$ sont uniform\'ement contract\'es par $f$ et
$f^{-1}$ respectivement,
\item[--] $E^c$ est de dimension $1$ et de type (N),
\item[--] $K$ est strictement contenu dans sa classe de r\'ecurrence par cha\^\i nes,
\end{itemize}
alors, l'un des deux cas suivants se produit.
\begin{enumerate}
\item $K$ est contenu dans une classe homocline d'indice $\dim(E^s)$ ou $\dim(E^s)$
qui poss\`ede des orbites p\'eriodiques dont le $(\dim(E^s)+1)$-\`eme exposant de Lyapunov
est arbitrairement proche de $0$.
\item Pour tout voisinage $U$ de $K$, il existe une perturbation $g\in \diff^1(M)$
de $f$ ayant une orbite p\'eriodique contenue dans $U$
telle que $W^{ss}(O)\setminus O$ et $W^{uu}(P)\setminus O$ s'intersectent.
\end{enumerate}
\end{proposition}

La d\'emonstration est bien plus d\'elicate que pour les propositions pr\'ec\'edentes.
En voici une id\'ee grossi\`ere (qui ne fonctionne que dans certains cas).
Fixons un point $z\not \in K$ appartenant \`a la classe de r\'ecurrence par cha\^\i nes
de $K$. Il existe des orbites qui quittent un voisinage arbitraire de $K$ pour
atteindre (par it\'erations pass\'ees ou futures) un voisinage arbitraire de $z$.
Les orbites qui s'\'echappent d'un petit voisinage de $K$ suivent l'``ensemble instable
par cha\^\i nes local'' de $K$. Puisque la dynamique centrale est pi\'eg\'ee
($E^c$ est de type (N)) on s'attend \`a ce que l'ensemble instable par cha\^\i nes local
de $K$ soit r\'eunion de plaques $W^{uu}_{loc}(x)$ pour $x\in K$.
Il est donc possible (dans certains cas) par perturbation de cr\'eer hors de $K$ une intersection
entre deux vari\'et\'es $W^{uu}(x)$ et $W^{ss}(y)$, avec $x,y\in K$.
Si l'on consid\`ere une orbite p\'eriodique $O$ suffisamment proche de $K$ en topologie de Hausdorff,
on peut alors cr\'eer une intersection entre les vari\'et\'es $W^{ss}(O)\setminus O$ et $W^{uu}(P)\setminus O$ .

\subsection{Ensembles instables par cha\^\i nes}
On peut esp\'erer que tout point de la classe de r\'ecurrence par cha\^\i nes contenant $K$
est accumul\'e par des orbites p\'eriodiques d'indice $\dim(E_{1})$.
Le r\'esultat suivant est d\'emontr\'e dans~\cite[section 5]{potrie}, voir aussi~\cite{yang-lyapunov-stable}
et~\cite[proposition 1.10]{model}.
La d\'emonstration n'utilise pas les mod\`eles centraux.

\begin{proposition}[Potrie]\label{p.instable-potrie}
Si $f$ appartient \`a un G$_\delta$ dense de $\diff^1(M)$
et si $K$ est un ensemble compact transitif par cha\^\i nes
ayant une d\'ecomposition domin\'ee $T_KM=E^s\oplus E^c\oplus E^u$
telle que
\begin{itemize}
\item[--] $E^s$ et $E^u$ sont uniform\'ement contract\'es par $f$ et
$f^{-1}$ respectivement,
\item[--] $E^c$ est de dimension $1$,
\item[--] l'exposant de Lyapunov le long de $E^c$ de toute mesure support\'ee par $K$ est nul,
\end{itemize}
alors, pour tout $\delta>0$, il existe un voisinage $U$ de $K$ tel que tout point $z\in U$ v\'erifiant~:
\begin{itemize}
\item[--] $z$ appartient \`a l'ensemble instable par cha\^\i nes local de $K$,
i.e. pour tout $\varepsilon>0$, il existe une $\varepsilon$-pseudo-orbite
contenue dans $U$ qui joint $K$ \`a $z$~;
\item[--] $z$ appartient \`a la classe de r\'ecurrence par cha\^\i nes de $K$,
\end{itemize}
v\'erifie \'egalement~:
\begin{itemize}
\item[] $z$ est limite d'une suite de points p\'eriodiques dont le $\dim(E^s)+1$-\`eme
exposant de Lyapunov appartient \`a $(-\delta,\delta)$.
\end{itemize}
\end{proposition}

\section{Application (1)~: dynamique simple versus intersections homoclines}
\label{s.palis-faible}\index{conjecture! faible de Palis}
Nous avons d\'emontr\'e dans~\cite{palis-faible}
un r\'esultat plus faible que la conjecture de Palis, \'egalement conjectur\'e
par Palis dans~\cite{palis-survey,palis-conjecture1,palis-conjecture2,palis-conjecture3}.
Nous avons introduit au chapitre~\ref{c.decomposition} les dif\-f\'e\-o\-mor\-phis\-mes Morse-Smale~: leur dynamique est tr\`es simple
puisque leur ensemble r\'ecurrent par cha\^\i nes est fini.
\`A l'oppos\'e, les diff\'eomorphismes ayant une classe homocline non triviale poss\`edent une infinit\'e
d'orbites p\'eriodiques~: ce sont les diff\'eomorphismes qui poss\`edent une intersection homocline transverse,
i.e. une orbite p\'eriodique hyperbolique $\cO$ dont les vari\'et\'es invariantes $W^s(\cO)\setminus \cO$
et $W^u(\cO)\setminus \cO$ ont un point d'intersection transverse.

\begin{theoreme}[Crovisier]\label{t.palis-faible}
L'espace $\diff^1(M)$ contient un ouvert dense qui est l'union disjointe $\cM\cS\cup\cI$
de deux ouverts~:
\begin{itemize}
\item[--] $\cM\cS$ est l'ensemble des diff\'eomorphismes Morse-Smale,
\item[--] $\cI$ est l'ensemble des diff\'eomorphismes qui poss\`edent une intersection homocline
transverse.
\end{itemize}
\end{theoreme}
Ce r\'esultat d\'ecoule de~\cite{peixoto} en dimension $1$, de~\cite{pujals-sambarino} en dimension $2$
et a \'et\'e d\'emontr\'e initialement par Bonatti-Gan-Wen~\cite{bgw} en dimension $3$.
Il implique que l'int\'erieur de l'ensemble des diff\'eomorphismes d'entropie non nulle
a la m\^eme adh\'erence que $\cI$ et que l'int\'erieur de l'ensemble des diff\'eomorphismes d'entropie nulle
a la m\^eme adh\'erence que $\cM\cS$. Voir aussi~\cite{pujals-sambarino, gan-surface} pour une discussion
de l'entropie dans le cas des dynamiques $C^1$ sur les surfaces et comparer au th\'eor\`eme
de Katok~\cite{katok} qui implique qu'un diff\'eomorphisme $C^2$ de surface ayant une entropie non nulle
a toujours une intersection homocline transverse.

\begin{proof}[D\'emonstration du th\'eor\`eme~\ref{t.palis-faible}]
Remarquons qu'un diff\'eomorphisme ayant une tangence homocline ou un
cycle h\'et\'erodimensionnel peut \^etre perturb\'e en un diff\'eomorphisme
ayant une intersection homocline transverse.
Par cons\'equent, il nous suffit de consid\'erer un diff\'eomorphisme $f$
loin des bifurcations homoclines. Nous pouvons \'egalement supposer que
$f$ appartient \`a un G$_\delta$ dense de $\diff^1(M)$ fix\'e \`a l'avance~:
nous supposerons en particulier que $f$ v\'erifie la propri\'et\'e de Kupka-Smale.

Si $f$ est hyperbolique, l'ensemble r\'ecurrent par cha\^\i nes
se d\'ecompose en un nombre fini de classe homoclines.
Si toutes les classes homoclines sont triviales, l'ensemble r\'ecurrent par cha\^\i nes
est fini et, puisque $f$ est un diff\'eomorphisme Kupka-Smale, pour tous
$x,y\in \cR(f)$, l'intersection $W^s(x)\cap W^u(y)$ est transverse~:
par cons\'equent $f$ est Morse-Smale.
Si l'une des classes homocline n'est pas triviale,
$f$ poss\`ede une intersection homocline transverse.

Si $f$ n'est pas hyperbolique, il poss\`ede un ensemble minimal non hyperbolique $K$.
Puisque $f$ v\'erifie la propri\'et\'e de Kupka-Smale, $K$ n'est pas une orbite p\'eriodique
et par cons\'equent
si $K$ est contenu dans une classe homocline, $f$ poss\`ede une intersection
homocline transverse. Nous pouvons donc supposer que $K$ est contenu dans une
classe ap\'eriodique.
D'apr\`es le corollaire~\ref{c.wen}, $K$ poss\`ede une structure
partiellement hyperbolique $T_KM=E^s\oplus E^c\oplus E^u$
avec $\dim(E^c)=1$. D'apr\`es les propositions~\ref{p.R}, \ref{p.H}, \ref{p.piege},
si le fibr\'e $E^c$ est de type (R), (H) ou (N),
ou si $K$ n'est pas vrill\'e et $E^c$ est de type (P),
$f$ poss\`ede une classe homocline non triviale et une intersection homocline transverse.
Si $K$ est vrill\'e, la proposition~\ref{p.vrillee} permet de construire une intersection homocline transverse. Dans tous les cas $f$ appartient donc \`a l'adh\'erence de
$\cI$.
\end{proof}

\section{Application (2)~: \'etude des quasi-attracteurs}\label{s.quasi-attracteur-tangence}

On peut utiliser les mod\`eles centraux pour obtenir des propri\'et\'es sur les quasi-attracteurs
des diff\'eomorphismes g\'en\'eriques loin des tangences homoclines~:
J. Yang a montr\'e~\cite{yang-lyapunov-stable} que ce sont des classes homoclines.
Voici une version l\'eg\`erement am\'elior\'ee de son r\'esultat.

\begin{theoreme}[Yang]\label{t.quasi-attracteur}
Pour tout diff\'eomorphisme dans un G$_\delta$ dense de $\diff^1(M)\setminus \overline{\cT}$, tout quasi-attracteur est une classe homocline $H$.

De plus, si $i$ est l'indices minimal des orbites p\'eriodiques qu'il contient, $H$ poss\`ede une d\'ecomposition domin\'ee
$T_HM=E^ s\oplus E^ c\oplus F$ telle que~:
\begin{itemize}
\item[--] $E^s\oplus E^c$ est de dimension $i$ et $E^c$ est de dimension $0$ ou $1$~;
\item[--] $E^s$ est uniform\'ement contract\'e, $E^c$ est de type (H)-attractif~;
\item[--] si $E^c$ est de dimension $1$ et non uniform\'ement contract\'e,
il existe des orbites p\'eriodiques dans $H$ dont l'exposant de Lyapunov le long de $E^c$ est arbitrairement proche de $0$.
\end{itemize}
\end{theoreme}
\begin{proof}
Soit $\Lambda$ un quasi-attracteur. Nous pouvons supposer que ce n'est pas un ensemble hyperbolique et consid\'erer un sous-ensemble
non hyperbolique minimal $K\subset \Lambda$.
D'apr\`es le th\'e\-o\-r\`e\-me~\ref{t.minimaux}, $K$ poss\`ede une d\'ecomposition
domin\'ee $T_KM=E^ s\oplus E^ c\oplus E^ u$ avec $\dim(E^c)\geq 1$.
Nous pouvons supposer que $K$ a \'et\'e choisi pour que la dimension
$\dim(E^s)$ soit minimale.

\begin{affirmation*}
$\Lambda$ est une classe homocline
qui contient des orbites p\'eriodiques
\begin{itemize}
\item[--] qui sont arbitrairement proches de $K$
en topologie de Hausdorff,
\item[--] ou dont l'indice est inf\'erieur ou \'egal \`a $\dim(E^s)$,
\item[--] ou dont le $(\dim(E^s)+1)$-\`eme exposant de Lyapunov est arbitrairement proche de $0$.
\end{itemize}
\end{affirmation*}
\begin{proof}
Lorsque $\dim(E^c)>1$, d'apr\`es le th\'eor\`eme~\ref{t.minimaux},
$\Lambda$ est une classe homocline
contenant des orbites p\'eriodiques dont le
$(\dim(E^s)+1)$-\`eme exposant de Lyapunov est arbitrairement proche de $0$.
Nous supposons donc que $\dim(E^c)=1$ et consid\'erons le type
de $E^c$.

Lorsque $E^c$ est de type (R) ou (H),
les propositions~\ref{p.R}, \ref{p.H}
entra\^\i nent directement que $\Lambda$ est une classe homocline
qui contient des orbites p\'eriodiques contenues dans des voisinages
arbitrairement petits de $K$.

Lorsque $E^c$ est de type (N) ou (P$_{SN}$), la proposition~\ref{p.piege}
montre que l'ensemble instable de $\Lambda$ s'accumule sur une orbite p\'eriodique $O$ contenue dans un voisinage arbitraire de $K$.
Puisque $\Lambda$ est un quasi-attracteur, il contient $O$ et $\Lambda$ est une classe homocline.

Nous nous pla\c{c}ons finalement dans le cas o\`u
$E^c$ est de type (P$_{SU}$) ou (P$_{UN}$).
L'ensemble $K$ poss\`ede un mod\`ele central (correspondant \`a l'une des orientations du fibr\'e $E^c$)
pour lequel il existe des voisinages
arbitrairement petits de $\widehat K\times \{0\}$ qui sont r\'epulsifs
et contenus dans l'ensemble instable par cha\^\i nes de $\widehat K\times \{0\}$.

En particulier, pour tout voisinage $U$ de $K$, et tout $x\in K$,
il existe une courbe $\gamma\subset U$ tangente \`a $E^c(x)$ en $x$
telle que
pour tout $z\in \gamma$ et tout $\varepsilon>0$, il existe une $\varepsilon$-pseudo-orbite
contenue dans $U$ qui joint $K$ \`a $z$.
Puisque $\Lambda$ est un quasi-attracteur, il contient $\gamma$
et la proposition~\ref{p.instable-potrie} s'applique.

Fixons $\delta>0$ proche de $0$.
Tout point $z$ proche de $x$ est limite d'une suite d'orbites p\'eriodiques
$(O_n)$ dont le $(\dim(E^s)+1)$-\`eme exposant de Lyapunov appartient \`a
$(-\delta,\delta)$. Nous pouvons supposer que $(O_n)$ converge en topologie
de Hausdorff vers un ensemble compact invariant $D\subset \Lambda$
qui, d'apr\`es le corollaire~\ref{c.tangence} poss\`ede
une d\'ecomposition domin\'ee $T_DM = E_1\oplus E_2$
avec $\dim(E_1)=\dim(E^s)$. Deux cas sont possibles.
\begin{itemize}
\item[--] Si $E_1$ n'est pas uniform\'ement contract\'e,
puisque $\dim(E^s)$ a \'et\'e choisie minimale, le th\'e\-o\-r\`e\-me~\ref{t.trichotomie}
montre que $\Lambda$ est une classe homocline contenant des points p\'eriodiques
d'indice inf\'erieur ou \'egal \`a $\dim(E^s)$.
\item[--] Si $E_1$ est uniform\'ement contract\'e, les vari\'et\'es
stables fortes de dimension $\dim(E^s)$ des orbites $O_n$
sont de taille uniforme~: pour $n$ suffisamment grand,
$W^s(O_n)$ intersecte donc la vari\'et\'e instable
forte (de dimension $\dim(E^u)$) d'un point $z'\in \gamma$
proche de $z$. Puisque $\Lambda$ est un quasi-attracteur, il contient $O_n$.
Nous en d\'eduisons que $\Lambda$ est une classe homocline
qui contient des orbites p\'eriodiques dont le
$(\dim(E^s)+1)$-\`eme exposant de Lyapunov est arbitrairement proche de $0$.
\end{itemize}

Dans tous les cas l'affirmation est d\'emontr\'ee.
\end{proof}
\medskip

Consid\'erons l'indice $i$ minimal des points p\'eriodiques contenus dans $\Lambda$.
D'apr\`es le corollaire~\ref{c.tangence}, il existe une d\'ecomposition
domin\'ee $T_\Lambda M=E\oplus F$ avec $\dim(E)=i$.
Si $E$ est uniform\'ement contract\'e, nous obtenons
la conclusion du th\'eor\`eme~\ref{t.quasi-attracteur}
avec $E^s=E$ et $\dim(E^c)=0$.

Si $E$ n'est pas uniform\'ement contract\'e, nous appliquons
le th\'eor\`eme~\ref{t.trichotomie}.
Le premier cas du th\'eor\`eme n'appara\^\i t pas puisque $i$ est minimal.
\begin{itemize}
\item[--] Dans le second cas, on obtient une d\'ecomposition domin\'ee
$E=E'\oplus E^c$ sur $\Lambda$ avec $\dim(E^c)=1$.
\item[--] Dans le troisi\`eme cas, il existe un ensemble
minimal $K\subset \Lambda$
avec structure partiellement hyperbolique $T_KM=E^s\oplus E^c\oplus E^u$
tel que $\dim(E^s)<\dim(E)$ et tel que toute orbite p\'eriodique
proche de $K$ a un exposant de Lyapunov central proche de $0$.
Puisque $i$ est minimal, $\Lambda$ ne contient pas d'orbite p\'eriodique
d'indice inf\'erieur ou \'egal \`a $\dim(E^s)$.
D'apr\`es l'affirmation,
la classe $\Lambda$ contient donc des orbites p\'eriodiques dont le
$\dim(E)$-\`eme exposant de Lyapunov est arbitrairement proche de $0$.
Nous obtenons \`a nouveau une d\'ecomposition domin\'ee
$E=E'\oplus E^c$ sur $\Lambda$ avec $\dim(E^c)=1$.
\end{itemize}
Dans les deux cas, $E$ se d\'ecompose et $H$ contient des orbites p\'eriodiques
dont l'exposant de Lyapunov le long de $E^c$ est proche de $0$.
Le fibr\'e $E'$ est uniform\'ement contract\'e~: dans le cas contraire
on pourrait r\'ep\'eter le raisonnement et conclure
que $\Lambda$ contient des points p\'eriodiques d'indice strictement inf\'erieurs \`a $i$.
\smallskip

Puisque $H$ contient des orbites p\'eriodiques hyperboliques d'indice $i$, le fibr\'e
$E^c$ ne peut \^etre que de type (H)-attractif ou de type (R).
La proposition~\ref{p.R2} montre que le type (R) n'appara\^\i t pas puisque $H$ ne contient pas de
point p\'eriodique d'indice $i-1$.
\end{proof}
\bigskip

Ceci permet de r\'epondre aux questions~\ref{q.interieur} et~\ref{q.closing-asymptotique} pour les diff\'eomorphismes
loin des tangences homoclines.
\begin{corollaire}[J. Yang~\cite{yang-lyapunov-stable}]
Pour tout diff\'eomorphisme dans un G$_\delta$ dense de $\diff^1(M)\setminus \overline{\cT}$, la r\'eunion des ensembles stables des orbites p\'eriodiques
est dense dans $M$.
\end{corollaire}

\begin{corollaire}[Potrie~\cite{potrie}]
Si $M$ est connexe,
pour tout diff\'eomorphisme dans un G$_\delta$ dense de $\diff^1(M)\setminus \overline{\cT}$, toute classe de r\'ecurrence par cha\^\i nes qui est stable au sens de Lyapunov pour $f$ et $f^{-1}$ co\"\i ncide avec $M$.
\end{corollaire}

Bonatti, Gan, Li et D. Yang~\cite{BGLY} on donn\'e r\'ecemment une r\'eponse aux questions~\ref{q.existence-quasi-attracteur} et~\ref{q.homocline-essentiel} dans ce cadre.

\begin{theoreme}[Bonatti-Gan-Li-Yang]
Pour tout diff\'eomorphisme appartenant \`a un G$_\delta$ dense de $\diff^1(M)\setminus \overline{\cT}$,
chaque quasi-attracteur est un attracteur essentiel.
\end{theoreme}
\bigskip

J. Yang montre \'egalement~\cite{yang-newhouse} que loin des tangences homoclines
le ph\'enom\`ene de Newhouse n'a lieu qu'au voisinage de classes
homoclines particuli\`eres.

\begin{theoreme}[J. Yang]
Pour tout diff\'eomorphisme dans un G$_\delta$ dense de $\diff^1(M)\setminus \overline{\cT}$, si $K$ est la limite de Hausdorff d'une suite de puits $(O_n)$,
$K$ est contenu dans une classe homocline ayant des points p\'eriodiques
d'indice $d-1$ dont le plus grand exposant de Lyapunov est arbitrairement proche de $0$.
\end{theoreme}
\begin{proof}
D'apr\`es le corollaire~\ref{c.homocline-bornee},
les puits $(O_n)$ ne sont pas $N$-uniform\'ement contract\'es \`a
la p\'eriode pour un entier $N\geq 1$. Puisque $f$ est g\'en\'erique, on d\'eduit du th\'eor\`eme~\ref{t.pliss}
que $\Lambda$ est limite de Hausdorff d'une suite de puits $(\tilde O_n)$ ayant un exposant de Lyapunov arbitrairement proche de $0$.
D'apr\`es le corollaire~\ref{c.tangence}, $\Lambda$ poss\`ede donc une d\'ecomposition domin\'ee
$E\oplus E^c$ avec $\dim(E^c)=1$. De plus $E^c$ n'est pas uniform\'ement dilat\'e.
Le th\'eor\`eme~\ref{t.trichotomie}
implique donc que l'un des cas suivants se produit.
\begin{itemize}
\item[--] $K$ est contenu dans une classe homocline ayant des orbites p\'eriodiques d'indice $d-1$
dont le plus grand exposant de Lyapunov est arbitrairement proche de $0$.
\item[--] $K$ contient un ensemble minimal $K'$ dont les exposants de Lyapunov le long de
$E^c$ pour toute mesure invariante support\'ee sur $K'$ sont nuls.
La proposition~\ref{p.H} entra\^\i ne alors que $K'$ est contenu dans une classe homocline
ayant des orbites p\'eriodiques d'indice $d-1$ arbitrairement proches de $K'$.
En particulier, $H$ contient des orbites p\'eriodiques d'indice $d-1$ dont le plus grand exposant de
Lyapunov est arbitrairement proche de $0$.
\end{itemize}
Le th\'eor\`eme est d\'emontr\'e dans tous les cas.
\end{proof}

\section{Application (3)~: loin des cycles h\'et\'erodimensionnels}

Les r\'esultats de ce chapitre permettent d'\'etendre~\cite{model} le r\'esultat local de Wen (corollaire~\ref{c.wen}).
\begin{theoreme}[Crovisier]\label{t.part-hyp}
Tout diff\'eomorphisme $f$ appartenant \`a un G$_{\delta}$ dense de $\diff^1(M)\setminus \overline{\cT\cup \cC}$
est partiellement hyperbolique~: son ensemble r\'ecurrent par cha\^\i nes $\cR(f)$ se d\'ecompose en un nombre
fini d'ensembles compacts disjoints $\Lambda_{1},\dots,\Lambda_{s}$. Chaque ensemble $\Lambda_{i}$, $1\leq i\leq s$,
poss\`ede une d\'ecomposition domin\'ee $T_{\Lambda_{i}}M=E^s\oplus E^c_{1}\oplus E^c_{2}\oplus E^u$ qui est partiellement
hyperbolique et telle que chaque fibr\'e central $E^c_{1}, E^c_{2}$ est de dimension $0$ ou $1$.

De plus,
\begin{itemize}
\item[--]
les classes homoclines de $f$ sont hyperboliques par cha\^\i nes,
\item[--] si $H(p)$ est une classe homocline ayant un fibr\'e $ E^c_i$ non hyperbolique,
elle poss\`ede des orbites p\'eriodiques homocliniquement reli\'ees \`a $p$ avec un exposant de Lyapunov
le long de $E^c_i$ arbitrairement proche de $0$,
\item[--] les classes ap\'eriodiques ont une dynamique minimale et
un fibr\'e central de dimension $1$ et de type (N).
\end{itemize}
\end{theoreme}

\begin{remarque}
Nous verrons en section~\ref{s.degenerescence}
que les fibr\'es $E^s$ et $E^u$ ne sont pas d\'eg\'en\'er\'es, sauf dans le cas des puits et des sources
(qui sont en nombre fini).
\end{remarque}

\begin{proof}
Consid\'erons tout d'abord un ensemble minimal $K$ ayant une structure partiellement hyperbolique
$T_{K}M=E^s\oplus E^c\oplus E^u$ avec $\dim(E^c)=1$ telle que toute mesure invariante support\'ee
sur $K$ ait un exposant central nul.
D'apr\`es la proposition~\ref{p.vrillee} et le lemme~\ref{l.connexion-forte}, l'ensemble $K$ n'est pas vrill\'e.
D'apr\`es les propositions~\ref{p.R}, \ref{p.H} et \ref{p.non-vrillee},
si $E^c$ est de type (R), (H) ou (P), tout voisinage de $K$ contient une orbite p\'eriodique
contenue dans la classe de r\'ecurrence par cha\^\i nes de $K$.
En particulier, $K$ est contenu dans une classe homocline ayant des points p\'eriodiques
d'indice $\dim(E^s)$ ou $\dim(E^s)+1$ et dont le $(\dim(E^s)+1)$-\`eme exposant de Lyapunov est arbitrairement proche de $0$.
Lorsque $E^c$ est de type (N) et que $K$ est strictement contenu dans sa classe de r\'ecurrence par cha\^\i nes,
nous obtenons la m\^eme conclusion d'apr\`es la proposition~\ref{p.N}. Ainsi
$K$ satisfait l'un des cas suivants.
\begin{itemize}
\item[--] $K$ est contenu dans une classe homocline ayant des points p\'eriodiques
d'indice $\dim(E^s)$ ou $\dim(E^s)+1$ et dont le $(\dim(E^s)+1)$-\`eme exposant de Lyapunov est arbitrairement proche de $0$,
\item[--] ou $K$ co\"\i ncide avec sa classe de r\'ecurrence par cha\^\i nes et $E^c$ est de type (N).
\end{itemize}
\medskip

Consid\'erons \`a pr\'esent une classe de r\'ecurrence par cha\^\i nes $C$ de $f$.
Si $C$ est une classe ap\'eriodique, elle contient un ensemble minimal non hyperbolique $K$
dont le fibr\'e central est de dimension $1$. D'apr\`es le corollaire~\ref{c.anosov-closing}, toutes les mesures
support\'ees sur $K$ ont un exposant de Lyapunov selon $E^c$ qui est nul. Nous pouvons
donc appliquer la discussion du premier paragraphe et conclure que $K=C$ et que le fibr\'e central est de type (N).
\medskip

Si l'on suppose que $C$ est une classe homocline $H(p)$, elle poss\`ede un unique indice.
D'apr\`es le corollaire~\ref{c.tangence}, elle supporte une d\'ecomposition domin\'ee $E\oplus F$ telle que $\dim(E)$
co\"{\i}ncide avec la dimension stable de $p$.
Si $E$ n'est pas uniform\'ement contract\'e, on peut appliquer le th\'eor\`eme~\ref{t.trichotomie}.
\begin{itemize}
\item[--] Nous excluons le premier cas puisque $H(p)$ poss\`ede un unique indice.
\item[--] Dans le second cas, la classe poss\`ede des points p\'eriodiques ayant un exposant stable proche de $0$.
\item[--] Dans le troisi\`eme cas, la classe contient un ensemble minimal  non hyperbolique $K$.
Le premier paragraphe montre que la classe contient des orbites p\'eriodiques ayant un exposant central proche de $0$.
Puisque $\dim(E^s)<\dim(E)$, l'exposant central est un exposant stable.
\end{itemize}
Nous concluons (des deux derniers cas) que le fibr\'e $E$ se d\'ecompose $E=E'\oplus E^c_1$ avec $\dim(E^c)=1$.
En reprenant le raisonnement avec le fibr\'e $E'$, nous en d\'eduisons que $E'$ est uniform\'ement contract\'e.
Puisque la classe contient des points p\'eriodiques pour lesquels $E^c_{1}$ est un fibr\'e stable, le type
de $E^c_{1}$ ne peut pas \^etre (N), (P) ou (H)-r\'epulsif. D'apr\`es la proposition~\ref{p.R2}, il ne peut \^etre de type (R).
Il est donc de type (H)-attractif.
Le m\^eme raisonnement montre que $F$ est uniform\'ement dilat\'e ou se d\'ecompose
en $F=E^c_{2}\oplus E^u$ avec $\dim(E^c_{2})=1$ et $E^c_{2}$ est de type (H)-r\'epulsif.
On en d\'eduit que $H(p)$ est hyperbolique par cha\^\i nes.
\end{proof}
\bigskip

Nous obtenons aussi un r\'esultat~\cite{cp1} sur la structure des quasi-attracteurs
(comparer au th\'eor\`eme~\ref{t.quasi-attracteur}).

\begin{corollaire}[Crovisier-Pujals]\label{c.cp-attracteur}
Pour tout diff\'eomorphisme $f$ appartenant \`a un G$_{\delta}$ dense de $\diff^1(M)\setminus \overline{\cT\cup \cC}$,
les quasi-attracteurs sont des classes homoclines hyperboliques par cha\^\i nes $H$ ayant une structure partiellement hyperbolique $T_H=E^{cs}\oplus E^{cu}=(E^s\oplus E^c)\oplus E^u$ avec $\dim(E^c)=0$ ou $1$.
\end{corollaire}
\begin{proof}
Consid\'erons une classe ap\'eriodique $K$. D'apr\`es le th\'eor\`eme pr\'ec\'edent, elle poss\`ede
une structure partiellement hyperbolique $T_KM=E^s\oplus E^c\oplus E^u$ avec $\dim(E^c)=1$.
De plus $E^c$ est de type (N). La proposition~\ref{p.piege} implique alors que l'ensemble instable de $K$
coupe la vari\'et\'e stable d'une orbite p\'eriodique. On en d\'eduit que $K$ ne contient pas sa vari\'et\'e
instable et n'est donc pas un quasi-attracteur.

Un quasi-attracteur est donc une classe homocline hyperbolique par cha\^\i nes $H(p)$.
Nous devons montrer qu'elle n'admet pas de fibr\'e non hyperbolique $E^c_2$ de dimension $1$ de type
(H)-r\'epulsif. Si c'\'etait le cas, il existerait des orbites p\'eriodiques homocliniquement
reli\'ees \`a $p$ ayant un exposant de Lyapunov le long de $E^c_2$ positif et arbitrairement proche de $0$.
Puisque $f$ satisfait une condition de g\'en\'ericit\'e, il existe des points p\'eriodiques $q$ d'orbites homocliniquement reli\'ees
\`a $p$ avec une demi-vari\'et\'e centrale tangente \`a $E^c_2$ arbitrairement courte~:
l'autre extr\'emit\'es de la vari\'et\'e centrale de $q$ est bord\'ee
par un point p\'eriodique $q'$ pour lequel $ E^c_2$ est un fibr\'e stable.
On en d\'eduit que $q'$
a une vari\'et\'e stable de dimension plus grande que celle de $p$.
Puisque $H(p)$ est un quasi-attracteur, il contient la vari\'et\'e instable de $q$~: par cons\'equent il
contient $q'$. La classe $H(p)$ contient donc des points p\'eriodiques d'indices diff\'erents.
La section~\ref{s.consequence} montre alors qu'il existe des perturbation de $f$ qui poss\`edent un cycle h\'et\'erodimensionnel. C'est une contradiction. La classe $H(p)$ n'a donc pas de fibr\'e non hyperbolique
de type (H)-r\'epulsif.
\end{proof}
\medskip

Nous terminons par un premier r\'esultat vers la finitude des quasi-atracteurs.

\begin{proposition}[Crovisier-Pujals]\label{p.fermeture}
Pour tout diff\'eomorphisme $f$ appartenant \`a un G$_{\delta}$ dense de $\diff^1(M)\setminus \overline{\cT\cup \cC}$,
l'union des quasi-attracteurs non-triviaux est ferm\'ee.
\end{proposition}
Nous verrons en section~\ref{s.degenerescence} que l'on peut enlever l'hypoth\`ese que les quasi-atracteurs sont triviaux.
\begin{proof}
Consid\'erons une suite de quasi-attracteurs non triviaux $(A_n)$ qui converge en topologie de Hausdorff
vers un ensemble compact invariant $K$, contenu dans une classe de r\'ecurrence par cha\^\i nes $\Lambda$.
Nous devons montrer que $\Lambda$ est un quasi-attracteur.
D'apr\`es le th\'eor\`eme~\ref{t.part-hyp}, $\Lambda$ est partiellement hyperbolique.
Nous notons $E^u$ son fibr\'e instable fort.

\begin{affirmation}
$\Lambda$ est une classe homocline $H(p)$.
\end{affirmation}
\begin{proof}
Supposons par l'absurde que $\Lambda$ est une classe ap\'eriodique.
Elle a une d\'e\-com\-po\-si\-tion domin\'ee $T_\Lambda M=E^s\oplus E^c\oplus E^u$,
avec $\dim(E^c)=1$ et $E^c$ de type (N). La proposition~\ref{p.degenere} implique que $E^u$ est non  d\'eg\'en\'er\'e.
Pour $n$ grand, le fibr\'e $ E^u$ est \'egalement d\'efini au-dessus des ensembles $A_n$.
Puisque les ensembles $ A_n$ sont des quasi-attracteurs, ils sont satur\'es en vari\'et\'es instables fortes
tangentes \`a $E^u$. Par passage \`a la limite, $K$ est \'egalement satur\'e en vari\'et\'es instables fortes.
Ceci contredit la proposition~\ref{p.piege} puisque $\Lambda$ ne contient pas de points p\'eriodiques.
\end{proof}
\smallskip

\begin{affirmation}
Nous pouvons nous ramener au cas o\`u la dimension stable de $p$ est strictement inf\'erieure \`a
celle des points p\'eriodiques des classes $A_n$.
\end{affirmation}
\begin{proof}
Supposons que la dimension instable de $p$ est inf\'erieure ou \'egale \`a celle des classes $A_n$.
Consid\'erons des familles de plaques $\cW^{cs},\cW^{cu}$ d\'efinissant sur $H(p)$
une structure de classe hyperbolique par cha\^\i nes.
Elles s'\'etendent \'egalement au-dessus des quasi-attracteurs $A_n$.

Si $E^{cu}$ est uniform\'ement dilat\'e, les quasi-attracteurs $A_n$ sont satur\'es en plaques de $\cW^{cu}$.

Sinon, il y a une d\'ecomposition $E^{cu}=E^c\oplus E^u$ et l'on peut aussi consid\'erer
une famille de plaques centrales $\cW^c$ tangentes \`a $E^c$ et pi\'eg\'ees par $f^{-1}$.
Nous remarquons que les plaques $\cW^c_x$ pour $x\in A_n$ sont contenues dans $A_n$.
En effet, si ce n'\'etait pas le cas il existerait un point p\'eriodique $q$ homocliniquement reli\'e \`a l'orbite de $p$
dont la vari\'et\'e instable ne contient pas enti\`erement la plaque $\cW^c_q$~:
il existe un point p\'eriodique $q'\in \cW^c_q$ qui borde la vari\'et\'e instable de $q$.
Puisque $A_n$ est un quasi-attracteur, il contient la vari\'et\'e instable de $q$ et donc le point $q'$.
Cependant le point $q'$ est d'indice diff\'erent de l'indice de $q$ et la section~\ref{s.consequence}
montre que l'on peut cr\'eer un cycle h\'et\'erodimensionnel par perturbation. C'est une contradiction.

Dans les deux cas les quasi-attracteurs $A_n$ sont satur\'es en plaques $\cW^{cu}$. Par passage \`a
la limite c'est aussi le cas de $K$.
\smallskip

D'apr\`es le lemme~\ref{l.hyp-chain}, il existe un ensemble dense de points p\'eriodiques $q$ homocliniquement reli\'es \`a l'orbite de $p$
tels que $\cW^{cs}_q\subset W^s(q)$ et $\cW^{cu}_q\subset W^u(q)$.
Il existe un tel point $q$ proche de $K$ dont la vari\'et\'e stable intersecte transversalement une
plaque centre-instable de $K$. On en d\'eduit finalement que $H(p)$ contient la vari\'et\'e instable de $p$.
D'apr\`es la section~\ref{s.consequence}, $H(p)$ est un quasi-attracteur. Ceci conclut la d\'emonstration.
Par cons\'equent, nous sommes ramen\'es au cas o\`u la dimension instable de $p$ est strictement plus grande
que celles des classes $A_n$.
\end{proof}
\smallskip

Nous en d\'eduisons que $H(p)$ a une d\'ecomposition domin\'ee
$T_{H(p)}=E^{cs}\oplus E^c\oplus E^u$ avec $\dim(E^c)=1$
telle que $E^c$ est de type (H)-r\'epulsif au-dessus de $H(p)$
et de type (H)-attractif au-dessus des quasi-attracteurs $A_n$.
Nous allons montrer que ceci m\`ene \`a une contradiction.

On consid\`ere un point $z\in K$ et on fixe un petit voisinage $U$ de $z$.
En utilisant l'expansion uniforme le long de $E^u$, on en d\'eduit que pour tout $n$ grand, il existe une
orbite de $A_n$ qui \'evite le voisinage $U$.
D'apr\`es le th\'eor\`eme de densit\'e (corollaire~\ref{c.pugh}),
il existe une orbite p\'eriodique $\tilde O_n$ contenue dans un petit voisinage de $A_n$
qui \'evite  $U$. En consid\'erant des plaques centre-stable pi\'eg\'ee suffisamment petites
au-dessus de la dynamique proche de $A_n$, on montre que $A_n$ contient une orbite
p\'eriodique $O_n$ rencontrant les m\^emes plaques centre-stable que $\tilde O_n$.
Par cons\'equent l'orbite p\'eriodique $O_n\subset A_n$ \'evite l'ouvert $U$.

\begin{affirmation}
Pour tout $n$ assez grand, il existe un voisinage $\cU_n$ de $f$ form\'e de dif\-f\'e\-o\-mor\-phis\-mes
pour lesquels les continuations de $W^{s}(O_n)$ et $W^u(p)$ s'intersectent.
\end{affirmation}
\begin{proof}
D'apr\`es l'affirmation pr\'ec\'edente, $E^c$ est un fibr\'e stable de $O_n$.
Il existe donc un point $x\in O_n$ tel que pour tout $k\geq 0$ on ait
$\prod_{i=0}^{k-1} \|Df(f^i(x))\|\leq 1$. La domination entre $E^{cs}$ et $E^c$
implique alors que $\cW^{cs}_x$ est contenue dans l'ensemble stable de $O_n$.
Par ailleurs, $K$ est contenu dans l'adh\'erence de la vari\'et\'e $W^u(p)$.
On en d\'eduit que $W^s(O_n)$ et $W^u(p)$ ont un point d'intersection transverse $z$, i.e.
$T_zM=T_zW^s(O_n)+T_zW^u(p)$. Cette propri\'et\'e est robuste aux petites perturbations dans $\diff^1(M)$.
\end{proof}

Pour $n$ grand, $W^u(O_n)$ rencontre un voisinage arbitrairement petit de $z$.
Par ailleurs, la vari\'et\'e $W^s(f^k(p))$ rencontre tout voisinage de $z$ pour un entier $k$.
Le lemme de connexion permet de cr\'eer une perturbation $g$, $C^1$-proche de $f$,
telle que $W^u(O_n)$ et $W^s(f^k(p))$ s'intersectent. On en d\'eduit que
$p$ et $O_n$ sont li\'es par des orbites h\'et\'eroclines de $g$.
Puisque leurs indices diff\`erent, nous avons cr\'e\'e un cycle h\'et\'erodimensionnel,
ce qui est une contradiction.
\end{proof}


\chapter{Hyperbolicit\'es topologique et uniforme}\label{c.hyp}

Ce chapitre traite de la conjecture de Palis.
Nous savons (th\'eor\`eme~\ref{t.part-hyp}) que les dif\-f\'e\-o\-mor\-phis\-mes loin des bifurcations homoclines sont partiellement hyperboliques. Nous devons \`a pr\'esent
montrer que les fibr\'es centraux des dynamiques partiellement hyperboliques du th\'eor\`eme~\ref{t.part-hyp} sont des fibr\'es uniform\'ement contract\'es ou dilat\'es.
\smallskip

La conjecture de Palis a \'et\'e r\'esolue sur les surfaces par Pujals et Sambarino~\cite{pujals-sambarino}.

\begin{theoreme}[Pujals-Sambarino]\label{t.ps}
Lorsque $M$ est de dimension $2$, tout diff\'eomorphisme peut \^etre approch\'e dans $\diff^1(M)$ par un
diff\'eomorphisme hyperbolique ou par un diff\'eomorphisme ayant une tangence homocline.
\end{theoreme}

En dimension sup\'erieure, nous avons montr\'e avec Pujals~\cite{cp1}
que, loin des bifurcations homoclines,
sur un ouvert dense de $M$ la dynamique se comporte comme une dynamique hyperbolique.
Plus pr\'ecis\'ement, un diff\'eomorphisme est dit \emph{essentiellement hyperbolique}\index{hyperbolicit\'e!essentielle}
s'il poss\`ede un nombre fini
d'attracteurs hyperboliques dont l'union des bassins est dense dans $M$, ainsi qu'un nombre fini de
r\'epulseurs hyperboliques dont l'union des bassins est denses dans $M$.

\begin{theoreme}[Crovisier-Pujals]\label{t.cp}
Tout diff\'eomorphisme peut \^etre approch\'e dans $\diff^1(M)$ par un diff\'eomorphisme
essentiellement hyperbolique ou par un diff\'eomorphisme qui pr\'esente une tangence homocline ou un cycle h\'et\'erodimensionnel.
\end{theoreme}
\medskip

La d\'emonstration du th\'eor\`eme~\ref{t.ps} (et ses g\'en\'eralisations)
repose essentiellement sur un argument de distorsion
et utilise pour cela le passage en r\'egularit\'e $C^2$.
Le d\'emonstration du th\'eor\`eme~\ref{t.cp} est plus g\'eom\'etrique~: c'est pour cette raison
que l'on obtient seulement l'hyperbolicit\'e des attracteurs et que l'on ne contr\^ole pas
la dynamique sur les pi\`eces de type ``selle''.

\section{Hyperbolicit\'e des fibr\'es extr\^emes}
Nous avons vu en section~\ref{s.degenere}, qu'au-dessus d'un ensemble transitif par cha\^\i nes $K$
qui n'est pas un puits, un fibr\'e extr\^eme de dimension $1$ est ``topologiquement hyperbolique''
(il est de type (H)). Dans certains cas, il est possible de d\'emontrer qu'il est uniform\'ement
hyperbolique.
\smallskip

Sur les surfaces, la sym\'etrie entre $E$ et $F$ permet~\cite{pujals-sambarino} d'obtenir l'hyperbolicit\'e de $K$.
\begin{theoreme}[Pujals-Sambarino]\label{t.pujals-sambarino}
Lorsque $M$ est de dimension $2$, pour tout diff\'eomorphisme appartenant \`a un G$_\delta$ dense de $\diff^1(M)$, tout ensemble compact invariant $K$ ayant une d\'ecomposition domin\'ee non triviale est hyperbolique.
\end{theoreme}

Ceci entra\^\i ne la conjecture de Palis sur les surface.

\begin{proof}[D\'emonstration du th\'eor\`eme~\ref{t.ps}]
Consid\'erons un diff\'eomorphisme qui ne peut pas \^etre approch\'e par un dif\-f\'e\-o\-mor\-phis\-me
ayant une tangence homocline. Nous pouvons l'approcher par un dif\-f\'e\-o\-mor\-phis\-me $f$ appartenant \`a un G$_\delta$
dense de $\diff^1(M)$.
D'apr\`es le th\'eor\`eme~\ref{t.part-hyp}, chaque classe de r\'ecurrence par cha\^\i nes de $f$ poss\`ede une d\'ecomposition
domin\'ee non triviale, elle est donc hyperbolique d'apr\`es le th\'eor\`eme pr\'ec\'edent.
\end{proof}

Ce r\'esultat se g\'en\'eralise~\cite{pujals-sambarino-dissipative} en dimension sup\'erieure lorsque $E$ est uniform\'ement contract\'e.
\begin{theoreme}[Pujals-Sambarino]\label{t.pujals-sambarino-dissipative}
Pour tout diff\'eomorphisme appartenant \`a un G$_\delta$ dense de $\diff^1(M)$, toute
classe homocline $H$ ayant une d\'ecomposition domin\'ee $T_HM=E\oplus F$ avec $\dim(F)=1$
et $E$ uniform\'ement contract\'e est hyperbolique.
\end{theoreme}

Nous avons \'etendu~\cite{cp1} ces travaux dans un cadre o\`u $E$ n'est plus uniform\'ement contract\'e~:
nous supposerons \`a la place que $E$ est finement pi\'eg\'ee et qu'il existe une famille de plaques
localement invariante $\cW$ tangente \`a $E$ (rappelons que d'apr\`es la section~\ref{s.chain-hyp}, une telle famille de plaques est essentiellement unique)
et telle que pour tout $x\in H$
l'intersection $H\cap \cW_x$ est totalement discontinue.

\begin{theoreme}[Crovisier-Pujals]\label{t.pc-2D}
Pour tout diff\'eomorphisme appartenant \`a un G$_\delta$ dense de $\diff^1(M)$ et toute
classe homocline $H$ ayant une d\'ecomposition domin\'ee $T_HM=E\oplus F$ avec $\dim(F)=1$ et v\'erifiant~:
\begin{itemize}
\item[--] $E$ est finement pi\'eg\'ee,
\item[--] il existe une famille de plaques $\cW$ localement invariante et tangente \`a $E$
telle que $\cW_x\cap H$ est totalement discontinu pour tout $x\in H$,
\end{itemize}
le fibr\'e $F$ est uniform\'ement dilat\'e ou $H$ est un puits.
\end{theoreme}

\begin{remarque}\label{r.dimsup}
Ces r\'esultats restent vrais pour des ensemble compacts invariants $K$
contenus dans une vari\'et\'e localement invariante $N$ normalement hyperbolique.
Par exemple, si $N$ est une surface et si $K$ poss\`ede une d\'ecomposition domin\'ee tangente \`a $N$
non triviale, alors $K$ est un ensemble hyperbolique.
\end{remarque}

\section{Technique de Ma\~n\'e-Pujals-Sambarino}
Ma\~n\'e avait d\'emontr\'e un r\'esultat similaire~\cite{mane-1D} pour les endomorphismes du cercle.
Pour les diff\'eomorphismes en dimension plus grande, l'id\'ee est d'exploiter la contraction
(topologique ou uniforme) le long des plaques tangentes \`a $E$ pour ``quotienter'' la dynamique
dans la direction centre-stable et se ramener ainsi \`a un cadre essentiellement unidimensionnel.
\medskip

Les th\'eor\`emes~\ref{t.pujals-sambarino}, \ref{t.pujals-sambarino-dissipative}
et~\ref{t.pc-2D} sont la contrepartie de r\'esultats non perturbatifs en classe $C^2$.
Nous consid\'erons donc dans la suite un diff\'eomorphisme $f\in \diff^2(M)$
et un ensemble compact invariant $K$
muni d'une d\'ecomposition domin\'ee $T_KM=E\oplus F$ avec $\dim(F)=1$. Nous supposerons que~:
\smallskip

\begin{itemize}
\item[a)] pour tout point p\'eriodique de $K$, le fibr\'e $F$ est uniform\'ement dilat\'e~;
\smallskip

\item[b)] $K$ ne contient pas de courbe ferm\'ee simple tangente \`a $F$, invariante par un it\'er\'e de $f$.
\end{itemize}
\smallskip

\noindent
En effet, si l'on suppose que $f$ est un diff\'eomorphisme Kupka-Smale,
tous les points p\'eriodiques qui ne v\'erifient pas la premi\`ere condition
sont des puits et les courbes de la seconde condition sont normalement hyperboliques et portent une dynamique tr\`es simple. En retirant toutes les courbes (en nombre fini), tous les puits de $K$
et leur bassins, on se ram\`ene aux conditions \'enonc\'ees ci-dessus.

On d\'emontre alors que $F$ est uniform\'ement dilat\'e au-dessus de $K$ par une ``r\'ecurrence''~:
on se ram\`ene ais\'ement \`a un ensemble $K$ ayant la propri\'et\'e suppl\'ementaire suivante~:
\smallskip

\begin{itemize}
\item[c)] Pour tout sous-ensemble
compact invariant $K'\subsetneq K$, le fibr\'e $F_{|K'}$ est uniform\'ement dilat\'e.
\end{itemize}
\smallskip

\noindent
Nous introduisons \'egalement une famille de plaques pi\'eg\'ees $\cW^{cs}$ tangente \`a $E$.
\bigskip

La d\'emonstration comporte plusieurs \'etapes.
\smallskip

\begin{itemize}
\item[\bf 1) Hyperbolicit\'e topologique.] En utilisant que pour tout sous-ensemble, $F$ est u\-ni\-for\-m\'e\-ment dilat\'e,
on montre qu'il existe une famille de plaques $\cW^{cu}$ localement invariante et tangente \`a $F$ telle que
$|f^n(\cW^{cu}_x)|\to 0$ lorsque $n\to \infty$.
C'est cet argument, de type Denjoy-Schwartz, qui permet de d\'emontrer le th\'eor\`em~\ref{t.codim-c2}.
\smallskip

\item[\bf 2) Construction de bo\^\i tes markoviennes.]
On construit alors une ``bo\^\i te'' munie d'une lamination centre-instable, i.e. une r\'e\-u\-ni\-on $B$ de courbes contenues dans des plaques de $\cW^{cu}$
telle que $B\cap K$ est d'int\'erieur non vide dans $K$. Nous demandons \'egalement \`a $B$ de v\'erifier
des propri\'et\'es de type ``rectangles
de partitions de Markov''.
\smallskip

\item[\bf 3) Hyperbolicit\'e uniforme.] En induisant dans $B$, les propri\'et\'es markoviennes de $B$,
la contraction topologique des plaques de $\cW$ et un contr\^ole de distorsion montrent
alors que $\|D_xf^{-n}_{|F}\|\to 0$ en tout $x\in K$ lorsque $n\to \infty$.
\end{itemize}
\medskip

Les \'etapes 1) et 3) sont identiques pour chacun des th\'eor\`emes~\ref{t.pujals-sambarino},
\ref{t.pujals-sambarino-dissipative} et~\ref{t.pc-2D} mais la partie 2) est sp\'ecifique.
\smallskip

\begin{itemize}
\item[--] \emph{Dans le cas des surfaces,} on exploite le fait que les plaques $\cW^{cu}$
et $\cW^{cs}$ sont de dimension $1$ pour chercher \`a construire un rectangle g\'eom\'etrique
bord\'e par deux plaques de $\cW^{cs}$ et deux plaques de $\cW^{cu}$.
\smallskip

\item[--] \emph{Dans le cas o\`u le fibr\'e $E$ est uniform\'ement dilat\'e,} il est possible de construire
une partition de Markov pour la classe homocline $H$. La bo\^\i te $B$ est construite \`a partir de l'un des
rectangles de la partition.
\smallskip

\item[--] \emph{Dans le cas o\`u $E$ est seulement finement pi\'eg\'ee,}
l'hypoth\`ese de discontinuit\'e dans les plaques
permet \`a nouveau de construire une partition adapt\'ee. Plus pr\'ecis\'ement,
pour tout $x\in H$, il existe un ensemble compact $C_x\subset \cW^{cs}_x$ v\'erifiant~:
\begin{itemize}
\item[\quad i)] pour $x,x'\in H$ on a $C_x=C_{x'}$ ou $C_x\cap C_{x'}=\emptyset$,
\item[\quad ii)] pour tout $x\in H$, on a $f(C_x)\subset C_{f(x)}$,
\item[\quad iii)] $(C_x)$ varie contin\^ument en topologie de Hausdorff avec $x\in H$.
\end{itemize}
La bo\^\i te $B$ s'obtient alors en fixant $x\in H$,
en consid\'erant des segments de courbes contenus dans les plaques $\cW^{cu}_y$ pour $y\in H\cap C_x$ \`a extr\'emit\'es contenues
dans deux plaques $\cW^{cs}_{x^-}$ et $\cW^{cs}_{x^+}$ proches de $\cW^{cs}_x$.
\end{itemize}

\section{Classes homoclines \`a plaques centre-stables discontinues}\label{s.discontinu}
Nous \'enon\c{c}ons \`a pr\'esent dans le cadre des ensembles hyperboliques un r\'esultat non perturbatif qui permet
de satisfaire l'hypoth\`ese topologique du th\'eor\`eme~\ref{t.pc-2D}.

\begin{theoreme}[Crovisier-Pujals]\label{t.discont1}
Consid\'erons un ensemble hyperbolique $K$ localement maximal muni d'une d\'e\-com\-po\-si\-tion domin\'ee
$T_KM=E^s\oplus E^u=(E^{ss}\oplus E^c)\oplus E^u$ avec $\dim(E^c)=\dim(E^u)=1$.
Alors l'un des trois cas suivants se produit~:
\begin{itemize}
\item[--] $K$ est contenu dans une sous-vari\'et\'e localement invariante
tangente \`a $E^c\oplus E^u$~;
\item[--] $K$ poss\`ede une connexion forte g\'en\'eralis\'ee~;
\item[--] pour tout $x\in K$ l'intersection $W^s_{loc}(x)\cap K$
est totalement discontinue.
\end{itemize}
\end{theoreme}

Ce r\'esultat s'\'etend aux classes homoclines hyperboliques par cha\^\i nes.

\begin{theoreme}[Crovisier-Pujals]\label{t.discont2}
Consid\'erons une classe homocline $H$ et
une d\'ecomposition domin\'ee $T_HM=E^{cs}\oplus E^{cu}=(E^{ss}\oplus E^c)\oplus E^{cu}$ avec $\dim(E^c)=\dim(E^{cu})=1$ telle que $E^{cs}$ et $E^{cu}$ soient finement pi\'eg\'es par $f$ et $f^{-1}$ respectivement.
Alors l'un des trois cas suivants se produit~:
\begin{itemize}
\item[--] $H$ est contenue dans une sous-vari\'et\'e localement invariante
tangente \`a $E^c\oplus E^{cu}$~;
\item[--] $H$ poss\`ede une connexion forte g\'en\'eralis\'ee~;
\item[--] il existe une famille de plaques $\cW^{cs}$ localement invariante
et tangente \`a $E^{cs}$ telle que pour tout $x\in H$ l'intersection $\cW^{cs}_x\cap H$
est totalement discontinue.
\end{itemize}
\end{theoreme}

\begin{proof}[Id\'ee de la d\'emonstration]
Supposons $H$ sans connexion forte g\'en\'eralis\'ee.
\begin{affirmation*}
Il n'existe pas d'ensemble connexe $C\subset H$ non r\'eduit \`a un point et contenu dans une vari\'et\'e stable forte.
\end{affirmation*}
\begin{proof}
Supposons le contraire. Nous utilisons les deux propri\'et\'es suivantes.
\begin{itemize}
\item[--] En it\'erant $C$ n\'egativement, on obtient des ensembles connexes de diam\`etre arbitrairement grand dans les
feuilles stables fortes.
\item[--] Les projections de $C$ par holonomie le long des plaques centre-instables,
pr\'eservent les feuilles stables fortes~:
en effet, si l'image $C'$ de $C$ par holonomie est contenue dans une plaque centre-stable
et rencontre plusieurs vari\'et\'es stables fortes, on choisit un point p\'eriodique $q$
proche de $C'$ et on en d\'eduit que la plaque centre-instable $\cW^{cu}_q$ de $q$ rencontre $C$.
Les it\'er\'es $f^{-n}(C)$ lorsque $n\to +\infty$ convergent vers une partie connexe non triviale
contenue dans $H\cap W^{ss}(q)$. Ceci contredit l'hypoth\`ese sur $H$. 
\end{itemize}
En transportant $C$ par holonomie sur une vari\'et\'e stable d'orbite p\'eriodique $q$
contenue dans $H$, ces propri\'et\'es impliquent que $W^{ss}(q)$ rencontre $H$ sur un ensemble connexe non trivial,
contredisant l'hypoth\`ese que $H$ n'a pas de connexion forte g\'en\'eralis\'ee.
\end{proof}
\smallskip

Supposons que $H$ contienne un ensemble connexe non trivial $C$ contenu dans une plaque
centre-stable. Puisque $C$ intersecte chaque vari\'et\'e stable forte selon un ensemble totalement discontinu,
$C$ est un graphe au-dessus de la direction centrale~: c'est une courbe.
\smallskip

Consid\'erons \`a pr\'esent une orbite p\'eriodique $O$ ayant un point $q$ proche de $C$.
La projection de $C$ sur la plaque centre-stable de $q$ d\'efini une nouvelle courbe $C_q\subset H$~;
puisque $W^{ss}(q)\setminus \{q\}$ n'intersecte pas $C$, cette courbe contient $q$.
En it\'erant, on obtient \'egalement une courbe en chaque point de $O$.
Puisque $H$ contient un ensemble dense de points p\'eriodiques, on construit par passage \`a la limite une courbe
$C_x\subset H$ en chaque point de $H$.

Supposons que $H$ n'est pas contenue dans une vari\'et\'e tangente \`a $E^c\oplus E^{cu}$.
D'apr\`es le th\'eor\`eme~\ref{t.whitney}, il existe deux points $x\neq y$
tels que $W^{ss}(x)=W^{ss}(y)$. Soit $q$ un point p\'eriodique proche de $y$.
La projection par holonomie de $C_x$ sur la plaque centre-stable de $q$ coupe $W^{ss}(q)$ en un point
de $W^{ss}(q)\setminus\{q\}$ appartenant \`a $H$. Nous avons \`a nouveau contredit l'existence
d'une connexion forte g\'en\'eralis\'ee.
Par cons\'equent, les composantes connexes de $H$ dans les plaques centre-stables sont triviales.
\end{proof}

\section{Application~: non d\'eg\'en\'erescence des fibr\'es uniformes}\label{s.degenerescence}

Nous d\'eduisons des sections pr\'ec\'edentes
un r\'esultat~\cite{cp1} qui compl\`ete le th\'eor\`eme~\ref{t.part-hyp}.

\begin{corollaire}[Crovisier-Pujals]
Pour tout diff\'eomorphisme appartenant \`a un G$_\delta$ dense de $\diff^1(M)\setminus \overline{\Tang\cup\Cycl}$,
toute classe de r\'ecurrence par cha\^\i nes qui n'est pas un puits ou une source poss\`ede une structure
partiellement hyperbolique avec fibr\'es stables et instables forts non d\'eg\'en\'er\'es.
\end{corollaire}
\begin{proof}
Nous appliquons le th\'eor\`eme~\ref{t.part-hyp} sur la structure partiellement hyperbolique des
classes de r\'ecurrence par cha\^\i nes.

Consid\'erons une classe ap\'eriodique. Elle poss\`ede une structure partiellement hyperbolique
avec central de dimension $1$ de type (N). La proposition~\ref{p.degenere} montre alors directement
que les fibr\'es uniformes sont non d\'eg\'en\'er\'es.

Consid\'erons une classe homocline non hyperbolique $H(p)$~:
elle admet une structure partiellement hyperbolique de la forme $E^s\oplus E^c\oplus E^u$
ou $E^s\oplus E^c_1\oplus E^c_2\oplus E^u$.
Dans le premier cas, le th\'eor\`eme~\ref{t.pujals-sambarino-dissipative} montre que $E^s$ et $E^u$
ne sont pas d\'eg\'en\'er\'es si la classe n'est pas un puits ou une source.
Dans le second cas, la classe est hyperbolique par cha\^\i nes avec une d\'ecomposition
$E^{cs}\oplus E^{cu}=(E^s\oplus E^c_1)\oplus (E^c_2\oplus E^u)$.
Si $E^u$ est d\'eg\'en\'er\'e, le th\'eor\`eme~\ref{t.discont2} montre que l'on est dans l'un des cas suivants.
\begin{itemize}
\item[--] La classe est contenue dans une surface tangente \`a $E^c_1\oplus E^c_2$
et d'apr\`es le th\'eor\`eme~\ref{t.pujals-sambarino} et la remarque~\ref{r.dimsup}
elle est hyperbolique~: c'est une contradiction.
\item[--] La classe poss\`ede une connexion forte. Si le fibr\'e $E^c_1$
n'est pas uniform\'ement contract\'e, il existe des orbites p\'eriodiques
homocliniquement reli\'ees \`a $p$ ayant un exposant selon $E^c_1$ arbitrairement faible
d'apr\`es le th\'eor\`eme~\ref{t.part-hyp}.
On d\'eduit du lemme~\ref{l.generalisee} que l'on perturber le diff\'eomorphisme pour cr\'eer un cycle h\'et\'erodimensionnel. C'est \`a nouveau une contradiction.
\item[--] La classe est totalement discontinue le long des plaques centre-stables.
\end{itemize}
On applique alors le th\'eor\`eme~\ref{t.pc-2D} et on en d\'eduit que $E^c_2$ est uniform\'ement dilat\'e.
\end{proof}

\begin{corollaire}[Crovisier-Pujals]\label{c.puits-fini}
Tout diff\'eomorphisme dans un G$_\delta$ dense de $\diff^1(M)\setminus \overline{\cT\cup\cC}$,
poss\`ede au plus un nombre fini de puits et de sources.
\end{corollaire}
\begin{proof}
Une suite de puits ou de source doit s'accumuler sur une partie d'une classe
de r\'ecurrence par cha\^\i nes qui n'est ni un puits ni une source.
\end{proof}

\section{Hyperbolicit\'e des quasi-attracteurs}\label{s.hyperbolicite-quasi-attracteur}
Voici un nouveau r\'esultat sur la g\'eom\'etrie des ensemble hyperboliques~:
quitte \`a perturber, l'existence de vari\'et\'e stables fortes rencontrant la classe en plusieurs
points peut \^etre d\'etect\'ee sur les orbites p\'eriodiques.
Ceci reprend un travail ant\'erieur de Pujals~\cite{pujals-3D1,pujals-3D2}.

\begin{theoreme}[Crovisier-Pujals]\label{t.dichot0}
Soit $K$ un attracteur hyperbolique avec une d\'ecomposition domin\'ee
$T_KM=E^{s}\oplus E^{u}=(E^{ss}\oplus E^c)\oplus E^u$, $\dim(E^c)=1$.

Il existe alors une perturbation $g\in \diff^1(M)$ de $f$ telle que
\begin{itemize}
\item[--] ou bien $W^{ss}(x)\cap K=\{x\}$ pour tout $x\in K$,
\item[--] ou bien $K$ poss\`ede une connexion forte g\'en\'eralis\'ee.
\end{itemize}
\end{theoreme}

Ce r\'esultat s'\'etend aux classes hyperboliques par cha\^\i nes.

\begin{theoreme}[Crovisier-Pujals]\label{t.dichot}
Soit $H$ un quasi-attracteur qui est une classe homocline avec une d\'e\-com\-po\-si\-tion domin\'ee
$T_HM=E^{cs}\oplus E^{cu}=(E^{s}\oplus E^c)\oplus E^u$, $\dim(E^c)=1$ telle que
$E^{cs}$ est finement pi\'eg\'e.

Il existe alors une perturbation $g\in \diff^1(M)$ de $f$ telle que
\begin{itemize}
\item[--] ou bien $W^{ss}(x)\cap H=\{x\}$ pour tout $x\in K$,
\item[--] ou bien $H$ poss\`ede une connexion forte g\'en\'eralis\'ee.
\end{itemize}
\end{theoreme}

Nous pouvons \`a pr\'esent terminer la d\'emonstration du th\'eor\`eme~\ref{t.cp}.

\begin{proof}[D\'emonstration du th\'eor\`eme~\ref{t.cp}]
Par sym\'etrie, il suffit de traiter le cas des attracteurs.

Nous savons d'apr\`es le th\'eor\`eme~\ref{c.cp-attracteur} que chaque quasi-attracteur qui n'est pas un ensemble
hyperbolique est une classe homocline avec une d\'ecomposition domin\'ee
$T_HM=E^{cs}\oplus E^{cu}=(E^{s}\oplus E^c)\oplus E^u$, $\dim(E^c)=1$.
Nous pouvons alors appliquer le th\'eor\`eme~\ref{t.dichot}.
Dans le premier cas, le th\'eor\`eme~\ref{t.whitney} montre que $H$
est contenue dans une sous-vari\'et\'e localement invariante tangente \`a $E^c\oplus E^u$.
Le th\'eor\`eme~\ref{t.pujals-sambarino-dissipative} et la remarque~\ref{r.dimsup} impliquent alors que
$H$ est un ensemble hyperbolique.
Dans le second cas, puisque $H$ poss\`ede des orbites p\'eriodiques faibles homocliniquement reli\'ees \`a $p$,
il existe d'apr\`es le lemme~\ref{l.generalisee} une perturbation de $f$ admettant un cycle h\'et\'erodimensionnel~:
c'est une contradiction. Nous avons donc montr\'e que tous les quasi-attracteurs sont des ensembles hyperboliques.

D'apr\`es les r\'esultats de la section~\ref{s.consequence}, l'union des bassins des attracteurs est dense dans $M$.
D'apr\`es le corollaire~\ref{c.puits-fini}, l'ensemble des puits de $f$ est fini.
D'apr\`es la proposition~\ref{p.fermeture}, l'union des attracteurs non triviaux est ferm\'ee~:
il n'y a donc qu'un nombre fini d'attracteurs.
\end{proof}

\section{Classification des connexions fortes}
Abordons maintenant certains arguments importants de la d\'emonstration du th\'eor\`eme~\ref{t.dichot}.

\begin{proposition}\label{p.trichot}
Consid\'erons une classe homocline $H(p)$ ayant les m\^emes propri\'et\'es qu'au th\'eor\`eme~\ref{t.dichot}.
Elle satisfait alors l'un des cas suivants~:
\begin{itemize}
\item[a)] il existe $g$ $C^1$-proche de $f$ tel que, pour tout
$x$ appartenant \`a la continuation $H(p_g)$ de $H(p)$ pour $g$, on ait $W^{ss}(x)\cap H(p_g)=\{x\}$~;
\item[b)] il existe $g$ $C^1$-proche de $f$ ayant une connexion forte g\'en\'eralis\'ee~;
\item[c)] il existe deux points p\'eriodiques $p_x,p_y$ homocliniquement reli\'es \`a l'orbite de $p$,
et un ouvert $\cU$ de diff\'eomorphismes $g$ $C^1$-proche de $f$ tels que pour tout $g\in \cU$
il existe deux points distincts $x\in W^u(p_x)$ et $y\in W^u(p_y)$ dans $H(p_g)$ avec $W^{ss}(x)=W^{ss}(y)$~;
\item[d)] il existe un point p\'eriodique $q$ homocliniquement reli\'e \`a l'orbite de $p$
et pour tout diff\'eomorphisme $g$ $C^1$-proche de $f$
deux points $x_g\neq y_g$ dans $W^s(q_g)\cap H(p_g)$ tels que $W^{ss}(x_g)=W^{ss}(y_g)$.
\end{itemize}
\end{proposition}

\begin{proof}[Id\'ee de la d\'emonstration lorsque $K=H(p)$ est un ensemble hyperbolique]
Nous pouvons sup\-po\-ser qu'il existe $x\neq y$ dans $K$ avec $W^{ss}(x)=W^{ss}(y)$
car sinon nous sommes dans le cas a).

Fixons une famille de plaques centre-instables $\cW^{uu}$ et une petite constante $\varepsilon>0$.
Nous comparons alors les vari\'et\'es instables $W^u_{loc}(x)$ et $W^u_{loc}(y)$ en projetant
$W^u_{loc}(y)$ sur la plaque centre-instable de $x$ par l'holonomie $\Pi^{ss}$ le long des vari\'et\'es stables fortes.
Trois cas sont possibles.

\begin{description}
\item[-- Intersection transverse.] Il existe $x\neq y$ dans $K$ avec $W^{ss}(x)=W^{ss}(y)$
tels que la projection $\Pi^{ss}(W^{u}_{loc}(y))$ intersecte les deux composantes connexes de
$B(x,\varepsilon)\cap(\cW^{cu}_x\setminus W^{uu}_{loc}(x))$.
\item[-- Int\'egrabilit\'e jointe.] Il existe $x\neq y$ dans $K$ avec $W^{ss}(x)=W^{ss}(y)$ et tels que
$\Pi^{ss}(W^{u}_{loc}(y))$ et $W^{u}_{loc}(x)$ co\"\i ncident dans $B(x,\varepsilon)$.
\item[-- Intersection strictement non transverse.]
Pour tout $x\neq y$ avec $W^{ss}(x)=W^{ss}(y)$, la projection $\Pi^{ss}(W^{u}_{loc}(y))$
est disjointe de l'une des composantes connexes de $B(x,\varepsilon)\cap(\cW^{cu}_x\setminus W^{uu}_{loc}(x))$
et intersecte l'autre.
\end{description}

Dans le cas transverse, on choisit $p_x,p_y$ p\'eriodiques proches de $x$ et $y$ respectivement.
Par continuit\'e du feuilletage instable, nous sommes dans le cas c).
\medskip

Dans le cas de l'int\'egrabilit\'e jointe, on choisit $q$ p\'eriodique proche de $x$. Sa vari\'et\'e stable locale coupe $W^u(x)$ et $W^u(y)$ en deux points $x',y'\in K$.
On a $W^{ss}(x')=W^{ss}(y')$.
Pour tout dif\-f\'e\-o\-mor\-phis\-me $C^1$-proche $g$,
les continuations hyperboliques $x'_g,y'_g\in K_g$ de $x',y'$ pour $g$
appartiennent toujours \`a la vari\'et\'e stable locale $W^s_{loc}(q_g)$.
Si $y'_g$ appartient \`a la vari\'et\'e $W^{ss}_{loc}(x'_g)$ pour tout $g$ $C^1$ proche de $f$,
nous sommes dans le cas d).
Si $y'_g$ n'appartient plus \`a la vari\'et\'e $W^{ss}_{loc}(x'_g)$, 
on consid\`ere un arc $(g_t)$ entre $f$ et $g$ dans un petit voisinage de $f$ dans $\diff^1(M)$.
Puisque $x'$ est accumul\'e par des points p\'eriodiques de $K$,
nous en d\'eduisons qu'il existe $g_t$ tel que $\Pi^{ss}(y'_{g_t})$ appartient \`a la vari\'et\'e instable d'un tel point p\'eriodique. Nous obtenons donc une connexion forte g\'en\'eralis\'ee (cas b).
\medskip

Il reste \`a examiner le cas strictement non transverse.
Puisqu'il n'y a pas d'intersection transverse, les points
$x\neq y$ avec $W^{ss}(x)=W^{uu}(y)$ ne sont accumul\'es par $K\cap W^s_{loc}(x)$ que d'un seul c\^ot\'e de $W^{ss}_{loc}(x)$ (voir la figure~\ref{f.intersection-stricte}).
De tels points sont appel\'es \emph{points de fronti\`ere stable forte}.
\begin{figure}[ht]
\begin{center}
\input{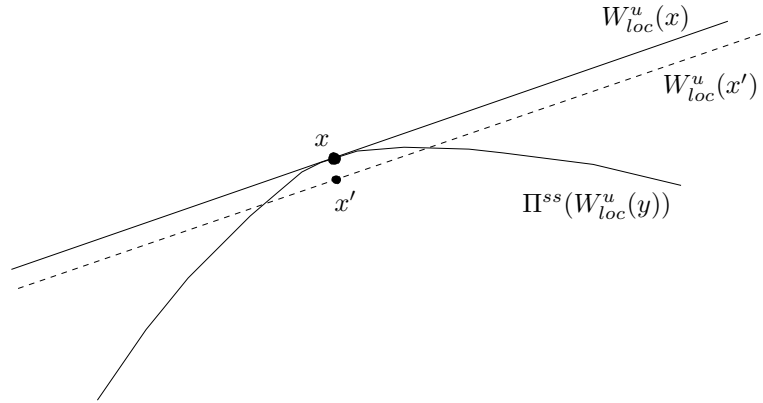}
\end{center}
\caption{Intersection strictement non transverse. \label{f.intersection-stricte}}
\end{figure}
\begin{lemme}
Les points de fronti\`ere stable forte appartiennent \`a la vari\'et\'e instable d'un point p\'eriodique
qui est un point de fronti\`ere stable forte. Ces points p\'eriodiques fronti\`ere sont en nombre fini.
\end{lemme}

Consid\'erons les points p\'eriodiques fronti\`ere stable forte $p_1,\dots,p_\ell$ de $f$ et leur vari\'et\'es instables locales.
S'il existe un arc $(g_t)$ de diff\'eomorphismes proches de $f$ et un point $y\in W^u_{loc}(p_i)$ tel que
\begin{itemize}
\item[--] $x_{g_0}:=\Pi^{ss}(y_{g_0})$ appartient \`a une autre vari\'et\'e $W^u_{loc}(p_j)$,
\item[--] $\Pi^{ss}(y_{g_1})$ appartient \`a la composante connexe de
$W^s_{loc}(x_{g_1})\setminus W^{ss}_{loc}(x_{g_1})$ dans laquelle $K$ accumule $x_g$,
\end{itemize}
nous concluons comme pr\'ec\'edemment que l'un des diff\'eomorphismes $g_t$ a une connexion forte.
Nous sommes dans le cas b).
\smallskip

Dans le cas contraire, s'il existe un diff\'eomorphisme $g$ $C^1$-proche de $f$ tel que toutes les intersections
$\Pi^{ss}(W^u_{loc}(p_i))\cap W^u_{loc}(p_j)$ sont disjointes, nous sommes dans le cas a).
\smallskip

Finalement, il existe un ouvert $\cU$ de diff\'eomorphismes proche de $f$ en topologie $C^1$ et
deux points $p_i$, $p_j$ tels que pour tout $g\in \cU$ les vari\'et\'es
$\Pi^{ss}(W^u_{loc}(p_i))$ et $W^u_{loc}(p_j)$ s'intersectent. Nous sommes alors dans le cas c).
\end{proof}

\section{Connexions de vari\'et\'es instables p\'eriodiques}
Nous discutons \`a pr\'esent le cas c) du th\'eor\`eme~\ref{t.dichot} (dans le cadre hyperbolique)
et montrons que l'on peut perturber $f$ pour obtenir une connexion forte.

\begin{proposition}\label{p.instable}
Consid\'erons un attracteur hyperbolique $K$ avec une d\'ecomposition domin\'ee
$T_KM=E^s\oplus E^u= (E^{ss}\oplus E^c)\oplus E^u$ telle que $\dim(E^c)=1$.
Supposons qu'il existe deux points p\'eriodiques $p_x,p_y\in K$ et, pour tout $g$
$C^1$-proche de $f$, deux points distincts $x\in W^{u}(p_{x,g})$ et $y\in W^u(p_{y,g})$ de $K_g$ tels que $W^{ss}(x)=W^{ss}(y)$.

Il existe alors $g$ $C^1$-proche de $f$ et $q\in H(p_g)$ p\'eriodique tel que $W^{ss}(q)$ intersecte $W^u(p_y)$.
\end{proposition}
\begin{proof}[Id\'ee de la d\'emonstration]
Soient~:
\begin{itemize}
\item[--] $\lambda_u>1$ une constante qui minore la dilatation des it\'er\'es positifs de $f$ le long de $E^u$,
\item[--] $\lambda\in (0,1)$ une constante qui minore la domination entre $E^{ss}$ et $E^c$,
\item[--] $\log(\lambda_c)$ l'exposant de Lyapunov de l'orbite de $p_x$ le long de $E^c$.
\end{itemize}

Quitte \`a le remplacer par un diff\'eomorphisme proche,
nous pouvons supposer que $f$ est de classe $C^2$
et lin\'eaire au voisinage de $p_x$.
Nous pouvons supposer \'egalement que $W^{ss}(p_x)\setminus \{p_x\}$ est disjoint de $K$
puisque dans le cas contraire le lemme de connexion permet de cr\'eer par perturbation
une intersection entre $W^{ss}(p_x)$ et $W^u(p_y)$ et de conclure directement la d\'emonstration
de la proposition.
Nous noterons $W^u_{loc}(p_y)$ et $W^u_{loc}(p_x)$ des vari\'et\'es locales de $p_x$ et $p_y$
contenant $x,y$ respectivement.

\paragraph{Connexion \`a retours lents.}
Nous dirons que la connexion est \`a retours lents s'il existe des \'echelles $d>0$ arbitrairement
petites
telles que le temps d'entr\'ee pour $x$ et $y$ dans la boule $U=B(f^{-1}(x),d)$ est sup\'erieur \`a $C_0.|\log(d)|$
pour une certaine constante $C_0>0$.
Dans ce cas, nous perturbons $f$ dans la boule $U$ pour
obtenir un diff\'eomorphisme $g$ tel que $g(f^{-1}(x))$ appartienne \`a la plaque centrale de $x$.
La distance $r$ entre $g(f^{-1}(x))$ et $x$ v\'erifie $r/d>\varepsilon$ o\`u $\varepsilon$
ne d\'epend que de la taille de la perturbation $C^1$ autoris\'ee.

Apr\`es perturbation les continuations $x_g$ et $y_g$ de $x,y$ appartiennent
\`a $g(W^u_{loc}(p_x))$ et \`a $W^u_{loc}(p_y)$ mais ne co\"\i ncident pas n\'ecessairement avec $g(f^{-1}(x))$
et $y$. On a les estim\'ees
$$d(x_g,g(f^{-1}(x)))<\lambda_u^{-n_x},\quad d(y_g,y)<\lambda_u^{-n_y},$$
o\`u $n_x,n_y$ sont les temps d'entr\'ee de $x$ et $y$ dans $U$.

Puisque $f$ est de classe $C^2$, l'holonomie $\Pi^{ss}_f$ du feuilletage stable fort (d\'efini sur un voisinage
de $K$) est h\"olderienne~\cite{PSW}~: il existe $\alpha$ tels que
$$d(\Pi^{ss}_f(y_g),\Pi^{ss}_f(y))< d(y_g,y)^\alpha.$$

Nous comparons \'egalement les holonomies $\Pi^{ss}_f, \Pi_g^{ss}$ pour $f$ et pour $g$~:
$$d(\Pi^{ss}_f(y_g),\Pi^{ss}_g(y_g))< \lambda^{n_{y_g}},$$
o\`u $n_{y_g}$ est le temps d'entr\'ee de $y_g$ dans $U$.

Puisque $\Pi^{ss}_f(y)=x$, on a
\begin{equation*}
\begin{split} d(x_g,\Pi_g^{ss}(y_g))>&\\
d(g(f^{-1}&(x)),x)-d(g(f^{-1}(x)),x_g)-d(\Pi_g^{ss}(y_g),\Pi^{ss}_f(y_g))
-d(\Pi^{ss}_f(y_g),\Pi^{ss}_f(y)).
\end{split}
\end{equation*}

Par cons\'equent $x_g\neq \Pi^{ss}_g(y_g)$ lorsque
$$r-\lambda_u^{-n_x}- \lambda_u^{-\alpha.{n_y}}-\lambda^{n_{y_g}}>0.$$
Les temps d'entr\'ee $n_x,n_y,n_{y_g}$ sont essentiellement les m\^emes car
$y_g$ est la continuation de $y$ et $x,y$ appartiennent \`a une m\^eme vari\'et\'e stable forte.
Nous notons $n=\inf(n_x,n_y,n_{y_g})$.
Puisque $r>\varepsilon.d$, la connexion entre $x$ et $y$ est bris\'ee par perturbation
lorsque $n>C_0.|\log(d)|$ pour une certaine constante $C_0>0$, ce qui est le cas puisque la connexion est \`a retours lents.

Consid\'erons un arc de diff\'eomorphismes $(g_t)$ proche de $f$ et joignant $f$ \`a $g$.
Par continuit\'e, il existe un diff\'eomorphisme $g_t$ tel que $W^{ss}_{loc}(y_{g_t})$
intersecte la vari\'et\'e instable d'un point p\'eriodique de $K$ proche de $x$.
Ceci conclut la d\'emonstration de la proposition dans ce cas.

\paragraph{Connexion \`a retours rapides.}
Dans ce cas, pour tout $d>0$, le temps de retour de $x$ et $y$
dans la boule $U=B(f^{-1}(x),d)$ est inf\'erieur \`a $C_0. |\log(d)|$.
Ceci implique le lemme cl\'e suivant.

\begin{lemme}
Il existe $a>b>0$ et $n$ arbitrairement grand tel que~:
\begin{itemize}
\item[--] $f^n(x)$ est \`a distance inf\'erieure \`a $e^{-a.n}$ de $x$,
\item[--] les it\'er\'es $f^k(x)$, $0<k<n$, sont \`a distance sup\'erieure \`a $e^{-b.n}$ de $x$.
\end{itemize}
\end{lemme}
\begin{proof}
Consid\'erons la suite des plus proches retours $f^{n_i}(x)$ de $x$ pr\`es de lui m\^eme.
Nous fixons un voisinage $W$ de l'union des vari\'et\'es stables et instables locale de l'orbite de $p_x$.
Nous supposons que l'orbite pass\'ee de $f^{-1}(x)$ est contenue dans $W$, mais $x\notin W$.
Pour chaque retour $f^{n_i}(x)$ de $x$ pr\`es de lui-m\^eme, nous notons
$\{f^{m_i}(x),\dots,f^{n_i-1}(x)\}$ le segment d'orbite maximal contenu dans $W$.

Pour chaque $i$, la longueur $\ell_i=n_i-m_i$ est de l'ordre de $|\log(d_i)|$, o\`u $d_i$ est la distance entre $f^{n_i}(x)$
et la vari\'et\'e $W^u_{loc}(x)$.
En effet, puisque $W^{ss}(p_x)\setminus \{p_x\}$ est disjoint de $K$, les retours pr\`es de $x$ et de $p_x$
se font le long de la direction centrale unidimensionnelle. Puisque $f$ est lin\'eaire au voisinage de
$p_x$, la distance $d_i$ est de l'ordre de $\lambda_c^{\ell_i}$.

Par hypoth\`ese nous avons donc
$$n_i\leq C_0.|\log(d_i)|\leq C_1. \ell_i.$$

Posons
$$R=\limsup_i \frac{\ell_i}{n_i}.$$
Pour $\varepsilon>0$ petit, nous choisissons $j_0$ tel que pour tout $j>j_0$
on ait
$$\frac{\ell_j}{n_j}<(1+\varepsilon).R.$$
Nous choisissons ensuite $n_i$ arbitrairement grand pour que
$$\frac{\ell_i}{n_i}>(1-\varepsilon).R.$$
Pour tout $j_0<j<i$ nous avons $n_j< n_i- \ell_i$ et donc
$$\ell_j<(1+\varepsilon).R.n_j<(1+\varepsilon).R.(1-(1-\varepsilon).R).n_i.$$
Nous posons $a_0=(1-\varepsilon)^2.R$ et $b_0=(1+\varepsilon)^2.R.(1-(1-\varepsilon).R)$.
Puisque $R$ appartient \`a $[C_1^{-1},1]$, nous avons $a_0>b_0>0$.
On a bien
$$d_i<\lambda_c^{(1-\varepsilon)\ell_i}<\lambda_c^{a_0.n_i}$$
et pour tout $j<i$,
$$d_j> \lambda_c^{(1+\varepsilon)\ell_j}>\lambda_c^{b_0.n_i}.$$
Le lemme est donc d\'emontr\'e avec $a=a_0|\log{\lambda_c}|$
et $b=b_0|\log \lambda_c|$.
\end{proof}
\medskip

Pour conclure la d\'emonstration de la proposition~\ref{p.instable} dans ce cas, nous consid\'erons
un disque $D$ contenu dans la vari\'et\'e stable forte de $y$ et contenant $x$ et $y$.
Pour tout point $z\in D$ et tout retour $f^n(z)$ pr\`es de $x$, le rayon de $f^n(D)$
est tr\`es petit devant la distance $d$ de $f^n(z)$ \`a $W^u_{loc}(p_x)$~:
si $\ell$ est le temps de passage pr\`es de l'orbite de $p_x$,
la distance $d$ est de l'ordre de $\lambda_c^\ell$ et
le rayon de $f^n(D)$ est inf\'erieur \`a $\lambda_s^\ell$, o\`u $\lambda_s<\lambda_c$
majore la contraction de $Df^n_{p_x}$ le long de $E^s$.

Choisissons $n$ comme dans le lemme pr\'ec\'edent~:
$f^n(D)$ est contenu dans une boule centr\'ee en $x$ et de rayon tr\`es petit devant les
distances \`a $x$ des images $f^k(D)$, $0< k <n$.
Il existe donc une perturbation $g$ de $f$
pr\`es de $f^{-1}(x)$ telle que $g^{n}(D)\subset D$. On en d\'eduit que $g$
poss\`ede un point p\'eriodique $q\in D$ dont la vari\'et\'e stable forte contient $D$.
Puisque $W^u_{loc}(p_y)$ n'a pas \'et\'e perturb\'e, $W^{ss}(q)$ et $W^u(p_y)$
s'intersectent.

\end{proof}

\section{Connexions contenues dans une vari\'et\'e stable p\'eriodique}
Pour terminer la d\'emonstration du th\'eor\`eme~\ref{t.dichot0}, nous devons expliquer
comment perturber dans le cas d) de la proposition~\ref{p.trichot} afin de cr\'eer une connexion forte.

\begin{proposition}
Consid\'erons un ensemble hyperbolique $K$ localement maximal muni d'une d\'ecomposition domin\'ee
$T_KM=E^s\oplus E^u=(E^{ss}\oplus E^c)\oplus E^u$ avec $\dim(E^c)=1$.
Supposons qu'il existe $q\in K$ p\'eriodique et
deux points $x\neq y$ de $K$
et tels que pour tout $g$ $C^1$ proche de $f$ les continuations $x_g$ et $y_g$
v\'erifient $x_g\in W^{ss}(y_g)$.

Il existe alors un arc de diff\'eomorphismes $(g_t)$ $C^1$-proche de $f$
et des points $x',y'\in K$ tels que $W^{ss}_{loc}(x'_{g_t})$ traverse
$W^u_{loc}(y'_{g_t})$ lorsque $t$ varie.
\end{proposition}
On obtient la connexion forte en rempla\c{c}ant $z$ par un point p\'eriodique proche.
\medskip

Nous supposerons pour simplifier que $q$ est un point fixe
et (quitte \`a remplacer $f$ par un diff\'eomorphisme proche)
que la dynamique est lin\'eaire au voisinage de $q$. Dans une carte, $f$ est de la forme~:
$$(x^{ss},x^c,x^u)\mapsto (A^s(x^{ss}),\lambda_c.x^c,A^u(x^u)),$$
o\`u $A^s,(A^u)^{-1}$ sont des applications lin\'eaires contractantes et $\lambda_c$ appartient \`a $(-1,0)\cup (0,1)$. Nous notons $d^{ss}$ la dimension du fibr\'e $E^{ss}$.

Dans la vari\'et\'e stable locale $W^{s}_{loc}(q)$ (qui co\"\i ncide avec $\RR^{d^{ss}+1}\times \{0\}$ localement), le feuilletage stable fort $\cF^{ss}$ co\"\i ncide avec le feuilletage en sous-espaces affines parall\`eles \`a $\RR^{d^{ss}}\times \{0\}$.
\smallskip

Il existe un point $z\in W^u(q)\cap W^{s}(q)\cap K$ d'orbite distincte de celles de $x$ et $y$.
Fixons un disque $D$ de $W^s(q)$ qui contient $z$.
Quitte \`a remplacer $z$ par un it\'er\'e n\'egatif, nous pouvons supposer que tous ses it\'er\'es n\'egatifs sont contenus dans le voisinage lin\'earisant de $q$.
Par petite perturbation, nous pouvons redresser $D$ et le feuilletage stable fort dans $D$~:
le disque $D$ est alors contenu dans un plan parall\`ele \`a $\RR^{d^{ss}+1}\times \{0\}$
et $\cF^{ss}$ est le feuilletage en sous-espaces parall\`eles \`a $\RR^{d^{ss}}\times \{0\}$.
\smallskip

Nous comparons \`a pr\'esent les espaces instables aux points de $K$ proches de $q$~:
nous nous int\'eressons ainsi aux angles entre deux espaces instables dans la direction centrale,
i.e. \`a l'angle entre leurs projections dans l'espace de coordonn\'ees $E^c(q)\oplus E^u(q)$.

En rempla\c{c}ant $x,y$ par des it\'er\'es positifs,
\emph{$E^u(x),E^u(y)$ sont arbitrairement proches de $E^u(q)$.}
\smallskip

\begin{figure}[ht]
\begin{center}
\input{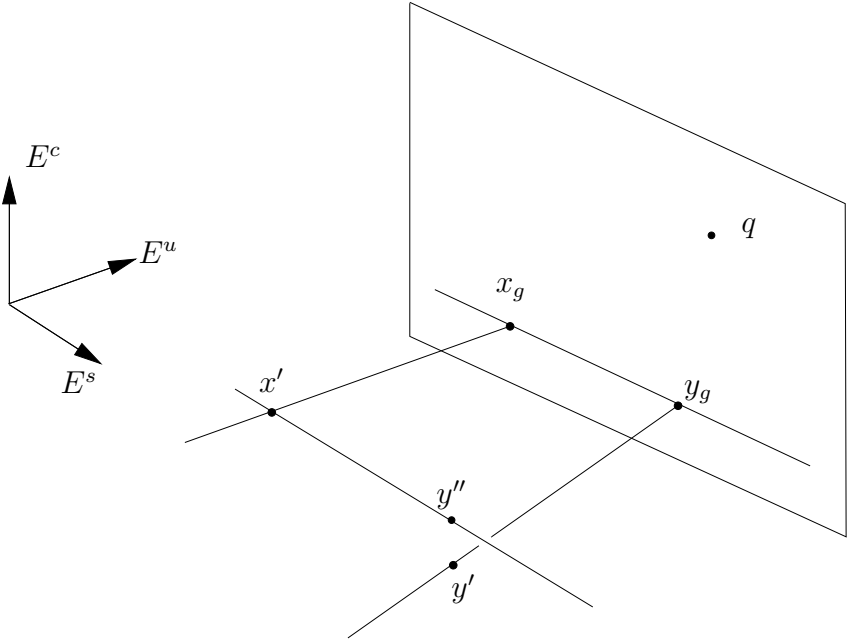}
\end{center}
\caption{Perturbation dans une vari\'et\'e stable p\'eriodique. \label{f.connexion-stable}}
\end{figure}

Nous fixons une petite constante $\alpha>0$.
Nous perturbons $f$ au voisinage de $f^{-1}(x)$ en un diff\'eomorphisme $g$ de sorte que
$g$ co\"\i ncide avec $f$ au voisinage de $f^{-1}(x)$ sur $W^s(q)$
mais $Dg_{f^{-1}(x)}.E^u(f^{-1}(x))$ est envoy\'e sur
un espace faisant un angle $\alpha$ avec $E^u(x)$.

La norme $C^0$ de la perturbation peut \^etre rendue arbitrairement $C^0$-petite si le support de la perturbation est contenu dans un
voisinage arbitrairement petit de $f^{-1}(x)$. En revanche l'angle $\alpha$ est donn\'e par la taille
de la perturbation $C^1$. On en d\'eduit que les continuations $x_g,y_g$ restent arbitrairement proches de $x,y$. Par cons\'equent,
\emph{$E^u(x_g)$ fait un angle proche de $\alpha$ avec $E^u(x)$ et $E^u(q)$~;
 l'espace $E^u(y_g)$ un angle arbitrairement petit avec $E^u(y)$ et $E^u(q)$.}
\smallskip

Les it\'er\'es $g^{-n}(D)$ pour $n\geq 0$ grand s'accumulent sur $W^s_{loc}(q)$ et recoupent
$W^u_{loc}(x_g), W^u_{loc}(y_g)$ en deux points $x',y'$. Pour $n$ grand les vecteurs $(x_g,x')$
et $(y_g,y')$ sont proches de $E^u(x_g)$ et $E^u(y_g)$.
Soit $y''$ la projection de $x'$ sur la plaque centre-instable de $y'_g$,
par l'holonomie le long des vari\'et\'es stables fortes.
Voir la figure~\ref{f.connexion-stable}.

Les orbites positives de $x'$ par $g$ et $f$ co\"{\i}ncident. Par cons\'equent la
vari\'et\'e stable forte locale de $x'$ est la m\^eme pour $f$ et pour $g$~:
c'est un sous-espace affine parall\`ele \`a $W^{ss}_{loc}(x')$.
On en d\'eduit que les vecteurs $(x_g,x')$ et $(y_g,y'')$ ont la m\^eme direction.
Par cons\'equent pour $g$, $(y_g,y''_g)$ fait un angle proche de $\alpha$ avec $E^u(y_g)$ et $E^u(q)$.

Consid\'erons \`a pr\'esent un arc de diff\'eomorphismes en faisant varier l'angle $\alpha$~:
par continuit\'e nous pouvons assurer que la direction du vecteur $(y_g,y''_g)$ traverse
$E^u_g$, en projection dans l'espace de coordonn\'ee $E^c(q)\oplus E^{ss}(q)$.
On en d\'eduit que $W^{ss}_{loc}(x'_g)$ traverse $W^u_{loc}(y_g)$. \qed

\section{Continuation des classes hyperboliques par cha\^\i nes}
Pour g\'en\'eraliser les raisonnements des sections pr\'ec\'edentes dans le cas des classes
hyperboliques par cha\^\i nes (i.e. dans le cadre du th\'eor\`eme~\ref{t.dichot}), nous avons besoin
de d\'efinir la notion de continuation de la classe (qui ne pose pas de probl\`eme dans le cadre hyperbolique).

On consid\`ere une classe de r\'ecurrence par cha\^\i nes qui est une classe homocline
$H(p)$ avec une structure partiellement hyperbolique $T_{H(p)}M=E^s\oplus E^c\oplus E^u$,
$\dim(E^c)=1$, telle que $E^c$ est finement pi\'eg\'ee.
Pour tout $g$ $C^1$-proche de $f$, la classe homocline $H(p_g)$ est contenue dans un petit
voisinage de $H(p)$, poss\`ede la m\^eme structure partiellement hyperbolique, et reste
hyperbolique par cha\^\i nes (th\'eor\`eme~\ref{t.chain-robust}).

Remarquons que $H(p)$ poss\`ede un ensemble dense $\cP$ de points p\'eriodiques $q$ 
pour lesquels l'exposant de Lyapunov central est uniform\'ement major\'e par une constante
strictement n\'egative, et il existe \'egalement une famille de plaques $\cW^{cs}$ tangente \`a $E^s\oplus E^c$,
pi\'eg\'ee, telle que $\cW^{cs}_q\subset W^s(q)$ pour chaque point $q\in \cP$.
\begin{lemme}\label{l.densite}
Il existe un voisinage $\cU\subset \diff^1(M)$ de $f$ form\'e de diff\'eomorphismes $g$
pour lesquels la continuation hyperbolique $q_g$ de chaque point $q\in \cP$ est bien d\'efinie
et tels que l'ensemble $\{q_g,\; q\in \cP\}$ est dense dans $H(p_g)$.
\end{lemme}

\begin{definition}
Pour $g,g'\in \cU$, deux points $x\in H(p_g)$, $x'\in H(p_{g'})$
\emph{ont la m\^eme continuation} s'il existe une suite $(q_n)$ de $\cP$ telle que
$(q_{n,g})$ converge vers $x$ et $(q_{n,g'})$ converge vers $x'$.
\end{definition}

Le lemme~\ref{l.densite} implique que pour tout $g,g'\in \cU$ et tout $x\in H(p_g)$,
il existe $x'\in H(p_{g'})$ ayant la m\^eme continuation que $x$.
Cette notion ne d\'epend pas du choix de l'ensemble de points p\'eriodiques $\cP$.
En g\'en\'eral la continuation d'un point n'est pas unique mais si $x'_1,x'_2$ ont la m\^eme
continuation que $x$, alors $x'_1$ et $x'_2$ appartiennent \`a la m\^eme plaque centre-stable.
\medskip

Pour les diff\'eomorphismes loin des connexions fortes, chaque point poss\`ede au plus deux continuations.
En effet, si $x'_1,x'_2, x'_3$ sont des continuations de $x$,
on introduit trois suites $(q_{1,n})$, $(q_{2,n})$, $(q_{3,n})$ dans $\cP$ telles que
leurs continuations pour $g$ convergent toutes vers $x$ et leurs continuations pour $g'$ convergent
vers $x'_1,x'_2,x'_3$ respectivement.
Si $(g_t)$ est un arc de diff\'eomorphismes dans $\cU$ joignant $g$ \`a $g'$, il existe $g_t$
ayant une connexion forte pour l'un des points p\'eriodiques consid\'er\'es.
Nous allons donner un \'enonc\'e plus pr\'ecis.
\medskip

Pour chaque $g\in \cU$, nous introduisons
$\widetilde{H(p_g)}$, l'ensemble des paires $(x,\sigma)$ o\`u $x$ est un point de $H(p_g)$
et $\sigma$ est une orientation de $E^c(x)$ telle que $x$ est accumul\'e par $H(p_g)$
dans la composante de $\cW^{cs}(x)\setminus W^{ss}_{loc}(x)$ d\'etermin\'ee par $\sigma$.
C'est une partie du fibr\'e unitaire associ\'e \`a $E^c_{H(p_g)}$.
La dynamique de $g$ s'\'etend donc en une dynamique $\widetilde g$ sur $\widetilde{H(p_g)}$.
On introduit aussi la projection $\pi_g\colon \widetilde{H(p_g)} \to H(p_g)$
telle que $\pi_g(x,g)=x$.

\begin{proposition}
Consid\'erons une classe de r\'ecurrence par cha\^\i nes de $f$ qui est une classe homocline
non triviale $H(p)$ munie d'une structure partiellement hyperbolique
$T_{H(p)}M=E^s\oplus E^c\oplus E^u$ avec $\dim(E^c)=1$ et telle que $E^{cs}=E^s\oplus E^c$
est finement pi\'eg\'e. Supposons aussi que pour tout diff\'eomorphisme $g$ $C^1$-proche de $f$,
la classe $H(p_g)$ n'a pas de connexion forte.

Il existe alors un voisinage $\cU\subset \diff^1(M)$ de $f$
tel que pour tout $g,g'\in \cU$ on ait les propri\'et\'es suivantes.
\begin{itemize}
\item[--] (Rel\`evement.) L'application $\pi_g\colon \widetilde{H(p_g)} \to H(p_g)$ est surjective et
semi-conjugue $\widetilde g$ et $g$.
\item[--] (Continuation du rel\`evement.) Pour tout $\widetilde x_g=(x_g,\sigma)\in \widetilde{H(p_g)}$,
il existe un unique $\widetilde x_{g'}=(x_{g'},\sigma')\in \widetilde{H(p_{g'})}$ tel que
$\pi_{g}(\widetilde x_g)=x_g$ et $\pi_{g'}(\widetilde x_{g'})=x_{g'}$ ont la m\^eme continuation
et les orientations $\sigma$ sur $E^c(x_g)$ et $\sigma'$ sur $E^c(x_{g'})$ co\"\i ncident.
Ceci d\'efinit une bijection $$\Phi_{g,g'}\colon \widetilde{H(p_g)} \to \widetilde{H(p_{g'})}.$$
\item[--] (Continuation de la projection.) Pour tout $x_g\in H(p_g)$ et $x_{g'}\in H(p_{g'})$ ayant la
m\^eme continuation il existe $\widetilde x\in \widetilde{H(p)}$ tel que
$\pi_g(\Phi_{f,g}(\widetilde x))=x_g$ et $\pi_{g'}(\Phi_{f,g'}(\widetilde x))=x_{g'}$.
\end{itemize}
\end{proposition}


\chapter{Centralisateurs de diff\'eomorphismes}\label{c.centralisateur}
Nous utilisons les r\'esultats des chapitres pr\'ec\'edents
pour d\'ecrire le centralisateur des dif\-f\'e\-o\-mor\-phismes
de classe $C^1$ dans le groupe $\diff^1(M)$.
Dans \cite{centralisateur-conservatif,
centralisateur} nous montrons que le centralisateur
d'un diff\'eomorphisme $C^1$-g\'en\'erique $f$ est trivial, i.e.
co\"\i ncide avec le groupe engendr\'e par $f$.
Ceci r\'epond \`a une conjecture de Smale~\cite{Sm1,Sm2}.
Pour obtenir ce r\'esultat, nous devons avoir une bonne
description de la dynamique globale~:
on peut donc voir ce chapitre comme un test de notre compr\'ehension
des dynamiques $C^1$-g\'en\'eriques.
La d\'emonstration n\'ecessite de savoir perturber un diff\'eomorphisme
en pr\'eservant sa dynamique topologique mais en modifiant
la dynamique de son application tangente.
Par ailleurs, dans \cite{gioia} nous construisons un ouvert non vide $\cO$
de $\diff^1(M)$ et une partie dense de $\cO$ constitu\'ee de diff\'eomorphismes
dont le centralisateur est non d\'enombrable, donc non trivial.

\section{Diff\'eomorphismes commutant}
Si $f,g$ sont des diff\'eomorphismes qui commutent,
$g$ pr\'eserve l'ensemble des orbites de $f$
ainsi que les invariants dynamiques diff\'erentiables et topologiques
de $f$. Un argument facile de transversalit\'e
(voir~\cite[proposition 4.5]{ghys}) montre que pour
tout $p$-uple g\'en\'erique $(f_1,\dots f_p)\in (\diff^r(M))^p$ avec $p\geq 2$ et $r\geq 0$,
le groupe $\langle f_1, \ldots, f_p\rangle$ est libre.
Ceci motive la question suivante pos\'ee par Smale~\cite{Sm1,Sm2}.

Consid\'erons une vari\'et\'e connexe compacte sans bord $M$
et pour tout diff\'eomorphisme $f\in \diff^r(M)$, d\'efinissons
son \emph{centralisateur}
$$Z^r(F)=\{g\in \diff^1(M), fg=gf\}.$$
Il contient le groupe cyclique $\langle f\rangle$.
Nous disons que \emph{$f$ a un centralisateur trivial} si
$Z^r(f)=\langle f\rangle$.

\begin{question*}[Smale]
\emph{Pour $r\geq 1$, consid\'erons l'ensemble
des diff\'eomorphismes $C^r$ ayant un centralisateur trivial.}
\begin{enumerate}
\item Cet ensemble est-il dense dans $\diff^r(M)$?
\item Cet ensemble contient-il un G$_\delta$ dense de $\diff^r(M)$?
\item Cet ensemble contient-il un ouvert dense de $\diff^r(M)$?
\end{enumerate}
\end{question*}
\medskip

Ce probl\`eme remonte au travail de Kopell~\cite{kopell}
qui a donn\'e une r\'eponse compl\`ete pour $r\geq 2$
et $M=S^1$~: dans ce cas,
l'ensemble des diff\'eomorphismes \`a centralisateur trivial
contient un ouvert dense.

En dimension plus grande,
la description du centralisateur demande une compr\'ehension fine
de la dynamique globale~: puisqu'un \'el\'ement du centralisateur d'un diff\'eomorphisme peut co\"\i ncider
avec l'identit\'e sur certaines parties de $M$, nous devons tenir compte des orbites pour $f$ d'une partie dense de $M$.
Seules certaines classes de diff\'eomorphismes ont \'et\'e trait\'ees.
\begin{itemize}
\item[--] En r\'egularit\'e $C^\infty$, l'ensemble des diff\'eomorphismes \`a
centralisateur trivial contient un ouvert dense de l'ensemble des diff\'eomorphismes
poss\'edant un puits ou une source et
satisfaisant l'axiome $A$ et la condition de transversalit\'e forte~\cite{PY1}.
Sur le tore $\TT^n$, il contient un ouvert dense de l'ensemble
des diff\'eomorphismes d'Anosov~\cite{PY2} (voir aussi~\cite{fisher}).
Il contient aussi un G$_\delta$ dense d'un ouvert
de diff\'eomorphismes partiellement hyperboliques avec fibr\'e central de dimension $1$~\cite{Bu}.
\item[--] En r\'egularit\'e $C^1$, Togawa~\cite{To1, To2} a montr\'e que
l'ensemble des diff\'eomorphismes de classe $C^1$ \`a centralisateur trivial
contient un G$_\delta$ de dense de l'ensemble des diff\'eomorphismes
satisfaisant l'axiome $A$~; en particulier,
les diff\'eomorphismes $C^1$-g\'en\'eriques du cercle ont un centralisateur
trivial~\cite{To1, To2}.
\end{itemize}
\medskip

Nous avons obtenu une r\'eponse compl\`ete
\`a la question de Smale en classe $C^1$.

\begin{theoreme}[Bonatti-Crovisier-Wilkinson~\cite{centralisateur}]\label{t.centralisateur}
Il existe dans $\diff^1(M)$ un G$_\delta$ dense 
form\'e de diff\'eomorphismes $f$ ayant un centralisateur $Z^1(f)$
trivial.
\end{theoreme}

\begin{theoreme}[Bonatti-Crovisier-Vago-Wilkinson~\cite{gioia}]
\label{t.centralisateur-large}
Il existe un ouvert non vide $\cU$ de $\diff^1(M)$
et une partie dense $\cD\subset \cU$ form\'ee
de diff\'eomorphismes $f$ ayant un centralisateur $Z^1(f)$
non d\'enombrable. En particulier l'ensemble des diff\'eomorphismes
\`a centralisateur trivial ne contient pas un ouvert dense de $\diff^1(M)$.
\end{theoreme}

Dans le cas conservatif (lorsque qu'une forme volume ou une
forme symplectique est pr\'e\-ser\-v\'ee), l'analogue
du th\'eor\`eme~\ref{t.centralisateur}
reste vrai~\cite{centralisateur-conservatif}.
L'analogue du th\'eor\`eme~\ref{t.centralisateur-large}
reste vrai dans le cadre symplectique mais nous n'avons
pas de construction pour l'espace des diff\'eomorphismes
pr\'eservant une forme volume.

\section{Strat\'egie pour montrer que le centralisateur est trivial}
\subsection{Centralisateur localement trivial}\label{ss.local}
Bien souvent, la premi\`ere \'etape pour montrer
que le centralisateur d'un diff\'eomorphisme est trivial consiste
\`a montrer qu'il est ``localement trivial''.

\begin{definition}\label{d.local}
Le centralisateur $Z^1(f)$ d'un diff\'eomorphisme $f$
est \emph{localement trivial} si pour tout $g\in Z^1(f)$,
il existe un ouvert dense $U\subset M$ et une application
localement constante $\alpha\colon U\mapsto \ZZ$ telle que
$g(x)=f^{\alpha(x)}(x)$ pour tout $x\in U$.
\end{definition}

Dans~\cite{kopell}, le contr\^ole de la distorsion
de la d\'eriv\'ee des applications $C^2$ est un des ingr\'edients
principaux pour montrer que le centralisateur est trivial.
Dans le cadre des diff\'eomorphismes $C^1$ g\'en\'eriques,
nous devons changer d'argument puisque la distorsion n'est pas born\'ee~:
curieusement, c'est cette fois le fait que la distorsion
explose qui nous permet d'amorcer la d\'emonstration.

\paragraph{Propri\'et\'e (MD$^{M\setminus \Omega}$).}
Un diff\'eomorphisme $f$ satisfait la \emph{propri\'et\'e de
mauvaise distorsion sur l'ensemble errant},
s'il existe une partie dense $\cD\subset M\setminus \Omega(f)$ telle
que pour tout $K>0$ et tous $x\in \cD$, $y\in M\setminus \Omega(f)$
appartenant \`a des orbites diff\'erentes, il existe
$n\geq 1$ tel que
$$|\log |\det Df^n(x)|-\log|\det Df^n(y)||>K.$$

\paragraph{Propri\'et\'e (MD$^\text{s}$).}
Un diff\'eomorphisme $f$ satisfait la \emph{propri\'et\'e de mauvaise
distorsion sur les vari\'et\'es stables} (MD$^\text{s}$)
si, pour toute orbite p\'eriodique hyperbolique $O$, il existe
une partie dense $\cD\subset W^s(O)$
(pour la topologie intrins\`eque de $W^s(O)$)
telle que pour tout $K>0$ et tous $x\in \cD$, $y\in W^s(O)$
appartenant \`a des orbites diff\'erentes, il existe
$n\geq 1$ tel que
$$|\log |\det Df_{|W^s(\cO)}^n(x)|-\log|\det Df_{|W^s(\cO)}^n(y)||>K.$$
\medskip

D'apr\`es les th\'eor\`emes~\ref{t.kupka-smale} et~\ref{t.franks}
ainsi que les r\'esultats de la section~\ref{s.consequence}, les propri\'et\'e suivantes sont $C^1$-g\'en\'eriques~:
\begin{enumerate}
\item\label{g1} toutes les orbites p\'eriodiques sont hyperboliques~;
\item\label{g2} des orbites p\'eriodiques distinctes ont des valeurs propres distinctes~;
\item\label{g3} pour toute composante connexe $U$ appartenant \`a l'int\'erieur de
l'ensemble non-errant $\Omega$, il existe une orbite p\'eriodique
$O\subset \interior(\Omega)$ telle que $U$ est contenue dans l'adh\'erence
de $W^s(O)$.
\end{enumerate}

\begin{consequence*}
Tout diff\'eomorphisme $f$ ayant les propri\'et\'es
$C^1$-g\'en\'eriques~\ref{g1}, \ref{g2}, \ref{g3} ci-dessus et
qui satisfait les propri\'et\'es (MD$^{M\setminus \Omega}$) et (MD$^\text{s}$)
a un centralisateur localement trivial.
\end{consequence*}
\begin{proof}
Les propri\'et\'es (MD) permettent d'``individualiser'' les orbites~:
si $f$ satisfait la propri\'et\'e (MD$^{M\setminus \Omega}$),
alors tout diff\'eomorphisme $g\in Z^1(f)$ doit pr\'eserver
chaque orbite de $\cD$. Ceci permet de construire
la fonction $\alpha$ de la d\'efinition~\ref{d.local}
sur $\cD$. Sur l'ensemble errant, la fonction $\alpha$
doit \^etre localement constante. Par cons\'equent,
$\alpha$ s'\'etend contin\^ument sur tout $M\setminus \Omega(f)$
(en particulier, elle est constante sur chaque composante de l'ensemble errant).

Si $f$ satisfait la propri\'et\'e (MD$^\text{s}$),
le m\^eme argument s'applique \`a la vari\'et\'e stable de toute
orbite p\'eriodique hyperbolique pr\'eserv\'ee par $g\in Z^1(f)$~;
en particulier $Dg$ et $Df^\alpha$ co\"\i ncident aux points de l'orbite $O$
et $\alpha$ est constante sur toute la vari\'et\'e stable.
Puisque les orbites p\'eriodiques de $f$ ont des valeurs propres distinctes,
$g$ fixe individuellement chacune d'entre elles et l'on peut donc appliquer
la propri\'et\'e (MD$^\text{s}$).
Pour toute composante connexe $U$
de $\interior(\Omega(f))$, il existe une orbite p\'eriodique
$O\subset \interior(\Omega(f))$ telle que $U$ est contenue dans l'adh\'erence
de $W^s(O)$. On en d\'eduit que $\alpha$ s'\'etend contin\^ument en une fonction
localement constante sur $\interior(\Omega(f))$.
\end{proof}
\medskip

La premi\`ere \'etape vers le th\'eor\`eme~\ref{t.centralisateur}
consiste donc \`a montrer le r\'esultat suivant.

\begin{theoreme}\label{t.MD}
Il existe un G$_\delta$ dense de $\diff^1(M)$
form\'e de diff\'eomorphismes qui satisfont les propri\'et\'es
(MD$^{M\setminus \Omega}$) et (MD$^\text{s}$).
\end{theoreme}

La propri\'et\'e (MD$^\text{s}$) peut s'obtenir comme variante
des arguments de~\cite{To1,To2}. La propri\'et\'e (MD$^{M\setminus \Omega}$)
est beaucoup plus d\'elicate \`a obtenir. En effet, les points d'une m\^eme
vari\'et\'e stable d'orbite p\'eriodique ont la m\^eme dynamique future, proche
d'une application lin\'eaire. Au contraire dans l'ensemble errant
les orbites distinctes peuvent avoir des comportements tr\`es diff\'erents.
La d\'emonstration de ce r\'esultat utilise la notion de perturbation fant\^ome
(section~\ref{s.fantome} plus bas) mais nous ne l'expliquons pas ici.

\subsection{Du local au global}

Nous introduisons une nouvelle propri\'et\'e.

\paragraph{Propri\'et\'e (GD).}
Un diff\'eomorphisme $f$ satisfait la \emph{propri\'et\'e de grande dilatation}
sur un ensemble $X\subset M$ si, pour tout $K>0$, il existe $n(K)\geq 1$ et pour tous
$x\in X$ et $n\geq n(K)$, il existe $j\in \ZZ$ tels que
$$\sup\{\|Df^n(f^j(x))\|,\|Df^{-n}(f^{j+n}(x))\|\}>K.$$
\medskip

\begin{consequence*}
Soit $f$ un diff\'eomorphisme d'une vari\'et\'e connexe $M$
de dimension $d\geq 2$ et $\per(f)$ l'ensemble de ses points p\'eriodiques.
Supposons que $f$ ait un centralisateur
localement trivial, ait toutes ses orbites p\'eriodiques hyperboliques et
satisfait la propri\'et\'e (GD) sur $M\setminus \per(f)$.
Alors $f$ a un centralisateur trivial.
\end{consequence*}
\begin{proof}
Consid\'erons $g$ dans le centralisateur de $f$
et $K>\max_{x\in M} (\|D_xg\|,\|D_xg^{-1}\|)$.
Puisque le centralisateur de $f$ est localement trivial,
on a $g=f^\alpha$ o\`u $\alpha\colon M\to \ZZ$
est une fonction localement constante sur un ouvert dense de $M$. La propri\'et\'e (GD) permet de
borner $\alpha$ par $n(K)$.
Puisque les orbites p\'eriodiques de $f$ sont hyperboliques,
l'ensemble des points p\'eriodiques $\cP_0$ de p\'eriode inf\'erieure
\`a $2n(K)$ est fini. Consid\'erons \`a pr\'esent une collection d'ouverts
$\{U_i\}_{|i|\leq n(K)}$ d'union dense dans $M$
et telle que $g=f^i$ sur $U_i$. Sur $M\setminus \cP_0$, les ensembles
$\overline {U_i}$ sont deux \`a deux disjoints. Puisque $M\setminus \cP_0$ est connexe,
seul un des ensembles $U_i$ est non vide. Ceci montre que $g=f^i$
pour un certain entier $i$. Le centralisateur de $f$ est donc trivial.
\end{proof}
\medskip

Pour obtenir le th\'eor\`eme~\ref{t.centralisateur},
nous montrons l'\'enonc\'e suivant.

\begin{theoreme}\label{t.GD}
Soit $f$ un diff\'eomorphisme dont les points p\'eriodiques sont hyperboliques.
Il existe $g\in \diff^1(M)$ arbitrairement proche de $f$ tel que la propri\'et\'e
(GD) soit satisfaite sur $M\setminus \per(f)$.

De plus,
\begin{itemize}
\item[--] $f$ et $g$ sont conjugu\'e par un hom\'eomorphisme $\Phi$ de $M$,~;
\item[--] pour tout orbite p\'eriodique $O$ de $f$, les d\'eriv\'ees de $f$
le long de $O$ et de $g$ le long de $\Phi(O)$ sont conjugu\'ees~;
\item[--] si $f$ satisfait les propri\'et\'es (MD$^{M\setminus \Omega}$) et (MD$^\text{s}$), alors $g$ les satisfait \'egalement.
\end{itemize}
\end{theoreme}

\subsection{Densit\'e et g\'en\'ericit\'e}
Nous pouvons d\'eduire des sections pr\'ec\'edentes
qu'il existe un ensemble dense de dif\-f\'e\-o\-mor\-phis\-mes de $\diff^1(M)$
ayant un centralisateur trivial.
En effet, on peut perturber tout diff\'eomorphisme $f\in \diff^1(M)$
pour que les propri\'et\'es g\'en\'eriques~\ref{g1}, \ref{g2} et \ref{g3}
de la section~\ref{ss.local} ainsi que (MD$^{M\setminus \Omega}$) et (MD$^\text{s}$) soient satisfaites. On perturbe ensuite $f$ en un diff\'eomorphisme
$g$ donn\'e par le th\'eor\`eme~\ref{t.GD} qui satisfait la propri\'et\'e (GD).
Les propri\'et\'es ~\ref{g1}, \ref{g2}, (MD$^{M\setminus \Omega}$) et (MD$^\text{s}$) sont pr\'eserv\'ees. Puisque $f$ et $g$ sont conjugu\'ees,
la propri\'et\'e~\ref{g3} l'est \'egalement. Ceci implique que
le centralisateur de $g$ est trivial.

Nous ne pouvons pas obtenir la g\'en\'ericit\'e des dif\-f\'e\-o\-mor\-phis\-mes
\`a centralisateur trivial de cette fa\c con puisque la propri\'et\'e
$GD$ n'est pas v\'erifi\'ee sur un G$_\delta$ dense.
Pour conclure nous faisons un argument de g\'en\'ericit\'e~:
pour des raisons de compacit\'e, nous travaillons avec l'espace des
hom\'eomorphismes bilipschitziens $\lip(M)$.
Les raisonnements des sections pr\'ec\'edentes permettent de montrer que
l'ensemble des diff\'eomorphismes $f$ ayant un centralisateur trivial
dans $\lip(M)$ forme une partie dense de $\diff^1(M)$.
En utilisant la proposition~\ref{p.baire}, nous en d\'eduisons qu'il contient
un G$_\delta$ dense, ce qui conclut la d\'emonstration du th\'eor\`eme~\ref{t.centralisateur}.

\section{Perturbations fant\^omes}\label{s.fantome}
Nous pr\'esentons \`a pr\'esent quelques ingr\'edients importants pour
la d\'emonstration du th\'e\-o\-r\`e\-me~\ref{t.GD}.
Certains de ces arguments sont aussi utilis\'es pour obtenir le
th\'eor\`eme~\ref{t.MD}.

Le th\'eor\`eme~\ref{t.GD} permet de perturber un
diff\'eomorphisme $f$ pour changer la dynamique de son application
tangente sans changer la classe de conjugaison topologique de $f$.
Pour cela, nous consid\'erons des domaines ouverts $U\subset M$ et 
un entier $n\geq 1$ tels que
les it\'er\'es $\overline U, f(\overline U),\dots, f^{n}(\overline U)$
soient deux-\`a-deux disjoints.
La perturbation $g_1$ de $f$ est conjugu\'ee \`a $f$ par un diff\'eomorphisme
$\phi_1$ qui co\"\i ncide avec l'identit\'e hors de l'union des $f^i(\overline U)$,
$0\leq i\leq n$. Pour la construire, nous modifions $f$ sur les ouverts
$U,f(U),\dots, f^{j-1}(U)$ pour un certain entier $1\leq j<n$
et nous utilisons les domaines restants $f^j(U),\dots,f^{n-1}(U)$
pour compenser la perturbation de sorte que $f^n$ et $g^n_1$
co\"\i ncident sur $U$. Une telle modification est appel\'ee
\emph{perturbation fant\^ome} de $f$.\index{perturbation fant\^ome}

Pour modifier les propri\'et\'es de l'application tangente $Df$, il est commode
de pouvoir supposer que les composantes connexes de chaque ouvert $f^i(U)$,
$0\leq i\leq n$ sont petites
de sorte que $f$ y soit proche d'une application lin\'eaire.
Par ailleurs, afin de modifier la dynamique le long de l'orbite de tout point non p\'eriodique,
nous souhaitons que $U$ contienne un it\'er\'e de chaque point $x\in M\setminus \per(f)$.
C'est exactement ce que permet le th\'eor\`eme~\ref{t.tour}
d'existence de tours topologiques.

Apr\`es avoir effectu\'e une telle perturbation fant\^ome, les
dynamiques des applications tangentes $Df$ et $Dg_1$
restent conjugu\'ees par $D\varphi_1$.
Pour modifier effectivement la dynamique, nous introduisons une suite
de perturbations fant\^omes $(g_i)$ convergeant dans
$\diff^1(M)$ vers un diff\'eomorphisme $g$.
On l'obtient en conjuguant
successivement par une suite de diff\'eomorphismes $(\varphi_i)$.
En choisissant la taille des composantes connexes du support
de chaque application $\varphi_i$ suffisamment petite,
nous garantissons que la suite $(\varphi_i)$ converge vers un hom\'eomorphisme
$\varphi$ de $M$ qui conjugue $f$ et $g$.

Cette technique de construction de perturbation comme limite de diff\'eomorphismes con\-ju\-gu\'ees
a d\'ej\`a \'et\'e utilis\'ee en dynamiques par Anosov et Katok~\cite{anosov-katok} pour des dynamiques lisses ou par Rees (voir~\cite{rees}) en r\'egularit\'e $C^0$.

\section{Probl\`emes}

Nous n'avons pas obtenu d'analogue au th\'eor\`eme~\ref{t.centralisateur-large}
dans le cadre conservatif.

\begin{question}
\emph{Consid\'erons une forme volume $v$ sur $M$.}\\
Existe-t-il un ouvert non vide $\cU$ de l'espace $\diff^1_v(M)$
des diff\'eomorphismes pr\'eservant $v$ et une partie dense $\cD$ dont les \'el\'ements
ont un centralisateur dans $\diff^1_v(M)$ non trivial~?
\end{question}
\medskip

Les th\'eor\`emes~\ref{t.centralisateur}
et~\ref{t.centralisateur-large} sugg\`erent que la topologie de l'ensemble des
diff\'eomorphismes \`a centralisateur trivial doit \^etre compliqu\'ee
et motive les questions suivantes.

\begin{question}
\begin{enumerate}
\item \emph{Consid\'erons l'ensemble des diff\'eomorphismes
ayant un centralisateur trivial.}
Quel est son int\'erieur~?\\
\emph{\cite{PY1} implique que pour $r\geq 2$, l'int\'erieur est non vide.
Pour $M=S^1$ et $r=1$, \cite{gioia} implique que l'int\'erieur est vide.}
\item Est-ce un ensemble bor\'elien~?\\
\emph{(Voir~\cite{FRW} pour une r\'eponse n\'egative dans le cadre mesurable.)}
\item \emph{L'ensemble $\{(f,g)\in (\diff^1(M))^2, fg=gf\}$
est ferm\'e.} Quelle est sa topologie locale~? Est-il localement connexe~?
\end{enumerate}
\end{question}
\medskip

Pour d\'emontrer le th\'eor\`eme~\ref{t.centralisateur}, nous
exhibons des propri\'et\'es dynamiques qui impliquent que le centralisateur est trivial.
Ceci illustre l'existence de liens entre la dynamique d'un dif\-f\'e\-o\-mor\-phis\-me
et ses propri\'et\'es alg\'ebriques dans le groupe $\diff^1(M)$.
En voici un autre exemple~: la relation de Baumslag-Solitar
$gfg^{-1}=f^n$, o\`u $n>1$ est un entier, implique que les points p\'eriodiques de $f$
ne sont pas hyperboliques.

\begin{question}
\begin{enumerate}
\item Consid\'erons un groupe finiement engendr\'e
$G=\langle a_1,\dots,a_k | r_1,\dots,r_m\rangle$
o\`u $\{a_i\}$ est une partie g\'en\'eratrice et les $r_i$ sont des relations.
Quelle est la taille de l'ensemble des diff\'eomorphismes
$f\in \diff^r(M)$ tels qu'il existe un morphisme injectif $\rho\colon G\to \diff^r(M)$
satisfaisant $\rho(a_1)=f$~?
\emph{Le th\'eor\`eme~\ref{t.centralisateur} implique que cet ensemble
est maigre lorsque $G$ est ab\'elien
(ou m\^eme nilpotent) et diff\'erent de $\ZZ $.}
\item Existe-t-il un diff\'eomorphisme $f\in \diff^r(M)$ tel que
pour tout $g\in \diff^r(M)$ le groupe engendr\'e par $f$ et $g$
est $\langle f\rangle$ lorsque $g\in \langle f\rangle$
ou le produit libre $\langle f\rangle*\langle g\rangle$~?
\end{enumerate}
\end{question}

\printindex
\end{document}